\documentclass{article}
\usepackage[english]{babel}
\usepackage{graphicx}
\usepackage{slashed}
\usepackage{stmaryrd}
\usepackage{calligra}
\usepackage{mathtools}
\usepackage{amsmath,amsfonts,amssymb}

\usepackage{ntheorem}
\usepackage{dsfont}
\usepackage{verbatim}
\usepackage{float}
\usepackage{accents}
\usepackage{mathrsfs}
\usepackage{cancel}
\usepackage{amsmath}
\usepackage{amssymb}
\usepackage{hyperref}
\usepackage{tikz}
\usepackage{pgfplots}
\usepackage{mathtools}
\let\savering\ring
\let\ring\relax
\usepackage{mathabx}
\let\ring\savering
\pgfplotsset{compat=1.14}
\usetikzlibrary{patterns}
\usetikzlibrary{positioning,arrows,arrows.meta}
\usepackage{geometry}
\usepackage{lipsum} 
\usepackage{titlesec}

\titleformat{\subsection}[runin]
       {\normalfont\bfseries}
       {\thesubsection}
       {0.5em}
       {}
       [.]
       
\titleformat{\subsubsection}[runin]
       {\normalfont\bfseries}
       {\thesubsubsection}
       {0.5em}
       {}
       [.]

 \DeclareFontFamily{U}{dutchcal}{\skewchar \font =45}
  \DeclareFontShape{U}{dutchcal}{m}{n}{
    <-> dutchcal-r}{}
  \DeclareFontShape{U}{dutchcal}{b}{n}{
    <-> dutchcal-b}{}
  \DeclareMathAlphabet{\mdutchcal}{U}{dutchcal}{m}{n}
  \SetMathAlphabet{\mdutchcal}{bold}{U}{dutchcal}{b}{n}
  \DeclareMathAlphabet{\mdutchbcal} {U}{dutchcal}{b}{n}

\usepackage{enumitem}
\theoremseparator{.} 
\newtheorem{Th}{Theorem}[section]
\newtheorem{Def}[Th]{Definition}
\newtheorem{Rq}[Th]{Remark}
\newtheorem{Pro}[Th]{Proposition}
\newtheorem{Cor}[Th]{Corollary}
\newtheorem{Lem}[Th]{Lemma}
\theoremstyle{empty}
\newtheorem{refproof}{Proof}

\newcommand{\R}{\mathbb{R}}
\newcommand{\C}{\mathscr{C}}
\newcommand{\T}{\mathbf{T}}
\newcommand{\X}{\mathbf{X}}
\newcommand{\XX}{\mathrm{X}}
\newcommand{\V}{\mathbf{V}}
\newcommand{\Z}{\mathbf{Z}}
\newcommand{\K}{\widehat{\mathbb{P}}_S}
\newcommand{\dr}{\mathrm{d}}

\newenvironment{proof}{\noindent\textit{Proof.~}}{\hfill$\Box$\bigbreak} 

\title{Scattering map for the Vlasov-Maxwell system around source-free electromagnetic fields}

\author{L\'eo Bigorgne\footnote{Univ Rennes, CNRS, IRMAR - UMR 6625, F-35000 Rennes, France.
{\em E-mail address:} {\tt leo.bigorgne@univ-rennes.fr}}}
\date{}
\begin{document}

\maketitle
    
\begin{abstract}
We construct an isometric modified scattering operator, mapping any sufficiently regular past scattering state, with a small distribution function, to the future one corresponding to forward evolution by the Vlasov-Maxwell system.

The main part of this work is devoted to the construction of a modified wave operator, which relates the future asymptotic dynamics of the solutions to their initial data. Then, by applying our previous results on modified scattering \cite{scat}, we are able to construct a small data scattering map for the Vlasov-Maxwell system. Our analysis relies in particular on the study of the  \textit{asymptotic Maxwell equations}. Its solution captures not only the large time behavior of the electromagnetic field of the plasma near future timelike infinity as well as future null infinity, but also the one of its weighted derivatives. 

The wave operator provides as well a class of global solutions to the Vlasov-Maxwell system which are large in $L^p_{x,v} \times L^2_x$ but initially dispersed. 
\end{abstract}

    \tableofcontents

\section{Introduction}

This paper is concerned with the Vlasov-Maxwell system on $\R_t \times \R^3_x \times \R^3_v$ in the perturbative regime. This set of equations are used to describe the evolution of collisionless plasma and can be written as
\begin{alignat}{2}
& \partial_t f+\widehat{v} \cdot \nabla_x f + ( E+\widehat{v} \times B) && \cdot \nabla_v f = 0, \label{VM1}  \\
&  \nabla_x \cdot E = \int_{\R^3_v}f \mathrm{d}v,  && \partial_t E = \nabla_x \times B -\int_{\R^3_v} \widehat{v}f \mathrm{d}v, \label{VM2}  \\
&  \nabla_x \cdot B = 0,   && \partial_t B = - \nabla_x \times E, \label{VM3}
\end{alignat} 
where
\begin{itemize}
\item $f:\R_{ t} \times \R^3_x \times \R^3_v \rightarrow \R_+$ is the density distribution function of the particles. 
\item $\widehat{v}=\frac{v}{v^0}$, where $v^0=\langle v \rangle :=\sqrt{1+|v|^2}$, is the relativistic speed of a particle of momentum $v \in \R^3_v$.
\item $\int_{\R^3_v} f \mathrm{d}v$ and $\int_{\R^3_v} \widehat{v}f \mathrm{d}v$ are respectively the total charge density and the total current density.
\item $E,B:\R_{ t} \times \R^3_x \rightarrow \R^3$ are respectively the electric and the magnetic field.
\end{itemize}
Although we consider, as it is usually done, plasmas composed by one species of particles of charge $q=1$ and mass $m=1$, all the results of this paper can be extended without any additional difficulty to the multispecies case, with particles of different charges and strictly positive masses\footnote{In contrast, in view of the analysis performed in \cite{massless}, linear scattering is expected to hold for massless particles.}. For convenience, we set in this article the speed of light $c$ to one, $c=1$, but in certain context it is important to keep it as a parameter in \eqref{VM1}--\eqref{VM3}. For instance, it is well known from \cite{Asano,DegondVMVP,SchaefferVMVP,BriHan} that in the non-relativistic limit $c \to +\infty$, the solutions to the Vlasov-Maxwell equations converge to the ones of the repulsive Vlasov-Poisson system
\begin{equation}\label{eq:VP}
\partial_t f+v \cdot \nabla_x f+ \nabla_x \phi \cdot \nabla_v f =0, \qquad \qquad \Delta \phi = \int_{\R^3_v} f \mathrm{d} v .
\end{equation}
A detailed introduction to kinetic equations, including these two systems, can be found in \cite{Glassey}. 

\subsection{Local and global well-posedness for the Vlasov-Maxwell system} An initial data set $(f_0,E_0,B_0)$ for \eqref{VM1}--\eqref{VM3} is composed by a function $f_0:\R^3_x \times \R^3_v \rightarrow \R_+$ and two fields $E_0,B_0:\R^3_x \rightarrow \R^3$ which satisfy the constraint equations 
\begin{equation}\label{eq:constraint}
\nabla_x \cdot E_0=\int_{\R^3_v} f_0 \, \mathrm{d}v, \qquad \qquad \nabla_x \cdot B_0=0.
\end{equation} The local existence and uniqueness problem for the solutions to the Vlasov-Maxwell system has been adressed by Wollman and Glassey-Strauss \cite{Woll,GlStrauss}. While global weak solutions were constructed by DiPerna-Lions \cite{DPLions} (see also \cite{Reinweak}), the existence of global in time classical solutions to the Vlasov-Maxwell system is only known in the pertubative regime \cite{GSt,GSc,Rein2,Sc,dim3,Wang,WeiYang,scat} or under certain symmetry assumptions \cite{GSc90,GSc97,Gsc98,LukStrain,XWang}. Although the problem remains open besides these specific cases, various continuation criteria have been obtained during the past decades \cite{GlStrauss,GlasseyStrauss0,GlStracrit,KlSta,BGP2,Pallard1,SAI,LukStrain2,Kunze,Pallard2,Patel}.

\subsection{Asymptotic stability of vacuum for the Vlasov-Maxwell system} The solutions arising from sufficiently small and regular initial data are known, in addition to be global in time, to decay with a rate corresponding to the solutions of the linearised equations around $0$,
\begin{alignat}{2}
&  \partial_t f^{\mathrm{lin}}+\widehat{v} \cdot \nabla_x f^{\mathrm{lin}}= 0, && \label{VM1lin}  \\
&  \nabla_x \cdot E^{\mathrm{lin}} = \int_{\R^3_v}f^{\mathrm{lin}} \mathrm{d}v, \qquad  && \qquad \partial_t E^{\mathrm{lin}} = \nabla_x \times B^{\mathrm{lin}} -\int_{\R^3_v} \widehat{v}f^{\mathrm{lin}} \mathrm{d}v, \label{VM2lin}  \\
&  \nabla_x \cdot B^{\mathrm{lin}} = 0,  \qquad   && \qquad \partial_t B^{\mathrm{lin}} = - \nabla_x \times E^{\mathrm{lin}}. \label{VM3lin}
\end{alignat} 
Thus, for any such solution $(f,E,B)$ to \eqref{VM1}--\eqref{VM3},
\begin{equation}\label{eq:intro}
\hspace{-10mm} \forall \, (t,x) \in \R_+ \times \R^3, \qquad \int_{\R^3_v} f(t,x,v) \mathrm{d} v \, \lesssim t^{-3}, \qquad |E|(t,x)+|B|(t,x) \lesssim \langle t+|x| \rangle^{-1} \langle t-|x| \rangle^{-1}.
\end{equation}
 This result was first obtained for compactly supported initial data by Glassey-Strauss \cite{GSt}, where estimates on the first order derivatives of the electromagnetic field were derived as well. Soon after, in the multi-species case, the smallness assumptions on the individual density distribution functions has been relaxed \cite{GSc}. Later, Schaeffer \cite{Sc} removed the support restriction on the velocity variable. However, his method lead to a loss on the estimate of $\int_v f \mathrm{d}v$. 

In recent years, all the compact support assumptions on the initial data were removed independently by \cite{dim3,Wang} using robust approaches. More precisely, they rely on vector field methods as well as, for the latter, Fourier analysis. It allowed them to obtain (almost) optimal pointwise decay estimates, similar to \eqref{eq:intro}, on the solutions and their higher order derivatives. Furthermore, in \cite{dim3}, the initial decay assumption in $v$ is optimal and improved estimates on certain \textit{null} components of the electromagnetic field are derived. Thereafter, Wei-Yang proved a global existence result allowing for large initial Maxwell fields \cite{WeiYang}. Thus, their result implies the asymptotic stability of electromagnetic fields in vacuum, of regularity $C^2$ and decaying fast enough, for the Vlasov-Maxwell system. Their work, based on the framework of Glassey-Strauss, does not require any compact support restriction on the data and provides the optimal decay rates \eqref{eq:intro}.

In \cite{scat}, we provided a shorter proof of the main results of \cite{dim3,Wang} and we allowed, as \cite{WeiYang}, the electromagnetic field to be large. Moreover, we investigate further the asymptotic dynamics of the solutions. Once the optimal decay estimates \eqref{eq:intro} are proved, the next question one may ask is whether or not $f$ and $(E,B)$ can be approached by a linear solution. 

\subsection{Modified scattering for the small data solutions to the Vlasov-Maxwell system} The purpose of scattering theory is to compare the dynamics of a perturbed system, here the Vlasov-Maxwell equations, with a free dynamics, which is meant to be much simpler. At first glance, we can consider the free dynamics of the linear Vlasov equation and the homogeneous Maxwell equations. However, we will see that in order to achieve the three requirements of a satisfactory scattering theory, a more involved free dynamics such as the one of the linearised system \eqref{VM1lin}--\eqref{VM3lin} must be considered. We refer to \cite{ReedSimon} for a general introduction on this subject. 

Ideally, one would like the asymptotic behavior of the solutions to be captured by past and future \textit{scattering states}. For the free relativistic transport equation \eqref{VM1lin}, the distribution function $f^{\mathrm{lin}}$ is constant along the characteristics, which are \textit{timelike straight lines}, the trajectories of isolated massive bodies. More precisely,
$$ \exists \, f_\infty : \R^3_x \times \R^3_v \to \R, \; \; \forall \, (t,x,v) \in \R_t \times \R^3_x \times \R^3_v, \qquad \qquad f^{\mathrm{lin}}(t,x+t\widehat{v},v)=f_\infty(x,v),$$
so that the asymptotic dynamics of $f^{\mathrm{lin}}$ can be fully described by the knowledge of $f_\infty$. In that case, the (future) scattering state $f_\infty$ turns out to also correspond to the initial data $f^{\mathrm{lin}}(0,\cdot , \cdot)$. Consider an electromagnetic field $(E^{\mathrm{hom}},B^{\mathrm{hom}})$ in vacuum, that is a solution to the homogeneous Maxwell equations
$$ \nabla_x \cdot E^{\mathrm{hom}}=0, \qquad  \nabla_x \cdot B^{\mathrm{hom}}=0, \qquad \partial_t E^{\mathrm{hom}} = \nabla_x \times B^{\mathrm{hom}}, \qquad \partial_t B^{\mathrm{hom}} = -\nabla_x \times E^{\mathrm{hom}}.$$ 
Then, under suitable assumptions on the data, $(rE^{\mathrm{hom}},rB^{\mathrm{hom}})$ converges along \textit{null rays}, the trajectories of photons. There exists $E_\infty, \, B_\infty : \R_u \times \mathbb{S}^2_\omega \to \R^3$, the radiation fields of $E^{\mathrm{hom}}$ and $B^{\mathrm{hom}}$ along \textit{future null infinity} $\mathcal{I}^+ \! \simeq \R_u \! \times \mathbb{S}^2_\omega$, such that
\begin{equation}\label{eq:introscatEB}
\hspace{-5mm} \forall \, (u,\omega) \in \R_u \! \times \mathbb{S}^2_\omega,   \; \; \qquad \lim_{r \to+\infty} rE^{\mathrm{hom}}(r+u,r\omega)=E_\infty(u,\omega), \quad \;  \lim_{r \to+\infty} rB^{\mathrm{hom}}(r+u,r\omega)=B_\infty(u,\omega).
\end{equation}
Moreover, the future scattering state $(E_\infty , B_\infty)$ satisfies constraint equations on $\mathcal{I}^+$ inherited from the Maxwell equations. By denoting abusively by $\omega$ the function $\omega \mapsto \omega \in \mathbb{S}^2 $, they can be written as
\begin{equation}\label{eq:Maxnull}
 \omega \cdot E_\infty  = 0, \qquad  \omega \cdot B_\infty   = 0, \qquad E_\infty+ \omega \times B_\infty =0, \qquad B_\infty- \omega \times E_\infty =0 .
\end{equation} 
Conversally, we proved in \cite[Theorem~$7.6$]{scat} that for any $L^2_{u,\omega} $ scattering state $(E_\infty,B_\infty)$ satisfying the constraint equations \eqref{eq:Maxnull}, there exists a unique $L^\infty(\R_+,L^2_x)$ solution $(E^{\mathrm{hom}},B^{\mathrm{hom}})$ to the vacuum Maxwell equations such that \eqref{eq:introscatEB} holds, at least in a weak sense.

Let us now review the scattering results that we derived in \cite{scat} for the solutions $(f,E,B)$ to the Vlasov-Maxwell system arising from sufficiently regular data and a small initial distribution function.

\begin{itemize}
\item For the Vlasov field, a \textit{linear scattering} statement would be
\begin{equation}\label{eq:linscatt}
\exists \, f_\infty \in L^1 (\R^3_x \times \R^3_v), \qquad \lim_{t \to + \infty} \|f(t,x+t\widehat{v},v)-f_\infty(x,v)\|_{L^1_{x,v}} =0 .
\end{equation}
In other words, $f$ would converge along the linear characteristics and would approach, for large times, the solution to the free relativistic transport equation \eqref{VM1lin} with scattering data $f_\infty$.
However, because of the long-range effects of the Lorentz force, \eqref{eq:linscatt} holds if and only if\footnote{In the multispecies case or if we allow the distribution function $f$ to take negative values, linear scattering occurs for solutions arising from nontrivial initial data.} $f(0,\cdot , \cdot)=0$. Instead, we could expect a weaker quantity, which contains less informations, to have a linear behavior for large times. It turns out that the spatial average of $f$, which is conserved in time for the solutions to \eqref{VM1lin}, converges,
$$ \exists \, Q_\infty \in L^{\infty}(\R^3_v), \quad \forall \, v \in \R^3_v, \qquad \qquad \lim_{t \to +\infty} \int_{\R^3_x}   f(t,x,v) \dr v = Q_\infty(v).$$
It allowed us to determine, as in Corollary \ref{Corlinbound} below, the leading order contribution of the source terms in the Maxwell equations \eqref{VM2}--\eqref{VM3}, which are the charge and the current density. In particular, once multiplied by their decay rate $t^3$, these two quantities have a self-similar asymptotic profile, as the corresponding linear ones. This enabled us to capture the large time behavior of the electromagnetic field along linear approximations of the trajectories of the massive particles,
$$\exists \, \mathrm{Lor}:\R^3_v \to \R^3, \quad \forall \, (x,v) \in \R^3_x \times \R^3_v, \qquad \lim_{t \to + \infty} t^2 \big( E(t,x+t\widehat{v})+\widehat{v} \times B(t,x+t\widehat{v}) \big) = \mathrm{Lor}(v).$$
It is worth noting that, as we prove in Section \ref{SecMaxasymp}, the long range Lorentz force $t^{-2} \, \mathrm{Lor}$ arises from a solution to the (asymptotic) Maxwell equations. Then, we obtained from this property that the density function $f$ verifies a \textit{modified scattering} dynamics.  Heuristically, the characteristics of the Vlasov equation satisfy, for $t \gg 1$,
$$ \dot{X} =\widehat{V}, \qquad \qquad \dot{V} \approx t^{-2} \, \mathrm{Lor}(V),$$
so that we expect $V(t) \to v$ as $t \to +\infty$. It yields to the approximations 
$$V(t) \approx v-\frac{1}{t}\mathrm{Lor}(v), \qquad \dot{X}(t) \approx \widehat{v}+\partial_t \C_{t,v} +O \big(t^{-2} \big), \qquad \C^k_{t,v}:=\frac{\log(t)}{v^0}\left(\widehat{v}\cdot \mathrm{Lor} (v) \widehat{v}^k -\mathrm{Lor}^k(v) \right),$$
and then $X(t) \approx x+t\widehat{v}+\C_{t,v}$. Consistently with these heuristic computations, we proved that the distribution function converges along these logarithmic corrections of the linear characteristics. There exists $f_\infty \in L^1 (\R^3_x \times \R^3_v)$ such that
\begin{equation}\label{eq:introscatf}
\hspace{-10mm} \lim_{t \to + \infty} \big\|f \big(t,x+t\widehat{v}+\C_{t,v},v \big)-f_\infty(x,v) \big\|_{L^1_{x,v}} =0 .
\end{equation}
A crucial property for the purpose of this paper is that the asymptotic Lorentz force is a functional of (the spatial average $Q_\infty$ of) $f_\infty$. More precisely, it arises from an asymptotic electromagnetic field $v \mapsto (\mathbb{E}[f_\infty],\mathbb{B}[f_\infty])(v)$, which is given, for $1 \leq k \leq 3$, by
\begin{align}
 \mathbb{E}^k[f_\infty](v) &:=- \frac{1}{4  \pi} \int_{\substack{|z| \leq 1 \\ |z+\widehat{v}| <1-|z|  }}\mathbf{Q}^k \big[  f_\infty \big] \left( \frac{\widecheck{  z+\widehat{v}  }}{1-|z|} \right)  \frac{\dr z}{(1-|z|)^4 |z|} , \label{kev:defasympelec} \\
 \mathbb{B}^k[f_\infty](v) &:= -  \frac{1}{4  \pi} \int_{\substack{|z| \leq 1 \\ |z+\widehat{v}| <1-|z|  }}\mathbf{Q}^{ij} \big[  f_\infty \big] \left( \frac{\widecheck{  z+\widehat{v}  }}{1-|z|} \right)  \frac{\dr z}{(1-|z|)^4 |z|}, \label{kev:defasympmag}
\end{align}
where $1 \leq i , \, j \leq 3$ are such that the signature of the permutation $(1,2,3) \mapsto (k,i,j)$ is equal to $1$, $w \mapsto \widecheck{w}$ is the inverse function of the relativistic speed $v \mapsto \widehat{v}$ and
$$ \mathbf{Q}^k \big[ f_\infty \big](v) =  \int_{\R^3_x} \langle v \rangle^5 \partial_{v^k} \big( \langle v \rangle f_\infty \big) (x,v) \dr x, \qquad \quad \mathbf{Q}^{ij} \big[ f_\infty \big](v) =  \int_{\R^3_x} \langle v \rangle^5  \big( v^i \partial_{v^j}f_\infty-v^j \partial_{v^i} f_\infty \big) (x,v) \dr x.$$
In fact, one can express $(\mathbb{E}[f_\infty],\mathbb{B}[f_\infty])$ in terms of $f_\infty$, instead of its first order $v$-derivatives, and then gain regularity.
\item In contrast, the electromagnetic field exhibits a \textit{linear scattering} dynamics. More precisely, it has a radiation field $(E_\infty,B_\infty)$ along future null infinity $\mathcal{I}^+$, that is \eqref{eq:introscatEB} holds for $(E,B)$. Furthermore, since Vlasov fields enjoy strong decay properties along null straight lines, $(E_\infty,B_\infty)$ verifies the constraint equations \eqref{eq:Maxnull} of the scattering states for the source free Maxwell equations. Thus, $(E,B)$ approaches for large time the homogeneous solution $(E^{\mathrm{hom}},B^{\mathrm{hom}})$ corresponding to this scattering state. 

It turns out that $(f_\infty, E_\infty,B_\infty)$ satisfy two other constraint equations, given by \eqref{eq:constr1}--\eqref{eq:constr2} below, that we identify in this paper and which are crucial for the purpose of constructing the modified wave operator. They imply that, contrary to a smooth solution to the homogeneous Maxwell equations, the field $(E,B)$ may have a nontrivial \textit{electromagnetic memory effect} \cite{BieriGarfinkle}. Thus, in some sense, $(E,B)$ has a linear behavior merely at first order. For this reason, $(f^{\mathrm{lin}},E^{\mathrm{lin}},B^{\mathrm{lin}})$, the solution to the linearised Vlasov-Maxwell system \eqref{VM1lin}--\eqref{VM3lin} with asymptotic data $(f_\infty, E_\infty,B_\infty)$, is a better approximation than $(f^{\mathrm{lin}},E^{\mathrm{hom}},B^{\mathrm{hom}})$. Note that we study $f^{\mathrm{lin}}$ throughout Section \ref{SubseclinVla} and $(E^{\mathrm{lin}},B^{\mathrm{lin}})$ in Proposition \ref{ProforF1}.
\item We also derived similar statements for the derivatives of the solutions. In particular, although there is a loss of derivatives in the process, the smoother the initial data set $(f_0,E_0,B_0)$ is, the more regular the scattering state $(f_\infty,E_\infty,B_\infty)$ is.
\item The total energy of the system,
\begin{equation}\label{eq:defnorm0}
 \mathbb{E}_{\mathrm{tot}}[f_0,E_0,B_0]:=\int_{\R^3_x} \int_{\R^3_v} \langle v \rangle \,  f_0( x,v) \mathrm{d} v \mathrm{d} x+ \frac{1}{2} \int_{\R^3_x} |E_0(x)|^2+|B_0(x)|^2 \mathrm{d}x,
\end{equation}
which is a conserved quantity\footnote{More precisely, for any sufficiently regular solution to \eqref{VM1}--\eqref{VM3}, $t \mapsto \mathbb{E}_{\mathrm{tot}}[f(t,\cdot , \cdot ),E(t,\cdot) ,B (t, \cdot)]$ is constant.}, is equal to the corresponding norm of the scattering state $(f_\infty,E_\infty,B_\infty)$,
\begin{align*}
 \mathbb{E}_{\infty}[f_\infty , E_\infty , B_\infty] := & \int_{\R^3_x} \int_{\R^3_v} \langle v \rangle \, f_\infty( x,v) \mathrm{d} v \mathrm{d} x \nonumber \\
 &+ \frac{1}{4} \int_{\R_u} \int_{\mathbb{S}^2_\omega} |E_\infty- \omega \times B_\infty|^2(u,\omega)+|B_\infty+ \omega \times E_\infty|^2(u,\omega) \mathrm{d} \mu_{\mathbb{S}^2_\omega} \mathrm{d} u. 
 \end{align*}
\end{itemize}
As formulated in \cite{ReedSimon} in a general setting, a satisfactory scattering theory for the Vlasov-Maxwell system must verify the following three properties.

\begin{enumerate}
\item \textit{Existence of scattering states}: for any sufficiently regular scattering state $(f_\infty,E_\infty,B_\infty)$, there exists a finite-energy solution $(f,E,B)$ to the Vlasov-Maxwell system evolving to this given scattering state.
\item \textit{Uniqueness of scattering states}: two finite-energy solutions to \eqref{VM1}--\eqref{VM3} corresponding to the same scattering state must coincide.
\item \textit{Asymptotic completeness}: the global finite-energy solutions to the Vlasov-Maxwell system are either bound states or evolve to such a scattering state.
\end{enumerate}
The first two properties would constitute a global well-posedness statement for the Vlasov-Maxwell system with asymptotic data. They allow to construct the (modified) wave operator $\mathscr{W}$, mapping any scattering state to the corresponding solution to \eqref{VM1}--\eqref{VM3}. Note that we do not know if bound states of finite energy exist for the Vlasov-Maxwell system. Though, if the distribution function is small, \eqref{eq:intro} implies that such solutions does not exist. 

According to the results proved in \cite{scat}, $3.$ holds true for the solutions arising from a small initial Vlasov field. The main goal of this article is to complete the construction of the scattering theory for the Vlasov-Maxwell system with a small distribution function. In other words, we will prove $1.$ and $2.$ for this class of solutions. It will allow us to relate the past asymptotic dynamics, of the solutions $(f,E,B)$ to \eqref{VM1}--\eqref{VM3} for which $f$ is small, to the future one.
\begin{Rq}
A global well-posedness result for the solutions to the Vlasov-Maxwell system would imply, together with Theorem \ref{Theo1} below, that existence and uniqueness of scattering states hold for the large data solutions as well.

 Indeed, given a smooth and large scattering state $(f_\infty,E_\infty,B_\infty)$, one can uniquely associate to it a solution $(f,E,B)$ to \eqref{VM1}--\eqref{VM3}, defined on $[T,+\infty[$ and evolving to $(f_\infty,E_\infty,B_\infty)$. However, Theorem \ref{Theo1} provides a lower bound on $T$ which, without a smallness assumption on $\epsilon$, can be large. Then, we \underline{could} extend $(f,E,B)$ on $\R_+$ and, by a density argument, construct a bijective isometric wave operator
$$ \mathscr{W} : (f_\infty,E_\infty,B_\infty) \mapsto \big(f(0,\cdot,\cdot),E(0,\cdot ),B(0,\cdot) \big), \qquad \mathbb{E}_{\infty}[f_\infty , E_\infty , B_\infty]=\mathbb{E}_{\mathrm{tot}}[f(0,\cdot,\cdot),E(0,\cdot ),B(0,\cdot)]<+\infty.$$
\end{Rq}

Finally, let us mention the work of Ben Artzi-Pankavich \cite{BAP}. Given a small data solution to the Vlasov-Maxwell system $(f,E,B)$ constructed by Glassey-Strauss in \cite{GSt}, they proved modified scattering for the distribution function $f$.

\subsection{Analogous results for the Vlasov-Poisson system}

In contrast with the case of the Vlasov-Maxwell system, global existence holds for the classical solutions to the Vlasov-Poisson system \cite{Pfa,LionsPerthame,SchaefferPoisson}. Sharp decay rates for the small data solutions to \eqref{eq:VP} were first obtained in \cite{Bardos} and then, with several improvements, by \cite{HRV,Poisson,Duan,smallSchaeffer} (see also \cite{rVP0,rVPWang,rVP} for the massive and massless relativistic Vlasov-Poisson systems). The study of the modified scattering dynamics of these solutions started with \cite{Choi} and, more recently, \cite{scattPoiss,Panka} clearly identified the asymptotic electrostatic force field responsible for this phenomenon. It is related, through the Poisson equation, to the spatial average of the limit distribution function $f_\infty$. The understanding of these asymptotic dynamics allowed Flynn-Ouyang-Pausader-Widmayer to construct a scattering map for the small data solutions to the Vlasov-Poisson system \cite{scattmap}.

Let us mention finally that the asymptotic stability of other steady states has been recently adressed. The perturbations of a point charge exhibits a modified scattering dynamics \cite{PausaWid,PWY} and the ones of the Poisson equilibrium scatter to linear solutions \cite{IPWW}. For the study of the solutions to the linearised
Vlasov-Poisson system near more general spatially homogenous equilibria, we refer to \cite{Toan,IPWW2}.

\subsection{General notations}

In this subsection as well as in the next one, we introduce certain quantities, used throughout this article, in order to properly state the main results. We work on the $1+3$ dimensional Minkowski spacetime $(\R^{1+3},\eta)$ and we will use two sets of coordinates, the Cartesian $(x^0=t,x^1,x^2,x^3)$, in which $\eta=\mathrm{diag}(-1,1,1,1)$, and null coordinates $(u,\underline{u},\theta, \varphi)$, where
$$u=t-r, \qquad \underline{u}=t+r, \qquad  r:= |x|=\sqrt{|x^1|^2+|x^2|^2+|x^3|^2},$$
and $(\theta,\varphi)\in ]0,\pi[ \times ]0,2 \pi[$ are spherical coordinates on the spheres of constant $(t,r)$. These coordinates are defined globally on $\R^{1+3}$ apart from the usual degeneration of spherical coordinates and at $r=0$. Sometimes, for a tensor field $\mathrm{T}$ defined on a subset of $\R_t \times \R^3_x$, it will be convenient to write
$$ \mathrm{T}(u,\underline{u},\omega):=\mathrm{T}\left( \frac{\underline{u}+u}{2},\frac{\underline{u}-u}{2} \omega \right), \qquad \qquad (u,\underline{u}, \omega ) \in \R \times \R \times \mathbb{S}^2 .$$
We denote by $\dr \mu_{\mathbb{S}^2}=\sin(\theta) \dr \theta \dr \varphi$ the volume form on the sphere $\mathbb{S}^2$. We will work with the null frame $(L,\underline{L},e_\theta,e_\varphi)$, where $L=2\partial_u$, $\underline{L}=2\partial_{\underline{u}}$ are null derivatives and $(e_\theta,e_\varphi)$ is the standard orthonormal basis on the spheres. More precisely,
$$L=\partial_t+\partial_r , \qquad \underline{L}=\partial_t-\partial_r, \qquad e_\theta=\frac{1}{r}\partial_{\theta}, \qquad e_\varphi=\frac{1}{r \sin \theta} \partial_{\varphi}.$$
The Einstein summation convention will often be used, for instance $v^{\mu}\partial_{x^{\mu}} f=\sum_{\mu = 0}^3v^{\mu}\partial_{x^{\mu}}f$. The lower case Latin indices goes from $1$ to $3$ and the Greek indices from $0$ to $3$. We will raise and lower indices with respect to the Minkowski metric $\eta$, so that $v_i=\eta_{i \mu}v^\mu=v^i$ and $v_0=-v^0$. Capital Roman letters, in particular $A$ and $B$, will correspond to spherical variables.

The four-momentum vector $\mathbf{v}=(v^{\mu})_{0 \leq \mu \leq 3}$ is parameterized by $v=(v^i)_{1 \leq i \leq 3} \in \R^3_v$ and $v^0=\sqrt{1+|v|^2}$ since the mass of the particles is equal to $1$. The relativistic speed $\widehat{v} \in \R^3$ is given by $\widehat{v}^i=\frac{v^i}{v^0}$ and, for convenience, we define $\widehat{v}^0 := 1$. Since we will perform the change of variables $y=x-\widehat{v}t$ in integrals over the domain $\R^3_v$, we will have to consider the inverse function of $v \mapsto \widehat{v}$. Note that we already used it as well in order to define the asymptotic electromagnetic field \eqref{kev:defasympelec}--\eqref{kev:defasympmag}.
\begin{Def}\label{Definvrela}
Let $\widecheck{\; \color{white} x \color{black} \; }$ be the operator defined, on the domain $\{ y \in \R^3 \; \; |y|<1 \}$, as 
$$y \mapsto \widecheck{y} = \frac{y}{\sqrt{1-|y|^2}}, \qquad \text{so that} \qquad \forall \, |y| <1, \; \forall \, v \in \R^3_v, \qquad \widehat{\widecheck{y}}=y, \quad \widecheck{\widehat{v}}=v .$$
\end{Def}
Sometimes, for convenience, we will write $(|v^0|^pg)(w)$ to denote $|w^0|^pg(w)$, where $w \in \R^3_v$ and $g:\R^3_v \rightarrow \R$. For $a \in \R$ and $x \in \R^3$, we will use the Japanese brackets $\langle a \rangle := (1+|a|^2)^{\frac{1}{2}}$ and $\langle x \rangle := \langle |x| \rangle$. Finally, the notation $D_1 \lesssim D_2$ will stand for the statement that $\exists \, C>0$ a positive constant independent of the solutions such as $ D_1 \leq C D_2$. Sometimes, we will write $\lesssim_{t_0}$ if $C$ depends on the initial time $t_0$.

\subsection{The Vlasov-Maxwell system in geometric form}

In order to capture the good behavior of certain geometric quantities associated to the solutions in the good null directions $(L,e_\theta,e_\varphi)$, which will in particular be useful for proving scattering results for the electromagnetic field, we represent $(E,B)$ as a $2$-form. Let $F_{\mu \nu}$ be the Faraday tensor and $J(f)_{\mu}$ be the four-current density, defined in cartesian coordinates as
\begin{equation}\label{eq:defF} F = \begin{pmatrix}
0 & E^1 & E^2 & E^3\\
-E^1 & 0 & -B^3 & B^2 \\
- E^2 & B^3 & 0& -B^1 \\
- E^3 &-B^2 & B^1 & 0
\end{pmatrix}, \qquad \qquad J(f) := \begin{pmatrix}
-\int_{\R^3_v}f \mathrm{d}v \\\noalign{\vskip2pt} \int_{\R^3_v}\widehat{v}_1 f \mathrm{d}v \\\noalign{\vskip2pt} \int_{\R^3_v}\widehat{v}_2f \mathrm{d}v \\\noalign{\vskip2pt} \int_{\R^3_v}\widehat{v}_3 f \mathrm{d}v
\end{pmatrix}.
\end{equation}
The null decomposition $(\underline{\alpha}(F),\alpha(F),\rho(F),\sigma(F))$ of $F$ is defined, for $A \in \{\theta, \varphi\}$, as
\begin{equation}\label{defnullcompoF}
 \underline{\alpha}(F)_{e_A}:=F_{e_A \underline{L}}, \qquad \alpha(F)_{e_A}:=F_{e_A L}, \qquad \rho(F):=\frac{1}{2}F_{\underline{L} L}, \qquad \sigma(F):=F_{e_\theta e_\varphi},
 \end{equation}
where $\underline{\alpha}(F)$ and $\alpha(F)$ are $1$-forms tangential to the $2$-spheres.
 \begin{Rq}
Note that $-J(f)_0$ is the charge density and $(J(f)_i)_{1 \leq i \leq 3}$ is the current density. The nonvanishing cartesian components of $F$ are equal to a cartersian component of $\pm E$ or $ \pm B$.

Concerning the elements of the null decomposition of $F$, the $1$-forms $\alpha(F)$ and $\underline{\alpha}(F)$ verify $$\alpha(F)_{e_\theta}\!=B \cdot e_\varphi-E \cdot e_\theta , \quad \alpha(F)_{e_\varphi}\!= -B \cdot e_\theta -E \cdot e_\varphi \quad  \underline{\alpha}(F)_{e_\theta}\!=-B \cdot e_\varphi-E \cdot e_\theta , \quad \underline{\alpha}(F)_{e_\varphi}\!= B \cdot e_\theta -E \cdot e_\varphi.$$ 
Finally, $\rho(F)=E \cdot \omega$ and $-\sigma(F)=B \cdot \omega$ are the radial components of the electric and the magnetic field. 
\end{Rq}
Then, the Vlasov equation \eqref{VM1} can be rewritten as
\begin{equation}\label{Vlasov1}
\T_F(f)=0, \qquad  \qquad \T_F : f \mapsto \partial_t f+\widehat{v} \cdot \nabla_x f+\widehat{v}^{\mu}{F_{\mu}}^j \partial_{v^j}f,
\end{equation}
and the Maxwell equations \eqref{VM2}--\eqref{VM3} are equivalent to the Gauss-Ampère law and the Gauss-Farady law
\begin{equation}\label{Maxwell23}
\nabla^{\mu} F_{\mu \nu}=J(f)_{\nu}, \qquad \qquad \nabla^\mu {}^* \!F_{\mu \nu}=0,
\end{equation}
where ${}^* \! F_{\mu \nu} = \frac{1}{2} F^{\lambda \sigma} \varepsilon_{ \lambda \sigma \mu \nu}$ is the Hodge dual of $F$ and $\varepsilon$ the Levi-Civita symbol. The Gauss-Faraday law can also be written as $\nabla_{[\lambda}F_{\mu \nu]}:=\nabla_\lambda F_{\mu \nu}+\nabla_\mu F_{ \nu \lambda}+\nabla_\nu F_{  \lambda \mu}=0$.

The intrinsic covariant differentiation on the spheres will be denoted by $\slashed{\nabla}$. The spherical divergence and curl of a $1$-form $\beta$ tangential to the $2$-spheres, such as $\underline{\alpha} (F)$ and $\alpha (F)$, are then
$$ \slashed{\nabla} \cdot \beta := \slashed{\nabla}_{e_\theta} \beta_{e_\theta}+\slashed{\nabla}_{e_\varphi} \beta_{e_\varphi}, \qquad  \slashed{\nabla} \times \beta := \slashed{\nabla}_{e_\theta} \beta_{e_\varphi}-\slashed{\nabla}_{e_\varphi} \beta_{e_\theta}, \qquad \qquad \slashed{\nabla}_{e_A}e_B := \nabla_{e_A} e_B-\Gamma_{AB}^r \partial_r,$$
where $\Gamma^\lambda_{\mu \nu}$ are the Christofel symbols in the basis $(\partial_t,\partial_r,e_\theta, e_\varphi)$, so that $\Gamma^0_{AB}=0$.

 We will further use the pointwise norms
\begin{align*}
 |\alpha(F)|^2 &:= |\alpha(F)_{e_\theta}|^2+|\alpha(F)_{e_\varphi}|^2, \qquad \qquad |\underline{\alpha}(F)|^2 := |\underline{\alpha}(F)_{e_\theta}|^2+|\underline{\alpha}(F)_{e_\varphi}|^2, \\
|F|^2 &:= \sum_{0 \leq \mu < \nu \leq 3} |F_{\mu \nu}|^2 = \frac{1}{2}|\alpha(F)|^2+\frac{1}{2}|\underline{\alpha}(F)|^2+|\rho (F)|^2+|\sigma(F)|^2.
\end{align*}

\subsection{Charged electromagnetic field}\label{Subsecpurecharge}

If the plasma is not neutral, one cannot expect the electromagnetic field that we will construct from the scattering data to decay faster than $r^{-2}$. Indeed, if $(f,F)$ is a sufficiently regular solution to \eqref{Vlasov1}--\eqref{Maxwell23} on $[T,+\infty[$, we obtain from Gauss's law that the total charge
$$ \forall\, t \geq T, \qquad \quad Q_F(t):=\lim_{r \rightarrow + \infty}  \int_{\mathbb{S}^2_\omega } \rho(F)(t,r\omega) r^2 \dr \mu_{\mathbb{S}^2_\omega}= \int_{\R^3_x} \int_{\R^3_v} f(t,x,v) dv dx=\int_{\R^3_x} \int_{\R^3_v} f_\infty (x,v) dv dx,$$
is a conserved quantity and that $|F| =o(r^{-2})$ implies $Q_F=0$. In order to avoid such a restrictive assumption, we introduce the pure charge part $\overline{F}$ of $F$,
\begin{equation}\label{defFoverlin}
\hspace{-4.5mm} \overline{F}(t,x) := \overline{\chi}(t-|x|)\frac{Q_F}{4 \pi |x|^2} \frac{x_i}{|x|} \dr t \wedge \dr x^i , \quad \; \rho\big(\overline{F} \big)(t,x)= \overline{\chi}(t-|x|)\frac{Q_F}{4 \pi |x|^2}, \; \; \, \alpha(\overline{F})=\underline{\alpha}(\overline{F})=\sigma(\overline{F})=0,
\end{equation}
where $\overline{\chi} \in C^\infty(\R,[0,1])$ is a cutoff function such that $\overline{\chi}(s)=1$ for $s \leq -2$ and $\chi(s)=0$ for $s \geq -1$. It corresponds to the electromagnetic field generated by a point charge $Q_F$ at $x=0$. Since $Q_{\overline{F}}=Q_F$, so that $F-\overline{F}$ is chargeless, one can then expect $F$ to have an asymptotic expansion of the form $F=\overline{F}+O(r^{-2-\delta})$, with $\delta >0$. In fact, $E=E^{\mathrm{df}}+E^{\mathrm{cf}}$ and $B=B^{\mathrm{df}}+B^{\mathrm{cf}}$ can be decomposed into their divergence-free and curl-free components. Then, $B^{\mathrm{cf}}=0$ and $E^{\mathrm{cf},i}=\overline{F}_{0i}+O(r^{-3})$ if $J(f)_0$ is sufficiently regular.

\subsection{Constraint equations satisfied by the scattering states} For a solution to the Maxwell equations with a source term decaying fast enough and arising from regular initial data, the null components $\alpha$, $\rho$ and $\sigma$ enjoy stronger decay properties than $\underline{\alpha}$. In particular, this holds true for any solution $(f,F)$ to the Vlasov-Maxwell system \eqref{Vlasov1}--\eqref{Maxwell23} arising from sufficiently regular data and a small distribution function. The constraints \eqref{eq:Maxnull} are the consequence of
$$\forall \, (u,\omega) \in \R_u \times \mathbb{S}^2_\omega, \qquad \; \; \lim_{r \to + \infty} r \alpha(F)(r+u,r\omega)=\lim_{r \to + \infty} r \rho(F)(r+u,r\omega)=\lim_{r \to + \infty} r \sigma(F)(r+u,r\omega)=0.$$
The radiation field along future null infinity of the electromagnetic field $F$ can then be represented by a $1$-form $\underline{\alpha}^\infty$, defined on $\R_u \times \mathbb{S}^2_\omega$ and tangential to the $2$-spheres, such that
$$\forall \, (u,\omega) \in \R_u \times \mathbb{S}^2_\omega, \qquad \; \; \lim_{r \to + \infty} r \underline{\alpha}(F)_{e_A}(r+u,r\omega)=\underline{\alpha}^\infty_{e_A}, \qquad \qquad A \in \{ \theta , \varphi \}.$$
We will prove in this article that there are in addition two constraint equations relating the limit distribution $f_\infty$ with the radiation field $\underline{\alpha}^\infty$. In order to write them, we consider the Faraday tensor $\mathbb{F}[f_\infty]$ associated through \eqref{eq:defF} to the asymptotic electromagnetic field $(\mathbb{E}[f_\infty],\mathbb{B}[f_\infty])$ introduced in \eqref{kev:defasympelec}--\eqref{kev:defasympmag}. It is a constant in time $2$-form defined on $\R \times \R^3$. Then, the following relations hold,
\begin{align}
  \hspace{-5.5mm}   \int_{\R_u} \! \slashed{\nabla} \cdot \underline{\alpha}^{\infty} (u,\omega) \dr u =& \, \frac{1}{2 \pi} \int_{\R^3_x} \int_{\R^3_v} f_\infty(x,v) - f_\infty(x,|v| \omega)\dr v \dr x -  \int_{\tau=0}^{+\infty}  \frac{2 \tau^2}{\langle \tau \rangle + \tau} \, \slashed{\nabla} \cdot  \underline{\alpha} \big( \mathbb{F} [ f_\infty ] \big) (\tau \omega) \frac{\dr \tau}{\langle \tau \rangle^3}  \label{eq:constr1} ,\\ 
  \hspace{-5.5mm}   \int_{\R_u} \! \slashed{\nabla} \times \underline{\alpha}^{\infty} (u,\omega) \dr u =& - \int_{\tau=0}^{+\infty}  \frac{2 \tau^2}{\langle \tau \rangle + \tau} \, \slashed{\nabla} \times  \underline{\alpha} \big( \mathbb{F} [ f_\infty ] \big) (\tau \omega) \frac{\dr \tau}{\langle \tau \rangle^3}. \label{eq:constr2}
\end{align}
These equations imply that the field $F$ may have a nonvanishing electromagnetic memory effect \cite{BieriGarfinkle}. It is a change of velocity, usually referred as a kick, experienced by a test charged particle which remained far from $r=0$ during its evolution.
\begin{Rq}
It is the field $F^{\mathrm{asymp}}[f_\infty]$, studied in Section \ref{SecMaxasymp} and capturing the asymptotic behavior of $F$, which gives rise to \eqref{eq:constr1}--\eqref{eq:constr2}. In particular, we have $t^2 F^{\mathrm{asymp}}[f_\infty](t,t\widehat{v})=\mathbb{F}[f_\infty](v)$ for $t \geq 4 \, \langle v \rangle^2$ and
$$ \forall \, (t,x) \in [T,+\infty[ \times \R^3, \qquad \big|F-F^{\mathrm{asymp}}[f_\infty] \big|(t,x) \lesssim \Lambda \langle t+|x| \rangle^{-1} \, \langle t-|x| \rangle^{-\frac{3}{2}}.$$
It is a solution to the Maxwell equations with a well-chosen source term, governing the large time behavior of the current density $J(f)$.   
\end{Rq}
\begin{Rq}
If $v=\tau \omega$, then $2 \tau^2(\langle \tau \rangle + \tau)^{-1}\langle \tau \rangle^{-3}=2|\widehat{v}|^2(1-|\widehat{v}|)$.
\end{Rq}
\subsection{Future and past scattering states}\label{Subsecsoncstr}  We are now able to define the set of the scattering states for the Vlasov-Maxwell system. 
 \begin{Def}
 A future scattering state for the Vlasov-Maxwell system is an ordered pair $(f_\infty , \underline{\alpha}^\infty)$ where
 \begin{itemize}
\item $f_\infty : \R^3_x \times \R^3_v \to \R$ is a distribution function,
\item $\underline{\alpha}^\infty$ is a $1$-form defined on $\mathcal{I}^+ \simeq \R_u \times \mathbb{S}^2$ which is tangential to the $2$-spheres,
\item $f_\infty$ and $\underline{\alpha}^\infty$ are sufficiently regular so that any quantity in \eqref{eq:constr1}--\eqref{eq:constr2} is well-defined and these equations are verified.
\end{itemize}
 \end{Def}   
The distribution function is usually assumed to be nonnegative for physical reasons. Since this restriction is not required for the results of this paper to hold, we allow it to take negative values. 

In order to relate the small data past asymptotic dynamics to the future one, we need to introduce the past scattering states as well. For this, given $(f,E,B)$, we define $(\widetilde{f},\widetilde{E}, \widetilde{B})$ as 
$$\widetilde{f}(t,x,v):=f(-t,x,-v), \qquad (\widetilde{E}, \widetilde{B})(t,x):=(E,-B)(-t,x).$$
 This operation is an involution leaving invariant the set of the solutions to the Vlasov-Maxwell system. The null components of the Faraday tensor $\widetilde{F}$ are related to the ones of $F$ as follows,
\begin{alignat*}{2}
 \rho \big( \widetilde{F} \big)(t,x)&=\rho \big( F \big)(-t,x), \qquad \qquad \; \sigma \big( \widetilde{F} \big)(t,x)&&=-\sigma \big( F \big)(-t,x), \\ 
 \alpha \big( \widetilde{F} \big)(t,x)&=\underline{\alpha} \big( F \big)(-t,x), \qquad \qquad \underline{\alpha} \big( \widetilde{F} \big)(t,x)&&=\alpha \big( F \big)(-t,x).
 \end{alignat*}
 \begin{Def}
 A past scattering state for the Vlasov-Maxwell system is an ordered pair $(f_{-\infty} , \alpha^{-\infty})$ such that $(\widetilde{f}_{\infty}, \widetilde{\underline{\alpha}}^{\infty})$ is a future scattering state, where
 $$ \widetilde{f}_{\infty}(x,v):=f_{-\infty}(x,-v), \qquad \qquad \widetilde{\underline{\alpha}}^{\infty}(u,\omega):=\alpha^{-\infty}(-u,\omega).$$
\end{Def}   
If $(f,F)$ is a solution to the Vlasov-Maxwell system admitting $(f_{-\infty} , \alpha^{-\infty})$ as past scattering state, $\alpha^{-\infty}$ is the radiation field of $F$ along past null infinity $\mathcal{I}^- \simeq \R_{\underline{u}} \times \mathbb{S}^2_\omega$, that is
$$ \forall \, (\underline{u} , \omega) \in \R_{\underline{u}} \times \mathbb{S}^2_\omega, \qquad \qquad \lim_{r \to +\infty} \alpha(F)(-r+\underline{u},r\omega)=\alpha^{-\infty}(\underline{u},\omega).$$

\subsection{Statement of the main results}

We now state our result concerning the construction of the modified wave operator. For this, recall the correction $\C_{t,v}$ to the linear spatial characteristics introduced in \eqref{eq:introscatf} and which is a functional of $f_\infty$. Certain estimates verified by the electromagnetic field will use notations not yet defined. We refer to Section \ref{SecenergyMaxweel} for the weighted $L^2_x$ norms $\mathcal{E}^{K}_{7}[ \cdot]$ and $\mathcal{E}^{K,1}[ \cdot]$ as well as Section \ref{SecMaxasymp} for $F^{\mathrm{asymp}}[f_\infty]$.

\begin{Th}\label{Theo1}
Let $N \geq 8$, and $(f_\infty , \underline{\alpha}^\infty)$ be a future scattering state for the Vlasov-Maxwell system. Assume that there exists $\epsilon >0$ and $\Lambda \geq 1$ such that
\begin{align}
 \sup_{|\kappa_z|+|\kappa_v| \leq N+1} \, \sup_{(z,v) \in \R^3 \times \R^3}  \langle z \rangle^{N+12} \,  \langle v \rangle^{25+|\kappa_v|} \, \big| \partial_{z}^{\kappa_z} \partial_{v}^{\kappa_v} f_\infty (z,v) \big|  & \leq \sqrt{\epsilon} ,  \label{eq:assumpfinfty} \\
 \sup_{\gamma_u+|\gamma_\omega| \leq N+1 } \int_{\R_u} \int_{\mathbb{S}^2_\omega} \langle u \rangle^{2} \log^2\! \big( 1+\langle u \rangle \big) \big| \nabla_{\partial_u}^{\gamma_u} \slashed{\nabla}^{\gamma_\omega} \underline{\alpha}^\infty (u,\omega) \big|^2 \dr \mu_{\mathbb{S}^2_\omega} \dr u &\leq  \Lambda. \label{eq:assumpunderalphainfty}
\end{align}
Let further $T \geq 2$ be an initial time. There exists an absolute constant $\varepsilon_0>0$ and $C_\Lambda := \exp \big( \exp \big( C \sqrt{\Lambda} \big) \big)$, where $C>0$ depends only on $N$, such that the next statement hold. If $C_\Lambda \epsilon  \log^{-1}(T) \leq \varepsilon_0$, then there exists a unique solution $(f,F)$ to the Vlasov-Maxwell system, defined on $[T,+\infty[$, such that
\begin{itemize}
\item the following asymptotic conditions hold,
\begin{alignat*}{2}
&\forall \, (x,v) \in \R^3_x \times \R^3_v, \qquad \qquad &&   \lim_{t \to +\infty} f \big(t,x+t\widehat{v}+\C_{t,v},v \big) = f_\infty(x,v) , \\
&\forall \, (u,\omega) \in \R_u \times \mathbb{S}^2_\omega , \qquad \qquad && \quad \; \, \lim_{r \to +\infty} r \underline{\alpha}(F)(r+u,r\omega)=\underline{\alpha}^\infty (u,\omega).
\end{alignat*}
\item The function $g(t,z,v):= f(t,z+t\widehat{v}+\C_{t,v},v)$ verifies the uniform $L^2_{z,v}$ bound
$$ \sup_{\kappa_t+|\kappa_z|+|\kappa_v| \leq N} \, \sup_{t \geq T}\int_{\R^3_z}\int_{\R^3_v}   \langle z \rangle^{16+2N-2|\kappa_v|} \, \langle v \rangle^{30+2|\kappa_v|} \big| \partial_t^{\kappa_t} \partial_{z}^{\kappa_z} \partial_{v}^{\kappa_v} g (t,z,v) \big|^2 \dr v \dr z  \leq   C_\Lambda \epsilon.$$
\item The electromagnetic field $F$ and its derivatives up to order $N$ are uniformly bounded in $L^2(\R^3_x)$ and
$$ \sup_{t \geq T} \, \mathcal{E}^K_{7} \big[ F-F^{\mathrm{asymp}}[f_\infty] \big](t)+\sup_{|\gamma|=7} \, \sup_{t \geq T} \,\mathcal{E}^{K,1} \big[ \nabla_{t,x} \mathcal{L}_{Z^\gamma}\big(F-F^{\mathrm{asymp}}[f_\infty] \big) \big](t) \lesssim \Lambda.$$
Moreover, if $T=2$ and $C_\Lambda  \epsilon\leq \varepsilon_0 $, then 
$$ \forall \, |\gamma| \leq N-1, \qquad \qquad  \int_{\R^3_x} \langle x\rangle^{2+2|\gamma|}\big| \nabla_{t,x}^\gamma \big( F-\overline{F} \, \big) \big|^2(2,x) \dr x \lesssim  \Lambda .$$
\end{itemize}
\end{Th}
\begin{Rq}\label{Rqhighermoment}
The finiteness of higher moments of $f_\infty$ is sufficient in order to propagate stronger norms for $f$, as it is stated by Proposition \ref{ProCauchpb}. If $N \geq 9$, the higher order derivatives of $F$ satisfy similar weighted $L^2_x$ estimates. However, without stronger velocity decay assumptions for $f_\infty$, the norms are slightly less convenient to write.
\end{Rq}
\begin{Rq}
In fact, we use a slightly weaker decay assumption on $\underline{\alpha}^\infty$ than \eqref{eq:assumpunderalphainfty}. Similarly, as suggested by Remark \ref{RqenergyforTh1}, $2-\delta$ powers of $|z|$ and $6-\delta$ powers of $|v|$ could be saved in \eqref{eq:assumpfinfty} but the $L^2_{z,v}$ estimates for $g$ would be weaker and would not allow us to construct the scattering map.
\end{Rq}
\begin{Rq}
If $\epsilon$ is large, the solution $(f,F)$ constructed here is a future-global-in-time large solution to the Vlasov-Maxwell system in the following sense. It is defined for all time $t \geq T$, $F$ is of size $\Lambda$ in $L^2_x$ and $f$ is of size $\epsilon$ in $L^p_{x,v}$, so that $(f,F)$ is large in $ L^p_{x,v} \times L^2_x$ for any $1 \leq p \leq \infty$. The solution is however initially dispersed since $T \gg 1 $ and, for instance, $ |J(f)|(T,x,v) \lesssim  T^{-3}$.
\end{Rq}
Using the involution $\sim$ introduced in the previous Section \ref{Subsecsoncstr}, Theorem \ref{Theo1} ensures as well that for any sufficiently regular past scattering state, there exists a unique solution to the Vlasov-Maxwell system converging to it. Thus, Theorem \ref{Theo1} and \cite[Theorem~$2.11$]{scat} provide us a small data scattering map for the Vlasov-Maxwell equations, for which the electromagnetic field is allowed to be large.
\begin{Th}\label{Theo3}
Let $N \geq 8$, $\epsilon >0$, $\Lambda \geq 1$ and $(f_{-\infty} , \alpha^{-\infty})$ be a past scattering state for the Vlasov-Maxwell system verifying the bounds \eqref{eq:assumpfinfty}--\eqref{eq:assumpunderalphainfty}. There exists an absolute constant $\varepsilon_1>0$ and $D_\Lambda := \exp \big( \exp \big( D \sqrt{\Lambda} \big) \big)$, where $D>C$ depends only on $N$, such that the following statement hold. If $\epsilon D_\Lambda \leq \varepsilon_1$, then
\begin{itemize}
\item there exists a unique classical solution $(f,F)$ to the Vlasov-Maxwell system \eqref{VM1}--\eqref{VM3}, defined for all time $t \in \R$, and a future scattering state $(f_{+\infty},\underline{\alpha}^{+\infty})$, such that
\begin{alignat*}{2}
& \lim_{t \to \pm\infty} f\big(t,x+t\widehat{v}\pm\C_{t,v},v \big) = f_{\pm\infty}(x,v), \qquad \qquad && \forall \, t >0, \quad \C_{\pm t,v}^k := -\log(|t|)\frac{\delta_k^i-\widehat{v}_k \widehat{v}^i}{v^0} \widehat{v}^\mu \mathbb{F}_{\mu i}[f_{\pm \infty}](v) , 
\end{alignat*}
for all $(x,v) \in \R^3_x \times \R^3_v$, as well as
\begin{alignat*}{2}
&\lim_{r \to +\infty} r \alpha(F)(-r+\underline{u},r\omega)=\alpha^{-\infty} (\underline{u},\omega), \qquad \qquad && \lim_{r \to +\infty} r \underline{\alpha}(F)(r+u,r\omega)=\underline{\alpha}^{+\infty} (u,\omega),
\end{alignat*}
for all $\omega \in \mathbb{S}^2_\omega$, $u \in \R_u$ and $\underline{u} \in \R_{\underline{u}}$. Moreover, the initial data satisfy
\begin{align*}
\sup_{|\kappa_z|+|\kappa_v| \leq N-4} \,\langle z \rangle^{8+|\kappa_z|} \, \langle v \rangle^{15+|\kappa_v|} \big| \partial_t^{\kappa_t} \partial_{z}^{\kappa_z} \partial_{v}^{\kappa_v} f (0,z,v) \big| & \leq   C_\Lambda \epsilon, \\
  \sup_{|\gamma| \leq N-3}\langle x\rangle^{\frac{5}{2}+|\gamma|}\big| \nabla_{t,x}^\gamma \big( F-\overline{F} \, \big) \big|(0,x) & \lesssim  \Lambda 
  \end{align*}
\item The future scattering state verifies the pointwise bounds
\begin{alignat*}{2}
\sup_{(z,v) \in \R^3 \times \R^3} \, \langle z \rangle^{4-|\kappa_v|} \, \langle v \rangle^{5+|\kappa_v|} \big| \partial_{z}^{\kappa_z} \partial_{v}^{\kappa_v} f_{+\infty}(z,v) \big| & \leq  D_\Lambda \sqrt{\epsilon} , \qquad \qquad &&|\kappa_z|+|\kappa_v| \leq N-6, \\
\sup_{(u,\omega) \in \R \times \mathbb{S}^2}  \langle u \rangle^{\frac{3}{2}} \big| \nabla_{\partial_u}^{\gamma_u} \, \slashed{\nabla}^{\gamma_\omega} \underline{\alpha}^{+\infty}  \big| (u,\omega)   & \lesssim \Lambda, \qquad \qquad && \gamma_{u}+|\gamma_\omega| \leq N-7 .
\end{alignat*}
\item The total energy of the system is conserved. For all $t \in \R$,
\begin{align*}
\int_{\R^3_x} \int_{\R^3_v} v^0 f(t,x,v) \dr x \dr v+\frac{1}{2}\int_{\R^3_x} |F(t,x)|^2 \dr x &= \int_{\R^3_x} \int_{\R^3_v} f_{-\infty}(x,v) \dr x \dr v+\frac{1}{4} \int_{\R_{\underline{u}}}\int_{\mathbb{S}^2_\omega} \big| \alpha^{-\infty} (\underline{u},\omega) \big|^2 \dr \mu_{\mathbb{S}^2_\omega} \dr \underline{u} \\
& = \int_{\R^3_x} \int_{\R^3_v} f_{+\infty}(x,v) \dr x \dr v+\frac{1}{4} \int_{\R_u}\int_{\mathbb{S}^2_\omega} \big| \underline{\alpha}^{+\infty} (u,\omega) \big|^2 \dr \mu_{\mathbb{S}^2_\omega} \dr u.
\end{align*}
\end{itemize}
\end{Th}
\begin{Rq}
Part of the loss of regularity is purely technical and could be avoided, Theorem \ref{Theo1} provides $L^2$ bounds on the solutions whereas we propagated $L^\infty$ estimates in \cite[Theorem~$2.11$]{scat}.
\end{Rq}

Consistently with Remark \ref{Rqhighermoment}, under stronger decay assumptions on $f_{-\infty}$, Proposition \ref{ProCauchpb} and the analysis carried out in \cite{scat} imply that the future scattering state $f_{+\infty}$ possibly belongs to a stronger weighted Sobolev space.
\begin{Cor}
Let $(f,F)$ be one of the solutions to the Vlasov-Maxwell system considered in Theorem \ref{Theo3}. If, for $N_z \geq 12$ and $N_v \geq 25$,
\begin{align*}
\mathbf{B}(N,N_z,N_v):= \sup_{|\kappa_z|+|\kappa_v| \leq N} \big\| \,  \langle z \rangle^{N_z+N} \, \langle v \rangle^{N_v+|\kappa_v|} \,  \partial_z^{\kappa_z} \partial_v^{\kappa_v} f_{-\infty} \big\|_{L^2(\R^3_z \times \R^3_v)} <+\infty, 
\end{align*}
then, there exists $C >0$, depending only on $(\Lambda,N,N_z,N_v)$ such that, for any $|\kappa_z|+|\kappa_v| \leq N-6$,
$$\sup_{(z,v) \in \R^3 \times \R^3} \, \langle z \rangle^{N_z-7-|\kappa_v|} \, \langle v \rangle^{N_v-16+|\kappa_v|} \big| \partial_{z}^{\kappa_z} \partial_{v}^{\kappa_v} f_{+\infty}(z,v) \big| \leq C\mathbf{B}(N,N_z,N_v).$$ 
\end{Cor}

The main ingredients of the proof of Theorem \ref{Theo1} are presented in Section \ref{Subseckeyidea}.

\subsection{Related open problems} We list here some questions connected to the results proved in this article.
\begin{enumerate}
\item In the multispecies case or if we allow $f$ to take negative values, can a solution to the Vlasov-Maxwell system satisfy linear scattering as $t \to - \infty$ and modified scattering as $t \to + \infty$? This problem was raised in the fifth part of \cite[Remark~$1.2$]{scattmap}.
\item Given a past and a future radiation field, does the knowledge of the scattering map allows us to determine the properties of the plasma? In other words, given $(E_{-\infty},B_{-\infty})$ and $(E_{+\infty},B_{+\infty})$, is there a unique solution $(f,E,B)$ to the Vlasov-Maxwell system such that the past (respectively the future) radiation field of $(E,B)$ corresponds to $(E_{-\infty},B_{-\infty})$ (respectively $(E_{+\infty},B_{+\infty})$)?

This question, which is part of inverse scattering problems, has been raised to me by Claude Bardos. $(E_{-\infty},B_{-\infty})$ would correspond to a signal emitted by a far away observer in the direction of the plasma and $(E_{+\infty},B_{+\infty})$ would be the signal measured by this observer after the interaction between the electromagnetic field and the particles.
\item In the linear case, $t^3J(f)(t,t\widehat{v})$ has an asymptotic expansion in powers of $t^{-1}$. In the nonlinear case, although the four-current density has a linear behavior at zeroth order, that is $t^3J(f)(t,t\widehat{v}) \to \int_z \langle v \rangle^5 f_\infty(z,v)\dr z$, we expect it to have a polyhomogeneous expansion, as the charge density in the context of the Vlasov-Poisson system \cite{latetime}. More generally, it would be interesting to determine the asymptotic expansion of $J(f)$ as well as $(E,B)$.
\end{enumerate} 

\subsection{Structure of the paper}

We introduce in Section \ref{Sec2} the tools used throughout this article. We further analyse the solutions to the linear Vlasov equation in order to motivate the introduction of the asymptotic Maxwell equations. Then, we present the key ideas of the proof. In Section \ref{Secenergy}, we prove energy estimates for both the distribution functions and the electromagnetic field. Section \ref{SecscattMax} is devoted to scattering results for the Maxwell equations. In particular, we construct different wave operators for the Maxwell equations with a vanishing or a strongly decaying source term. 

The asympotic Maxwell equations are introduced in Section \ref{SecMaxasymp}. We perform a thorough analysis of its solution $F^{\mathrm{asymp}}[f_\infty]$, which captures the large time dynamics of $F$. In Section \ref{SecPicard}, we set up the Picard iteration which will allow us to construct the class of solutions $(f,F)$ to the Vlasov-Maxwell system presented in Theorem \ref{Theo1}. If $N \geq 9$, for the purpose of working in a more natural functional space and simplifying the presentation of the proof, we will first make a stronger assumption on $f_\infty$ than \eqref{eq:assumpfinfty}. The modifications to bring to the proof in order to deal with the general case are presented in Section \ref{Secweakerassump}. We control the sequence of approximate solutions to \eqref{VM1}--\eqref{VM2} in a suitable functional space throughout Sections \ref{SecVlasov}--\ref{SecMax}. In particular, we study the Vlasov equation (respectively the Maxwell equations), with a prescribed electromagnetic field (respectively a prescribed current density) and asymptotic data. We conclude the proof of Theorem \ref{Theo1} by proving in Section \ref{SecCauchy} that the sequence of approximate solutions is Cauchy for a weaker energy norm. Finally, Section \ref{secApp} contains elliptic estimates.

\subsubsection*{Acknowledgements} I would like to thank Renato Velozo Ruiz for several insightful discussions. This work was conduced within the France 2030 framework porgramme, the Centre Henri Lebesgue ANR-11-LABX-0020-01.

\section{Preliminaries}\label{Sec2}

\subsection{Commutation vector fields}

Our analysis of the electromagnetic field will rely on vector field methods, initially introduced in \cite{Kl85} in order to study solutions to wave equations and then adapted to the Maxwell equations by \cite{CK}. It will in particular allow us to obtain improved decay estimates for the good null components $\alpha$, $\rho$ and $\sigma$. These kind of approaches are usually based on
\begin{itemize}
\item a set of weighted derivatives (vector fields) \textit{commuting} with the linear operator of the equations studied.
\item Energy estimates, used to control $L^p$ norms of the solutions and their derivatives, such as Proposition \ref{ProMoravac} below. Vector fields are here used as \textit{multipliers}.
\item Weighted Sobolev inequalities, such as Proposition \ref{decayMaxell} stated in the next Section \ref{Secenergy}, providing decay estimates on the fields.
\end{itemize}
The set of the commutators for the Maxwell equations is composed of Killing vector fields and the scaling vector field $S$, which is merely conformal Killing. They then generate (conformal) isometries of the Minkowski space.
\begin{Def}
Let $\mathbb{K}$ be the set composed by the vector fields
$$ \partial_t, \qquad \partial_{x^k}, \qquad \Omega_{ij}:=x^i\partial_{x^j}-x^j\partial_{x^i}, \qquad \Omega_{0k}:=t\partial_{x^k}+x^k \partial_t, \qquad S:=t\partial_t+ x^\ell \partial_{x^\ell}=t\partial_t+r\partial_r,$$
where $1 \leq k \leq 3$ and $1\leq i < j \leq 3$.
\end{Def}
\begin{Rq}
The vector fields $\partial_{x^\mu}$ (respectively $\Omega_{ij}$ and $\Omega_{0k}$) are the infinitesimal generators of translations (respectively rotations and Lorentz boosts). Usually, the set $\mathbb{K} \setminus \{ S \}$ is denoted by $\mathbb{P}$.
\end{Rq}

To commute the Maxwell equations, it is important to differentiate the electromagnetic field geometrically. For a $2$-form $F$ and a vector field $Z=Z^{\mu}\partial_{x^{\mu}}$, the Lie derivative $\mathcal{L}_Z(F)$ of $F$ with respect to $Z$ is defined, in coordinates, as
\begin{equation}\label{keva:defLie}
 \mathcal{L}_Z(F)_{\mu \nu} = Z(F_{\mu \nu})+\partial_{\mu}(Z^{\lambda})F_{\lambda \nu}+\partial_{\nu}(Z^{\lambda}) F_{\mu \lambda}.
 \end{equation}
 Similarly, if $J$ is $1$-form, its Lie derivative with respect to $Z$ is
  $$\mathcal{L}_Z(J)_\nu= Z(J_\nu)+\partial_{x^\nu}(Z^\lambda)J_\lambda.$$
  
 \begin{Lem}\label{LemComMaxwell}
 Let $F$ be a solution to the Maxwell equations with source term $J$, 
 $$\nabla^{\mu} F_{\mu \nu}=J_\nu, \qquad \qquad \nabla^{\mu} {}^* \! F_{\mu \nu}=0.$$
 Then, for all Killing vector field $Z \in \mathbb{K} \setminus \{S \}$, we have
 $$ \nabla^{\mu} \mathcal{L}_Z(F)_{\mu \nu}=\mathcal{L}_Z(J)_\nu, \qquad \nabla^{\mu} {}^* \! \mathcal{L}_Z(F)_{\mu \nu}=0 \,; \qquad \quad \; \nabla^{\mu} \mathcal{L}_S(F)_{\mu \nu}=\mathcal{L}_S(J)_\nu+2J_\nu, \qquad \nabla^{\mu} {}^* \! \mathcal{L}_S(F)_{\mu \nu}=0.$$
 \end{Lem}
Thus, if $F$ is solution to the vacuum Maxwell equations $\nabla^{\mu} F_{\mu \nu}=\nabla^{\mu} {}^* \! F_{\mu \nu}=0$, then same holds true for $\mathcal{L}_Z F$ for all conformal Killing vector field $Z \in \mathbb{K}$. 

In the context of the Vlasov-Maxwell system, controlling the derivatives of $F$ requires to bound $Z J(f)$. Motivated by this problem, Fajman-Joudioux-Smulevici developed a vector field method for Vlasov equations compatible with the one for wave equations \cite{FJS}. Their approach relies on the fact that for any Killing vector field $Z$, its complete lift\footnote{We refer to \cite[Section~$2G$]{FJS} for more details about the relations between the Vlasov operator on a Lorentzian manifold and the complete lift of its Killing vector fields.} $\widehat{Z}$ commute with $v^0 \T_0$, where $\T_0 := \partial_t+\widehat{v} \cdot \nabla_x$ is the linear transport operator. The scaling, which verifies $[\T_0 , S]=\T_0$, will also be useful.
\begin{Def}
Let $\K$ be the set composed by
$$\widehat{ \partial_t}= \partial_t, \quad \widehat{\partial_{x^k}}=\partial_{x^k}, \quad \widehat{\Omega}_{ij}:=x^i\partial_{x^j}-x^j\partial_{x^i}+v^i\partial_{v^j}-v^j\partial_{v^i}, \quad \widehat{\Omega}_{0k}:=t\partial_{x^k}+x^k \partial_t+v^0 \partial_{v^k}, \quad  \widehat{S}=t\partial_t+r\partial_r+3,$$
where $1 \leq k \leq 3$ and $1\leq i < j \leq 3$.
\end{Def}
\begin{Rq}
Although the differential operator $S+3$ is not the complete lift of $S$, we will denote it by $\widehat{S}$ for simplicity. We chose to work with it in order to lighten the presentation and simplify the study of the commuted Maxwell equations (see Proposition \ref{Com} below).
\end{Rq}
\begin{Rq}
The good commutation properties of $S$ with the free transport operator $\T_0$ are explained by the relations $S=t\T_0+(x^i-t\widehat{v}^i)\partial_{x^i}$ and $\T_0(x-t\widehat{v})=0$, so that $x-t\widehat{v}$ is conserved along the linear flow.
\end{Rq}

For the purpose of considering higher order derivatives, we introduce an ordering on $\mathbb{K}$ and $\K$,
$$ \mathbb{K}=\{Z^i \, | \, 1 \leq i \leq 11\}, \qquad \qquad \K=\{\widehat{Z}^i\,  | \, 1 \leq i \leq 11\}.$$
It will be convenient to assume that $Z^{11}=\widehat{Z}^{11}=S$ and $\widehat{Z^i}=\widehat{Z}^i$ for any $1 \leq i \leq 10$. Moreover, for a multi-index $\kappa \in \llbracket 1,11 \rrbracket^p$ of length $|\kappa|=p$, we denote by $\mathcal{L}_{Z^{\kappa}}$ the Lie derivative $\mathcal{L}_{Z^{\kappa_1}} \dots \mathcal{L}_{Z^{\kappa_p}}$ of order $|\kappa|$. Similarly, we define $\widehat{Z}^{\kappa}$ as $\widehat{Z}^{\kappa_1} \dots \widehat{Z}^{\kappa_p}$. Note the equivalence between the pointwise norms
\begin{equation}\label{equinorm}
  \; \sup_{|\gamma| \leq N} \left|\mathcal{L}_{Z^{\gamma}}(F)\right| \lesssim \sup_{|\beta| \leq N} \, \sup_{0 \leq \mu,\nu \leq 3} \left| Z^{\beta}(F_{\mu \nu}) \right| \lesssim \sup_{|\gamma| \leq N} \left|\mathcal{L}_{Z^{\gamma}}(F)\right|.
\end{equation} 
In order to exploit certain hierarchies in the commuted equations, we will denote, for a multi-index $\beta$, by $\beta_H$ (respectively $\beta_T$) the number of homogeneous vector fields (respectively translations) composing $\widehat{Z}^\beta$ as well as $Z^\beta$. For instance,
$$ Z^\beta= S\partial_t \Omega_{23} \partial_{x^1} \Omega_{02},  \qquad \qquad |\beta|=5, \quad \beta_H=3, \quad \beta_T=2.$$
The following geometric formula, proved in \cite[Lemma~$2.8$]{massless}, will be fundamental for us.
\begin{Lem}\label{LemCom}
Let $H$ be a $2$-form and $h : [T,+\infty[ \times \R^3_x \times \R^3_v \rightarrow \R$ be a function, both sufficiently regular. Let further $Z \in \mathbb{K} \setminus \{ S \}$ be a Killing vector field and $\widehat{Z} \in \K \setminus \{ S \}$ be its complete lift. Then, we have
\begin{align*}
\mathcal{L}_Z \big( J(h) \big) = J(\widehat{Z}h),  \qquad \qquad \qquad \mathcal{L}_S \big( J(h) \big) = J(Sh)+J(h)
\end{align*}
as well as
\begin{align*}
& \widehat{Z} \left( v^\mu {H_\mu}^j \partial_{v^j} h \right)= v^\mu {\mathcal{L}_Z(H)_\mu}^j \partial_{v^j} h +v^\mu {H_\mu}^j \partial_{v^j} \widehat{Z}h, \\
& S \left( v^\mu {H_\mu}^j \partial_{v^j} h \right)=v^{\mu} {\mathcal{L}_S(H)_{\mu}}^j \partial_{v^j} h-2v^\mu {H_\mu}^j \partial_{v^j} h +v^\mu {H_\mu}^j \partial_{v^j}S h.
\end{align*}
\end{Lem}

Iterating Lemmata \ref{LemComMaxwell} and \ref{LemCom}, we obtain that the structure of the Vlasov-Maxwell equations \eqref{Vlasov1}--\eqref{Maxwell23} are preserved by commutation.
\begin{Pro}\label{Com}
Let $(f,F)$ be a sufficiently regular solution to the Vlasov-Maxwell system. For any multi-index $\beta$, there exists $C^{\beta}_{\gamma, \kappa} \in \mathbb{Z}$ such that
\begin{align*} 
\T_F \big( \widehat{Z}^{\beta} f \big) &= \sum_{|\kappa| \leq |\beta|-1} \,  \sum_{|\gamma|+|\kappa| \leq |\beta|} C^{\beta}_{\gamma,\kappa} \; \widehat{v}^\mu {\mathcal{L}_{Z^{\gamma}} (F)_{\mu}}^j \partial_{v^j} \widehat{Z}^{\kappa} f , \\
\nabla^{\mu} \mathcal{L}_{Z^\beta}(F)_{\mu \nu} &=  J(\widehat{Z}^\beta f), \qquad \qquad \nabla^\mu {}^* \mathcal{L}_{Z^\beta}(F)_{\mu \nu} = 0.
\end{align*}
\end{Pro}

The vector field method developed by Fajman-Joudioux-Smulevici \cite{FJS} for the study of relativistic transport equations lead to a proof of the stability of Minkowski spacetime for both the massive and massless Einstein-Vlasov system \cite{FJS3,EVmassless} (these results were also obtained by \cite{LindbladTaylor,WangMinko,Taylor} using different approaches). Our analysis carried out in \cite{dim3,scat} of the small data solutions to the Vlasov-Maxwell system rely on such techniques as well. 

Vector field methods have also been used in order to derive boundedness and decay estimates for the solutions to the massless Vlasov equation on a fixed Schwarzschild or slowly rotating Kerr black hole \cite{ABJ,Schwa}. However, due to the difficulties related to the trapped null geodesics, the control of certain derivatives of the solutions still remain to be understood. In this perspective, Velozo Ruiz-Velozo Ruiz studied in \cite{VRVR} the small data solutions to the Vlasov-Poisson system with an external trapping potential by developing a commuting vector field approach adapted to such problems. It turns out that these solutions exhibit, in dimension $2$, a modified scattering dynamic \cite{VPTPscat}.

Finally, in the non-relativistic setting, such approaches have been applied to the Vlasov-Poisson and Vlasov-Yakawa systems \cite{Poisson,Duan} as well as in the collisional regime \cite{ Chatu1,Chatu2,Chatu3}.

\subsection{Weights capturing the good behavior of distribution functions for $t \sim |x|$ and $|x| \geq t$}\label{subsecweight}

Since the (massive) particles follow timelike trajectories, most of the matter should be located in the interior of the light cone when $t \gg 1$. We then expect Vlasov fields to enjoy stronger decay estimates near and in the exterior of the light cone $t=|x|$. Moreover, as for electromagnetic fields, we expect the four-current density $J(f)$ to have a better behavior in the good null directions $(L,e_\theta,e_\varphi)$. For this, we introduce $(v^L,v^{\underline{L}},v^{e_1},v^{e_2})$, the null components of the four-momentum vector $\mathbf{v}$, and $\slashed{v}=(v^{e_\theta},v^{e_\varphi})$, its angular part, so that
$$\mathbf{v}=v^L L+ v^{\underline{L}} \underline{L}+v^{e_\theta}e_\theta+v^{e_\varphi}e_\varphi, \qquad v^L=\frac{v^0+\frac{x_i}{r}v^i}{2} , \qquad v^{\underline{L}}=\frac{v^0-\frac{x_i}{r}v^i}{2}, \qquad |\slashed{v}|^2=|v^{e_\theta}|^2+|v^{e_\varphi}|^2.$$
Sometimes, we will write $v^{\underline{L}}(x)$ or $v^{\underline{L}}(\omega)$ in order to emphasise the dependence of this quantity in the spherical variables $\omega=x/|x|$. For convenience, we define as well
$$ \widehat{v}^0 := \frac{v^0}{v^0}=1, \qquad  \widehat{v}^L := \frac{v^L}{v^0}, \qquad  \widehat{v}^{\underline{L}} := \frac{v^{\underline{L}}}{v^0}, \qquad \slashed{\widehat{v}}:= \frac{\slashed{v}}{v^0} , \qquad \widehat{v}^{e_A} := \frac{v^{e_A}}{v^0}, \quad A \in \{ \theta,\varphi \}.$$
In the linear setting, the improved decay properties can be obtained by exploiting the weight
\begin{equation*}
 \langle x-t\widehat{v} \rangle, \qquad \qquad \T_0(\langle x-t\widehat{v} \rangle)= \widehat{v}^\mu \partial_{x^\mu} (\langle x-t\widehat{v} \rangle)=0,
 \end{equation*}
which is then constant, equal to $\langle x \rangle$, along the linear flow $t \mapsto (t,x+t\widehat{v} , v)$. Consequently, if $\T_0(f)=0$, one can propagate boundedness for $\langle x-t\widehat{v} \rangle^a f$ and obtain extra decay for $f$ in the favourable regions using the next inequalities.
\begin{Lem}\label{gainv}
For all $(t,x,v) \in \R_+ \times \R^3_x \times \R^3_v$, there holds
$$1+|\slashed{v}|^2  \leq 4 \langle v \rangle^2 \widehat{v}^{\underline{L}}, \qquad \qquad \widehat{v}^{\underline{L}}(x) \lesssim \frac{\max(t-|x|,1)}{\langle t+|x| \rangle} \langle x-t\widehat{v} \rangle, \qquad \qquad \big|\slashed{\widehat{v}}\big|(x) \lesssim \frac{\langle x-t\widehat{v} \rangle}{\langle t+|x| \rangle}.$$
It implies in particular
\begin{equation*}
 \forall \; (t,x,v) \in \R_+ \times \R^3_x \times \R^3_v, \qquad 1 \lesssim \frac{\max(t-|x|,1)}{\langle t+|x| \rangle}\langle v \rangle^2 \langle x-t\widehat{v} \rangle.
 \end{equation*}
\end{Lem}
\begin{proof}
The first inequality is implied by the mass-shell relation $\eta(\mathbf{v},\mathbf{v})=-1$, which implies $-4v^Lv^{\underline{L}}+|\slashed{v}|^2=-1$, and $v^L \leq v^0$. For the second one, use first
$$  2 tv^{\underline{L}}(x)=(t-|x|)v^0-\frac{x_i}{|x|}v^0\big(t\widehat{v}^i-x^i \big), \qquad  |x|v^{\underline{L}}(x)=(|x|-t|)v^{\underline{L}}(x)+tv^{\underline{L}}(x)$$
in order to deal with the interior of the light cone $t \geq |x|$. If $|x| \geq t$, remark further that $|x|-t\leq |x-t\widehat{v}|$. Finally, for the last one we have
$$ |\slashed{v}(x)|^2=\frac{1}{|x|^2}\sum_{1 \leq i < j \leq 3} |x^iv^j-x^jv^i|^2 , \quad \; x^iv^j-x^jv^i=v^j(x^i-t\widehat{v}^i)-v^i(x^j-t\widehat{v}^j)=\frac{x^i}{t} (t\widehat{v}^j-x^j) - \frac{x^j}{t} (t\widehat{v}^i-x^i)    .$$
\end{proof}

\subsection{Linear Vlasov fields}\label{SubseclinVla}

The purpose of this section, where we will recall results concerning the asymptotic behavior of solutions to the free massive relativistic transport equation $\T_0(f)=\partial_t f+ \widehat{v} \cdot \nabla_x f=0$, is twofold.
\begin{enumerate}
\item We will initialise the Picard iteration for the proof of Theorem \ref{Theo1} by considering $f_1(t,x,v):=f_\infty (x-t\widehat{v},v)$, the unique solution to the free Vlasov equation equal to $f_\infty$ at $t=0$. Moreover, certain of the estimates proved in this section will in fact be useful later in the paper as well.
\item The results of this section will give insights concerning distribution functions satisfying modified scattering $f(t,x+t\widehat{v}+\C_{t,v},v) \to f_\infty(x,v)$. In particular, the velocity averages of such a distribution function, which decays as $t^{-3}$, behave as the ones of $f_1$. More precisely $|J(\widehat{Z}^\beta f)-J(\widehat{Z}^\beta f_1) |(t,x)= O(t^{-4}\log(t))$. It will motivate us to introduce the asymptotic Maxwell equations.
\end{enumerate}
Let $N \in \mathbb{N}$ and assume that $f_\infty$ satisfy the assumptions of Theorem \ref{Theo1}. Recall that $[\T_0 , \widehat{Z}]=0$ for any $\widehat{Z} \in \K \setminus \{ S \}$, $[\T_0 , \widehat{S}]=\T_0$ and $\T_0(\langle v \rangle)=\T_0 (x-t\widehat{v})=0$. Thus, for any $(N_z,N_v) \in \mathbb{N}^2$,
$$ \forall \, (t,x,v) \in \R_+ \times \R^3_x \times \R^3_v, \quad \langle x-t \widehat{v} \rangle^{N_z} \langle v \rangle^{N_v} \widehat{Z}^\beta f_1(t,x,v) = \langle x \rangle^{N_z} \langle v \rangle^{N_v} \widehat{Z}^\beta f_1 (0,x,v) , \qquad |\beta| \leq N+1. $$
Recall that $\beta_H$ denotes the number of homogeneous vector fields composing $\widehat{Z}^\beta$. We then deduce, in view of the decay assumption \eqref{eq:assumpfinfty} on $f_\infty$, that 
\begin{equation}\label{eq:boundlin}
\hspace{-10mm} \forall \, (t,x,v) \in \R_+ \times \R^3_x \times \R^3_v, \qquad \langle x-t \widehat{v} \rangle^{12+N-\beta_H} \langle v \rangle^{25} \big|\widehat{Z}^\beta f_1\big|(t,x,v)  \lesssim \sqrt{\epsilon}, \qquad \qquad |\beta| \leq N+1. 
 \end{equation} 
 Moreover, using the equation $\T_0(\widehat{Z}^\beta f)=0$ and performing integration by parts in $x$, we have further the conservation of the spatial averages,
\begin{equation}\label{eq:consspatf1}
\hspace{-6mm} \forall \, (t,v) \in \R_+ \times \R^3_v , \quad \langle v \rangle^{N_v} \int_{\R^3_x} \widehat{Z}^\beta f_1(t,x,v) \dr x = \langle v \rangle^{N_v} \int_{\R^3_x} \widehat{Z}^\beta f_1(0,x,v) \dr x , \qquad \qquad |\beta| \leq N+1 .
 \end{equation}
\subsubsection{Asymptotic behavior of the current density}\label{subsubseclin}  The goal now is to derive the large time behavior of $J(\widehat{Z}^\beta f_1)$, which is closely related to the spatial average of $\widehat{Z}^\beta f_1$. For this, we recall a standard identity, proved for instance in \cite[Lemma~$2.9$]{scat}.
\begin{Lem}\label{cdv}
For all $(t,x) \in \R_+^* \times \R^3$, the Jacobian determinant of $v \mapsto x-t \widehat{v}$ is equal to $-t^3 |v^0|^{-5}$.
\end{Lem} 
 
In order to be able to apply the next results later in this paper, we perform an analysis in a slightly more general setting. Let
$$
f: [T,+\infty[ \times \R^3_x \times \R^3_v \rightarrow \R, \qquad \qquad h(t,z,v) := f(t,z+t\widehat{v},v), \qquad \qquad T \geq 0,
$$
and remark that $h$ is contant if $f=f_1$. Here, we will assume that $f$ behaves well along the linear characteristics $t \mapsto (t,x+t\widehat{v},v)$ but we will not require $h$ to be constant. Though, we stress that in the next statements, all the norms appearing in the right hand sides of the inequalities will be conserved in time if $f=f_1$. 

The optimal $L^\infty_x$ and $L^2_x$ decay can be obtained through the next estimate.
\begin{Pro}\label{Prosimpledecay}
There holds, for all $(t,x) \in [T,+ \infty[ \times \R^3$,
$$ \int_{\R^3_v} |f|(t,x,v) \dr v \lesssim \frac{1}{\langle t+|x| \rangle^3}\sup_{(z,v) \in \R^3 \times \R^3} \langle x-t\widehat{v} \rangle^4 \, \langle v \rangle^{5}\big|  f (t,x,v) \big|.$$
Similarly, we have for all $t \geq T$,
 \begin{align*}
\int_{\R^3_x} \langle t+|x| \rangle^{3} \bigg|\int_{\R^3_v} |f|(t,x,v) \mathrm{d} v \bigg|^2  \dr x   \lesssim   \int_{\R^3_x} \int_{\R^3_v} \langle x-t\widehat{v} \rangle^4 \, \langle v \rangle^{5}\big|  f (t,x,v) \big|^2 \mathrm{d} v   \dr x.
\end{align*}.
\end{Pro}
\begin{proof}
By applying Hölder's inequality in the variable $v$, we reduce the problem to the proof of
\begin{equation}\label{eq:Lemfordecayvelocityaverage00}
\int_{\R^3_v} \frac{\dr v}{\langle x - t\widehat{v} \rangle^4 \, \langle v \rangle^5} \lesssim \frac{1}{\langle t+|x| \rangle^3}.
\end{equation}
If $t \leq 4 \langle x \rangle$, we have $2\langle x - t\widehat{v} \rangle \geq \langle t+|x| \rangle$ and the inequality holds. Otherwise, we use Lemma \ref{cdv} in order to perform the change of variables $y=x-t\widehat{v}$. Since $ t \gtrsim \langle t+|x| \rangle$ in that case, this concludes the proof.
\end{proof}

Applying this result to $\langle x-t\widehat{v}\rangle^a \, \langle v \rangle^{2a} f(t,x,v)$ and using Lemma \ref{gainv}, one can derive improved decay estimates for Vlasov fields near and in the exterior of the light cone $t=|x|$. 
\begin{Cor}\label{Corpointestistrong}
For any $a \in \R_+$, we have
$$ \forall \, (t,x) \in [T,+\infty[ \times \R^3, \qquad \quad \int_{\R^3_v} |f|(t,x,v) \dr v \lesssim \frac{| \max (t-|x|,1) |^a}{\langle t+|x| \rangle^{3+a}}\sup_{(z,v) \in \R^3 \times \R^3} \langle x-t\widehat{v} \rangle^{4+a} \, \langle v \rangle^{5+2a}\big|  f (t,x,v) \big|.$$
\end{Cor}

By requiring more regularity on $f$, one can refine the previous result. In this paper, we will mainly use $L^2_x$ estimates but analogous ones hold in $L^p_x$, $1 \leq p \leq \infty$.
\begin{Pro}\label{Protechfordecay}
For all $t \geq T$, we have 
\begin{align*}
\int_{\R^3_x} \langle t+|x| \rangle^{5} \bigg|\int_{\R^3_v} f(t,x,v) \mathrm{d} v -  \frac{\mathds{1}_{|x|<t}}{t^3}\int_{\R^3_y} \Big[ |v^0|^5 f \Big] \bigg( t,y,\frac{\widecheck{\; x \;}}{t}  \bigg)& \dr y \bigg|^2  \dr x  \\
\lesssim &  \sup_{|\beta| \leq 1} \int_{\R^3_x} \int_{\R^3_v} \langle x-t\widehat{v} \rangle^8 \, \langle v \rangle^{14}\big| \widehat{Z}^\beta f (t,x,v) \big|^2 \mathrm{d} v   \dr x.
\end{align*}
\end{Pro}
\begin{proof}
Let $t \geq T$. Recall from Lemma \ref{gainv} that $\langle x \rangle \lesssim \langle x-t \widehat{v} \rangle \, \langle v \rangle^2$ for all $|x| \geq t$. Consequently, using also the Cauchy-Schwarz inequality as well as $v \mapsto \langle v \rangle^{-4} \in L^1(\R^3_v)$,
\begin{align*}
\int_{|x| \geq t}\langle x \rangle^5 \bigg|\int_{\R^3_v} f(t,x,v) \mathrm{d} v \bigg|^2  \dr x  & \lesssim \int_{|x| \geq t} \bigg|\int_{\R^3_v} \langle x-t\widehat{v} \rangle^{\frac{5}{2}} \, \langle v \rangle^5 \, f(t,x,v) \mathrm{d} v \bigg|^2  \dr x \\
& \leq \int_{|x| \geq t} \int_{\R^3_v} \langle x-t\widehat{v} \rangle^5 \, \langle v \rangle^{14} \, |f|^2(t,x,v) \mathrm{d} v   \dr x.
\end{align*}
Consider now the case $|x| <t$ and $t\geq 1$. We perform the change of variables $y=x-t\widehat{v}$ in order to get
$$ t^3 \int_{\R^3_v} f(t,x,v) \mathrm{d} v = \int_{|y-x|<t} \Big[ |v^0|^5 h\Big]\bigg( t,y, \frac{\widecheck{x-y}}{t} \bigg) \mathrm{d} y,$$
where we recall that $\widecheck{\; \color{white} x \color{black} \; }$, the inverse of the relativistic speed operator, is introduced in Definition \ref{Definvrela}. Let $h_0(t,y,w) :=[|v^0|^5 h](t,y,\widecheck{w})= (1-|w|^2)^{-5/2} h(t,y,\widecheck{w})$, so that
$$ t^3 \int_{\R^3_v} f(t,x,v) \mathrm{d} v = \int_{|y-x|<t} \Big[ |v^0|^5 h\Big]\bigg( t,y, \frac{\widecheck{\; x \;}}{t} \bigg) \mathrm{d} y+ \int_{|y-x|<t} \int_{\tau=0}^1 \frac{y}{t} \cdot \nabla_w h_0 \Big( t,y, \frac{ x-\tau y }{t} \Big) \dr \tau \mathrm{d} y .$$
We point out that $|x| < t$ and $|y-x|<t$ implies $|x-\tau y|<t$. We then get from the Cauchy-Schwarz inequality in $y$ that
$$ \int_{|x|<t} t^{5} \bigg|\int_{\R^3_v} f(t,x,v) \mathrm{d} v- \frac{1}{t^3}\int_{\R^3_y} \Big[ |v^0|^5 h\Big]\bigg( t,y, \frac{\widecheck{ \; x \;}}{t} \bigg) \mathrm{d} y \bigg|^2  \dr x  \leq  \mathcal{I}_1+  \mathcal{I}_2 ,$$
where
\begin{align*}
  \mathcal{I}_1  &:= \int_{|x|<t} \frac{1}{t} \int_{|y-x|\geq t} \langle y \rangle^4 \Big| |v^0|^5 h\Big|^2\bigg( t,y, \frac{\widecheck{\; x \;}}{t} \bigg) \mathrm{d} y   \dr x , \\
  \mathcal{I}_2& := \int_{|x| < t} \frac{1}{t} \int_{|y-x|<t} \int_{\tau=0}^1 \frac{|y|^2}{t^2} \langle y \rangle^4 \big|\nabla_w h_0 \big|^2 \Big( t,y, \frac{ x-\tau y }{t} \Big) \dr \tau \mathrm{d} y  \dr x.
 \end{align*}
In order to bound $\mathcal{I}_1$ recall that $|x|<t$ and remark that, for $v=\widecheck{x/t}$ and any $y \in \R$ such that $|y-x| \geq t$,
$$
 1  = |v^0|^2 \left( 1-\frac{|x|^2}{t^2}\right)  \leq |v^0|^2 \frac{|y|(t+|x|)}{t^2} \leq 2 \frac{|y||v^0|^2}{t}.
$$
It allows us to get
$$
\mathcal{I}_1  \leq 4 \int_{|x|<t} \frac{1}{t^3} \int_{|y-x|\geq t}\langle y \rangle^6 \Big| |v^0|^7 h\Big|^2\bigg( t,y, \frac{\widecheck{ \; x \;}}{t} \bigg) \mathrm{d} y   \dr x \leq 4 \int_{\R^3_y} \int_{|x|<t}  \langle y \rangle^6 \Big| |v^0|^7 h\Big|^2\bigg( t,y, \frac{\widecheck{ \; x \;}}{t} \bigg)    \frac{\dr x}{t^3} \dr y .
$$
Thus, performing the reverse change of variables $x=t\widehat{v}$, one obtains
$$
\mathcal{I}_1 \leq  4 \int_{\R^3_y} \int_{\R^3_v} \langle y \rangle^6 |v^0|^9 | h|^2 (t,y,v) \dr v \dr y . 
$$
For $\mathcal{I}_2$, since $|\nabla_w \widecheck{w}| \lesssim |1-|w|^2|^{-3/2}=\langle \widecheck{w} \rangle^3=|v^0(\widecheck{w})|^3$, we have
$$ \left| \nabla_w h_0 \right|(t,y,w) \lesssim \Big| |v^0|^7 h \Big|(t,y,\widehat{w}) +\Big| |v^0|^8 \nabla_v h \Big|(t,y,\widehat{w}) .$$
We then deduce from Fubini's theorem that
$$\mathcal{I}_2 = \int_{\tau=0}^1 \int_{|y|<2t} \int_{\{ x \in \R^3 \, | \, |x|<t, \, |x-y|<t\}} \langle y \rangle^6 \Big| |v^0|^7 h \Big|^2 \bigg( t,y, \frac{\widecheck{ x-\tau y }}{t} \bigg)  +\langle y \rangle^6 \Big| |v^0|^8 \nabla_v h \Big|^2 \bigg( t,y, \frac{\widecheck{ x-\tau y }}{t} \bigg)\frac{\dr x}{t^3} \mathrm{d} y  \dr \tau .$$
Hence, the change of variables $x/t=\tau y/t+\widehat{v}$ yields
\begin{align*}
\mathcal{I}_2 & \leq \int_{\tau=0}^1 \int_{|y|<2t} \int_{\R^3_v} \langle y \rangle^6  |v^0|^9 |h |^2 ( t,y, v )  +\langle y \rangle^6 |v^0|^{9} \big|v^0\nabla_v h \big|^2 ( t,y,v) \dr v \mathrm{d} y  \dr \tau \\
& \leq  \int_{\R^3_y} \int_{\R^3_v} \langle y \rangle^6  |v^0|^9 \Big(|h |^2  + \big|v^0\nabla_v h \big|^2\Big) ( t,y,v) \dr v \mathrm{d} y   .
\end{align*}
To conclude it suffices to use $|v^0 \nabla_v h|(t,y,v)\lesssim \langle y \rangle \, \sup_{|\beta|=1} \big| \widehat{Z}^\beta f \big|(t,y+t\widehat{v},v)$, which will be proved thereafter in Lemma \ref{Lemrelftoh}, and to make, for fixed $v$, the linear change of variables $y=x-t\widehat{v}$.
\end{proof}

We now would like to relate $\widehat{Z}^\beta f$ to weighted derivatives of $h$ for two reasons. First, in order to conclude the proof of the previous proposition. Secondly, if $f$ and its derivatives satisfy \textit{linear scattering}, $\widehat{Z}^\beta f (t,x+t\widehat{v},v)$ converges as $t \to +\infty$ to $h_\infty^\beta (x,v)$, we would like to be able to express $h^\beta$ in terms of derivatives of $h_\infty$, the limit of $h$. If $f=f_1$, it will give us the relation between $\widehat{Z}^\beta f_1$ and the derivatives of $f_\infty$. 
\begin{Lem}\label{Lemrelftoh}
The following relations hold, for all $(t,z,v) \in [T,+\infty[ \times \R^3_z \times \R^3_v$,
$$ \big( \partial_{x^k} f \big)(t,z+t\widehat{v},v) = \partial_{z^k} h(t,z,v), \qquad \big(S f \big) (t,z+t\widehat{v},v)  = S h(t,z,v), \qquad  \big( \widehat{\Omega}_{ij} f \big)(t,z+t\widehat{v},v)  = \widehat{\Omega}_{ij} h(t,z,v),$$
for any $1 \leq i < j \leq 3$ and $ 1 \leq k \leq 3$. For the time derivative and the Lorentz boosts, we have
\begin{align*}
\big(\partial_t f \big)(t,z+t\widehat{v},v) & =\partial_t h(t,z,v) -\widehat{v}\cdot \nabla_z h(t,x,v) ,\\
\big(\widehat{\Omega}_{0k}  f\big)(t,z+t\widehat{v},v) & = v^0 \partial_{v^k} h(t,z,v)-z^k\widehat{v}\cdot \nabla_z h(t,z,v)+z^k\partial_t h(t,z,v)+\widehat{v}^k  t\partial_t  h(t,z,v) .
\end{align*}
\end{Lem}
\begin{proof}
Since $f(t,x,v)=h(t,x-t\widehat{v},v)$ and $h(t,z,v)=f(t,z+t\widehat{v},v)$, we have, for any $ 1 \leq k \leq 3$,
\begin{align*}
 \partial_{x^k}f(t,x,v)&= \partial_{z^k}h(t,x-t\widehat{v},v), \qquad \qquad \partial_t f(t,x,v) = \big(\partial_t h -\widehat{v} \cdot \nabla_z h\big)(t,z-t\widehat{v},v) , \\
  v^0 \partial_{v^k}f(t,x,v)&= v^0 \partial_{v^k}h(t,x-t\widehat{v},v)-t\partial_{z^k}h(t,x-t\widehat{v},v)+t\widehat{v}_k \big[ \widehat{v} \cdot \nabla_z h \big](t,x-t\widehat{v},v)  .
\end{align*}
It implies the stated relations. For instance,
\begin{align*}
\big( \widehat{\Omega}_{ij} f \big)(t,z+t\widehat{v},v)&=\!\big(z^i+t\widehat{v}^i \big)\partial_{x^j}f(t,z+t\widehat{v},v)-\big(z^j+t\widehat{v}^j\big)\partial_{x^i}f(t,z+t\widehat{v},v)+\!\big[ v^i \partial_{v^j}f-v^j \partial_{v^i}f \big](t,z+t\widehat{v},v) \\
&= z^i \partial_{z^j} h(t,z,v)-z^j \partial_{z^i} h(t,z,v)+v^i \partial_{v^j} h(t,z,v)-v^j \partial_{v^i} h(t,z,v)=\widehat{\Omega}_{ij}h(t,z,v).
\end{align*}
\end{proof}

Motivated by these identities, we introduce the following set of weighted derivatives, which will turn out to be very convenient to use in the context of modified scattering as well.
\begin{Def}
Let $\K^\infty$ be the set composed of the following vector fields, defined on $\R_t \times \R^3_z \times \R^3_v$,
$$  \partial_{z^k}, \quad \widehat{\Omega}_{ij}\!=\!z^i \partial_{z^j}-z^j \partial_{z^i}+v^i \partial_{v^j}-v^j \partial_{v^i}, \quad \widehat{S}\!=\!t\partial_t +z^i \partial_{z^i}+3 , \quad \partial_t^\infty \! := \partial_t-\widehat{v} \cdot \nabla_z , \quad \widehat{\Omega}_{0k}^\infty \! := \! z^k \partial_t^\infty +\widehat{v}^k  t\partial_t+v^0 \partial_{v^k}, $$
where $1 \leq k \leq 3$ and $1 \leq i < j \leq 3$. We will denote by $\widehat{Z}_\infty$ a generic vector field of $\K^\infty$ and we consider an ordering on this set compatible with the one on $\K$, in the sense that the label of $\partial_t^\infty$ (respectively $\partial_{z^k}$, $\widehat{\Omega}_{ij}$, $\widehat{\Omega}_{0k}^\infty$ and $\widehat{S}$) corresponds to the one of $\partial_t$ (respectively $\partial_{x^k}$, $\widehat{\Omega}_{ij}$, $\widehat{\Omega}_{0k}$ and $\widehat{S}$). Note then that, given a multi-index $\kappa$, $\kappa_H$ still corresponds to the number of homogeneous vector fields composing $\widehat{Z}_\infty^\kappa$.
\end{Def}
\begin{Rq}
When applied to a function independent of $t$, such as $f_\infty$, the operators $\partial_t^\infty$, $\widehat{\Omega}_{0k}^\infty$ and $\widehat{S}$ correspond to $-\widehat{v}\cdot \nabla_z$, $z^i \partial_{z^i}+3$ and $-z^k \widehat{v} \cdot \nabla_z +v^0 \partial_{v^k}$.
\end{Rq}

One can derive the next estimate of $J(\widehat{Z}^\beta f_1)$ by applying Proposition \ref{Protechfordecay} to $f=\widehat{v}_\nu \widehat{Z}^\beta f_1$ as well as the uniform bounds \eqref{eq:boundlin} and the conservation of the spatial averages \eqref{eq:consspatf1}. Note further that $\widehat{v}_\nu = \frac{ x_\nu}{t}$ for $v= \frac{\widecheck{\, x \,}}{t}$.
\begin{Cor}\label{Corlinbound}
Recall that $\widehat{v}_0=-1$ and $x_0=-t$. For any $|\beta| \leq N$ and all $t \geq T$,
$$ \int_{\R^3_x} \langle t+|x| \rangle^{5} \bigg|\int_{\R^3_v}\widehat{v}_\nu \widehat{Z}^\beta f_1(t,x,v) \mathrm{d} v -  \frac{\mathds{1}_{|x|<t}}{t^3}\frac{x_\nu}{t}\int_{\R^3_z} \Big[ |v^0|^5 \widehat{Z}^\beta_\infty f_\infty \Big] \bigg( z,\frac{\widecheck{\; x \;}}{t} \bigg) \dr y \bigg|^2  \dr x \lesssim  \epsilon .$$
\end{Cor}

In the context of the Vlasov-Maxwell system, $f$ does not verify linear scattering but converges along correction of the linear characteristics $t \mapsto (t,x+t\widehat{v}+\C_{t,v},v)$ to a limit distribution, say $h_\infty$. However, if $\widehat{Z} \in \K$ is a homogeneous vector field, $\widehat{Z} f$ does not satisfy modified scattering. Instead, we proved in \cite[Theorem~$2.11$]{scat} that
\begin{itemize}
\item there exists $\widehat{Z}^{\mathrm{mod}}$, a logarithmical correction of $\widehat{Z}$, such that modified scattering holds for $\widehat{Z}^{\mathrm{mod}} f$. Moreover, its limit distribution is $\widehat{Z}_\infty h_\infty$.
\item The spatial average of $\widehat{Z}f$ verifies $ \big| \int_{x} \widehat{Z} f(t,x,v) \dr x -\int_{z} \widehat{Z}_\infty h_\infty (z,v) \dr z \big| \lesssim t^{-1}\log^2(t)$. In some sense, although $\widehat{Z}f$ does not behave as a linear solution for large time, its spatial average does.
\end{itemize}
For this reason, the vector fields $\widehat{Z}_\infty$ remain relevant in the nonlinear problem. However, the coordinate system $(t,y,v)=(t,x+t\widehat{v},v)$ is not well adapted anymore in order to study the solutions for large times. 

\subsubsection{Improved estimates for the derivatives}\label{Subsubsecderivlin} It turns out that $J(\partial_{x^\mu} f_1)$ enjoys stronger decay properties than $J(f_1)$. This is captured by the identities
\begin{equation}
\hspace{-6mm} (t^2-r^2) \partial_t = t S-x^i\Omega_{0i}, \qquad \qquad (t^2-r^2)\partial_{x^k} = t\Omega_{0k}-x^k S-x^i \Omega_{ki}, \quad 1 \leq k \leq 3.   \label{eq:transladerivatives}
\end{equation}
As in the previous subsection, we first derive properties holding for a general distribution function $f$. By performing an induction, based on the relations \eqref{eq:transladerivatives} as well as integration by parts in $v$, one can derive the next result.
\begin{Lem}\label{LemgainderivVla}
Let $f :[T,+\infty[ \times \R^3_x \times \R^3_v \to \R$ be a sufficiently regular function. For any muti-index $\kappa$,
$$ \forall \, (t,x) \in [T,+\infty[ \times \R^3, \qquad \langle t-|x| \rangle^{|\kappa|} \bigg| \partial_{t,x}^\kappa \int_{\R^3_v} f(t,x,v) \dr v \bigg| \lesssim \sup_{|\beta| \leq |\kappa|} \int_{\R^3_v} \big| \widehat{Z}^\beta f \big|(t,x,v) \dr v.$$
\end{Lem}

Then, by applying Lemma \ref{gainv}, one can in fact strengthen this last estimate near the light cone.

\begin{Cor}\label{CorgainderivVla}
For any muti-index $\kappa$,
$$ \forall \, (t,x) \in [T,+\infty[ \times \R^3, \qquad \langle t+|x| \rangle^{|\kappa|} \bigg| \partial_{t,x}^\kappa \int_{\R^3_v} f(t,x,v) \dr v \bigg| \lesssim \sup_{|\beta| \leq |\kappa|} \int_{\R^3_v}  \langle x-t\widehat{v} \rangle^{|\kappa|} \, \langle v \rangle^{2|\kappa|} \big| \widehat{Z}^\beta f \big|(t,x,v) \dr v.$$
\end{Cor}

These two results will be sufficient for the purpose of this paper. However, for completeness and in order to ensure that the source term of the commuted asymptotic Maxwell equations, introduced below in Section \ref{SecMaxasymp}, have the correct form, we identify now the leading order term in the asymptotic expansion of $J(\partial_{x^\mu} f_1)$. Note that one cannot directly obtain it through Proposition \ref{Protechfordecay} since the spatial average of $\partial_{x^\mu} f_1$ vanishes. It turns out that it is easier to derive first the asymptotic behavior of the following derivatives of $f_1$,
$$ Y_{\mu \nu}:=\widehat{v}_\mu \partial_{x^\nu}-\widehat{v}_\nu \partial_{x^\mu}, \qquad 0 \leq \mu , \, \nu \leq 3 .$$
The average in $v$ of $Y_{\mu \nu} f$ is the source term of the wave equation satisfied by the cartesian component $F_{\mu \nu}$ of the electromagnetic field if $(f,F)$ is a solution to the Vlasov-Maxwell system. These derivatives enjoy a kind of \textit{null condition} and behave better than $\partial_{x^\mu}f_1$. The large time behavior of these latter quantities will then be obtained by exploiting the relations
\begin{equation}\label{link:Ytod} 
\partial_t= |v^0|^2 \big( \T_0+ \widehat{v}^i Y_{0i} \big), \qquad \quad \qquad \partial_{x^k}=|v^0|^2 \big( -Y_{0k}-\widehat{v}^k \T_0+\widehat{v}^i Y_{ki} \big), \quad 1 \leq k \leq 3 .
\end{equation}

\begin{Lem}\label{LemYderivasymp}
We have, for any $0 \leq \mu < \nu \leq 3$ and all $(t,x) \in [T+\infty[ \times \R^3$,
$$ \bigg|(-1)^{\delta_\mu^0} \,  t \int_{\R^3_v} Y_{\mu \nu} f(t,x,v) \dr v - \int_{\R^3_v} \widehat{\Omega}_{\mu \nu} f(t,x,v) \dr v \bigg| \lesssim \frac{1}{\langle t+|x| \rangle} \sup_{|\beta| \leq 1}  \int_{\R^3_v}\langle v \rangle^{2} \, \langle x-t\widehat{v} \rangle^2 \, \big| \widehat{Z}^\beta f \big|(t,x,v) \dr v.$$
\end{Lem}
\begin{proof}
Note first the relations
\begin{align*}
 -tY_{0k}&=t\left(\partial_{x^k}+\widehat{v}_k \partial_t \right)=\Omega_{0k}-(x^k-t\widehat{v}^k) \partial_t=-v^0 \partial_{v^k}+\widehat{\Omega}_{0k}-\widehat{v}^k- \partial_t \big[ (x^k-t\widehat{v}^k) \cdot \big],  \\
 tY_{ij}&=t\left(\widehat{v}_i\partial_{x^j}-\widehat{v}_j \partial_{x^i} \right) =-\big(v^i \partial_{v^j}-v^j \partial_{v^i} \big)+ \widehat{\Omega}_{ij}-(x^i-t\widehat{v}^i) \partial_{x^j}+(x^j-t\widehat{v}^j) \partial_{x^i}.
\end{align*}
Then, perform integration by parts in $v$ and apply Corollary \ref{CorgainderivVla} to $(x^\ell-t\widehat{v}^\ell)f$, $1 \leq \ell \leq 3$. It remains to remark that, for all $\widehat{Z} \in \K$, we have $|\widehat{Z} (x-t\widehat{v})|\lesssim \langle x-t\widehat{v} \rangle$.
\end{proof}

We are finally able to describe the asymptotic behavior of the derivatives of the current density $J(f_1)$. We focus on the time decay as Vlasov fields enjoy strong spatial decay in the exterior of the light cone $\{|x| \geq t\}$.
\begin{Cor}\label{CorlinexpanderivativJ}
Let $1 \leq k \leq 3$ and define the operators
\begin{equation}\label{defDop}
 D_t : h \mapsto -\widehat{\Omega}_{0i}^\infty \big( v^iv^0 h \big)-h, \qquad \qquad D_{x^k} : h \mapsto \widehat{\Omega}_{0k}^\infty \big( |v^0|^2 h \big)+\widehat{\Omega}_{ki}\big( v^i v^0 h \big)  .
 \end{equation}
  There holds, for all $t \geq 1$,
\begin{align*}
\int_{\R^3_x} t^{5} \bigg|J \big( \partial_t f_1 \big)_\nu (t,x) -  \frac{\mathds{1}_{|x|<t}}{t^4}\frac{x_\nu}{t}  \int_{\R^3_y}  \Big[ |v^0|^5 D_t f_\infty  \Big] \bigg(  y,\frac{\widecheck{\; x \;}}{t} \bigg) \dr y +\delta_\nu^0\frac{\mathds{1}_{|x|<t}}{t^4} \int_{\R^3_y}  \Big[ |v^0|^5  f_\infty  \Big] \bigg(  y,\frac{\widecheck{\; x \;}}{t} \bigg) \dr y \bigg|^2 \! \dr x  & \lesssim  \epsilon  , \\
\int_{\R^3_x} t^{5} \bigg|J \big( \partial_{x^k} f_1 \big)_\nu (t,x) -  \frac{\mathds{1}_{|x|<t}}{t^4}\frac{x_\nu}{t}  \int_{\R^3_y}  \Big[ |v^0|^5 D_{x^k} f_\infty \Big] \bigg(  y,\frac{\widecheck{\; x \;}}{t} \bigg) \dr y-\delta_\nu^k \frac{\mathds{1}_{|x|<t}}{t^4} \int_{\R^3_y}  \Big[ |v^0|^5  f_\infty  \Big] \bigg(  y,\frac{\widecheck{\; x \;}}{t} \bigg) \dr y \bigg|^2 \! \dr x  & \lesssim  \epsilon .
\end{align*}
\end{Cor}
\begin{proof}
Fix $0 \leq \nu \leq 3$ and $1 \leq k \leq 3$. Since $\partial_t=-\widehat{v} \cdot \nabla_x +\T_0$ and $\T_0(\widehat{v}^k f_1)=0$, it suffices to deal with the case of the spatial derivatives. Next, using the relation \eqref{link:Ytod}, we get
$$ \widehat{v}_\nu \partial_{x^k}f_1=|v^0|^2 \big[ -Y_{0k}-\widehat{v}^k \T_0+\widehat{v}^i Y_{ki} \big](\widehat{v}_\nu f_1)= - Y_{0k} \big(\widehat{v}_\nu |v^0|^2 f_1 \big)+Y_{ki}\big(\widehat{v}_\nu v^i v^0 f_1 \big).$$
Furthermore,
$$  \widehat{\Omega}_{0k} \big(\widehat{v}_\nu |v^0|^2 f_1 \big)+\widehat{\Omega}_{ki}\big(\widehat{v}_\nu v^i v^0 f_1 \big)=  \widehat{v}_\nu  \widehat{\Omega}_{0k} \big( |v^0|^2 f_1 \big)+\widehat{v}_\nu\widehat{\Omega}_{ki}\big( v^i v^0 f_1 \big)  + \delta_\nu^k (1-|\widehat{v}|^2)|v^0|^2 f_1   $$
and $(1-|\widehat{v}|^2)|v^0|^2=1$. Then, apply Lemma \ref{LemYderivasymp} to $f=\widehat{v}_\nu |v^0|^2 f_1$ as well as $f=\widehat{v}_\nu v^iv^0 f_1$ and then Proposition \ref{Protechfordecay} to $\widehat{v}_\nu  \widehat{\Omega}_{0k} \big( |v^0|^2 f_1 \big)$, $\widehat{v}_\nu\widehat{\Omega}_{ik}\big(v^i v^0 f_1 \big) $ as well as $f_1$, which are all solutions to the linear Vlasov equation. It remains to use the conservation laws \eqref{eq:boundlin}--\eqref{eq:consspatf1}.
\end{proof}

\subsection{The Vlasov equation in variables adapted to modified scattering}

Let $F$ be a sufficiently regular electromagnetic field. One of the main step of the proof of Theorem \ref{Theo1} consists in proving that under suitable assumptions on $F$, in particular $t^2 F(t,x+t\widehat{v})\sim \mathbb{F}[f_\infty](v)$, there exists a unique function $f$ satisfying $\T_F(f)=0$ and $f(t,x+t \widehat{v}+ \C_{t,v},v) \to f_\infty (x,v)$ as $t \to + \infty$. For this, it is convenient to work in a more appropriate coordinate system. We then introduce
\begin{equation}\label{defXC} \XX_\C : (t,x,v) \mapsto x+t\widehat{v}+\C_{t,v}, \qquad \qquad \C_{t,v}^k:= - \log(t)\frac{\delta_k^j-\widehat{v}_k \widehat{v}^j}{v^0} \widehat{v}^\mu \mathbb{F}_{\mu j}[f_\infty](v),
\end{equation}
and we remark that $(t,x,v) \mapsto (t, \XX_\C (t,x,v) , v)$ is a diffeomorphism from $\R_+^* \times \R^3_x \times \R^3_v$ to $\R_+^* \times \R^3_{z} \times \R^3_v$. To a distribution function $f: [T,+\infty[ \times \R^3_x \times \R^3_v \to \R$, we associate $g : [T,+\infty[ \times \R^3_{z} \times \R^3_v \to \R$ defined as
\begin{equation}\label{defg}
 \forall \, (t,z ,v) \in[T,+\infty[ \times \R^3_z \times \R^3_v, \qquad g(t,z , v) = f(t,z+t\widehat{v}+ \C_{t,v},v)=f\big(t,\mathrm{X}_\C(t,z,v),v\big).
 \end{equation}
Then, the modified scattering statement for $f$ is equivalent to $g(t,z,v) \to f_\infty (z,v)$. Furthermore, we have
\begin{align}
\partial_t g (t,z,v) & = \left(\partial_t  +\widehat{v}\cdot \nabla_x -\frac{\delta^i_j-\widehat{v}^i\widehat{v}_j}{tv^0} \widehat{v}^\mu {\mathbb{F}_\mu}^{j}[f_\infty](v) \partial_{x^i} \right)f \big(t,\XX_\C(t,z,v),v \big), \label{dtg} \\
\partial_{z^j} g(t,z,v) &= \partial_{x^j}f \big(t,\XX_\C(t,z,v),v \big) , \label{dzg} \\
v^0\partial_{v^j} g(t,z,v) &= \left( v^0\partial_{v^j}+t(\delta_j^i-\widehat{v}^j \widehat{v}^i)\partial_{x^i}+v^0\partial_{v^j}\C^i_{t,v}\partial_{x^i} \right) f \big(t,\XX_\C(t,z,v),v \big) \label{dvg}.
\end{align}
Consequently, $f(t,x,v)$ is solution to the Vlasov equation $\T_F(f)=0$ on $[T,+\infty[$ if and only if $g$ satisfies
$$ \forall \, (t,z,v) \in [T,+\infty[\times \R^3_z \times \R^3_v, \qquad \T_F^\infty(g)(t,z,v)=0,$$
where $\T_F^\infty$ is the Vlasov operator $\T_F$ in the coordinate system $(t,z,v)=(t,\XX_\C(t,x,v),v)$,
\begin{equation}\label{eq:Vlasov}
\T_F^\infty:=\partial_t-\frac{\delta_j^i-\widehat{v}_j \widehat{v}^i}{tv^0} \widehat{v}^\mu \Big( t^2  {F_{\mu}}^j(t,\XX_\C)-{\mathbb{F}_\mu}^{j}[f_\infty](v) \Big)\partial_{z^i}+\frac{\widehat{v}^\mu}{v^0} {F_{\mu}}^j(t,\XX_\C)\left( v^0\partial_{v^j}-v^0 \partial_{v^j} \C^i_{t,v} \partial_{z^i} \right).
\end{equation}
We will then study the equation $\T_F^\infty (g)=0$ with the asymptotic data $g(t,z,v) \to f_\infty (z,v)$. We will further control the derivatives $\widehat{Z}^\beta_\infty g$ up to order $N$. Note that they provide as much informations as the standard derivatives, in the sense that
$$ \sup_{|\beta| \leq p } \langle z \rangle^{\,p-\beta_H} \big| \widehat{Z}^\beta_\infty g \big|\lesssim  \sup_{\kappa_t+|\kappa_z|+|\kappa_t| \leq p} \big(t+\langle z \rangle \big)^{\kappa_t} \langle z \rangle^{|\kappa_z|} \, \langle v \rangle^{|\kappa_v|} \big| \partial_t^{\kappa_t} \partial_z^{\kappa_z} \partial_v^{\kappa_v} g \big| \lesssim \sup_{|\beta| \leq p } \langle z \rangle^{\,p-\beta_H} \big| \widehat{Z}^\beta_\infty g \big|,$$
on $ [T,+\infty[ \times \R^3_z \times \R^3_v$ and for any $p \leq N$. We chose to use them for the reasons presented in Section \ref{subsubseclin} and since it is much easier to relate these derivatives to $\widehat{Z}^\beta f$, which have to be controlled in order to estimate the source terms in the commuted Maxwell equations.
\begin{Lem}\label{Lemrelftog}
For $g(t,z,v)=f(t,z+t\widehat{v}+\C_{t,v},v)$, we have for any $1 \leq k \leq 3$ and $1 \leq i < j \leq 3$,
\begin{align}
 \big(\partial_{x^k} f \big) (t,z+t\widehat{v}+\C_{t,v},v)& = \partial_{z^k} g(t,z,v), \nonumber \\ \nonumber
\big(\partial_t  f \big)(t,z+t\widehat{v}+\C_{t,v},v)  &=\partial_t^\infty g(t,z,v) -\partial_t \C_{t,v}^\ell \partial_{z^\ell}g(t,z,v)  , \\ \nonumber
  \big(S f \big) (t,z+t\widehat{v}+\C_{t,v},v) & = S g(t,z,v)+ \big( \C^\ell_{t,v}-t\partial_t \C_{t,v}^\ell \big)\partial_{z^\ell} g(t,z,v), \\ \nonumber
  \big( \widehat{\Omega}_{ij} f \big)(t,z+t\widehat{v}+\C_{t,v},v) & = \widehat{\Omega}_{ij} g(t,z,v)+\big[\C^i_{t,v} \partial_{z^j}-\C^j_{t,v} \partial_{z^i} \big]g(t,z,v)- \big(v^i \partial_{v^j}-v^j\partial_{v^i} \big)\big( \C_{t,v}^\ell \big) \partial_{z^\ell} g(t,z,v) ,\\ \nonumber
\big(\widehat{\Omega}_{0k}  f \big)(t,z+t\widehat{v}+\C_{t,v},v)&= \widehat{\Omega}_{0k}^{\infty} g(t,z,v)+\C^k_{t,v} \partial_t^\infty g(t,z,v)-v^0 \partial_{v^k}\big( \C_{t,v}^\ell \big) \partial_{z^\ell} g(t,z,v) \\ 
& \quad  -\big(t \widehat{v}^k +z^k+\C^k_{t,v} \big)\partial_t \C_{t,v}^\ell \partial_{z^\ell}g(t,z,v) . \nonumber
\end{align}
\end{Lem}
\begin{proof}
Recall from Lemma \ref{Lemrelftoh} that if $h(t,z,v):=g(t,z-\C_{t,v},v)$, we have $\big( \widehat{Z} f \big) (t,x+t\widehat{v},v)=\widehat{Z}_\infty h(t,x,v)$ for any $\widehat{Z} \in \K$. It then remains to relate the derivatives of $h$ to the one of $g$.
\end{proof}

\subsection{Key elements of the proof}\label{Subseckeyidea}

We will construct the solution $(f,F)$ by a Picard iteration. For this, we have to study the Cauchy problem for both the Maxwell and the Vlasov equations with asymptotic data. 

\subsubsection{The Maxwell equations with scattering data}\label{subseubsecMaxide}

Given a sufficiently regular distribution function $f:[T,+\infty[ \times \R^3_x \times \R_v^3 \to \R$ satisfying modified scattering to $f_\infty$, we would like to prove a well-posedness result for the asymptotic Cauchy problem
\begin{equation}\label{eq:introkeyelMax}
\nabla^\mu G_{\mu \nu} = J(f)_\nu, \qquad \nabla^\mu {}^* \! G_{\mu \nu} =0, \qquad \qquad \qquad \lim_{r \to + \infty} r \underline{\alpha} (G)(r+u,r\omega)=\underline{\alpha}^{\infty} (u, \omega ).
\end{equation}
The regularity of $J(f)$ and the assumptions on $\underline{\alpha}^\infty$ will allow us to prove that the asymptotic Cauchy problem \eqref{eq:introkeyelMax} is well-posed in $L^\infty([T,+\infty[,H^N(\R^3))$. In view of the assumptions on $f$, our linear analysis suggests
\begin{enumerate}
\item that $J(f) \sim t^3$, so that we should not expect $G$ to decay faster than $t^{-2}$ in, say, the region $t \geq 2r$. This is related with the long range effect of the electromagnetic field, preventing the distribution function to scatter linearly. The problem then is that a naive energy estimate will not provide us enough informations on $G$ in order to close our Picard iteration.
\item Since the modified scattering dynamics of $f$ is captured by $f_\infty$, we expect $J(f)$ to have the same leading order term than the current density $J(f_1)$ of the linear solution, which is given by Corollary \ref{Corlinbound}. 
\item The same should hold for $J(\widehat{Z}^\beta f)$ and $ J(\widehat{Z}^\beta f_1)$, up to order $N-1$. However, since a loss of one derivative is required in the estimates of Corollary \ref{Corlinbound}, we will not be able to derive an asymptotic expansion for the top order derivatives $J(\widehat{Z}^\kappa f)$, $|\kappa|=N$. 
\end{enumerate}

\begin{Rq}
Since $f$ merely enjoys a modified scattering dynamics, the next terms in the asymptotic expansion of $J(f)$ should be polyhomogeneous. In contrast, the expansion of $J(f_1)$ is polynomial.
\end{Rq}
\begin{Rq}
For the forward problem, in order to prove global well-posedness for the small data solutions to the Vlasov-Maxwell system, one also has to deal with a difficulty specific to the top order derivatives. However, the loss of regularity was not a problem when we proved in \cite{scat} that asymptotic completeness hold.
\end{Rq}

Motivated by $1.$ and $2.$, we consider the solution $F^{\mathrm{asymp}}[f_\infty]$ to the asymptotic Maxwell equations,
$$\nabla^\mu F^{\mathrm{asymp}}_{\mu \nu}[f_\infty] = J^{\mathrm{asymp}}_\nu[f_\infty],  \qquad \qquad \nabla^\mu {}^* \! F^{\mathrm{asymp}}_{\mu \nu}[f_\infty] =0,$$
where $J^{\mathrm{asymp}}_\nu[f_\infty]$ is the leading order term of $J(f_1)_\nu$, as given by Corollary \ref{Corlinbound}. In order to uniquely define it, we further assume that at time $t=1$, the divergence free parts of both the electric and magnetic field of $F^{\mathrm{asymp}}[f_\infty]$ vanish. It implies that $F^{\mathrm{asymp}}[f_\infty](1,\cdot)$ is completely determined by 
$J^{\mathrm{asymp}}_0[f_\infty](1,\cdot)$ through a Poisson equation. It turns out that $F^{\mathrm{asymp}}[f_\infty]$ fully captures the asymptotic behavior of $G$.
\begin{itemize}
\item Along the trajectories of free massive particles, we expect 
\begin{equation}\label{eq:ketelconv}
|t^2G(t,x+t\widehat{v})-\mathbb{F}[f_\infty](v)| \lesssim \langle x \rangle^{\frac{3}{2}} \, \langle v \rangle^3 t^{-1/2},
\end{equation}
which is indeed a convergence estimate satisfied by $F^{\mathrm{asymp}}[f_\infty]$. In fact, by allowing the initial time, here set equal to $1$, to go to zero, we have in a certain sense $t^2 F^{\mathrm{asymp}}[f_\infty](t,t\widehat{v})=\mathbb{F}[f_\infty](v)$. We refer to Proposition \ref{Probehavior} for more details.
\item The decay of $G$ in the region where $|x| \geq 2t$, near spatial infinity, is governed by its pure charge part $\overline{F}$ introduced in \eqref{defFoverlin}, which is the tail of its electric field. It turns out that the same property holds for $F^{\mathrm{asymp}}[f_\infty]$ and its pure charge part is equal to the one of $G$.
\item The radiation field $\underline{\alpha}^{\mathrm{asymp}}[f_\infty]$ of $F^{\mathrm{asymp}}[f_\infty]$ strongly decays and verifies the constraint equations \eqref{eq:constr1}--\eqref{eq:constr2}. The derivation of these two equations requires a thorough analysis of the null properties of $F^{\mathrm{asymp}}[f_\infty]$ since certains quantities are at the threshold of integrability. It implies in particular that $\underline{\alpha}^\infty - \underline{\alpha}^{\mathrm{asymp}}[f_\infty]$ is the radiation field of a strongly decaying solution to the vacuum Maxwell equations.
\end{itemize}
These results, proved in Section \ref{SecMaxasymp}, reduces the problem of improving the estimate for $G$ to the analysis of
\begin{equation*}
\nabla^\mu H_{\mu \nu} = J(f)_\nu-J^{\mathrm{asymp}}_\nu[f_\infty], \quad \; \nabla^\mu {}^* \! H_{\mu \nu} =0, \qquad \quad  \lim_{r \to + \infty} r \underline{\alpha} (H)(r+u,r\omega)=\underline{\alpha}^{\infty} (u, \omega )-\underline{\alpha}^{\mathrm{asymp}}[f_\infty](u,\omega).
\end{equation*}
To deal with this new Cauchy problem with scattering data, we construct in Section \ref{SecscattMax} a wave operator for the Maxwell equations with a prescribed strongly decaying source terms, mapping strongly decaying radiation fields to strongly decaying solutions. The improvement of the estimates for $\mathcal{L}_{Z^\gamma}G$, with $|\gamma| \leq N-1$, is similar and rely on the next property. The structure of the asymptotic Maxwell equations is preserved by commutation,
$$\nabla^\mu \mathcal{L}_{Z^\gamma} ( F^{\mathrm{asymp}}[f_\infty] )_{\mu \nu} = J^{\mathrm{asymp}}_\nu \big[ \widehat{Z}^\gamma_\infty f_\infty \big],  \qquad \qquad \nabla^\mu {}^* \! \mathcal{L}_{Z^\gamma} ( F^{\mathrm{asymp}}[f_\infty] )_{\mu \nu} =0,$$
for any $Z^\gamma$ composed by homogeneous vector fields $\Omega_{0k}$, $\Omega_{ij}$ and $S$. In particular, we have
$$ t^2\mathcal{L}_{Z^\gamma} ( F^{\mathrm{asymp}}[f_\infty] )(t,x+t\widehat{v}) = \mathbb{F} \big[ \widehat{Z}^\gamma_\infty f_\infty \big]+O \big( t^{-1} \big).$$

\begin{Rq}
Although $\mathcal{L}_{Z^\gamma}F^{\mathrm{asymp}}[f_\infty] $ and $ F^{\mathrm{asymp}}[\widehat{Z}^\gamma_\infty f_\infty]$ share part of their asymptotic properties, they are not equal (see Proposition \ref{ProcommuFasympZ}). Let us also mention that when $Z^\gamma$ contains a translation $\partial_{x^\lambda}$, a similar commutation formula holds but we will not require it since $J(\widehat{Z}^\gamma f)$ already strongly decays.
\end{Rq}

Finally, we deal with $3.$ and improve the estimate of the top order derivatives in Section \ref{toporder} by exploiting the Glassey-Strauss decomposition of the derivatives of the field. It allows for a gain of regularity reminiscent of the elliptic regularity in the context of the Vlasov-Poisson system. More precisely, we will be able to derive the asymptotic behavior of $\nabla_{t,x} \mathcal{L}_{Z^\gamma}G$, where $|\gamma|=N-1$, by taking advantage of the strong decay properties of $J(\widehat{Z}^\gamma f)-J^{\mathrm{asymp}}[\widehat{Z}^\gamma_\infty f_\infty]$. Fix for instance $1 \leq k \leq 3$ and $0 \leq \mu < \nu \leq 3$. The idea is to exploit that
$$ \Box \, \partial_{x^k} \mathcal{L}_{Z^\gamma}(G)_{\mu \nu} = \partial_{x^\nu} \partial_{x^k} J(\widehat{Z}^\gamma f)_\mu-\partial_{x^\mu} \partial_{x^k} J(\widehat{Z}^\gamma f)_\nu $$
in order to use the representation formula for the wave equation. Then, Glassey-Strauss decomposed $\partial_{x^\lambda}$ as a combination of derivatives tangential to backward light cones and the linear Vlasov operator $\T_0$, which is transverse to light cones. In such a way, one can perform integration by parts twice and express $\partial_{x^k} \mathcal{L}_{Z^\gamma}(G)_{\mu \nu}$ as a functional of $J(\widehat{Z}^\gamma f)$ instead of its derivatives. In order to exploit this decomposition, we would like $\nabla_{\partial_{x^k}} \mathcal{L}_{Z^\gamma} F^{\mathrm{asymp}} [f_\infty]$ to verify a similar property. For this, we identify a singular solution to the linear Vlasov equation giving rise to the current density $J^{\mathrm{asymp}}[\widehat{Z}^\gamma_\infty f_\infty]$. An approximation argument then provides us the corresponding decomposition. 

\subsubsection{The Vlasov equation with asymptotic data}

We consider this time, for an electromagnetic field $F$ behaving as the solution of \eqref{eq:introkeyelMax}, the Cauchy problem
\begin{equation}\label{eq:introkeyelMax0}
\T_F^\infty (g) =0, \qquad \qquad \lim_{t \to + \infty} g(t,x,v) =f_\infty (x,v).
\end{equation}
In order to prove well-posedness and to propagate regularity backward in time for $g$, we will be lead to commute the Vlasov equation. As it is suggested by Proposition \ref{Com}, the structure of the nonlinear terms are conserved by commutation with $\widehat{Z}_\infty \in \K^\infty$. If $\widehat{Z}_\infty \neq \widehat{S}$, we will prove
\begin{equation*}
\big| \T_F^\infty \big( \widehat{Z}_\infty g \big) \big| \lesssim \Big| \frac{\widehat{v}^\mu}{tv^0} \Big( t^2  \mathcal{L}_Z(F)_{\mu j}(t,\XX_\C)-\mathbb{F}_{\mu j}\big[ \widehat{Z}_\infty f_\infty \big](v) \Big) \Big| \big|\nabla_z g \big|+\Big| \frac{\widehat{v}^\mu}{
v^0} \mathcal{L}_Z(F)_{\mu j}(t,\XX_\C) \Big|\sup_{\widehat{Z}_\infty'} \big|\widehat{Z}_\infty' g \big|+\text{better terms,}
\end{equation*}
where we recall $\XX_\C=z+t\widehat{v}+\C_{t,v}$. Now, in view of decay estimates such as \eqref{eq:ketelconv} satisfied by the electromagnetic field for the forward problem, we expect
\begin{align}
\Big| \frac{\widehat{v}^\mu}{tv^0} \Big( t^2  \mathcal{L}_Z(F)_{\mu j}(t,\XX_\C)-\mathbb{F}_{\mu j}\big[ \widehat{Z}_\infty f_\infty \big](v) \Big) \Big|& \lesssim \Lambda \,\langle z \rangle^{\frac{3}{2}}  \, \langle v \rangle^2 t^{-\frac{3}{2}} \log^{\frac{3}{2}}(t), \label{eq:introconvesti} \\
 \Big| \frac{\widehat{v}^\mu}{v^0} \mathcal{L}_Z(F)_{\mu j}(t,\XX_\C) \Big| & \lesssim \Lambda \, \langle t+|z| \rangle^{-1} \, \langle t-|\XX_\C| \rangle^{-1}. \label{eq:introtr}
 \end{align}
 We then identify two problems, which are in fact related.
\begin{enumerate}
\item Although $z$ and $v$ are almost conserved by the flow of $\T_F^\infty$, allowing for the propagation of moments for the solution and its derivatives, the polynomial loss $\langle z \rangle^{\frac{3}{2}}  \, \langle v \rangle^2$ in the right hand side of \eqref{eq:introconvesti} seems too strong in order to close the estimates through a Grönwall type inequality.
\item The decay rate of the electromagnetic field degenerates near the light cone, that is for $t\sim |\XX_\C|$ in \eqref{eq:introtr}.
\end{enumerate}
In order to deal with $2.$, one could expect to exploit the strong decay properties of massive fields near and in the exterior of the light cone, the regions where $t\sim r$ and $r \geq t$. In the most favourable case, where $f_\infty$ and then $g$ are compactly supported, we have $\langle t-|\XX_\C| \rangle \geq \delta t$ on the support of $g$, for a certain $0 < \delta <1$ depending on the support of $f_\infty$ and $\Lambda$. Equivalently, $f(t,x,v)=0$ for, say, $|v| \geq R$ or $|x| \geq R+tR/\langle R \rangle$ and $\langle t-r \rangle \gtrsim t$ on the support of $f(t,\cdot,\cdot)$. Otherwise, these strong decay properties are captured by Lemma \ref{gainv}, at the cost of moments in $z$ and $v$. Then, as for the problem $1.$, we would not be able to close the energy estimates in such a way. Indeed, bounding a given norm of the solution would require to control a stronger one, carrying more powers of $|z|$ and $|v|$. The idea then consists in exploiting the \textit{null structure} of the Lorentz force. Expanding $\widehat{v}^\mu {\mathcal{L}_Z F_\mu}^j$ according to the null frame $(\underline{L},L,e_\theta,e_\varphi)$, we obtain terms having as a factor
\begin{itemize}
\item either a good component $\alpha$, $\rho$ or $\sigma$ of $\mathcal{L}_Z F$, which turn out to decay as $\langle t+r \rangle^{-2}$,
\item or a good null component of the four-momentum vector, $\widehat{v}^{\underline{L}}$, $\widehat{v}^{e_\theta}$ or $\widehat{v}^{e_\varphi}$. They will allow us to take advantage of the $t-r$ decay in the energy estimates.
\end{itemize}
These observations yield an improvement of \eqref{eq:introtr} where the $t-r$ decay has been either transformed into $t+r$ decay or made usable in order to prove boundedness for the solutions. Similar considerations allow to obtain an alternative version of \eqref{eq:ketelconv},
\begin{equation}\label{eq:keyintro2}
\hspace{-5mm} \Big| \frac{\widehat{v}^\mu}{tv^0} \Big( t^2  \mathcal{L}_Z(F)_{\mu j}(t,\XX_\C)-\mathbb{F}_{\mu j}\big[ \widehat{Z}_\infty f_\infty \big](v) \Big) \Big| \big| \nabla_z g \big| \lesssim \Lambda  \bigg( \frac{|\widehat{v}^{\underline{L}} |^{\frac{1}{2}} \log(t)}{t \, \langle t-|\XX_\C| \rangle^{\frac{1}{2}}}+\frac{\log(t)}{t^{\frac{3}{2}}} \bigg) \langle z \rangle \,\big| \nabla_z g \big|.
\end{equation}
Note first that the right hand side does not carry any $v$ weight anymore. Although the time decay is weaker than for \eqref{eq:ketelconv}, exploiting the $t-|\XX_\C|$ decay through the factor $|\widehat{v}^{\underline{L}}|^{1/2}$ will allow us to compensate for this. However, the weight $\langle z \rangle$, even though it is crucially weaker than in \eqref{eq:ketelconv}, still has to be handled. For this, we take advantage of hierarchies in the commuted equations.
\begin{itemize}
\item If $\widehat{Z}_\infty$ is a translation $\partial_t^\infty$ or $\partial_{z^k}$, then $\mathbb{F}[\widehat{Z}_\infty f_\infty]=0$ and $Z=\partial_t$ or $Z=\partial_{x^k}$. Consequently, the derivative of the electromagnetic field has a better behavior, $|\mathcal{L}_Z F| \lesssim \Lambda \, \langle t+r \rangle^{-1} \, \langle t-r \rangle^{-2}$, implying that $|\T_F^\infty (\widehat{Z}_\infty g ) |$ strongly decays. 
\item If $\widehat{Z}_\infty$ is an homogeneous vector field, we remark that the worst error term \eqref{eq:keyintro2} carries the factor $\langle z \rangle \,\big| \nabla_z g \big|$. Since the derivatives $\partial_{z^k}g$ behaves better, it makes the system of the commuted equations \textit{triangular}. More concretely, we should be able to control the quantities $\langle z \rangle \,\big| \nabla_z g \big|$ and $|\widehat{Z}_\infty g|$.  
\end{itemize}
More generally, we will prove boundedness for the $L^2$ norms of quantities of the form $\langle z \rangle^{N_z-\beta_H} \langle v \rangle^{N_v} \, \widehat{Z}_\infty^\beta g$, where we recall that $\beta_H$ is the number of homogeneous vector fields composing $\widehat{Z}_\infty^\beta$.
\begin{Rq}
Considering hierarchies of the form $\langle z \rangle^{N_z-a\beta_H} \langle v \rangle^{N_v} \, \widehat{Z}_\infty^\beta g$, with $a>1$, would make good error terms problematic. This explains why we could not lose more powers of $\langle z \rangle$ in \eqref{eq:keyintro2}.
\end{Rq}

Finally, let us mention that working with large electromagnetic fields forces us to close the energy estimates through Grönwall type inequalities. For this reason, we carefully chose the foliation used to define the energy norms of Vlasov field in \eqref{defnormVlasovL2} below. In particular, it allows for an application of a two variables Gronwall inequality, in $(t,u)$, where $u=t-r$.

\section{Energy and decay estimates}\label{Secenergy}

We fix, for all this section, $T \geq 0$ as well as $h :[T,+\infty[ \times \R^3_x \times \R^3_v \to \R$ and a $2$-form $F$, defined on $[T,+ \infty[ \times \R^3$, both sufficiently regular. We will have to control the flux of the solutions through two different families of hypersurfaces, in addition to the $\{ t= t_{\mathrm{cst}} \}$ ones, that we start by introducing. 

\subsection{Ingoing and outgoing light cones}\label{subsec2} The future scattering state of a smooth electromagnetic field is a tensor field depending on the variables $(u,\omega)\in \R \times \mathbb{S}^2$, given by the limits of $rE(u,\underline{u},\omega)$ and $rB(u,\underline{u},\omega)$ as $\underline{u} \to +\infty$. This motivates the introduction, for any $\underline{u} \geq 0$, of the ingoing light cones
$$ \underline{C}_{\underline{u}}:= \big\{(t,x)  \in [T,+\infty[ \times \mathbb{R}^3  \, \big| \; t+|x|=\underline{u} \big\}, \qquad \qquad \dr \mu_{\underline{C}_{\underline{u}}} =\frac{1}{\sqrt{2}}r^2  \dr \mu_{\mathbb{S}^2} \dr u,$$
where $\dr \mu_{ \underline{C}_{\underline{u}}}$ is the induced volume form of $\underline{C}_{\underline{u}}$, in accordance with the choice of the null vector field $\sqrt{2}^{\, -1}\underline{L}$ as its Lorentzian normal vector.

\begin{center}
\begin{tikzpicture}
\fill[color=gray!35] (3,0)--(7,0)--(7,4)--(0,4)--(0,3)--(3,0);
\draw [-{Straight Barb[angle'=60,scale=3.5]}] (0,-0.3)--(0,4);
\draw [-{Straight Barb[angle'=60,scale=3.5]}] (-0.2,0)--(7,0) ;
\draw (3,0.2)--(3,-0.2);
\draw (-0.18,3)--(0.2,3);
\draw (0,-0.5) node{$r=0$};
\draw (-0.7,0) node{$t=T$};
\draw (3,-0.5) node{$\underline{u}$};
\draw (-0.5,3) node{$\underline{u}$};
\draw (-0.5,3.8) node{$t$};
\draw (6.8,-0.5) node{$r$}; 
\draw[densely dashed] (3,0)--(0,3) node[scale=1.2,below, midway] {$\underline{C}_{\underline{u}}$};
\draw (4.8,2.5) node[ color=black!100, scale=1.2] {$\{t+|x| \geq \underline{u}, \; t \geq T \}$};

      \node[align=center,font=\bfseries, yshift=-2em] (title) 
    at (current bounding box.south)
    {The sets $\underline{C}_{\underline{u}}$.};
\end{tikzpicture}
\end{center}

Controlling the flux of the solutions through truncated outgoing light cones will allow us to take advantage of the null structure of the system. Let, for all $T \leq t_0 \leq s \leq +\infty$ and $u \in \R$,
$$ C_u^{t_0,s}:= \big\{(t,x)  \in [T,+\infty[ \times \mathbb{R}^3 \, \big| \; t-|x|=u , \; t_0 \leq t \leq s \big\},  \qquad C_u^{t_0}:=C_u^{t_0,+\infty},\qquad \qquad \dr \mu_{C_u} =\frac{1}{\sqrt{2}}r^2  \dr \mu_{\mathbb{S}^2}\dr \underline{u},$$
with $\dr \mu_{ C_u}$ the induced volume form of $C_{u}:=C^T_u$, chosen in accordance with the choice of the null vector field $\sqrt{2}^{\,-1}L$ as its normal vector. For a sufficiently regular function $\psi :[T,+\infty[ \times \R^3 \to \R$,
 $$\int_{\underline{C}_{\underline{u}}} \psi \,  \dr \mu_{\underline{C}_{\underline{u}}}\! = \int_{u=2T-\underline{u}}^{\underline{u}} \int_{\mathbb{S}^2_\omega} \psi(u,\underline{u},\omega ) \frac{r^2}{\sqrt{2}} \dr \mu_{\mathbb{S}^2_\omega} \dr \underline{u}, \quad \; \; \int_{C_u^{t_0,s}} \psi \, \dr \mu_{C_u} \! = \int_{\underline{u}=\max ( 2t_0-u, u)}^{2s-u} \int_{\mathbb{S}^2_\omega} \psi(u,\underline{u},\omega ) \frac{r^2}{\sqrt{2}} \dr \mu_{\mathbb{S}^2_\omega} \dr \underline{u} .$$
\begin{Rq} 
The future scattering states are in fact defined on a part of the conformal boundary of the Minkowski space, future null infinity $\mathcal{I}^+$, which corresponds to the future end points of the null rays $t-|x|=u$. It can be viewed as $\underline{C}_{+\infty}$. More precisely, $$(t,r,\omega) \mapsto \big(T(t,r)=\tan^{-1}(t+r)+\tan^{-1}(t-r), R(t,r)=\tan^{-1}(t+r)-\tan^{-1}(t-r), \omega \big) \in \R \times \mathbb{S}^3$$ is a conformal diffeomorphism between Minkowski spacetime and the interior of the triangle $0 \leq R \leq \pi$, $|T|=\pi-R$ of the space $\R \times \mathbb{S}^3$, equipped with the metric $-\dr T^2+\dr R^2+\sin^2( R) \dr \mu_{\mathbb{S}^2}$. Then $$ \mathcal{I}^+ := \{ (T,R,\omega) \in \R \times \mathbb{S}^3 \; / \; 0 < R < \pi, \quad T=\pi-R \}.$$
The past scattering states, obtained as the limit of $rE(u,\underline{u},\omega)$ and $rB(u,\underline{u},\omega)$ as $u \to -\infty$, are defined on past null infinity $\mathcal{I}^-=\{ 0<R < \pi, \; T=R-\pi \}$, which can be viewed as $t-|x|=-\infty$. 

All the timelike geodesics, that is the lines $t \mapsto (t,x+\widehat{v}t)$ with $v \in \R^3_v$, terminate at future timelike infinity $i^+=(0,\pi)$ as $t \to +\infty$ and past timelike infinity $i^-=(0,-\pi)$ as $t \to -\infty$.
\end{Rq}

\begin{center}
\begin{tikzpicture}
\draw [-{Straight Barb[angle'=60,scale=3.5]}] (0,-0.3)--(0,4);
\draw [-{Straight Barb[angle'=60,scale=3.5]}] (-0.2,0)--(7,0) ;
\draw (3,0.2)--(3,-0.2);
\draw (0,-0.5) node{$r=0$};
\draw (-0.7,0) node{$t=T$};
\draw (-0.7,1) node{$t=t_0$};
\draw (3,-0.5) node{$t_0-u$};
\draw (-0.5,3.8) node{$t$};
\draw (-0.7,2) node{$t=u'$};
\draw (6.8,-0.5) node{$r$}; 
\draw[densely dashed] (0,1)--(3,1) ;
\draw (3,1)--(6,4) ;
\draw (0,2)--(2,4) ;
\draw (1.4,2.85) node[scale=1.2]{$C_{u'}^T$};
\draw (5.1,2.55) node[scale=1.2]{$C_u^{t_0}$};

  \node (II)   at (11,1.8)   {};

\path  
  (II) +(90:2.5)  coordinate    (IItop)
       +(-90:2.5) coordinate    (IIbot)
       +(0:2.5)   coordinate    (IIright)
       ;
\draw 
	
      (IIbot) --
          node[midway, above , sloped]    {\footnotesize{$r=0$}}   
          (IItop) --
          node[midway, above, sloped] {$\mathcal{I}^+$,  \hspace{0.8mm} \footnotesize{ $\underline{u}=+\infty$}} 
          (IIright) -- 
          node[midway, below, sloped] {$\mathcal{I}^-$, \hspace{0.8mm} \footnotesize{$ u=-\infty$}}
      (IIbot) -- cycle;
\draw (10.9,-0.8) node[scale=0.8]{$i^-$};
\draw (13.7,1.8) node[scale=0.8]{$i^0$};
\draw (10.9,4.4) node[scale=0.8]{$i^+$}; 
\draw (10.9,4.6) node[scale=0.8]{}; 

\draw[densely dashed] (11.7,0)--node[scale=0.6,midway, above, sloped]{$\underline{u}=\mathrm{cst}$}(11,0.7)--node[scale=0.6,midway, above, sloped]{$u=\mathrm{cst}$}(12.8,2.5);

\node[align=center,font=\bfseries, yshift=-2em] (title) 
    at (current bounding box.south)
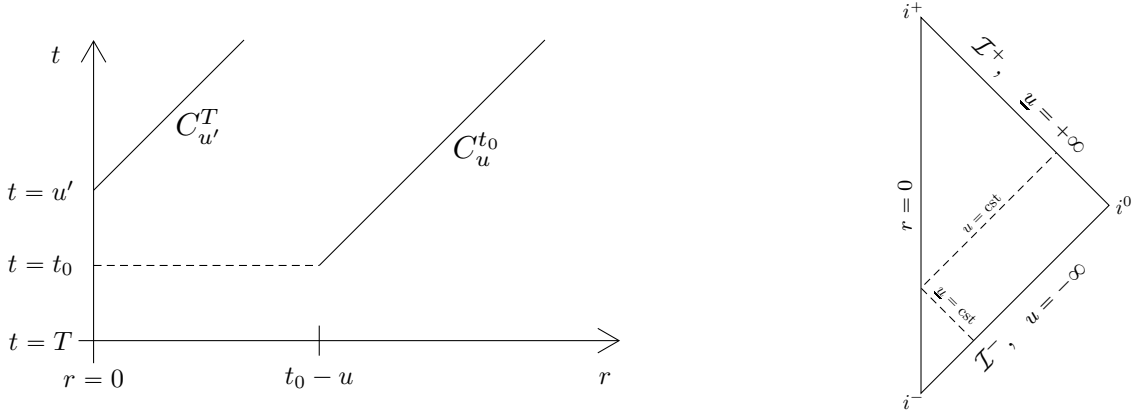
    {The sets $C_u^{t_0}$ and the Penrose diagram of the Minkowski space.};
\end{tikzpicture}
\end{center}

\subsection{Energy inequality for Vlasov fields}\label{SubsecenergyVla}

In order to ultimately prove boundedness in $L^2_{z,v}$ for quantities such as $\langle z \rangle^{p} \, \langle v \rangle^q \, \widehat{Z}^\beta_\infty g$, we will first estimate a refined norm of the solutions. In view of the hypersurfaces on which we will control them and even if we will estimate the error terms in the coordinate system $(t,z,v)$, it is more sensible to perform the energy inequalities using the variables $(t,x,v)$. Let, for $t_0 \geq T$ and $u \in \R$,
\begin{equation}\label{defnormVlasovL2}
 \overline{\mathbb{E}} \big[  h  \big] (t_0,u) :=  \int_{|x| \leq t_0 -u}  \int_{\R^3_v}  \big|h \big|^2(t, x,v) \dr v \dr x+ \sqrt{2}  \int_{C_{u}^{t_0}} \int_{\R^3_v} \widehat{v}^{\underline{L}} \big|h \big|^2 \dr v \dr \mu_{C_u},
 \end{equation}
where, more explicitly,
$$\sqrt{2}\int_{C_u^{t_0}} \int_{\R^3_v} v^{\underline{L}} \big|h \big|^2 \dr v \dr \mu_{C_u} = \int_{\underline{u}=\max(u,2t_0-u)}^{+ \infty} \int_{\mathbb{S}^2_\omega} \int_{\R^3_v} v^{\underline{L}}(\omega) \big|h \big|^2\Big( \frac{\underline{u}+u}{2}, \frac{\underline{u}-u}{2}\omega \Big) \dr v  r^2 \dr \mu_{\mathbb{S}^2_\omega} \dr \underline{u}, \qquad r= \frac{\underline{u}-u}{2}.$$
We now make a few remarks concerning the choice of the foliation of $[T,+\infty[ \times \R^3_x \times \R^3_v$.
\begin{itemize}
\item The domain of the first integral is the part of a constant $t$ hypersurface located in $\{ t-|x| \geq u \}$. The one of the second integral is the truncated outgoing null cone $C_u^{t_0}$, the part of $\{t-|x|=u \}$ located in $\{ t \geq t_0 \}$.
\item Controlling the $L^2$ norm of the distribution function on null cones $C_u^{t_0}$ will allow us to exploit the decay in $t-r$ in the energy estimates.
\item We consider such truncations of the constant $t$ hypersurfaces in order to be able to apply Gronwall's inequality when we will perform energy estimates (see Lemma \ref{LemGronwall2variables}). The error terms will decay either as $t^{-\frac{4}{3}}$ or as $\langle u \rangle^{-\frac{4}{3}}$.
\end{itemize}
The next energy inequality is adapted to the study of solutions to the Vlasov equation with asymptotic data.
\begin{Pro}\label{ProenergyL2Vlasov}
We have, for all $t_0 \geq T$ and $u \in \R$,
$$\overline{\mathbb{E}} \big[ h \big] (t_0,u) \leq 2\int_{t=t_0}^{+ \infty} \int_{|x| \leq t-u} \int_{\R^3_v} \big| \T_F  \big( h \big) h\big|(t,x,v) \dr v \dr x \dr t+ \lim_{t \to + \infty} \int_{\R^3_x}  \int_{\R^3_v}  \big|h \big|^2(t, x,v) \dr v \dr x.$$
\end{Pro}
\begin{proof}
The proof is based on the divergence theorem, applied to the current density of $h^2$, which verifies 
$$\nabla_\mu J \big(h^2 \big)^\mu = \partial_{x^\mu} \int_{\R^3_v} \widehat{v}^\mu h^2 \dr v= \int_{\R^3_v} \T_F \big( h^2 \big) \dr v-\int_{\R^3_v} \partial_{v^j} \big( \, \widehat{v}^\mu {F_{\mu}}^j h^2  \, \big) \dr v=\int_{\R^3_v} \T_F \big( h^2 \big) \dr v.$$
 For any $T \leq t_0 \leq s$ and $u \in \R$, we introduce the domain
$$ D^{t_0,s}_{u} := \big\{ (t,x) \in \R^{1+3} \, \big| \;  t_0 \leq t \leq s , \; t-|x| \geq u \big\}.$$ 
Its boundary is composed by 
\begin{itemize}
\item $\{ (t_0,x) \, | \, x \in \R^3, \; |x| \leq t_0 -u \}$ and $\{ (s,x) \, | \, x \in \R^3, \; |x| \leq s -u \}$. Their unit outward normal vector, for the euclidian metric, are respectively $-\partial_t$ and $\partial_t$. Note that they can be empty if $u$ is large.
\item The truncated null cone $C_u^{t_0,s}$.  Its euclidian unit outward normal vector is $-\sqrt{2}^{\, -1} (\partial_t-\partial_r)=-\sqrt{2}^{ \, -1} \underline{L}$.
\end{itemize}
The euclidian divergence theorem, applied to $J(h^2)$, in $D_u^{t_0,s}$, yields
$$ \int_{|x| \leq u-t_0} J \big( h^2 \big)^0(\tau,x) \dr x+\int_{C_u^{t_0,s}} \sqrt{2} J \big( h^2 \big)^{\underline{L}} \dr \mu_{C_u} = \int_{|x| \leq u-s} J \big( h^2 \big)^0(s,x) \dr x+2\int_{D_u^{t_0,s}} \int_{\R^3_v} \T_F(h) h \, \dr v \dr x \dr t  .$$
Since $J(h^2)^0= \int_v h^2\dr v$ and $J(h^2)^{\underline{L}}= \int_v \widehat{v}^{\underline{L}} \, h^2\dr v$, we get the result by letting $s \to +\infty$.
\end{proof}

\subsection{Energy inequalities for electromagnetic fields}\label{SecenergyMaxweel}

The results of this section will be proved by applying the divergence theorem to the energy-momentum tensor of $F$, contracted with well-chosen vector fields (multipliers).

\begin{Def}\label{Defenergytensor}
The energy-momentum tensor $\mathbb{T}[F]_{\mu \nu}$ of the $2$-form $F$ is defined as 
$$ \mathbb{T}[F]_{\mu \nu}:=F_{\mu \beta} {F_{\nu}}^{\beta}-\frac{1}{4}\eta_{\mu \nu} F_{\xi \lambda} F^{\xi \lambda}.$$
The null components of $\mathbb{T}[F]$ are given by
$$ \mathbb{T}[F]_{L L}=|\alpha(F)|^2, \qquad \mathbb{T}[F]_{\underline{L} \underline{L} }=|\underline{\alpha}(F)|^2, \qquad \mathbb{T}[F]_{L \underline{L}}=|\rho(F)|^2+|\sigma(F)|^2.$$
Moreover, if F a is solution to the Maxwell equations with a magnetic current,
\begin{equation}\label{Maxmagn}
 \nabla^\mu F_{\mu \nu} = J_{\nu}, \qquad \nabla^\mu {}^* \! F_{\mu \nu} = \widetilde{J}_{\nu}, \tag{Mmc}
 \end{equation}
 then
$$  \nabla^{\mu} T[F]_{\mu \nu}=F_{\nu \lambda} J^{\lambda}+{}^* \! F_{\nu \lambda} \widetilde{J}^\lambda.$$
\end{Def}
\begin{Rq}
If $F$ is a solution to the Maxwell equations, the magnetic current source term $\widetilde{J}$ vanishes in \eqref{Maxmagn}. However, in order to solve the vacuum Maxwell equations with data prescribed at future null infinity $\mathcal{I}^+$, we will consider a sequence of approximate solutions $(F^n)_{n \geq 1}$ such that $\nabla^\mu {}^* \! F^n_{\mu \nu} \neq 0$.
\end{Rq}
\subsubsection{Estimates for the forward problem}
We assume, for the next two Propositions, that $F$ is a solution to the Maxwell equations with sufficiently regular source terms $J$,
\begin{equation}\label{eq:Maxeqn}
 \nabla^\mu F_{\mu \nu}=J_\nu, \qquad \qquad \nabla^\mu {}^* \! F_{\mu \nu} =0 .
 \end{equation}
We start by proving energy estimates by using $\partial_t$ as a multiplier.
\begin{Pro}\label{Propartialt}
If $F$ verifies \eqref{eq:Maxeqn} on $[T,+\infty[ \times \R^3$, then
$$ \sup_{t \geq T}\| F(t,\cdot) \|_{L^2_x}^2+ \sup_{\underline{u} \geq T}  \int_{\underline{C}_{\underline{u}}}  |\underline{\alpha} (F)|^2 \dr \mu_{\underline{C}_{\underline{u}}} +\sup_{u \in \R} \int_{C_u^T} | \rho  (F)|^2+| \sigma (F)|^2  \dr \mu_{C_u} \lesssim  \| F(T,\cdot) \|_{L^2_x}^2+\bigg|  \int_{\tau=T}^{+\infty}  \| J (\tau,\cdot) \|_{L^2_x} \dr s \bigg|^2 \!.$$
\end{Pro}
\begin{proof}
Apply first the divergence theorem to $\mathbb{T}[F]_{\mu 0}$ in the region $\{ (\tau,x) \in \R_+ \times \R^3 \, | \, T \leq \tau \leq t \}$. As $2\mathbb{T}[F]_{0 0}=|F|^2$ and $\nabla^\mu \mathbb{T}[F]_{\mu 0} = F_{0 \nu} J^{\nu}$ , we get, for all $t \geq T$,
$$ \int_{\R^3_x} |F|^2(t,x) \dr x-\int_{\R^3_x} |F|^2(T,x) \dr x=-2\int_{\tau=0}^t \int_{\R^3_x} F_{0\nu}J^\nu \dr x \dr \tau \leq 2\sup_{T \leq \tau \leq t} \| F(\tau, \cdot) \|_{L^2_x}\int_{\tau=T}^t  \| J(\tau, \cdot) \|_{L^2_x} \dr \tau.$$
Fix now $u \in \R$ and apply the divergence theorem in the region $\{ (\tau,x) \in \R_+ \times \R^3 \, | \,  \tau-|x| \leq u, \; \tau \geq T \}$. As its boundary is $\{(T,x) \, | \, |x| \geq T-u \} \cup C_u^T$, we have
\begin{align*}
 \sqrt{2}\int_{C_u^T} \mathbb{T}[F]_{L 0} \dr \mu_{C_u}  -\int_{ |x| \geq T-u} |F|^2(T,x) \mathrm{d} x& =-2\int_{\tau=T}^{\underline{u}} \int_{|x| \geq \tau-u}  F_{0\nu}J^\nu \dr x \dr \tau \\
 & \leq 2\sup_{\tau \geq T} \| F(\tau, \cdot) \|_{L^2_x}\int_{\tau=T}^{+\infty}  \| J(\tau, \cdot) \|_{L^2_x} \dr \tau.
 \end{align*}
Then we use that $\mathbb{T}[F]_{\underline{L} L}=|\rho (F)|^2+|\sigma (F)|^2$ and $\mathbb{T}[F]_{LL}=|\alpha(F)|^2 \geq 0$. Finally, consider $\underline{u} \geq 0$ and apply the divergence theorem in the region $\{ (\tau,x) \in \R_+ \times \R^3 \, | \,  \tau+|x| \leq \underline{u}, \; \tau \geq T \}$. As its boundary is composed by $\{ (T,x) \, | \,  |x| \leq \underline{u}-T \}$ and $\underline{C}_{\underline{u}}$, we obtain
\begin{align*}
 \sqrt{2}\int_{\underline{C}_{\underline{u}}} \mathbb{T}[F]_{\underline{L} 0} \dr \mu_{\underline{C}_{\underline{u}}}  -\int_{ |x| \leq \underline{u} -T} |F|^2(T,x) \mathrm{d} x& =-2\int_{\tau=T}^{\underline{u}} \int_{|x| \leq \underline{u}-\tau}  F_{0\nu}J^\nu \dr x \dr \tau \\
 & \leq 2\sup_{T \leq \tau \leq \underline{u}} \| F(\tau, \cdot) \|_{L^2_x}\int_{\tau=T}^t  \| J(\tau, \cdot) \|_{L^2_x} \dr \tau.
 \end{align*}
It remains to use $\mathbb{T}[F]_{\underline{L} \underline{L}}=|\underline{\alpha}(F)|^2$ and $\mathbb{T}[F]_{\underline{L} L}=|\rho (F)|^2+|\sigma (F)|^2 \geq 0$.
\end{proof}

Let $K$ be the sum of $\partial_t$ and the Morawetz vector field,
\begin{equation}\label{defKs}
K:= \frac{1}{2} \langle t+r \rangle^{2} L+\frac{1}{2}\langle t-r \rangle^{2}\underline{L},
\end{equation}
which is a conformal Killing vector field of Minkowski spacetime with conformal factor $4t$,
\begin{equation*}
\nabla_\mu K_\nu+\nabla_\nu K_\mu =4t \eta_{\mu \nu}.
\end{equation*}
As $\mathbb{T}[F]$ is symmetric and trace-free, that is ${\mathbb{T}[F]_{\mu}}^\mu=0$, we have
$$ 2\mathbb{T}[F]_{\mu \nu} \nabla^{\mu} K^{\nu} = \mathbb{T}[F]_{\mu \nu}\left( \nabla^{\mu} K^{\nu}+\nabla^{\nu} K^{\mu} \right)=4t\mathbb{T}[F]_{\mu \nu} \eta^{\mu \nu}=4t {\mathbb{T}[F]_{\mu}}^\mu=0 ,$$
so that
\begin{equation}\label{eq:divMora}
\nabla^{\mu} \left( \mathbb{T}[F]_{\mu \nu} K^{\nu} \right)  = \nabla^{\mu} \left(\mathbb{T}[F]_{\mu \nu}\right)K^{\nu}.
\end{equation}
The vector field $K$, used a multiplier, then provides a conservation law for solutions to the vacuum Maxwell equations, that is \eqref{eq:Maxeqn} for $J=0$. More generally, $K$ can be used in order to propagate the weighted norms 
\begin{align}
\mathcal{E}^{K}[F](t) &:=  \int_{\R^3_x} \!\left[ \langle t+r \rangle^{2}|\alpha (F)|^2\!+\langle t-r \rangle^{2} |\underline{\alpha} (F)|^2\!+\big(\langle t+r \rangle^{2}+\langle t-r \rangle^{2} \big)(|\rho (F) |^2+|\sigma(F)|^2)\right](t,x)   \mathrm{d} x , \nonumber \\
 \mathcal{F}^{K}[F]_t^s(u) & := \sqrt{2} \int_{C_u^{t,s}} \langle \underline{u} \rangle^2 |\alpha (F)|^2+\langle u \rangle^2 \big(|\rho (F) |^2+|\sigma(F)|^2 \big) \dr \mu_{C_u} . \label{keva:defenergy}
 \end{align}
 In this paper, the forward problems will not require to control the fluxes through the outgoing null cones.
\begin{Pro}\label{ProMoravac}
Assume that $F$ is solution to \eqref{eq:Maxeqn} on $[T,+\infty[ \times \R^3$. Then, 
$$ \sup_{t \geq T} \mathcal{E}^K[F](t) \leq 2\mathcal{E}^K[F](T)+8\bigg|\int_{\tau=T}^{+\infty} \bigg|\int_{\R^3_x} \langle \tau+|x| \rangle^2 \big| J (\tau,x) \big|^2 \dr x \bigg|^{\frac{1}{2}} \dr \tau \bigg|^2.$$
\end{Pro}
\begin{proof}
Observe that, in view of the null components of $K$ and $\mathbb{T}[F]$, we have with $u=t-r$ and $\underline{u}=t+r$, 
\begin{align}
 2\mathbb{T}[F]_{\underline{L} \nu} K^{\nu}&=\langle u \rangle^{2}|\underline{\alpha} (F)|^2+\langle \underline{u} \rangle^{2} \big(|\rho (F) |^2+|\sigma(F)|^2\big)  , \label{eq:compoK1} \\
2 \mathbb{T}[F]_{L \nu} K^{\nu}&=\langle \underline{u} \rangle^{2}|\alpha (F)|^2+\langle u \rangle^2  \big(|\rho (F) |^2+|\sigma(F)|^2\big). \label{eq:compoK2} 
  \end{align}
Note further that $4 \mathbb{T}[F]_{0 \nu}K^{\nu}=2 \mathbb{T}[F]_{L \nu}K^{\nu}+2 \mathbb{T}[F]_{L \nu}K^{\nu}$ and fix $t \geq T$. By applying the divergence theorem to $ 4\mathbb{T}[F]_{\mu \nu}K^{\nu}$, in the time slab $T \leq \tau \leq t$, we get
$$ \mathcal{E}^K[F](t) = \mathcal{E}^K[F](T)-4\int_{\tau=T}^t \int_{\R^3_x} \big[ K^\mu F_{\mu \nu} J^\nu \big](\tau,x) \dr x \dr \tau.$$
Expanding $K^\mu F_{\mu \nu} $ according to the null frame $(\underline{L},L,e_\theta,e_\varphi)$, we obtain, as $2K=\langle t+r \rangle^2 L+ \langle t-r \rangle^2 \underline{L}$,
\begin{align*}
& \big| K^\mu F_{\mu \nu} \big| \leq  \, \langle t+r \rangle^2 \big|F_{L \nu} \big|+\langle t-r \rangle^2 \big|F_{\underline{L} \nu}  \big| \leq   \langle t+r \rangle^2 \big( \big| \alpha (F ) \big|+\big| \rho (F)\big|\big)  + \langle t-r \rangle^2 \big( \big| \underline{\alpha} (F)\big|+\big| \rho (F)\big|\big)  .
 \end{align*}
 Consequently, the Cauchy-Schwarz inequality in $x$ yields
 $$ \mathcal{E}^K[F](t) \leq \mathcal{E}^K[F](T)+4   \int_{\tau=T}^t \big|\mathcal{E}^K[F](\tau) \big|^{\frac{1}{2}} \bigg| \int_{\R^3_x} \langle \tau+|x| \rangle^2 \big| J(\tau,x) \big|^2 \dr x \bigg|^{\frac{1}{2}} \dr \tau$$
 and it remains to use a Grönwall type inequality.
\end{proof}

\subsubsection{Estimates for the backward problem}

For forward problems, the norm $\mathcal{E}^K[ \cdot ]$ can be strengthened in the exterior of the light cone $\{ |x| \geq t \}$ by a factor $\langle t-r \rangle^a$, $a \geq 0$. For the backward problem, as pointed out by \cite{SchlueLindblad}, it is in the interior region $\{|x| \leq t \}$ that one can use these weights in order to derive improved estimates. It is explained by
\begin{align}
 \hspace{-5mm} \nabla^\mu \big( \mathbb{T}[F]_{\mu \nu} K^\nu  | \max (t-r,1) |^a  \big)&= \nabla^\mu \big( \mathbb{T}[F]_{\mu \nu}  K^\nu \big) | \max (t-r,1) |^a + \mathbb{T}[G]_{L \nu} K^\nu \,  \nabla^L \big( | \max (t-r,1) |^a \big) \nonumber \\
& \leq  \nabla^\mu \big( \mathbb{T}[F]_{\mu \nu} K^\nu \big)  | \max (t-r,1) |^a  = \nabla^\mu \big( \mathbb{T}[F]_{\mu \nu} \big) K^\nu   | \max (t-r,1) |^a ,       \label{eq:pppppp}
\end{align}
which ensues from
$$L(t-r)=e_\theta(t-r)= e_\varphi (t-r)=0, \qquad \qquad 2\nabla^L= - \nabla_{\underline{L}}=-2\nabla_{\partial_u}, \qquad \qquad \mathbb{T}[G]_{L \nu} K^\nu \geq 0 .$$
Hence, when will apply the divergence theorem, the error term arising from the factor $ | \max (t-r,1) |^a$ will have a good sign. For the construction of the modified wave operator, we will crucially use this property in order to deal with the top order derivatives. Consider then, for all $T \leq t \leq s$, $u \in \R$ and any $a \geq 0$, the energy norms and energy fluxes
$$
\mathcal{E}^{K,a}[F](t) := \mathcal{E}^{K} \big[ | \max (t-r,1) |^a \, F \big](t), \qquad \qquad \mathcal{F}^{K,a}[F]_t^s(u) := \big| \max (u,1) \big|^a \, \mathcal{F}^{K}[F]_t^s(u).
$$
Let further the higher order energy norms, for $N \in \mathbb{N}$,
\begin{equation}\label{keva:defenergy2}
 \mathcal{E}^{K,a}_N[F](t) := \sup_{|\gamma| \leq N} \mathcal{E}^{K,a} [ \mathcal{L}_{Z^\gamma}F ](t), \qquad \qquad  \mathcal{F}^{K,a}_N[F]_t^s(u)  :=  \sup_{|\gamma| \leq N} \mathcal{F}^{K,a} [ \mathcal{L}_{Z^\gamma}F ]_t^s(u) .
 \end{equation}
 
\begin{Pro}\label{ProenergyforscatMax}
Let $F$ be a solution to the Maxwell equations with a magnetic current \eqref{Maxmagn}. Then, for all $T \leq t \leq s$ and any $a \geq 0$,
$$ \mathcal{E}^{K,a}[F](t) +\sup_{u \in \R} \mathcal{F}^{K,a}[F]_t^s(u) \leq  2\mathcal{E}^{K,a}[F](s)+8 \int_{\tau=t}^s \int_{\R^3_x} \big| \max (\tau-|x|,1) \big|^{2a}  \Big| F_{\lambda \nu} J^\nu K^\lambda + {}^* \! F_{\lambda \nu} \widetilde{J}^\nu K^\lambda \Big|(\tau , x) \dr x \dr \tau .$$
\end{Pro}
\begin{proof}
Let $X$ be the multiplier $ | \max (t-|x|,1)|^{2a} K$. Fix further $T \leq t \leq s$. By applying the Lorentzian divergence theorem to $ \mathbb{T}[F]_{\mu \nu} X^\nu$ in the domain $\{(\tau,x) \in \R_+ \times \R^3 \, | \, t \leq \tau \leq s \}$, we get, 
\begin{align*}
 \int_{\R^3_x} \Big[ \mathbb{T}[F]_{0 \nu} X^\nu  \Big](t,x) \dr x-\int_{\R^3_x} \Big[ \mathbb{T}[F]_{0 \nu} X^\nu \Big](s,x) \dr x &=\int_{\tau=t}^s \int_{\R^3_x} \Big[ \nabla^\mu \big( \mathbb{T}[F]_{\mu \nu} X^\nu \big) \Big] (\tau,x)\dr x \dr \tau  \\
 & \leq \int_{\tau=t}^s \int_{\R^3_x}\Big[ \nabla^\mu \big( \mathbb{T}[F]_{\mu \nu} \big) X^\nu  \Big] (\tau,x)\dr x \dr \tau ,
 \end{align*}
 since $\nabla^\mu \big( \mathbb{T}[F]_{\mu \nu} X^\nu \big) \leq \nabla^\mu \big( \mathbb{T}[F]_{\mu \nu} \big) X^\nu $ according to \eqref{eq:pppppp}. Simarly, fix $u \in \R$ and apply the divergence theorem in the region $\{ (\tau , x ) \, | \, \tau - |x| \geq u , \; t \leq \tau \leq s \}$. It yields
 \begin{align*}
\int_{|x| \leq t+u} \Big[\mathbb{T}[F]_{0 \nu} X^\nu \Big](t,x) \dr x+\frac{1}{\sqrt{2}} \int_{C_u^{t,s}} \mathbb{T}[F]_{L\nu} X^\nu  \dr \mu_{C_u}-&\int_{|x| \leq s+u} \Big[\mathbb{T}[F]_{0 \nu} X^\nu \Big](s,x) \dr x \\
 & \qquad \leq \int_{\tau=t}^s \int_{|x| \leq \tau+u} \Big[ \nabla^\mu \big( \mathbb{T}[F]_{\mu \nu} \big) X^\nu \Big] (\tau,x)\dr x \dr \tau  .
 \end{align*}
We then deduce that 
\begin{align*}
 \int_{\R^3_x} \Big[ \mathbb{T}[F]_{0 \nu} X^\nu  \Big](t,x) \dr x+\frac{1}{\sqrt{2}} \sup_{u \in \R} \int_{C_u^{t,s}} \mathbb{T}[F]_{L\nu} X^\nu  \dr \mu_{C_u} &  \leq 2\int_{\R^3_x} \Big| \mathbb{T}[F]_{0 \nu} X^\nu \Big|(s,x) \dr x \\
 & \quad \, +2 \int_{\tau=t}^s \int_{\R^3_x}\Big| \nabla^\mu \big( \mathbb{T}[F]_{\mu \nu} \big) X^\nu  \Big| (\tau,x)\dr x \dr \tau .
 \end{align*}
Recall that $\nabla^\mu  \mathbb{T}[F]_{\mu \nu} = F_{\nu \lambda} J^\lambda+{}^*\! F_{\nu \lambda} \widetilde{J}^\nu$ as well as \eqref{eq:compoK1}--\eqref{eq:compoK2} in order to compute the other integrands.
\end{proof}

\subsection{Pointwise decay estimates for electromagnetic fields}

In the linear case, the cartesian components of $F$ decay as $t^{-1}$ near the light cone since they are solution to the homogeneous wave equation. We explain here how to derive improved estimates in the region $t \sim r$ for the good null components $\alpha(F)$, $\rho (F)$ and $\sigma (F)$ when the energy norm $\mathcal{E}^K_2[\cdot]$ of $F$ is uniformly bounded. For this, we will need results which rely on the relations
\begin{align}\label{eq:nullderiv}
&(t-r)\underline{L}=S-\frac{x^i}{r}\Omega_{0i}, \quad (t+r)L=S+\frac{x^i}{r}\Omega_{0i}, \quad re_{\theta}=-\cos (\varphi) \Omega_{13}-\sin(\varphi)\Omega_{23} ,\quad re_{\varphi}= \Omega_{12} .
\end{align}
Since $t\Omega_{ij}=x^i\Omega_{0j}-x^j\Omega_{0i}$, we can gain time decay through $e_\theta$ and $e_\varphi$ as well. These identities reflect that the solutions enjoy better decay properties in the directions tangential to the light cone, $L$, $e_\theta$ and $e_\varphi$.
\begin{Lem}\label{LemComgoodderiv}
The following properties hold.
\begin{itemize}
\item For any rotation $\Omega \in \{\Omega_{12}, \, \Omega_{13} , \, \Omega_{23} \} $, the operator $\mathcal{L}_{\Omega}$ commutes with the null decomposition,
$$\mathcal{L}_{\Omega}  \big( \underline{\alpha}(G) \big)=\underline{\alpha}  \left( \mathcal{L}_{\Omega} G \right), \quad \;  \mathcal{L}_{\Omega}  \big( \alpha(G) \big)=\alpha  \left( \mathcal{L}_{\Omega} G \right), \quad \; \mathcal{L}_{\Omega}  \big( \rho(G) \big)=\rho  \left( \mathcal{L}_{\Omega} G \right), \quad \; \mathcal{L}_{\Omega} \big( \sigma(G) \big)=\sigma \left( \mathcal{L}_{\Omega} G \right).$$
The operators $\nabla_L$ and $\nabla_{\underline{L}}$ commute as well with the null decomposition, e.g. $\nabla_L \alpha (G)=\alpha (\nabla_L G)$.
\item For any null component $\zeta \in \{ \underline{\alpha}, \alpha ,  \rho , \sigma \}$, we have, for all $(t,x) \in \R_+ \times \R^3$,
$$\langle t-|x| \rangle\left| \nabla_{\underline{L}} \zeta( G) \right|\!(t,x)+\langle t+|x|\rangle \left| \nabla_L \zeta( G) \right|\!(t,x)+\langle x \rangle\left| \slashed{\nabla} \zeta( G) \right|\!(t,x)\lesssim  \sup_{|\gamma| \leq 1} \left| \zeta (\mathcal{L}_{Z^{\gamma}}G)\right|\!(t,x).$$
\end{itemize}
\end{Lem}
\begin{proof}
These are well-known properties, which are proved for instance in \cite[Appendix~$D$]{massless}.
\end{proof}

In the same spirit, we have the following property.
\begin{Lem}\label{improderiv}
Let $Z^{\gamma} \in \mathbb{K}^{|\gamma|}$ composed by at least one translation $\partial_{x^\lambda}$, that is $\gamma_T \geq 1$. If $F$ is sufficiently regular, we have  for any null component $\zeta \in \{ \underline{\alpha},\alpha,  \rho , \sigma \}$,
$$ \big| \zeta \big( \mathcal{L}_{Z^\gamma} F \big) \big|(t,x) \lesssim  \sup_{|\kappa| \leq |\gamma| }  \frac{\big| \zeta \big( \mathcal{L}_{Z^\kappa} F \big) \big|(t,x)}{\langle t-|x| \rangle}+\frac{\big| \mathcal{L}_{Z^\kappa} F \big|(t,x)}{\langle t+|x| \rangle}.$$ 
\end{Lem}
\begin{proof}
Let $0 \leq \lambda \leq 3$. As, for any $Z \in \mathbb{K}$, $[Z, \partial_{x^\lambda}]=0$ or there exists $0 \leq \nu \leq 3$ such that $[Z,\partial_{x^\lambda}]=\pm \partial_{x^\nu}$, it suffices to consider the case where $Z^\gamma=\partial_{x^\lambda} Z^\beta$, $|\beta|=|\gamma|-1$. Introduce then $G= \mathcal{L}_{Z^\beta}F$.
\begin{itemize}
\item If $\lambda =0$, we use $\mathcal{L}_{\partial_t}=\nabla_{\partial_t}$, $2\partial_t=L+\underline{L}$ and Lemma \ref{LemComgoodderiv}. 
\item Otherwise $1 \leq \lambda \leq 3$ and we consider $(t,x) \in \R_+ \times \R^3_x$. In view of \eqref{eq:transladerivatives} and since $\partial_{x^\mu} \in \mathbb{K}$, we have
$$ \langle t-|x| \rangle \, \big| \mathcal{L}_{\partial_{x^{\lambda}}}G \big|(t,x)=  \langle t-|x| \rangle \, \big|\nabla_{\partial_{x^\lambda}} G \big|(t,x) \lesssim \sup_{Z \in \mathbb{K}} \big|\nabla_{Z} G \big|(t,x) \lesssim \sup_{|\kappa| \leq 1} \big|\mathcal{L}_{Z^\kappa} G \big|(t,x).$$
It implies the result if $t \leq 1$ or $|x| \geq 2t$. Otherwise $|x| \leq 2t$ and we write $t\partial_{x^k}= \Omega_{0k}-x^k\partial_t$, so that
$$ t\big| \zeta \big(\mathcal{L}_{\partial_{x^{\lambda}}}G \big) \big|(t,x) \leq \big| \zeta \big(\nabla_{\Omega_{0k}}G \big) \big|(t,x)+|x|\big| \zeta \big(\nabla_{\partial_t}G \big) \big|(t,x)\leq \big| \nabla_{\Omega_{0k}}G \big|(t,x)+|x|\big| \zeta \big(\mathcal{L}_{\partial_t}G \big) \big|(t,x).$$
Since the case $\lambda =0$ holds and $t \gtrsim \langle t+|x| \rangle$, this concludes the proof.
\end{itemize}
\end{proof}

In order to derive improved estimates for certain null components of $F$, we will sometimes take advantage of the null structure of the Maxwell equations. For this, we have to express them in the null frame instead of the Cartersian one as it is done in \eqref{VM2}--\eqref{VM3}.
\begin{Lem}
Assume that $F$ is solution to the Maxwell equations with a magnetic current \eqref{Maxmagn}, that is 
\begin{equation*}
 \nabla^\mu F_{\mu \nu} = J_{\nu}, \qquad \qquad \nabla^\mu {}^* \! F_{\mu \nu} = \widetilde{J}_{\nu}.
 \end{equation*}
Then, these equations are equivalent to
\begin{align}
\hspace{-8mm} r^{-1}\nabla_{\underline{L}}(r\alpha(F))_{e_A}-\slashed{\nabla}_{e_A}\rho(F)-\varepsilon^{AB} \slashed{\nabla}_{e_B} \sigma(F)&=J_{e_A}+\varepsilon^{AB}\widetilde{J}_{e_B}, \qquad \qquad A \in \{\theta , \varphi \},\label{null1} \\
\hspace{-8mm} r^{-1} \nabla_{L}(r\underline{\alpha}(F))_{e_A}+\slashed{\nabla}_{e_A}\rho(F)-\varepsilon^{AB} \slashed{\nabla}_{e_B} \sigma(F) &=J_{e_A}-\varepsilon^{AB}\widetilde{J}_{e_B}, \qquad \qquad A \in \{\theta , \varphi \}, \label{null2} \\
\hspace{-8mm} r^{-2}\nabla_{\underline{L}}(r^2 \rho(F))+\slashed{\nabla} \cdot \underline{\alpha}(F) &= J_{\underline{L}}, \label{null3} \\
\hspace{-8mm} r^{-2}\nabla_{\underline{L}}(r^2 \sigma(F))+\slashed{\nabla} \times \underline{\alpha}(F) &= \widetilde{J}_{\underline{L}} , \label{null4} \\
\hspace{-8mm} -r^{-2} \nabla_L(r^2 \rho(F))+\slashed{\nabla} \cdot \alpha(F) &=J_L   , \label{null5} \\
\hspace{-8mm} -r^{-2}\nabla_L(r^2 \sigma(F))-\slashed{\nabla} \times \alpha(F) &=\widetilde{J}_L. \label{null6}
\end{align}
\end{Lem}
\begin{Rq}
These equations, which are \cite[$(M_1'')$--$(M''_6)$]{CK} in the vacuum case case $J=\widetilde{J}=0$, suggests that $\alpha (F)$, $\rho (F)$ and $\sigma (F)$ have a better behavior than $F$. Indeed, using \eqref{null1} and \eqref{null3}--\eqref{null4}, one can prove that the bad derivative $\underline{L}$ of these composents decays at least as $r^{-2}$ near the light cone $\{t=r\}$, where $\nabla_{\underline{L}}F$ merely decay as $r^{-1}$.
\end{Rq}

Using part of \eqref{null1}--\eqref{null6}, we are able to deduce improved estimates for the derivative of $r \underline{\alpha}(F)$, $r^2 \rho (F)$ and $r^2 \sigma (F)$ in the $\underline{u}$ direction, which will turn out to be crucial in order to prove that these quantities converge along null rays in the framework of this paper.

\begin{Cor}\label{Corgoodnull}
Assume that $F$ is solution to \eqref{Maxmagn}. Then, we have on $[T,+ \infty[ \times \R^3$,
\begin{align*}
 \left|\nabla_L \big(r \underline{\alpha}(F)\big) \right| &  \lesssim r|J|+r\big| \widetilde{J} \, \big|+\sup_{|\kappa| \leq 1} \left| \rho (\mathcal{L}_{Z^\kappa}F) \right|+\left|\sigma (\mathcal{L}_{Z^\kappa}F) \right| , \\
 \left|\nabla_L \big(r^2 \rho(F)\big) \right|+ \left|\nabla_L \big(r^2 \sigma(F)\big) \right| &  \lesssim r^2|J_L|+r^2\big| \widetilde{J}_L \, \big|+\sup_{|\kappa| \leq 1} r\left| \alpha (\mathcal{L}_{Z^\kappa}F) \right|.
\end{align*}
\end{Cor}
\begin{proof}
It suffices to combine \eqref{null2}, \eqref{null5} and \eqref{null6} with Lemma \ref{LemComgoodderiv}.
\end{proof}

We now prove decay estimates for $F$, which rely on Sobolev embeddings as well as \eqref{null1}.
\begin{Pro}\label{decayMaxell}
We have, for all $(t,x) \in [T,+\infty[ \times \R^3$,
\begin{align*}
|\underline{\alpha}(F)|(t,x) & \lesssim \langle t+|x| \rangle^{-1} \, \langle t-|x| \rangle^{-\frac{3}{2}}   \left| \mathcal{E}^K_2 [F](t) \right|^{\frac{1}{2}} ,\\
|\alpha ( F) |(t,x)+ |\rho(F)|(t,x)+ |\sigma (F)|(t,x) & \lesssim  \langle t+|x| \rangle^{-2} \,  \langle t-|x| \rangle^{-\frac{1}{2}}  \left| \mathcal{E}^K_2 [F](t) \right|^{\frac{1}{2}} .
 \end{align*}
If there exists $h:[T,+\infty[ \times \R^3_x \times \R^3_v \to \R$ such that $F$ is solution to
$$ \nabla^\mu F_{\mu \nu} = J(h)_\nu, \qquad \qquad \nabla^\mu {}^*  \! F_{\mu \nu} =0,$$ 
then the best null component $\alpha (F)$ verifies
\begin{align*} 
|\alpha (F)|(t,x) & \lesssim  \langle t+|x| \rangle^{-\frac{5}{2}} \bigg( \,  \left| \mathcal{E}^K_2 [F](t) \right|^{\frac{1}{2}}+ \sup_{|\kappa| \leq 1} \langle t \rangle^{-\frac{1}{2}}\Big\| \langle v \rangle^{\frac{5}{2}} \,  \langle x-t\widehat{v} \rangle^3 \, \widehat{Z}^\kappa h (t,x,v) \Big\|_{L^2(\R^3_x \times \R^3_v)} \, \bigg) .
\end{align*}
\end{Pro}
\begin{proof}
According to the weighted Klainerman-Sobolev of \cite[Lemma~$3.3$]{SchlueLindblad}, for any sufficiently regular function $\psi : [T,+\infty[ \times \R^3 \to \R$, we have
$$ \forall \, (t,x) \in [T,+\infty[ \times \R^3, \qquad \langle t-|x| \rangle^{\frac{3}{2}} \, \langle t+|x| \rangle \, |\psi|(t,x) \lesssim \sup_{|\gamma| \leq 2} \big\| \big\langle t-| \cdot| \big\rangle \, Z^\gamma \psi (t, \cdot) \big\|_{L^2(\R^3_x)} \, .$$
Applying it to any cartesian component $F_{\mu \nu}$, $0 \leq \mu , \, \nu \leq 3$, of the electromagnetic field, we get
$$ \forall \, (t,x) \in [T,+\infty[ \times \R^3, \qquad \langle t-|x| \rangle^{\frac{3}{2}} \, \langle t+|x| \rangle \, |F|(t,x) \lesssim \left| \mathcal{E}^K_2 [F](t) \right|^{\frac{1}{2}} ,$$
which implies the estimate for $\underline{\alpha} (F)$. In order to derive the better behavior of the good null components of $F$, we apply \cite[Proposition~$3.13$]{massless} (the vacuum case was originally treated in \cite[Theorem~$3.1$-$3.2$]{CK}). It provides the stated estimates for $\rho (F)$ as well as $\sigma (F)$ and the first estimate for $\alpha (F)$ is derived in its proof. It provides as well
\begin{equation}\label{decay:alphageneral}
 |\alpha (F)|(t,x)  \lesssim  \langle t+|x| \rangle^{-\frac{5}{2}} \,\bigg( \, \left| \mathcal{E}^K_2 [F](t) \right|^{\frac{1}{2}}+ \sup_{|\kappa| \leq 1} \big\| |\cdot |^2 \,  \cancel{ \mathcal{L}_{Z^{\kappa}}J}(h) \, (t,\cdot) \big\|_{L^2(\R^3)} \bigg), 
 \end{equation}
where $ \cancel{\mathcal{L}_{Z^{\kappa}}J}(h):= \big( |\mathcal{L}_{Z^{\kappa}}J(h)_{e_\theta}|^2+|\mathcal{L}_{Z^{\kappa}}J(h)_{e_\varphi}|^2 \big)^{1/2}$ is the angular part of $\mathcal{L}_{Z^{\kappa}}J(h)$.
Apply Lemmata \ref{LemCom} and \ref{gainv} in order to get
$$ \sup_{|\kappa| \leq 1}  \big| \cancel{ \mathcal{L}_{Z^{\kappa}}J}(h) \big|(t,x) \lesssim  \sup_{|\beta| \leq 1} \int_{\R^3_v} \Big|\slashed{\widehat{v}} \,  \widehat{Z}^\beta h \Big|(t,x,v) \dr v  \lesssim   \sup_{|\beta| \leq 1} \frac{1}{\langle x \rangle}  \int_{\R^3_v}  \langle x-t \widehat{v} \rangle \big| \widehat{Z}^\beta h \big|(t,x,v) \dr v.$$
By applying the Cauchy-Schwarz inequality in $v$, one obtains
\begin{align*} \sup_{|\kappa| \leq 1} \big\| |\cdot |^2 \,  \cancel{ \mathcal{L}_{Z^{\kappa}}J}(h) \, (t,\cdot) \big\|_{L^2(\R^3)}^2 & \lesssim \sup_{|\kappa| \leq 1} \int_{\R^3_x} |x|^2 \bigg| \int_{\R^3_v} \langle x-t\widehat{v} \rangle \big| \widehat{Z}^\kappa h \big|(t,x,v) \dr v \bigg|^2 \dr x \\
& \leq \sup_{|\kappa| \leq 1} \int_{\R^3_x}  \int_{\R^3_v} \langle v \rangle^5 \, \langle x-t\widehat{v} \rangle^6 \big| \widehat{Z}^\kappa h \big|^2(t,x,v) \dr v  \int_{\R^3_v} \frac{|x|^2 \dr v}{\langle x-t\widehat{v} \rangle^4 \, \langle v \rangle^5 } \dr x .
\end{align*}
It remains to bound the second integral in $v$ by $ \langle t\rangle^{-1} $ using \eqref{eq:Lemfordecayvelocityaverage00}.
\end{proof}

\subsection{Control of the initial data for the Maxwell equations}\label{Subsecinidata}

Let $t_0 \geq 0$ be an initial time, $J$ be a sufficiently regular current defined on $[t_0,+\infty[ \times \R^3$ and $F$ be the unique solution to
\begin{equation}\label{eq:Ffordataaa}
 \nabla^\mu F_{\mu \nu} = J_\nu,  \qquad  \nabla^\mu {}^* \! F_{\mu \nu} =0,   \qquad \qquad  \nabla \times E(t_0,\cdot) = \nabla \times B(t_0,\cdot)=0, 
 \end{equation}
where $E^i:=F_{0i}$ and $B^i:= -\varepsilon^{ijk} F_{jk} $ are the electric and magnetic fields. We assume here that their initial divergence free part, according to Helmholtz decomposition, vanish. Thus, there exists $\phi : \R^3 \to \R$ such that 
 $$B(t_0,\cdot)=0, \qquad \qquad E(t_0,\cdot)=\nabla \phi, \qquad \Delta \phi = -J_0(t_0, \cdot ).$$
 In particular $\nabla \times E(t_0,\cdot) = \nabla \times B(t_0,\cdot)=0$ is equivalent to 
 $$ \nabla_{\partial_{x^i}}F_{0 j} (t_0,\cdot)= \nabla_{\partial_{x^j}} F_{0 i} (t_0,\cdot), \qquad \nabla_{\partial_{x^i}}{}^* \!F_{0 j}(t_0,\cdot)= \nabla_{\partial_{x^j}} {}^*\!F_{0 i} (t_0,\cdot) , \qquad \qquad 1 \leq i, \, j \leq 3.$$
In this subsection, we control $F(t_0,\cdot)$ and its derivatives in two cases. For this, note that the Maxwell equations and the initial assumptions imply
\begin{align*}
\partial_t (\partial_t^{n-1}  B)(t_0,\cdot)&= -\partial_t^{n-1}  ( \nabla_x \times E)(t_0,\cdot)=0 ,
\end{align*}
for any $n \in \mathbb{N}^*$. Thus, $B$ and all its spacetime derivatives initially vanish. In order to control the initial time derivatives of the electric field, we will use the following relation, obtained similarly. For any component $1 \leq i \leq 3$ and multi-index $\beta$,
\begin{align}\label{eq:iniMax2}
\partial_t (\partial_t^{n-1} \partial_x^\beta E^i)(t_0,\cdot)=  -\partial_t^{n-1} \partial_x^\beta J_i(t_0,\cdot).
\end{align}
First, we assume that $J$ is generated by a distribution function.
\begin{Pro}\label{Proinidata0}
Let $N \geq 0$, $f: [t_0,+\infty[ \times \R^3_x \times \R^3_v \to \R$ be a sufficiently regular function and assume that $J=J(f)$. Then, for any $|\gamma| \leq N$,
$$ \int_{\R^3_x} \big|\mathcal{L}_{Z^\gamma}F \big|^2(t_0,x)\dr x \lesssim_{t_0} \sup_{|\beta| \leq N} \int_{\R^3_x}\int_{\R^3_v} \langle x \rangle^3 \, \langle v \rangle^4 \big| \widehat{Z}^\beta f \big|^2(t_0,x,v) \dr v \dr x .$$
\end{Pro}
\begin{proof}
In view of the definition of the Lie derivative and the fact that the elements of $\mathbb{K}$ are translations $\partial_{x^\mu}$ or homogeneous vector fields, we have, for $|\gamma| \leq N$, 
$$ \big|\mathcal{L}_{Z^\gamma}F \big|(t_0,x) \lesssim_{t_0} \sup_{|\xi| \leq |\gamma|}   \, \langle x \rangle^{|\xi|} \big| \nabla_{t,x}^\xi F\big|(t_0,x)=\sup_{|\xi| \leq |\gamma|}   \, \langle x \rangle^{|\xi|} \big| \nabla_{x}^\xi E\big|(t_0,x)+\sup_{n+|\beta| \leq |\gamma|, \, n \geq 1}   \, \langle x \rangle^{n+|\beta|} \big|\nabla_t^n \nabla_{x}^\beta E \big|(t_0,x),$$
since $B$ and all its derivatives initially vanish. As $E(t_0,\cdot)=\nabla \phi$, we estimate the spatial derivatives of the initial electric field by applying Lemma \ref{Lemdatahigh0} to $\phi$. The time derivatives can be controlled using \eqref{eq:iniMax2} as well as Lemma \ref{LemgainderivVla}, which provides
$$ \int_{\R^3_x} \langle x \rangle^{2n+2|\beta|} \bigg| \int_{\R^3_v} \widehat{v} \partial_t^{n-1} \partial_x^\beta f(t_0,x,v) \dr v \bigg|^2 \dr x \lesssim_{t_0} \sup_{|\zeta| \leq n-1+|\beta|} \int_{\R^3_x} \langle x \rangle^{2} \bigg| \int_{\R^3_v} \big| \widehat{Z}^\zeta f \big|(t_0,x,v) \dr v \bigg|^2 \dr x .$$
We finally handle the integrals on the right hand side by applying the Cauchy-Schwarz inequality in $v$.
\end{proof}

In the perspective of studying the asymptotic Maxwell equations, we prove stronger estimates for a compactly supported source term.
\begin{Pro}\label{Proinidata}
Let $N \geq 1$ and assume that $t_0 \leq 3$ and that $J(t_0,\cdot)$ is supported in $\{ |x| \leq 3 \}$. Recall
$$ \overline{F}(t,x) = \overline{\chi}(t-|x|)\frac{Q}{4 \pi |x|^2} \frac{x_i}{|x|} \dr t \wedge \dr x^i, \qquad \qquad Q=-\int_{\R^3_x} J_0(t_0,x)\dr x,$$
the pure charge part of $F$. Then, for any $|\gamma| \leq N-1$ and $|\kappa| \leq 1$,
$$ \forall \, x \in \R^3, \qquad  \langle x \rangle^{3+|\kappa|}  \left|\nabla_{t,x}^\kappa \mathcal{L}_{Z^\gamma}(F-\overline{F}) \right|\!(t_0,x) \lesssim \sup_{|\xi| \leq N}\big\| \nabla_{t,x}^\xi J (t_0,\cdot) \big\|_{L^{\infty}(\R^3_x)}.$$
Moreover, the following $L^2$ estimate holds for any $|\gamma| \leq N$ and $|\kappa| \leq 1$,
$$  \int_{\R^3_x} \langle x \rangle^{\frac{5}{2}+2|\kappa|}  \left|\nabla_{t,x}^\kappa \mathcal{L}_{Z^\gamma}(F-\overline{F}) \right|^2 \! (t_0,x) \dr x \lesssim \sup_{|\xi| \leq N}\big\| \nabla_{t,x}^\xi J (t_0,\cdot) \big\|^2_{L^2(\R^3_x)}.$$
\end{Pro}
\begin{proof}
Note that $|\partial_{t,x}^\kappa \partial_t \overline{F}_{0i}|(t_0,x) \lesssim |Q| \mathds{1}_{|x| \leq 5}$. Hence, by similar computations as the one performed in the previous proof, we have for any $|\kappa| \leq 1$ and $|\gamma| \leq N$,
\begin{align*}
\langle x\rangle^{|\kappa|}  \left| \nabla_{t,x}^\kappa \mathcal{L}_{Z^\gamma}(F-\overline{F})\right|(t_0,x)  \lesssim  & \sup_{1 \leq i \leq 3} \sup_{|\xi| \leq |\gamma|+1}   \, \langle x \rangle^{|\xi|} \big| \partial_{x}^\xi (F_{0i}-\overline{F}_{0i})\big|(t_0,x) \\
&+\sup_{n+|\beta| \leq |\gamma|+1, \, n \geq 1}   \, \langle x \rangle^{n+|\beta|} \big|\nabla_t^n \nabla_{x}^\beta E \big|(t_0,x)+|Q|\mathds{1}_{|x| \leq 5}.
 \end{align*}
Using \eqref{eq:iniMax2}, we can directly control the time derivatives of $E(t_0,\cdot)$, which turn out to be compactly supported. Next,
$$\sup_{1 \leq i \leq 3} \sup_{|\xi| \leq |\gamma|+1}   \, \langle x \rangle^{|\xi|} \big| \partial_{x}^\xi (F_{0i}-\overline{F}_{0i})\big|(t_0,x) \leq  \sup_{|\xi| \leq |\gamma|+1}   \, \langle x \rangle^{|\xi|} \Big| \nabla_{x}^\xi  \nabla \Big( \phi (x)+\frac{Q}{4\pi|x|} \overline{\chi}(t_0-|x|)  \Big)\Big|+|Q| \mathds{1}_{|x| \leq 5}.$$
If $Q  \neq 0$, one cannot apply Lemma \ref{Lemdatahigh} directly. Instead, remark that $\widetilde{\phi}=\phi+\frac{1}{4\pi r} Q\overline{\chi}(t_0-r)$ verifies
$$ \Delta \widetilde{\phi} = \widetilde{\psi}, \qquad \widetilde{\psi}:= \int_{\R^3_v} f(t_0,\cdot,v)\dr v+\frac{Q}{4\pi r} \chi''(t_0-r), \qquad \int_{\R^3_x}\widetilde{\psi} =0 . $$
As $\chi =1$ on $]-\infty,-2]$ and $\chi =0$ on $[-1,+\infty[$, the result ensues from Lemma \ref{Lemdatahigh}.
\end{proof}

\subsection{Study of the pure charge part of the electromagnetic field}\label{SubsecPurecharge}

The previous Proposition \ref{Proinidata} suggests that $\overline{F}$ plays a crucial role in the analysis of the asymptotic behavior of electromagnetic field of charge $Q \neq 0$. In fact, as noticed by \cite{LS}, the long range effect of the charge is only relevant in the exterior of the light cone $|x| \geq t$. In particular, in the context of the Vlasov-Maxwell system, the asymptotic behavior of $r^2 \rho (F)$ along null rays $r \mapsto (r+u,r\omega)$ will be governed, as $u \to -\infty$, by $r^2\rho(\overline{F})$. We collect here classical properties of $\overline{F}$.
\begin{Pro}\label{Propurecharge}
Let $Q \in \R$. The following properties hold.
\begin{enumerate}
\item The null components of $\overline{F}$ verify, for all $(t,x) \in \R_+ \times \R^3$,
$$ \rho \big( \overline{F} \, \big) (t,x) = \frac{Q}{4\pi |x|^2} \overline{\chi} (t-|x|), \qquad \qquad \underline{\alpha}\big( \overline{F} \, \big) (t,x) =\alpha \big( \overline{F} \, \big) (t,x) =\sigma \big( \overline{F} \, \big) (t,x) =0.$$
\item $\overline{F}$ is a solution to the Maxwell equations $\nabla^\mu \overline{F}_{\mu \nu} = \overline{J}_{\mu \nu}$ and $\nabla^{\mu} {}^* \! \overline{F}_{\mu \nu} =0$, where the source term $\overline{J}$ is compactly supported and given by
$$  \overline{J}_0(t,x):=\frac{Q}{4\pi |x|^2} \overline{\chi}'(t-|x|), \qquad \qquad \overline{J}_i(t,x):= -\frac{Q}{4\pi |x|^2}\frac{x_i}{|x|} \overline{\chi}'(t-|x|), \quad 1 \leq i \leq 3.$$
Its null components are
$$\overline{J}_{\underline{L}}(t,x)=\frac{Q}{2\pi |x|^2} \overline{\chi}'(t-|x|), \qquad \qquad \overline{J}_L(t,x)=\overline{J}_{e_\theta}(t,x)=\overline{J}_{e_\varphi}(t,x)=0.$$
\item The Lie derivatives with respect to the scaling vector field and the rotational derivatives are chargeless,
$$ \mathcal{L}_S \overline{F} (t,x)= (t-|x|)\overline{\chi}'(t-|x|)\frac{Q}{4 \pi |x|^2} \frac{x_i}{|x|} \dr t \wedge \dr x^i, \qquad \qquad \qquad \mathcal{L}_{\Omega_{ij}} \overline{F} =0, \quad 1 \leq i < j \leq 3.$$
\item Let $Z^\gamma \! \in \! \mathbb{K}^{|\gamma|}$. There exists $C_\gamma >0$ such that the null components of $\mathcal{L}_{Z^\gamma} \overline{F}$ verify, $\forall \, (t,x) \in \R_+ \! \times \R^3$,
\begin{align*}
\big|\underline{\alpha} \big(\mathcal{L}_{Z^\gamma} \overline{F} \, \big)\big|(t,x)+ \big|\rho \big(\mathcal{L}_{Z^\gamma} \overline{F} \, \big)\big|(t,x)+\big|\sigma \big(\mathcal{L}_{Z^\gamma} \overline{F} \, \big)\big|(t,x)& \leq C_\gamma \frac{Q }{ \langle t+|x| \rangle^2}\mathds{1}_{|x| \geq t+1} , \\[3pt]
 \big|\alpha \big(\mathcal{L}_{Z^\gamma} \overline{F} \, \big)\big|(t,x) &\leq C_\gamma \frac{ Q \, \langle t-|x| \rangle}{ \langle t+|x| \rangle^3} \mathds{1}_{|x| \geq t+1} .
 \end{align*}
 The ones of $\mathcal{L}_{Z^\gamma}\overline{J}$ satisfy, for any $A \in \{\theta , \varphi \}$ and for all $(t,x) \in \R_+ \times \R^3$,
 $$ \langle t+|x| \rangle^2\big| \mathcal{L}_{Z^\gamma}\big( \overline{J} \, \big)_{\underline{L}} \big|(t,x) + \langle t+|x| \rangle^3\big| \mathcal{L}_{Z^\gamma}\big( \overline{J} \, \big)_{e_A} \big|(t,x) + \langle t+|x| \rangle^4\big| \mathcal{L}_{Z^\gamma}\big( \overline{J} \, \big)_{L} \big|(t,x) \leq C_\gamma Q \mathds{1}_{1 \leq |x|-t \leq 2} .$$ 
\end{enumerate}
\end{Pro}
\begin{proof}
The first part of the statement is given by direct computations. The equations $\nabla^{\mu} {}^* \! \overline{F}_{\mu \nu}=0$, equivalent to $\nabla_{[ \lambda} \overline{F}_{\mu \nu]}=0$, follow from $\overline{F}_{ij}=0$ and from the fact that the electric field associated to $\overline{F}$ is radial, so that $\nabla_i \overline{F}_{ 0j}=\nabla_j \overline{F}_{0i} $. For the other ones, we have
\begin{align*}
 \nabla^{i} \overline{F}_{i0} & = -\frac{Q}{4 \pi} \partial_{x^i} \! \bigg( \frac{x^i}{|x|^3} \overline{\chi}(t-|x|) \! \bigg)\! = -\frac{Q}{4 \pi}\! \left(  \left( \frac{3}{|x|^3}   -3\frac{x_i x^i}{|x|^5} \right) \overline{\chi}(t-|x|)-  \frac{x^ix_i}{|x|^4} \overline{\chi}'(t-|x|) \! \right)\! = \frac{Q}{4 \pi |x|^2} \overline{\chi}'(t-|x|), \\ 
\nabla^{\mu} \overline{F}_{\mu i} & = -\partial_t \overline{F}_{0i} = -\frac{Q}{4\pi} \frac{x^i}{|x|^3} \overline{\chi}'(t-|x|).
\end{align*}
For the third part, recall from \eqref{keva:defLie} the expression of the Lie derivative and use $S(x_i/|x|^3)=-2x_i/|x|^3$, $S(t-|x|)=t-|x|$ as well as $\Omega_{ij}(|x|)=\Omega_{ij}(t)=0$. Finally, the estimates for the null components of the derivatives of $\overline{F}$ and $\overline{J}$ were obtained in \cite[Proposition~$5.3$]{LS} and \cite[Equations~$(3.52a)$--$(3.52c)$]{LS}. They can be derived, through an induction, by exploiting the following properties.
\begin{itemize}
\item For any $Z \in \mathbb{K}$ and any $X, \, Y \in \{ \underline{L}, \, L, \, e_\theta, \, e_\varphi \}$, we have
$$ \mathcal{L}_{ZZ^\gamma} \big(\, \overline{J} \, \big)_X=Z \big( \mathcal{L}_{Z^\gamma} \big( \,\overline{J} \, \big)_X \big)-\mathcal{L}_{Z^\gamma} \big( \, \overline{J} \, \big)_{[Z,X]}, $$
as well as
$$ \mathcal{L}_{ZZ^\gamma}\big(\, \overline{F} \, \big)_{XY}=Z \big( \mathcal{L}_{Z^\gamma} \big(\, \overline{F} \, \big)_{XY} \big)-\mathcal{L}_{Z^\gamma} \big(\, \overline{F} \, \big)_{[Z,X] Y}-\mathcal{L}_{Z^\gamma} \big(\, \overline{F} \, \big)_{X [Z,Y]}.$$
\item The components of the commutation vector fields with respect to the null frame are given by
\begin{align*}
&2\partial_t=L+\underline{L}, \quad 2S=\underline{u}L+u \underline{L}, \quad \partial_{x^k}=-\frac{\omega_k}{2}\underline{L}+\frac{\omega_k}{2}L+\omega_k^{e_A}e_A, \quad \Omega_{0k}=\frac{\omega_k}{2}(\underline{u} L-u \underline{L})+t \omega_k^{e_A} e_A, \\
& \Omega_{ij}=(x^i \omega_j^{e_A}-x^j \omega_i^{e_A})e_A, \quad \; \; \omega_i := \langle \partial_{x^i}, \partial_r \rangle=\frac{x^i}{|x|}, \quad \omega^{e_A}_i:= \langle \partial_{x^i},e_A \rangle, \quad \; \; 1 \leq k \leq 3, \quad 1 \leq i < j \leq 3, 
\end{align*}
where $\omega_i^{e_A}$ is a smooth function of the spherical variables $(\theta, \varphi) \in ]0,\pi[\times ]0,2\pi[$ (see \eqref{kevatalenn:omega} below).
\item The commutator of elements of $\mathbb{K}$ and the null frame are, for $1 \leq i < j \leq 3$, $1 \leq k \leq 3$ and $A \in \{ \theta , \varphi \}$,
\begin{align} 
\nonumber &[\partial_t,L]=[\partial_t, \underline{L}]  = 0, \quad [\partial_t , e_A]=0 , \quad [S,L]  = -L,  \quad [S, \underline{L}]=- \underline{L} , \quad [S,e_A]=-e_A,  \\ \nonumber & [\Omega_{ij}, L]=[\Omega_{ij}, \underline{L}] = 0 , \quad [\Omega_{ij}, e_A]= -e_A(\Omega^B_{ij})e_B-\Omega_{ij}^B[e_A,e_B]^D e_D , \quad \Omega_{ij}^B = r(\omega_i \omega_j^{e_B}-\omega_j \omega_i^{e_B}), \\ \nonumber & [\Omega_{0k}, L] = \frac{t-r}{r} \omega_k^{e_B} e_B-\omega_k L, \quad [\Omega_{0k}, \underline{L}] =- \frac{t+r}{r} \omega^{e_B}_k e_B+\omega_k \underline{L},\\ \nonumber &
 [\Omega_{0k},e_A] = \frac{\omega^{e_A}_k}{2r}((t-r) \underline{L}-(t+r)L)+t \omega_k^{e_B} \slashed{\Gamma}^D_{BA}e_D , \\ \nonumber &  [\partial_{x^k},L]=- [\partial_{x^k},\underline{L}] =\frac{1}{r}(\partial_{x^k}-\frac{x^k}{r}\partial_r)= \frac{1}{r^2} \left( \frac{x^j}{r} \Omega_{kj} \right), \quad [\partial_{x^k},e_A]=\partial_{x^k}(\omega_\ell^{e_A})\partial_{x^\ell},
\end{align}
where $\slashed{\Gamma}_{AB}^D$ are the Christofell symbols in the nonholonomic basis $(e_\theta, e_\varphi)$ of $\mathbb{S}^2$, $\slashed{\nabla}_{e_A}e_B = \slashed{\Gamma}_{AB}^D e_D$. In particular, for any $X \in \mathcal{N}:=\{\underline{L},L,e_\theta,e_\varphi \}$, $[\partial_{x^k},X]=\sum_{Y \in \mathcal{N}} r^{-1}a_Y(\omega)Y$, where $a_Y$ is a smooth function of $(\theta, \varphi)$.
\item For any $1 \leq k \leq 3$ and $1 \leq i < j \leq 3$, we have, with $u=t-|x|$,
 \begin{equation}\label{eqtminusr}
  \partial_t (u)=1, \qquad \partial_{x^k}(u)=-\frac{x^k}{|x|}, \qquad S(u)=u, \qquad \Omega_{ij}(u)=0, \qquad \Omega_{0k}(u)=-\frac{x^k}{|x|}u.
  \end{equation}
\end{itemize}
\end{proof}

We now give a schematic form for $\mathcal{L}_{Z^\gamma} \overline{F}$, which will be useful in order to prove scattering results.
\begin{Pro}\label{Prorad}
Let $Z^\gamma \in \mathbb{K}^{|\gamma|}$ containing $\gamma_T \geq 0$ translations $\partial_{x^\lambda}$. Then, for any cartesian components $0 \leq \mu , \, \nu \leq 3$, we can write $\mathcal{L}_{Z^\gamma} (\overline{F} \,)_{\mu \nu}(t,x)$ as a linear combination of terms of the form
$$ \frac{Q}{|x|^{2+\gamma_T}} P \Big(\frac{x}{|x|},\frac{t}{|x|} \Big) (t-|x|)^p\chi^{(q)}(t-|x|),$$
where $0 \leq p \leq q$ and $P$ is a polynomial. Thus, $r^2\mathcal{L}_{Z^\gamma} (\overline{F} \,)_{\mu \nu}(r+u,r\omega)$ converges as $r \to +\infty$, for all $(u,\omega) \in \R_u \times \mathbb{S}^2_\omega$. If $u \leq -2$, the limit is independent of $u$.
\end{Pro}
\begin{proof}
The result follows from an induction, that the elements of $\mathbb{K}$ are either translations $\partial_{x^\lambda}$ or homogeneous vector fields as well as \eqref{eqtminusr}.
\end{proof}

In Section \ref{SecMaxasymp}, for simplicity, we will work with a similar but slightly different quantity to $\overline{F}$. The motivation is that its cartesian componenents will be solution to the homogeneous wave equation. The singular electromagnetic field $\frac{Qx_i}{4\pi|x|^3} \dr t \wedge \dr x^i$, generated by a point charge $Q$ located at $x=0$, derives from the potential $A=Q(4\pi  r)^{-1} \dr t$. It turns out that $A$ satisfies the Lorenz gauge $\partial_{x^\mu} A^\mu=0$, so that $\Box A_\lambda =0$ on $\R\times (\R^3\setminus \{ 0 \})$. To deal with our evolution problem and the singularity of the Newton potential, we introduce 
$$\overline{A}(t,x):= \overline{\chi}(t-|x|) A(t,x)=\frac{Q}{4 \pi |x|} \overline{\chi} (t-|x|) \dr t,$$
which is smooth and verifies $\Box \overline{A}_\mu =0$ as well. It motivates the introduction of
$$ \widetilde{F}(t,x) := \mathrm{d} \overline{A}(t,x) =\overline{F}(t,x)+\frac{Qx_i}{4 \pi |x|^2} \overline{\chi}' (t-|x|) \mathrm{d} t \wedge \dr x^i,$$
which, in view of $[\Box,\partial_{x^\lambda}]=0$ and $\Box \overline{A}_\lambda =0$, verifies $\Box \widetilde{F}_{\mu \nu}=0$.

\begin{Pro}\label{Propurechargetilde}
For any multi-index $\gamma$ and $0 \leq \mu, \, \nu \leq 3$, we have $ \Box \, \mathcal{L}_{Z^\gamma} ( \widetilde{F} )_{\mu \nu} =0$. Moreover, 
$$ \forall \, (t,x) \in \R_+ \times \R^3, \qquad \big| \mathcal{L}_{Z^\gamma} \big( \widetilde{F}-\overline{F} \big) \big|(t,x) \leq C_\gamma Q \, \langle t+|x| \rangle^{-1} \, \mathds{1}_{1 \leq |x|-t \leq 2} .$$
\end{Pro}
\begin{proof}
These results essentially follows from the next properties.
\begin{itemize}
\item For any $Z \in \mathbb{K} \setminus \{ S \}$, $|\Box, Z]=0$ and $[\Box, S]=2 \, \Box$.
\item For any $Z=Z^\lambda \partial_{x^\lambda} \in \mathbb{K}$ and any $2$-form $H$, we have $\mathcal{L}_Z(H)_{\mu \nu}= Z(H_{\mu \nu})+\partial_{x^\mu}(Z^\lambda)H_{\lambda \nu}+\partial_{x^\nu}(Z^\lambda)H_{\mu \lambda }$.
\item $Z^\lambda$ is either a monomial, $\pm x^\mu$, or a constant, $0$ or $1$.
\item We have $|\chi^{(n)} (t-|x|) | \leq \mathds{1}_{1 \leq |x|-t \leq 2}$ for any $n \geq 1$. Hence, $|x| \gtrsim \langle t+|x| \rangle$ on the support of $\chi^{(n)}$.
\end{itemize}
Using $1.$, $2.$ and $3.$, together with an induction, one gets $ \Box \, \mathcal{L}_{Z^\gamma} ( \widetilde{F} )_{\mu \nu} =0$ for any $Z^\gamma \in \mathbb{K}^{\gamma|}$. The pointwise decay estimates ensue from $3.$, $4.$ and \eqref{eqtminusr}.
\end{proof}

\section{Scattering results for the Maxwell equations}\label{SecscattMax}
We fix, for all this section, an initial time $T \geq 0$.
\subsection{Forward evolution of electromagnetic fields}

This section is devoted to general scattering results for the Maxwell equations in Minkowski spacetime. We start by giving a definition of the radiation field of an electromagnetic field which holds in a very general setting. 
\begin{Def}\label{Defrad}
Let $F$ be a $2$-form defined on $[T,+\infty[ \times \R^3$. If there exists a $1$-form $\underline{\alpha}^{\mathcal{I}^+}$, defined on $\R_u \times \mathbb{S}^2_\omega$ and tangential to the $2$-spheres, such that
$$ r\underline{\alpha}(F)(r+u,r\omega) \xrightharpoonup[r \to + \infty]{} \underline{\alpha}^{\mathcal{I}^+}(u,\omega) \qquad \qquad \text{in $\mathcal{D}' \big( \R_u \times \mathbb{S}^2_\omega \big)$,}$$
we will say that $\underline{\alpha}^{\mathcal{I}^+}$ is the radiation field of $F$ along future null infinity $\mathcal{I}^+$.
\end{Def}
\begin{Rq}
The convergence could also be written $r\underline{\alpha}(F)(u,\underline{u},\omega) \xrightharpoonup[]{} \underline{\alpha}^{\mathcal{I}^+}(u,\omega)$, as $\underline{u} \to +\infty$, where it is implied that $r\underline{\alpha}(F)(\cdot,\underline{u},\cdot)$ is extended by $0$ for $u<2T-\underline{u}$ and $u >\underline{u}$.
\end{Rq}

We prove now that any solution $F$ to the Maxwell equations scatters provided that the initial data as well as the source term are sufficiently regular and enjoy enough decay. Then, we relate the derivatives of its radiation field to the ones of $\mathcal{L}_{Z^\gamma} F$. For this, we will use the following quantities depending only on the spherical variables,
\begin{align}\label{kevatalenn:omega}
&\omega_i:= \langle \partial_{x^i}, \partial_r \rangle= \frac{x^i}{|x|} , \qquad \omega^{e_A}_i:=\langle \partial_{x^i}, e_A \rangle, \qquad 1 \leq i \leq 3, \quad A \in \{\theta, \varphi \},\\ \nonumber &\omega^{e_\theta}_1=\cos (\varphi ) \cos (\theta), \quad \omega^{e_\theta}_2=\sin (\varphi) \cos (\theta), \quad \omega^{e_\theta}_3= -\sin (\theta), \quad  \omega^{e_\varphi}_1=-\sin(\varphi), \quad \omega^{e_\varphi}_2=\cos (\varphi), \quad \omega^{e_\varphi}_3=0,
\end{align}
By an abuse of notation, we will sometimes denote the functions $(u,\omega) \mapsto u$, $(u,\omega) \mapsto \omega_i$ and $(u,\omega) \mapsto \omega_i^{e_A}$ by $u$, $\omega_i$ and $\omega_i^{e_A}$. The assumptions of the next result will be satisfied by the electromagnetic field and its derivatives in the context of the study of the solutions to the Vlasov-Maxwell system with a small distribution function.
\begin{Pro}\label{blackboxscat}
Let $F$ be a sufficiently regular solution to the Maxwell equations \eqref{eq:Maxeqn}, with source term $J$. Let $J^\gamma$ be the current such that
$$ \forall \, |\gamma| \leq 1, \qquad  \nabla^\mu \mathcal{L}_{Z^\gamma}(F)_{\mu \nu} = J_\nu^\gamma, \qquad \nabla^\mu {}^* \! \mathcal{L}_{Z^\gamma}(F)_{\mu \nu} =0$$
and assume that there exists $D \geq 0$ such that
$$\sup_{|\gamma| \leq 1} \int_{\R^3_x} \left| \mathcal{L}_{Z^\gamma} F \right|^2(T,x) \dr x+\sup_{|\gamma| \leq 1} \, \sup_{t \geq T} \int_{\R^3_x} \langle t+|x| \rangle^3 |J^\gamma|^2(t,x) \dr x \leq D.$$
Then, $\sup_{t \ge T} \| \mathcal{L}_{Z^\gamma}F(t,\cdot) \|_{L^2_x}^2 \lesssim D$ for any $|\gamma| \leq 1$ and $F$ has a radiation field $\underline{\alpha}^{\mathcal{I}^+}$ along $\mathcal{I}^+$ satisfying
$$ \forall \, U \geq 1, \; \underline{u} \geq 2U, \quad \int_{|u| \leq U} \int_{\mathbb{S}^2_\omega} \left| \big[ r \underline{\alpha}(F) \big](u, \underline{u},\omega)-\underline{\alpha}^{\mathcal{I}^+}\!(u,\omega) \right|^2 \dr \mu_{\mathbb{S}^2_\omega} \dr u \lesssim \frac{UD }{ \langle \underline{u} \rangle^{ \frac{3}{4}}}, \quad \qquad \left\| \underline{\alpha}^{\mathcal{I}^+} \right\|_{L^2(\R_u \times \mathbb{S}^2_\omega)} \lesssim D^{\frac{1}{2}}  .$$
For any $Z \in \mathbb{K}$, there exists a $1$-form $\underline{\alpha}^{\mathcal{I}^+}_{ \, Z}$, defined on $\R_u \times \mathbb{S}^2_\omega$ and tangential to the $2$-spheres, such that
\begin{equation*}
 \big[ r \underline{\alpha}(\mathcal{L}_{Z}F) \big](\cdot , \underline{u},\cdot) \xrightharpoonup[\underline{u} \to + \infty]{} \underline{\alpha}^{\mathcal{I}^+}_{ \, Z} \qquad \text{ in $L^2 \big(\R_u \times \mathbb{S}^2_\omega \big)$}.
\end{equation*}
Moreover, we have for any $1 \leq k \leq 3$ and $1 \leq i < j \leq 3$,
\begin{alignat*}{2}
  \underline{\alpha}^{\mathcal{I}^+}_{\, \partial_t} &= \nabla_u \underline{\alpha}^{\mathcal{I}^+}, \qquad &&\underline{\alpha}^{\mathcal{I}^+}_{\, \partial_{x^k}}=-\omega_k \nabla_u \underline{\alpha}^{\mathcal{I}^+}, \qquad \qquad \underline{\alpha}^{\mathcal{I}^+}_{\, S}=u\nabla_u \underline{\alpha}^{\mathcal{I}^+}+\underline{\alpha}^{\mathcal{I}^+}, \\
\underline{\alpha}^{\mathcal{I}^+}_{\, \Omega_{ij}}& = \mathcal{L}_{\Omega_{ij}} \big( \underline{\alpha}^{\mathcal{I}^+} \big), \qquad && \underline{\alpha}^{\mathcal{I}^+}_{\, \Omega_{0k}}=- \omega_k u\nabla_u  \underline{\alpha}^{\mathcal{I}^+}-2\omega_k \underline{\alpha}^{\mathcal{I}^+}+ \omega^{e_A}_k \slashed{\nabla}_{e_A}  \underline{\alpha}^{\mathcal{I}^+} .
\end{alignat*}
\end{Pro}
\begin{proof}
Working with the null coordinates $(u,\underline{u},\omega)$, where $\omega \in \mathbb{S}^2$, $u=t-r$ and $\underline{u}=t+r$, will be convenient here. Note first that the assumption on the source term and Proposition \ref{Propartialt} provide, for all $|\gamma| \leq 1$,
\begin{align}
  \sup_{t \in \R_+}\| \mathcal{L}_{Z^\gamma}F(t,\cdot) \|_{L^2_x}^2+ \sup_{\underline{u} \in \R_+} \int_{u=2T-\underline{u}}^{\underline{u}}  \int_{\mathbb{S}^2_\omega}  \big|r \underline{\alpha} (\mathcal{L}_{Z^\gamma}F)\big|^2(u,\underline{u},\omega) \dr \mu_{\mathbb{S}^2_\omega} \dr u &  \nonumber \\
  + \sup_{u \in \R} \int_{\underline{u}=2T-u}^{+\infty} \int_{\mathbb{S}^2_\omega}  \big[|r \rho  (\mathcal{L}_{Z^\gamma}F)|^2+|r \sigma (\mathcal{L}_{Z^\gamma}F)|^2 \big](u,\underline{u},\omega) \dr \mu_{\mathbb{S}^2_\omega} \dr \underline{u}  & \lesssim \| \mathcal{L}_{Z^\gamma}F(0,\cdot) \|_{L^2_x}^2+\bigg|\int_{\tau=0}^{+\infty} \frac{ D^{\frac{1}{2}} \, \mathrm{d} \tau}{\langle \tau \rangle^{\frac{3}{2}}} \bigg|^2  \lesssim D. \nonumber
 \end{align}
In the following, in order to avoid technicalities related to the domain of integrations, we extend $F$ and its derivatives by $0$ for $u > \underline{u}$ and $u<2T-\underline{u}$.

 Let us prove now that $F$ admits a radiation field. For this, recall that $L=2\partial_{\underline{u}}$. Then, remark that for a smooth function $\psi :\R_+ \times \R^3_x \to \R$ and for all $0 \leq \underline{u}_1 \leq \underline{u}_2$, $\omega \in \mathbb{S}_\omega$,
$$
 \forall \, |u| \leq \underline{u}_1, \qquad | \psi(u,\underline{u}_2,\omega)-\psi(u,\underline{u}_1,\omega)| \leq 2\int_{\underline{u}=\underline{u}_1}^{\underline{u}_2}|L \psi|(u,\underline{u},\omega) \dr \underline{u}.
 $$
We then obtain, applying Minkowski's integral inequality, that for $U \ge 1$ and for all $2U \leq \underline{u}_1 \leq \underline{u}_2$,
\begin{align*}
\bigg| \int_{|u|\leq U}  \int_{\mathbb{S}^2_\omega}&\left| \big[r \underline{\alpha}(F)\big](u,\underline{u}_1,\omega)-\big[r \underline{\alpha}(F)\big](u,\underline{u}_2,\omega) \right|^2\dr  \mu_{\mathbb{S}^2_\omega} \dr u \bigg|^{\frac{1}{2}} \\
& \qquad \qquad \qquad \quad \lesssim \int_{\underline{u}=\underline{u}_1}^{\underline{u}_2}  \bigg| \int_{|u| \leq U}  \int_{\mathbb{S}^2_\omega}\left|\nabla_L \big[r \underline{\alpha}(F)\big](u,\underline{u},\omega)  \right|^2 \dr \mu_{\mathbb{S}^2_\omega} \dr u \bigg|^{\frac{1}{2}} \! \dr \underline{u} \\
& \qquad \qquad \qquad \quad \lesssim  \frac{1}{\langle \underline{u}_1 \rangle^{\frac{1}{4}}}\bigg| \int_{\underline{u}=\underline{u}_1}^{+\infty}  \int_{|u| \leq U} \int_{\mathbb{S}^2_\omega} \frac{\langle \underline{u}\rangle^{\frac{3}{2}}}{r^2}\left|r \nabla_L \big[r \underline{\alpha}(F)\big](u,\underline{u},\omega)  \right|^2 \dr \mu_{\mathbb{S}^2_\omega} \dr u  \dr \underline{u}  \bigg|^{\frac{1}{2}} .
\end{align*}
The goal now consists in bounding the right hand side by $\langle \underline{u}_1 \rangle^{-\frac{3}{8}}$. For this, recall from Corollary \ref{Corgoodnull} that
\begin{equation}\label{eq:Lunderlinealpha}
 \left|\nabla_L \big(r \underline{\alpha}(F)\big) \right|  \lesssim r|J|+\sup_{|\kappa| \leq 1} \left| \rho (\mathcal{L}_{Z^\kappa}F) \right|+\left|\sigma (\mathcal{L}_{Z^\kappa}F) \right| .
\end{equation}
Consider $|\kappa| \leq 1$ and remark that we have $2r=\underline{u}-u \gtrsim \langle \underline{u} \rangle$ on the domain of integration since $\underline{u} \geq 2U$ and $|u| \leq U$. We then have, using the uniform boundedness of the flux of $\rho(\mathcal{L}_{Z^\kappa} F)$ and $\sigma(\mathcal{L}_{Z^\kappa} F)$ on outgoing null cones,
$$ \int_{\underline{u}=\underline{u}_1}^{+\infty}  \int_{|u| \leq U} \int_{\mathbb{S}^2_\omega} \frac{\langle \underline{u}\rangle^{\frac{3}{2}}}{r^2}\left|\big(\left|r \rho (\mathcal{L}_{Z^\kappa}F) \right|+\left|r\sigma (\mathcal{L}_{Z^\kappa}F) \right|\big)  \right|^2\! (u,\underline{u},\omega) \dr \mu_{\mathbb{S}^2_\omega} \dr u \, \dr \underline{u} \lesssim \int_{|u| \leq U}  \frac{D \,\dr u}{\langle \underline{u}_1 \rangle^{\frac{1}{2}}}\lesssim \frac{U  D }{\langle \underline{u}_1 \rangle^{\frac{1}{2}}} . $$
As the domaine of integration is included in $\{t+|x| \geq \underline{u}_1 \}$ and $ r^2\dr \mu_{\mathbb{S}^2_\omega} \dr u \, \dr \underline{u}=2\dr x \dr t$, we have, in view of the assumptions on the source term,
\begin{align*}
 \int_{\underline{u}=\underline{u}_1}^{+\infty}  \int_{|u| \leq U} \int_{\mathbb{S}^2_\omega}  \frac{\langle \underline{u}\rangle^{\frac{3}{2}}}{r^2}\left| r^2 J \right|^2\! (u,\underline{u},\omega)  \dr \mu_{\mathbb{S}^2_\omega} \dr u \, \dr \underline{u}  \leq \frac{2}{\langle \underline{u}_1 \rangle^{\frac{1}{4}}}\int_{t=\underline{u}_1}^{+\infty}  \frac{1}{\langle t \rangle^{\frac{5}{4}}}\int_{\R^3_x} \langle t+|x| \rangle^{3} |J|^2(t,x)\dr x \dr t \lesssim \frac{D}{\langle \underline{u}_1 \rangle^{\frac{1}{4}}}.
\end{align*} 
The previous estimates imply that $\underline{u} \mapsto \big(r \underline{\alpha}(F)\big)(\cdot,\underline{u},\cdot)$ is Cauchy, and then converges, in $L^2([-U,U] \times \mathbb{S}^2_\omega)$. Furthermore, the rate of convergence is bounded below by $U D \, \langle \underline{u} \rangle^{-3/8}$. 

We proved that for any $Z \in \mathbb{K}$, $\underline{u} \mapsto [ r\underline{\alpha}( \mathcal{L}_Z F) ](\cdot, \underline{u} , \cdot)$ is uniformly bounded in $L^2(\R_u \times \mathbb{S}^2_\omega)$. It implies the existence of a strictly increasing and unbounded sequence $(\underline{u}_n)_{n \geq 0}$ as well as $1$-forms tangential to the $2$-spheres $\underline{\alpha}^{\mathcal{I}^+}_{ \,Z}$, such that 
\begin{equation}\label{eq:Lemstatement}
\forall \, Z \in \mathbb{K}, \qquad \qquad  \big( r \underline{\alpha}(\mathcal{L}_{Z}F) \big)(\cdot , \underline{u}_n,\cdot) \rightharpoonup \underline{\alpha}^{\mathcal{I}^+}_{ \, Z} \qquad \text{as $n \to +\infty$, in $L^2 \big(\R_u \times \mathbb{S}^2_\omega \big)$}.
\end{equation}
We now prove that the quantities $\underline{\alpha}^{\mathcal{I}^+}_{\, Z}$ are uniquely defined, that is that they do depend on the sequence $(\underline{u}_n)_{n \geq 0}$, by relating them to $\underline{\alpha}^{\mathcal{I}^+}$. It will imply the last part of the statement of the Proposition. For this, we fix $U \ge 1$ and we will write $\psi \rightharpoonup \psi^{\mathcal{I}^+}$ if we have the convergence
$$ \psi (\cdot , \underline{u},\cdot )  \rightharpoonup \psi^{\mathcal{I}^+}  \qquad \text{as $\underline{u} \to +\infty$, in $\mathcal{D}' \big( \, ]-U,U \, [  \times \mathbb{S}^2_\omega \big)$}.$$ 
If the convergence only holds for a striclty increasing and divergent sequence of advanced times $\mathbf{s}=(\underline{u}_n)_{n \in \mathbb{N}}$, we will write $\psi \rightharpoonup_{\mathbf{s}} \psi^{\mathcal{I}^+}$. Recall the quantities $\omega_j^A := \langle e_A,\partial_{x^j} \rangle$, which only depend on the variable $\omega \in \mathbb{S}^2$. Let us prove, for any $1 \leq j \leq 3$ and $B \in \{ \theta , \varphi \}$,
\begin{equation}\label{eq:weakconv1toprove}
r\underline{\alpha}(F) \rightharpoonup \underline{\alpha}^{\mathcal{I}^+},  \quad  \frac{r}{2}\underline{L} \big( \underline{\alpha}(F)_{e_B} \big) \rightharpoonup \partial_u \! \left(\underline{\alpha}^{\mathcal{I}^+}_{e_B}\right),  \quad r^2\omega_j^A e_A\big( \underline{\alpha}(F)_{e_B} \big) \rightharpoonup \omega_j^A e_A \! \left( \underline{\alpha}^{\mathcal{I}^+}_{e_B} \right) 
\end{equation}
and that there exists $\mathbf{s}=(\underline{u}_n)_{n \in \mathbb{N}}$, depending possibly on $U$, such that, for any $|\gamma| \leq 1$ and $Z \in \mathbb{K}$,
\begin{equation}\label{eq:weakconv2toprove}
r\underline{\alpha}(\mathcal{L}_Z F) \rightharpoonup_\mathbf{s} \underline{\alpha}^{\mathcal{I}^+ }_{ \, Z} , \quad r^2 L\big( \underline{\alpha}(F)_{e_B}\big) \rightharpoonup_{\mathbf{s}} -\underline{\alpha}^{\mathcal{I}^+}_{e_B},\quad r \rho (\mathcal{L}_{Z^\gamma}F) \rightharpoonup_\mathbf{s} 0, \quad r \sigma (\mathcal{L}_{Z^\gamma}F) \rightharpoonup_\mathbf{s} 0, \quad  |F| \rightharpoonup_\mathbf{s} 0.
\end{equation}
Then, the result would ensue from \cite[Proposition~$B.2$]{scat}. In order to prove these weak convergences and conclude the proof, we will need the following three estimates. As $r \gtrsim \underline{u} \gtrsim \langle t \rangle$ for $\underline{u} \geq 2U$ and $|u| \le U$, we have
\begin{align}
\int_{\underline{u}=2U}^{+ \infty} \int_{|u| \leq U} \int_{\mathbb{S}^2_\omega}\left|F \right|^2(u,\underline{u},\omega ) \dr \mu_{\mathbb{S}^2_\omega} \dr u \dr \underline{u} &\lesssim \int_{\underline{u}=2 U}^{+ \infty} \int_{|u| \leq U} \int_{\mathbb{S}^2_\omega} \frac{1}{|\underline{u}|^2 }  |F|^2(u,\underline{u},\omega ) r^2 \dr \mu_{\mathbb{S}^2_\omega} \dr u \dr \underline{u} , \nonumber \\
& \lesssim \int_{t=2U}^{+\infty} \frac{1}{\langle t \rangle^2} \|F(t,\cdot)\|_{L^2_x}^2  \dr t \lesssim  \int_{t=0}^{+\infty} \frac{D \, \dr t}{\langle t \rangle^2} \lesssim D . \label{eq:forFweak0}
\end{align}
Using the uniform boundedness of $\rho (\mathcal{L}_{Z^\gamma}F)$ and $\sigma (\mathcal{L}_{Z^\gamma}F)$ on outgoing null cones $t-r=u$, we have
\begin{equation}\label{eq:55555555555} 
\forall \, U \geq 1, \qquad \int_{ \underline{u} = U}^{+ \infty} \int_{|u| \leq U} \int_{\mathbb{S}^2_\omega} \left[\big| r \rho (\mathcal{L}_{Z^\gamma}F) \big|^2+\big| r \sigma (\mathcal{L}_{Z^\gamma}F) \big|^2 \right](u,\underline{u},\omega) \dr \mu_{\mathbb{S}^2_\omega} \, \dr u \, \dr \underline{u} \lesssim UD.
\end{equation}
Finally, for all  $t \in \R_+, \, r >0$,
$$ \int_{\mathbb{S}^2_\omega} |J|^2(t,r\omega) \dr \mu_{\mathbb{S}^2_\omega} \leq 2 \int_{s=r}^{+\infty} \! \int_{\mathbb{S}^2_\omega} | J||\nabla_{\partial_r} J|(t,r\omega) \dr \mu_{\mathbb{S}^2_\omega} \dr s \lesssim \! \frac{1}{r^2 \langle t+r \rangle^3} \sup_{|\gamma| \leq 1} \int_{| x| \geq r} \langle t+|x| \rangle^3 |J^\gamma|^2(t,x) \dr x,$$
since $\partial_{x^i}J=J^{\kappa_i}$, if $\kappa_i$ is such that $Z^{\kappa_i}=\partial_{x^i}$. Then, using the assumption on the source terms $J^\gamma$, we get for all $\underline{u} \geq 0$,
\begin{equation}\label{Jweak0}
\int_{|u| \leq U} \int_{\mathbb{S}^2_\omega} r^4 |J|^2(u,\underline{u},\omega ) \dr \mu_{\mathbb{S}^2_\omega} \dr u \lesssim \frac{U D}{ \langle \underline{u} \rangle}, \qquad r^2|J|(\cdot , \underline{u} , \cdot) \xrightharpoonup[\underline{u} \to + \infty]{} 0 \quad \text{in $\mathcal{D}'(\R_u \times \mathbb{S}^2_\omega)$}.
\end{equation}
We now prove the weak convergences \eqref{eq:weakconv1toprove}--\eqref{eq:weakconv2toprove} as follows.
\begin{itemize}
\item We already proved $r \underline{\alpha}(F) \rightharpoonup \underline{\alpha}^{\mathcal{I}^+}$. It is in fact a strong convergence in $L^2([-U,U] \times \mathbb{S}^2_\omega)$.
\item Note that $r\underline{L}=\underline{L}r+1$ and $\underline{L}=2 \partial_u$. As $r \underline{\alpha}(F) \rightharpoonup \underline{\alpha}^{\mathcal{I}^+}$, we have $\underline{\alpha}(F) \rightharpoonup 0$ since $2r= \underline{u}-u \sim \underline{u}$ on compact subsets of $\R_u$, so that $r\underline{L} \big( \underline{\alpha}(F)_{e_B} \big) \rightharpoonup 2\partial_u (\underline{\alpha}^{\mathcal{I}^+}_{e_B})$.
\item Consider $(t,r) \in \R_+^2$, $\psi \in C^{\infty}_c(\R_u \times \mathbb{S}^2)$ and denote by $\vec{v}_i$ the vector field $\omega_i^{e_A}e_A$, which is the projection on the $2$-spheres of $\partial_{x^i}$. Since $(r e_\theta, r e_\varphi)= (\partial_\theta, \partial_{\varphi}/\sin(\theta))$, we have
\begin{align*}
\omega_i^A r^2\big( e_A (\underline{\alpha}(F)_{e_B}) \big)(t,r\omega) \psi(u,\omega)&=  r \psi(u,\omega)  \vec{v}_i \cdot \slashed{\nabla}\big( \underline{\alpha}(F)_{e_B} (t,r\omega)\big), \\
 \omega_i^A  e_A (\underline{\alpha}(F)_{e_B}^{\mathcal{I}^+} \big)(u,\omega) \psi(u,\omega)&= \psi(u,\omega)  \vec{v}_i \cdot \slashed{\nabla}\big( \underline{\alpha}(F)_{e_B}^{\mathcal{I}^+}\big) (u,\omega).
\end{align*}
As $r \underline{\alpha}(F) \rightharpoonup \underline{\alpha}(F)^{\mathcal{I}^+}$, $r^2\omega_j^A e_A\big( \underline{\alpha}(F)_{e_B} \big) \rightharpoonup \omega_j^A e_A \big( \underline{\alpha}^{\mathcal{I}^+}_{e_B} \big)$ ensues from the divergence theorem on $\mathbb{S}^2$.
\item 
From \eqref{eq:forFweak0}--\eqref{eq:55555555555}, we get that there exists $\mathbf{s}'=(\underline{u}'_n)_{n \geq 0}$ such that $\underline{u}'_n \to + \infty$ as $n \to + \infty$ and 
$$\big\| F(\cdot, \underline{u}'_n, \cdot) \big\|^2_{L^2([-U,U] \times \mathbb{S}^2_\omega)}+\big\| r \rho (\mathcal{L}_{Z^\gamma}F) (\cdot, \underline{u}'_n, \cdot) \big\|^2_{L^2([-U,U] \times \mathbb{S}^2_\omega)}+\big\| r \sigma (\mathcal{L}_{Z^\gamma}F) (\cdot, \underline{u}'_n, \cdot) \big\|^2_{L^2([-U,U] \times \mathbb{S}^2_\omega)} \lesssim \frac{D}{ \langle \underline{u}'_n \rangle},$$
for any $|\gamma| \leq 1$. We then deduce $F \rightharpoonup_{\mathbf{s}'} 0$, $r \rho (\mathcal{L}_{Z^\gamma}F)\rightharpoonup_{\mathbf{s}'} 0 $ and $r \rho (\mathcal{L}_{Z^\gamma}F) \rightharpoonup_{\mathbf{s}'} 0$.
\item Remark first that $r^2 L (\underline{\alpha}(F)_{e_B})=rL(r\underline{\alpha}(F)_{e_B})-r\underline{\alpha}(F)_{e_B}$ and $r\underline{\alpha}(F)_{e_B} \rightharpoonup  \underline{\alpha}^{\mathcal{I}^+}$. Then, using \eqref{eq:Lunderlinealpha} and $\nabla_Le_B=0$, we get 
$$ \left| rL\big(r\underline{\alpha}(F)_{e_B} \big) \right| \lesssim r^2|J|+\sum_{|\gamma| \leq 1} r|\rho ( \mathcal{L}_{Z^\gamma} F)|+r|\sigma ( \mathcal{L}_{Z^\gamma} F)| \rightharpoonup_{\mathbf{s}'} 0,$$
where, in the last step, we used \eqref{Jweak0} and the weak convergences that we just proved.
\item Finally, for any $\mathbb{Z} \in \mathbb{K}$, $n \mapsto \big( r\underline{\alpha}( \mathcal{L}_Z F) \big)(\cdot, \underline{u}_n' , \cdot)$ is bounded in $L^2(\R_u \times \mathbb{S}^2_\omega)$. Consequently, there exists a subsequence $\mathbf{s}$ of $\mathbf{s}'=(\underline{u}'_n)_{n \geq 0}$ and $\underline{\alpha}^{\mathcal{I}^+}_{\, Z}$ such that $r \underline{\alpha} ( \mathcal{L}_Z F)\rightharpoonup_{\mathbf{s}} \underline{\alpha}^{\mathcal{I}^+}_{ \, Z}$.
\end{itemize}
\end{proof}

Motivated by the previous result and since our approach is based on vector field methods, we introduce particular weighted derivatives for radiation fields. 
\begin{Def}\label{Defrecurrad}
Let $\underline{\alpha}^{\mathcal{I}^+}$ be a $1$-form, defined on $\R_u \times \mathbb{S}^2$ and tangential to the $2$-spheres. For any multi-index $\gamma$, we define recursively, at least in a weak sense, the $2$-forms $\underline{\alpha}^{\mathcal{I}^+}_{\, \gamma}$ as follows.
\begin{itemize}
\item $\underline{\alpha}^{\mathcal{I}^+}_{\, \gamma}:=\underline{\alpha}^{\mathcal{I}^+}$ if $|\gamma|=0$.
\item If $Z^\gamma = \partial_t Z^\kappa$, then $\underline{\alpha}^{\mathcal{I}^+}_{\, \gamma}:=\nabla_u \underline{\alpha}^{\mathcal{I}^+}_{\, \kappa}$.
\item If $Z^\gamma = \partial_{x^i} Z^\kappa$ with $1 \leq i \leq 3$, then $\underline{\alpha}^{\mathcal{I}^+}_{\, \gamma}:=-\omega_i \nabla_u \underline{\alpha}^{\mathcal{I}^+}_{\, \kappa}$.
\item If $Z^\gamma = S Z^\kappa$, then $\underline{\alpha}^{\mathcal{I}^+}_{\, \gamma}:=u\nabla_u \underline{\alpha}^{\mathcal{I}^+}_{\, \kappa}+\underline{\alpha}^{\mathcal{I}^+}_{\, \kappa}$.
\item If $Z^\gamma = \Omega_{0i} Z^\kappa$ with $1 \leq i \leq 3$, then $\underline{\alpha}^{\mathcal{I}^+}_{ \,\gamma}:=- \omega_i u\nabla_u  \underline{\alpha}^{\mathcal{I}^+}_{\, \kappa}-2\omega_i \underline{\alpha}^{\mathcal{I}^+}_{\, \kappa}+ \omega^{e_A}_i \slashed{\nabla}_{e_A}  \underline{\alpha}^{\mathcal{I}^+}_{\, \kappa}$.
\item If $Z^\gamma = \Omega_{jk} Z^\kappa$ with $1 \leq j < k \leq 3$, then $\underline{\alpha}^{\mathcal{I}^+}_{\, \gamma}:=\mathcal{L}_{\Omega_{jk}} \big( \underline{\alpha}^{\mathcal{I}^+}_{\, \kappa} \big)$.
\end{itemize}
\end{Def}

We now establish an improved scattering result which requires stronger assumptions on the data and the source term. It shows that the conditions we will impose on the radiation fields for the backward problem are consistent. The hypotheses of the next result will be satisfied by $F-F^{\mathrm{asymp}}[f_\infty]$.

\begin{Pro}\label{Thscatforsmooth}
Let $F$ be a solution to the Maxwell equations \eqref{eq:Maxeqn}, with source term $J$. Let $J^\gamma$ be the current such that
$$ \forall \, |\gamma| \leq 3, \qquad  \nabla^\mu \mathcal{L}_{Z^\gamma}(F)_{\mu \nu} = J_\nu^\gamma, \qquad \nabla^\mu {}^* \! \mathcal{L}_{Z^\gamma}(F)_{\mu \nu} =0$$
and assume that there exist $\delta >0$ and $D \geq 0$ such that
$$\mathcal{E}^K_3[F](T)+\sup_{|\gamma| \leq 3} \, \sup_{(t,x) \in [T,+\infty[ \times \R^3} \,  \langle t+|x| \rangle^{7+2\delta} |J^\gamma|^2(t,x)  \leq D.$$
Then, $F$ has a radiation field along future null infinity $\underline{\alpha}^{\mathcal{I}^+}$ and
\begin{equation}\label{rhoinfysigmainfty0}
 \rho^{\mathcal{I}^+}  \! : (u,\omega) \mapsto -\frac{1}{2} \slashed{\nabla} \cdot\int_{s=-\infty}^u  \underline{\alpha}^{\mathcal{I}^+} (s,\omega) \dr s, \qquad \sigma^{\mathcal{I}^+}\! :(u,\omega) \mapsto -\frac{1}{2} \slashed{\nabla} \times \int_{s=-\infty}^u   \underline{\alpha}^{\mathcal{I}^+} (s,\omega) \dr s  
\end{equation}
are well-defined and go to $0$ as $|u| \to +\infty$. Moreover, for all $(t,x) \in [T,+\infty[ \times \R^3$, we have
\begin{spreadlines}{1.8ex}
\begin{align}
  |\alpha(F)|(t,x) & \lesssim \frac{\delta^{-1}\sqrt{D}}{ \langle t+|x|\rangle^{\frac{5}{2}}} ,\label{esti1best}\\ 
\left|r^2 \rho(F)(t,x) - \rho^{\mathcal{I}^+} \! \left(t-|x|,\frac{x}{|x|}\right)  \right| +\left|r^2 \sigma(F)(t,x) - \sigma^{\mathcal{I}^+} \! \left(t-|x|,\frac{x}{|x|}\right)  \right| & \lesssim \frac{\delta^{-1}\sqrt{D}}{ \langle t+|x|\rangle^{\frac{1}{2}}}, \label{esti1goodbis} \\ 
\left|r\underline{\alpha}(F)(t,x) - \underline{\alpha}^{\mathcal{I}^+} \! \left( t-|x|,\frac{x}{|x|} \right) \right| & \lesssim \frac{\delta^{-1}\sqrt{D}}{ \langle t+|x|\rangle \, \langle t-|x|\rangle^{\frac{1}{2}}}, \label{esti2badnull}
\end{align}
\end{spreadlines}
\end{Pro}
\begin{proof}
The energy inequality of Proposition \ref{ProMoravac}, applied to $\mathcal{L}_{Z^\gamma}F$ for any $|\gamma | \leq 3$, and the assumptions on the data as well as $J$ give 
$$ \sup_{t \geq T} \mathcal{E}^K_3 \big[ F \big](t) \leq 2 \mathcal{E}^K_3 \big[ F \big](T)+ 8 \sup_{|\gamma|  \leq 3} \bigg|\int_{\tau=T}^{+\infty} \frac{1}{\langle \tau \rangle^{1+\delta}} \bigg|\int_{\R^3_x} \langle \tau+|x| \rangle^{4+2\delta} \big|J^{\gamma} (\tau,x) \big|^2 \dr x \bigg|^{\frac{1}{2}} \dr \tau \bigg|^2 \leq 20 \delta^{-2} D.$$
Now, applying Proposition \ref{decayMaxell} as well as \eqref{decay:alphageneral}, we get, for any $|\gamma| \leq 1$ and for all $(t,x) \in [T,+\infty[ \times \R^3$,
\begin{alignat}{2}\label{eq:decayalphas}
\hspace{-3mm} &|\underline{\alpha}(\mathcal{L}_{Z^{\gamma}}F)|(t,x) \lesssim \delta^{-1}D^{\frac{1}{2}} \, \langle t+|x|\rangle^{-1}\langle t-|x|\rangle^{-\frac{3}{2}}, \qquad  &&|\alpha(\mathcal{L}_{Z^{\gamma}}F)|(t,x)  \lesssim \delta^{-1} D^{\frac{1}{2}} \, \langle t+|x|\rangle^{-\frac{5}{2}}, \\
\hspace{-3mm} & |\rho(\mathcal{L}_{Z^{\gamma}}F)|(t,x) \lesssim \delta^{-1} D^{\frac{1}{2}} \, \langle t+|x|\rangle^{-2}\langle t-|x|\rangle^{-\frac{1}{2}}, \qquad  &&|\sigma(\mathcal{L}_{Z^{\gamma}}F)|(t,x)  \lesssim \delta^{-1}D^{\frac{1}{2}} \, \langle t+|x|\rangle^{-2}\langle t-|x|\rangle^{-\frac{1}{2}}. \label{eq:decayrhosig}
\end{alignat}
Applying Corollary \ref{Corgoodnull}, we then obtain, for all $(t,x) \in [T,+\infty[ \times \R^3$,
\begin{align*}
 \left|\nabla_L \left( r^2\rho(F) \right) \right|(t,x)+ \left|\nabla_L \left( r^2 \sigma(F) \right) \right|(t,x)  \lesssim \frac{\delta^{-1}D^{\frac{1}{2}}}{ \langle t+|x|\rangle^{\frac{3}{2}}}, \qquad \qquad \left|\nabla_L \big( r \underline{\alpha}(F) \big) \right|(t,x) \lesssim \frac{ \delta^{-1}D^{\frac{1}{2}}}{ \langle t+|x|\rangle^2 \, \langle t-|x|\rangle^{\frac{1}{2}}}.
\end{align*}
We use now the null coordinates $u=t-|x|$, $\underline{u}=t+|x|$, $\omega=x/|x|$ and we recall that $L=2\partial_{\underline{u}}$. Fix $(u,\omega) \in \R_u \times \mathbb{S}^2$ and integrate the previous estimates between $|u| \leq \underline{u} \leq \underline{z}$. We get
\begin{align}
 \big|r^2\rho(F)(u,\underline{u},\omega)-r^2\rho(F)(u,\underline{z},\omega) \big|+\big|r^2\sigma(F)(u,\underline{u},\omega)-r^2\sigma(F)(u,\underline{z},\omega) \big| & \lesssim  \delta^{-1}D^{\frac{1}{2}} \,  \langle \underline{u} \rangle^{-\frac{1}{2}}, \label{eq:rhoetsigma} \\
 \big|r\underline{\alpha}(F) (u,\underline{u},\omega)-r \underline{\alpha}(F)(u,\underline{z},\omega) \big| & \lesssim \delta^{-1}D^{\frac{1}{2}} \, \langle \underline{u} \rangle^{-1}\langle u\rangle^{-\frac{1}{2}}. \label{eq:underlinealpha}
\end{align}
This implies the existence of $\underline{\alpha}^{\mathcal{I}^+}$, $\rho^{\mathcal{I}^+}$ and $\sigma^{\mathcal{I}^+}$ as well the stated rate of convergences. Moreover, for fixed $(\underline{u},\omega) \in [T,+\infty[ \times \mathbb{S}^2$, we obtain by integrating the null Maxwell equations \eqref{null3}--\eqref{null4} between $u \geq -\underline{u}+2T$ and $\underline{u}$. As $r=0$ when $t+r=t-r$, we get
$$ \bigg| r^2 \rho(F)(u,\underline{u},\omega)-\int_{s=u}^{\underline{u}} r^2\slashed{\nabla} \cdot \underline{\alpha}(F) (s,\underline{u},\omega ) \dr s \bigg| \lesssim \frac{D^{\frac{1}{2}}}{\langle \underline{u} \rangle^{\frac{3}{2}+\delta}}, \quad \; r^2 \sigma(F)(u,\underline{u},\omega)+\int_{s=u}^{\underline{u}} r^2 \slashed{\nabla} \times  \underline{\alpha}(F) (s,\underline{u},\omega ) \dr s =0.$$
Using $re_\theta = \partial_{\theta }$, $re_\varphi = \sin^{-1}(\theta) \partial_{\varphi}$ together with \eqref{eq:rhoetsigma}--\eqref{eq:underlinealpha}, we then deduce that the relations \eqref{rhoinfysigmainfty0} between $\rho^{\mathcal{I}^+}$, $\sigma^{\mathcal{I}^+}$ and $\underline{\alpha}^{\mathcal{I}^+}$ hold in $\mathcal{D}'(\R_u \times \mathbb{S}^2)$ and then, as $\rho^{\mathcal{I}^+}$ and $\sigma^{\mathcal{I}^+}$ are continous, in $C^0(\R_u \times \mathbb{S}^2)$. Furthermore, \eqref{eq:decayrhosig}--\eqref{eq:rhoetsigma} imply $|\rho^{\mathcal{I}^+}|(u,\omega)+|\sigma^{\mathcal{I}^+}|(u,\omega) \lesssim \langle u \rangle^{-\frac{1}{2}}$, so that these two functions go to $0$ as $|u| \to + \infty$.
\end{proof}

\subsection{Backward evolution from scattering data at future null infinity for the Maxwell equations}\label{SecscatMaxlin}

This section is devoted to an important step of the proof of Theorem \ref{Theo1}. More precisely, we will study the inverse problems of the one addresed by Proposition \ref{Thscatforsmooth}. For this, we start by introducing the functional space of the radiation fields that we will consider here.

\begin{Def}\label{Defscatstae}
Let $\mathcal{S}^K_{\mathcal{I}^+}\!$ be the set of the $1$-forms $\underline{\alpha}^{\mathcal{I}^+}\!$ on $\R_u \! \times \mathbb{S}^2_\omega$ which are tangential to the $2$-spheres and such that the functions 
\begin{equation}\label{rhoinfysigmainfty}
 \rho^{\mathcal{I}^+}  \! : (u,\omega) \mapsto -\frac{1}{2} \slashed{\nabla} \cdot \int_{s=-\infty}^u  \underline{\alpha}^{\mathcal{I}^+} (s,\omega) \dr s, \qquad \sigma^{\mathcal{I}^+}\! :(u,\omega) \mapsto -\frac{1}{2}\slashed{\nabla} \times \int_{s=-\infty}^u  \underline{\alpha}^{\mathcal{I}^+} (s,\omega) \dr s  
\end{equation}
are well defined and, for almost every $\omega \in \mathbb{S}^2$,
\begin{equation}\label{defmoynulle}
\slashed{\nabla} \cdot \int_{\R_u}  \underline{\alpha}^{\mathcal{I}^+} (u,\omega) \dr u = \slashed{\nabla} \times \int_{\R_u}  \underline{\alpha}^{\mathcal{I}^+} (u,\omega) \dr u =0. 
\end{equation}
If $\underline{\alpha}^{\mathcal{I}^+} \in \mathcal{S}_{\mathcal{I}^+}^K$ is sufficiently regular, let for any $N \in \mathbb{N}$ and all $(u,\omega) \in \R_u \times \mathbb{S}^2$ the pointwise norm
$$\big\| \underline{\alpha}^{\mathcal{I}^+} \! \big\|_{N}(u,\omega) := \sup_{\kappa_u+|\kappa_a| \leq N} \langle u \rangle^{\kappa_u+1}\left| \nabla_{\partial_u}^{\kappa_u} \! \slashed{\nabla}^{\kappa_a} \underline{\alpha}^{\mathcal{I}^+}\right|\!(u,\omega)+\langle u \rangle^{\kappa_u}\left| \nabla_{\partial_u}^{\kappa_u} \slashed{\nabla}^{\kappa_a} \rho^{\mathcal{I}^+}\right|\!(u,\omega)+\langle u \rangle^{\kappa_u} \! \left| \nabla_{\partial_u}^{\kappa_u} \slashed{\nabla}^{\kappa_a} \sigma^{\mathcal{I}^+}\right|\!(u,\omega)$$
and, for $a \geq 0$, the higher order energy norm
$$\overline{\mathcal{E}}^{ \, a}_{N}\big[\underline{\alpha}^{\mathcal{I}^+} \big]:= \int_{\R_u} \int_{\mathbb{S}^2_\omega}\langle u \rangle^{2a} \big\| \underline{\alpha}^{\mathcal{I}^+} \big\|^2_{N}(u,\omega) \dr \mu_{\mathbb{S}^2} \dr u .$$
\end{Def}
In order to prove a scattering result with assumptions allowing for the gain of regularity on the source term  discussed in Section \ref{subseubsecMaxide}, we need to associate to any locally integrable current $J$ a sequence of electromagnetic fields. This sequence will be used in order to construct the solution to the Maxwell equations with a prescribed source term and asymptotic data in Proposition \ref{ProscatMax}. 
\begin{Def}\label{DefSJn}
Let $J$ be a four-current density defined on $[T,+\infty[ \times \R^3$. Let, for any integer $n \geq T$, $S(J,n)$ be the unique solution to
$$ \nabla^{\mu } S(J,n)_{\mu \nu} = \chi\big(t /n \big)J_\nu, \qquad \nabla^\mu {}^* \! S(J,n)_{\mu \nu} =0, \qquad \qquad S(J,n)(n,\cdot)=0,$$
where $\chi \in C^\infty(\R,[0,1])$ is a cutoff function such that $\chi (s)=1$ for all $s \leq 1/4$ and $\chi(s)=0$ for all $s \geq 1/2$.
\end{Def}
\begin{Rq}
We introduced $\chi$ to ensure that the derivatives of $S(J,n)$ vanish as well for $t=n$. 
\end{Rq}
We now state the main result of this section. A higher order version holds but we do not prove it in order to avoid technicalities and since it is not required for the proof of Theorem \ref{Theo1}.
\begin{Pro}\label{ProscatMax}
Let $a \geq 0$, $\delta >0$, $\underline{\alpha}^{\mathcal{I}^+} \in \mathcal{S}_{\mathcal{I}^+}^K$ and $J$ be a four-current density defined. Assume
$$ \overline{\mathcal{E}}^{\, a+\delta}_{1}\big[ \underline{\alpha}^{\mathcal{I}^+} \big] < + \infty$$
and that there exists a constant $B_{\mathrm{source}}>0$ such that, for any $n \geq T$, the field $S(J,n)$ generated by the source term $J$ verifies
\begin{equation}\label{eq:condiSJn} 
\forall \, t \in [T,+\infty[, \qquad \mathcal{E}^{K,a}\big[ S(J,n) \big] (t) \lesssim \langle t \rangle^{-\delta} B_{\mathrm{source}}.
\end{equation}
Then, there exists a unique solution $F$ to the Maxwell equations 
$$ \nabla^\mu F_{\mu \nu} = J, \qquad  \qquad \nabla^{\mu} {}^* \! F_{\mu \nu} =0   ,$$
admitting $\underline{\alpha}^{\mathcal{I}^+}$ as radiation field along future null infinity $\mathcal{I}^+$. Moreover, 
$$ \exists \, C_\delta >0, \quad  \forall \, t \geq T, \qquad  \mathcal{E}^{K,a}[ F](t) \leq C_\delta \, \overline{\mathcal{E}}^{\, a+\delta}_{1}\big[ \underline{\alpha}^{\mathcal{I}^+} \big]+B_{\mathrm{source}}.$$
\end{Pro}
\begin{Rq}\label{RqappliTh}
In the context of Theorem \ref{Theo1}, if $N$ is the order of regularity of the solution, the construction of $\mathcal{L}_{Z^\gamma}(F-F^{\mathrm{asymp}}[f_\infty])$, with $|\gamma| \leq N-1$, will rely on this result, applied for $a=0$. We will suitably estimate the top order derivatives $\nabla_{t,x} \mathcal{L}_{Z^\gamma}(F-F^{\mathrm{asymp}}[f_\infty])$, $|\gamma| =N-1$, by appealing to Proposition \ref{ProscatMax} for $a =1$.
\end{Rq}
 \subsubsection{Preparatories} Before starting the proof of Proposition \ref{ProscatMax}, we give a sufficient condition on $J$ for \eqref{eq:condiSJn} to hold. In the course of Theorem \ref{Theo1}, we will use it for all but the top order derivatives of the source term $J(f)-J^{\mathrm{asymp}}[f_\infty]$. In view of Remark \ref{RqappliTh}, we focus here on the case $a=0$.
\begin{Pro}\label{CorassumpJforcondi}
Let $\delta >0$, $B>0$ and $J$ be a four-current defined on $[T,+\infty[ \times \R^3$ satisfying
$$ \forall \, t \in [T,+\infty[, \qquad  \int_{\R^3_x} \langle t+|x| \rangle^{4+\delta} \big|  J \big|^2(t,x) \dr x \leq B .$$
Then, there exists an absolute constant $D>0$ such that $S(J,n)$ verifies \eqref{eq:condiSJn} for all $n \geq T$ and for $a=0$ as well as $B_{\mathrm{source}}=D \delta^{-1} B$.
\end{Pro}
\begin{proof}
Let $n \geq T$ and recall that $S(J,n)=0$ on $[n,+\infty[ \times \R^3$. Furthermore, the energy inequality of Proposition \ref{ProenergyforscatMax} provides
$$ \forall \, t \in [T,n], \qquad \mathcal{E}^{K} \big[ S(J,n) \big] (t) \leq 8  \int_{\tau=t}^n \int_{\R^3_x} \big| K^\mu  S(J,n)_{\mu \nu} J^\nu \big|(\tau,x) \dr x \dr \tau.$$
Expanding $K^\mu  S(J,n)_{\mu \nu} $ according to the null frame $(\underline{L},L,e_\theta,e_\varphi)$, we get, as $2K=\langle t+r \rangle^2 L+ \langle t-r \rangle^2 \underline{L}$,
\begin{align*}
& \big| K^\mu  S(J,n)_{\mu \nu}  \big| \leq  \, \langle t+r \rangle^2 \big| S(J,n)_{L \nu} \big|+\langle t-r \rangle^2 \big| S(J,n)_{\underline{L} \nu}  \big| \\
& \qquad \qquad \quad \lesssim  \langle t+r \rangle^2 \big( \big| \alpha \big(  S(J,n) \big)\big|+\big| \rho \big(  S(J,n) \big)\big|\big)  + \langle t-r \rangle^2 \big( \big| \underline{\alpha} \big(  S(J,n) \big)\big|+\big| \rho \big(  S(J,n) \big)\big|\big)  .
 \end{align*}
 The Cauchy-Schwarz inequality in $x$ then provides, for all $T \leq t \leq n$,
 \begin{align*}
  \mathcal{E}^{K} \big[ S(J,n) \big] (t) & \lesssim \int_{\tau=t}^n \big| \mathcal{E}^K\big[S(J,n) \big](\tau) \big|^{\frac{1}{2}} \bigg| \int_{\R^3_x} \langle t+|x| \rangle^2 \big|J \big|^2(\tau,x) \dr x \bigg|^{\frac{1}{2}}\dr \tau \\
  & \lesssim B^{\frac{1}{2}}  \sup_{t \leq \tau \leq n } \big| \mathcal{E}^K\big[S(J,n) \big](\tau) \big|^{\frac{1}{2}} \int_{\tau=t}^n \frac{\dr \tau}{\langle \tau\rangle^{1+\frac{\delta}{2}}} \lesssim \delta^{-1} B^{\frac{1}{2}}  \langle t \rangle^{-\frac{\delta}{2}}  \sup_{t \leq \tau \leq n } \big| \mathcal{E}^K\big[S(J,n) \big](\tau) \big|^{\frac{1}{2}},
  \end{align*}
  which implies the result.
\end{proof}

\subsubsection{Set up of the proof of Proposition \ref{ProscatMax}}

We fix $a \geq 0$, $\delta >0$, $\underline{\alpha}^{\mathcal{I}^+}$ as well as $J$ satisfying the assumptions of Proposition \ref{ProscatMax}. We consider further $\rho^{\mathcal{I}^+}$ and $\sigma^{\mathcal{I}^+}$, given by \eqref{defmoynulle}. We will construct $F$ as the limit of a sequence $(F^n)$ of approximate solutions to the Maxwell equations with source term $J$. All these $2$-forms will have the same leading order term $F^{\mathrm{lead}}$, defined as
\begin{alignat*}{2}
\alpha \big(F^{\mathrm{lead}} \big)(t,r\omega)&=0, \qquad \qquad &&\underline{\alpha}\big(F^{\mathrm{lead}} \big)(t,r\omega)=\frac{1}{r} \underline{\alpha}^{\mathcal{I}^+}(t-r,\omega) \chi \big( \langle t-r \rangle /r \big),  \\
 \rho \big(F^{\mathrm{lead}} \big)(t,r\omega)&=\frac{1}{r^2} \rho^{\mathcal{I}^+}(t-r,\omega) \chi \big( \langle t-r \rangle /r), \qquad \qquad  && \sigma \big(F^{\mathrm{lead}} \big)(t,r\omega)=\frac{1}{r^2} \sigma^{\mathcal{I}^+}(t-r,\omega) \chi \big( \langle t-r \rangle /r \big),
\end{alignat*}
where we recall that $\chi \in C^\infty (\R,[0,1])$ verifies $\chi(s)=1$ for $s \leq 1/4$ and $\chi(s)=0$ for $s \geq 1/2$. 

\subsubsection{Study of the approximate solution}

We start by estimating how close $F^{\mathrm{lead}}$ is to be a solution to the vacuum Maxwell equations. 
\begin{Lem}\label{LemJJstar}
$F^{\mathrm{lead}}$ verifies the Maxwell equations with a magnetic current \eqref{Maxmagn}, with source terms $J^{\mathrm{err}}$ and $\widetilde{J}^{\mathrm{err}}$ satisfying, for all $(t,r\omega) \in \R_+ \times \R^3$,
\begin{align*}
\big|J^{\mathrm{err}}_{\underline{L}} \big|(t,r\omega)+\big|\widetilde{J}^{\mathrm{err}}_{\underline{L}} \big|(t,r\omega) & \lesssim \frac{1}{\langle t+r \rangle^3} \big\| \underline{\alpha}^{\mathcal{I}^+} \big\|_{1}(t-r,\omega) \, \mathds{1}_{\langle t-r \rangle \leq \frac{r}{2}}, \\
 \sup_{A \in \{ \theta , \varphi \}}\big|J^{\mathrm{err}}_{e_A} \big|(t,r\omega)+\big|\widetilde{J}^{\mathrm{err}}_{e_A} \big|(t,r\omega)  & \lesssim \frac{1}{\langle t+r \rangle^3} \big\| \underline{\alpha}^{\mathcal{I}^+} \big\|_{1}(t-r,\omega) \, \mathds{1}_{\langle t-r \rangle \leq \frac{r}{2}}, \\
 \big|J^{\mathrm{err}}_L \big|(t,r\omega)+\big|\widetilde{J}^{\mathrm{err}}_L \big|(t,r\omega) & \lesssim \frac{\langle t-r \rangle}{\langle t+r \rangle^4} \big\| \underline{\alpha}^{\mathcal{I}^+}  \big\|_{1}(t-r,\omega) \, \mathds{1}_{\langle t-r \rangle \leq \frac{r}{2}} .
 \end{align*}
Furthermore, the null components of $F^{\mathrm{lead}}$ verify $\alpha (F^{\mathrm{lead}})=0$ and
\begin{align*}
 \big| \underline{\alpha} \big(F^{\mathrm{lead}}  \big)\big|(t,r\omega) & \lesssim \langle t+r \rangle^{-1} \langle t-r \rangle^{-1}\big\| \underline{\alpha}^{\mathcal{I}^+} \big\|_{0}(t-r,\omega) \, \mathds{1}_{\langle t-r \rangle \leq \frac{r}{2}} , \\
    \big| \rho \big( F^{\mathrm{lead}}  \big)\big|(t,r\omega)+\big| \sigma \big(  F^{\mathrm{lead}}  \big)\big|(t,r\omega) & \lesssim  \langle t+r \rangle^{-2} \big\| \underline{\alpha}^{\mathcal{I}^+} \big\|_{0}(t-r,\omega) \, \mathds{1}_{\langle t-r \rangle \leq \frac{r}{2}} .
    \end{align*}
\end{Lem}
\begin{proof}
We first determine $J^{\mathrm{err}}$ and $\widetilde{J}^{\mathrm{err}}$ using the definition of $F^{\mathrm{lead}}$ and the equations \eqref{Maxmagn} expressed in the null frame. First, since $e_\theta =r^{-1} \partial_\theta$ and $e_\varphi=r^{-1} \sin^{-1}(\theta) \partial_{\varphi}$, we get from \eqref{null1}--\eqref{null2} that, for any $A \in \{\theta , \varphi \}$,
\begin{align*}
 J_{e_A}^{\mathrm{err}}&=-\frac{1}{r^3}\varepsilon^{AB}\slashed{\nabla}_{e_B}\sigma^{\mathcal{I}^+} \! (t-r,\omega) \chi \big(\langle t-r\rangle/r \big)-\frac{\langle t-r \rangle}{2r^3} \underline{\alpha}^{\mathcal{I}^+}_{e_A}(t-r,\omega) \chi ' \big( \langle t-r \rangle /r \big), \\
 \widetilde{J}_{e_A}^{\mathrm{err}} &= \frac{1}{r^3}\varepsilon^{AB}\slashed{\nabla}_{e_B}\rho^{\mathcal{I}^+} \! (t-r,\omega) \chi \big( \langle t-r\rangle/r \big)-\frac{\langle t-r \rangle}{2r^3} \varepsilon^{AB}\underline{\alpha}^{\mathcal{I}^+}_{e_B}(t-r,\omega) \chi ' \big( \langle t-r \rangle /r \big).
 \end{align*}
Next, as $\nabla_{\underline{L}}\rho^{\mathcal{I}^+}+\slashed{\nabla} \cdot  \underline{\alpha}^{\mathcal{I}^+}=\nabla_{\underline{L}}\sigma^{\mathcal{I}^+}+\slashed{\nabla}  \times  \underline{\alpha}^{\mathcal{I}^+}=0$ holds, we have from \eqref{null3}--\eqref{null4}
\begin{align*}
 J_{\underline{L}}^{\mathrm{err}}=\frac{(t-r)r+\langle t-r \rangle^2}{\langle t-r \rangle r^4}\rho^{\mathcal{I}^+} \! (t-r,\omega) \chi'\big(\langle t-r\rangle/r \big), \qquad 
 \widetilde{J}_{\underline{L}}^{\mathrm{err}} = \frac{(t-r)r+\langle t-r \rangle^2}{\langle t-r \rangle r^4}\sigma^{\mathcal{I}^+} \! (t-r,\omega) \chi' \big(\langle t-r\rangle/r \big).
 \end{align*}
Finally, we obtain from \eqref{null5}--\eqref{null6} that
$$ J_{L}^{\mathrm{err}}=\frac{\langle t-r \rangle}{r^4}\rho^{\mathcal{I}^+} \! (t-r,\omega) \chi' \big(\langle t-r\rangle/r \big), \qquad  \widetilde{J}_{L}^{\mathrm{err}} = \frac{\langle t-r \rangle}{r^4} \sigma^{\mathcal{I}^+} \! (t-r,\omega) \chi' \big(\langle t-r\rangle/r \big).$$
Recall further the definition of $F^{\mathrm{lead}}$. We then deduce, as $\langle t-r \rangle \leq r/2$ and then $\langle t+r \rangle \lesssim r$ on the support of $\chi$, that all the stated estimates hold.
\end{proof}
We now bound the energy norm of $F^{\mathrm{lead}}$.
\begin{Cor}\label{CorFlead}
For all $t \in \R_+$, we have $\mathcal{E}^{K,a} \big[ F^{\mathrm{lead}} \big] (t) \lesssim \overline{\mathcal{E}}^{\,a}_{0} \big[ \underline{\alpha}^{\mathcal{I}^+} \big]$.
\end{Cor}
\begin{proof}
In view of the support of the cutoff function $\chi$, $F^{\mathrm{lead}}(t,r\omega)$ vanishes for $0 \leq t \leq 1$ or $r \leq 2$. According to Lemma \ref{LemJJstar}, we have
$$ \forall \, t \in \R_+, \qquad \mathcal{E}^{K,a} \big[ F^{\mathrm{lead}} \big] (t) \lesssim  \int_{r=0}^{+\infty} \int_{\mathbb{S}^2_\omega} \langle t-r \rangle^{2a} \,  \big\| \underline{\alpha}^{\mathcal{I}^+}  \big\|^2_{0}(t-r,\omega) \dr \mu_{\mathbb{S}^2_\omega} \frac{r^2 \, \dr r}{\langle t+r \rangle^2} \leq \overline{\mathcal{E}}^{\,a}_{0} \big[ \underline{\alpha}^{\mathcal{I}^+} \big].$$
\end{proof}

\subsubsection{Sequence of approximate solutions}
Recall from Definition \ref{DefSJn} the field $S(J,n)$ and consider, for any time $n \in \mathbb{N}^*$ such that $n \geq T$, the $2$-forms $F^n$ and $R^n$ defined as
$$F^n=F^{\mathrm{lead}}+R^n+S(J,n), \qquad \qquad \nabla^\mu R^n_{\mu \nu}=-\chi(t/n) J_\nu^{\mathrm{err}}, \quad \nabla^\mu {}^* \! R^n = -\chi (t/n) \widetilde{J}_\nu^{\mathrm{err}}, \quad R^n(n,\cdot)=0.$$
Since $\chi(s)=0$ for all $s \geq 1/2$, the derivatives of $R^n$ vanish as well at $t=n$. Furthermore, as $\chi(s)=1$ for all $s \leq 1/4$, $F^n$ is a solution to the Maxwell equations with source term $J$ on $[0,n/4]$. Our goal now is to prove that $F^n$ converges to $F$, as $n \to +\infty$, with $F$ satisfying the conclusion of Proposition \ref{ProscatMax}. We start by a result which will allow us to estimate the energy of $F^n$ and to prove that it is a Cauchy sequence.
\begin{Lem}\label{Lemforenergy}
Let $n \in \mathbb{N}^*$, $G$ be a sufficiently regular $2$-form and $D$ be either $J^{\mathrm{err}}$ or $\widetilde{J}^{\mathrm{err}}$. Then, for all $0 \leq t \leq n$,
$$ \int_{\tau=t}^{n}\! \int_{\R^3_x} \! | \max (\tau-|x|,1) |^{2a} \! \left| K^{\mu} G_{\mu \nu} D^\nu\right| \! (\tau,x) \dr x \dr \tau \! \lesssim  \! \frac{\delta^{-1}}{\langle t \rangle^{\frac{\delta}{2}}} \left| \mathcal{E}^{a+\delta}_{1}\big[\underline{\alpha}^{\mathcal{I}^+}  \big] \right|^{\frac{1}{2}} \! \left| \sup_{t \leq \tau \leq n} \! \mathcal{E}^{K,a}[G](\tau)+ \sup_{u \in \R} \mathcal{F}^{K,a}[G]_t^n(u) \right|^{\frac{1}{2}} \!.$$ 
\end{Lem}
\begin{proof}
As $2K=\langle t+r\rangle^{2} L+\langle t-r \rangle^{2} \underline{L}$, there holds
\begin{align}
 \hspace{-11mm} \left| K^\mu G_{\mu \nu} D^\nu \right| & = \left| 2\langle t-r \rangle^{2} \rho (G)  D^L-2\langle t+r \rangle^{2} \rho (G) D^{\underline{L}}- \langle t-r \rangle^{2} \underline{\alpha}(G)_{e_A} D^{e_A}-\langle t+r \rangle^{2} \alpha (G) _{e_A} D^{e_A} \right| \nonumber \\
& \leq \langle t-r \rangle^{2} |\rho (G)| |D_{\underline{L}}|+\langle t+r \rangle^{2} |\rho (G)| |D_{L}|+\langle t-r \rangle^{2} |\underline{\alpha} (G)| |\slashed{D}|+\langle t+r \rangle^{2} |\alpha (G)|| \slashed{D}|, \label{eq:expanullframeeee}
\end{align}
where $|\slashed{D}|^2:=|D_{e_\theta}|^2+|D_{e_\varphi}|^2$. According to Lemma \ref{LemJJstar}, we have
\begin{align*}
 \left| K^{\mu} G_{\mu \nu} D^\nu \right| & \lesssim \!  \left( \frac{\langle t-r\rangle^2}{\langle t+r \rangle^3}|\rho (G)|+\frac{\langle t-r \rangle}{\langle t+r \rangle^{2}}|\rho (G)|+\frac{\langle t-r\rangle^2}{\langle t+r \rangle^3} |\underline{\alpha} (G)|+\frac{ 1}{ \langle t+r \rangle} |\alpha (G)| \right)  \big\|  \underline{\alpha}^{\mathcal{I}^+} \! \big\|_{1}.
 \end{align*}
We then deduce, as $|t-r| \leq t+r$, $r \leq t+r$ and $t \leq t+r$, that
 \begin{align*}
\left| K^{\mu} G_{\mu \nu} D^\nu \right|  & \lesssim  \bigg( \frac{\langle t+r \rangle}{\langle t \rangle^{1+\delta}}|\rho (G)|+\frac{\langle t-r\rangle}{\langle t \rangle^{1+\delta}} |\underline{\alpha} (G)|+\frac{ \langle t+r\rangle}{\langle \tau-r \rangle^{\frac{1+\delta}{2}} \, \langle \tau+r \rangle^{\frac{1+\delta}{2}}} |\alpha (G)| \bigg)  \frac{\langle t-r \rangle^{\delta}}{r} \big\|  \underline{\alpha}^{\mathcal{I}^+} \big\|_{1}.
 \end{align*}
Consequently, we have
$$ \int_{\tau=t}^n \int_{\R^3_x} | \max (\tau-|x|,1) |^{2a} \left| K^{\mu} G_{\mu \nu} D^\nu\right| (\tau,x) \dr x \dr \tau \lesssim  \mathbf{I}_{\rho, \underline{\alpha}}(t)+\mathbf{I}_{\alpha}(t), $$
where, using the shorthand $w_{\tau-|x|}:= | \max (\tau-|x|,1) |^{a}$,
\begin{align*}
\mathbf{I}_{\rho, \underline{\alpha}}(t)&:= \int_{\tau=t}^n \frac{1}{\langle \tau \rangle^{1+\delta}} \int_{\R^3_x}  w_{\tau-|x|} \Big( \langle \tau+r \rangle \, |\rho (G)|+\langle \tau-r\rangle \, |\underline{\alpha} (G)| \Big) \frac{\langle \tau-r \rangle^{a+\delta}}{r} \big\|  \underline{\alpha}^{\mathcal{I}^+} \big\|_{1}  \bigg(\tau-r,\frac{x}{|x|} \bigg) \dr x \dr \tau, \\
\mathbf{I}_{\alpha}(t) &:= \int_{\tau=t}^n  \int_{\R^3_x} w_{\tau-|x|} \frac{ \langle \tau+r\rangle}{\langle \tau-r \rangle^{\frac{1+\delta}{2}} \, \langle \tau+r \rangle^{\frac{1+\delta}{2}}} |\alpha (G)| \frac{\langle \tau-r \rangle^{a+\delta}}{r} \big\|  \underline{\alpha}^{\mathcal{I}^+} \big\|_{1}  \bigg(\tau-r,\frac{x}{|x|} \bigg) \dr x \dr \tau .
\end{align*}
Apply the Cauchy-Schwarz inequality in the variable $x$ in order to get
\begin{align*}
\mathbf{I}_{\rho , \underline{\alpha}}(t) & \lesssim \int_{\tau=t}^n \frac{1}{\langle \tau \rangle^{1+\delta}} \left| \mathcal{E}^{K,a}[G](\tau)  \int_{r=0}^{+\infty} \int_{\mathbb{S}^2_\omega} \langle \tau-r \rangle^{2a+2\delta}  \big\|  \underline{\alpha}^{\mathcal{I}^+} \big\|_{1}^2\!(\tau-r,\omega) \dr \mu_{\mathbb{S}^2_\omega} \dr r\right|^{\frac{1}{2}} \! \dr \tau \\
& \lesssim  \left| \sup_{t \leq \tau \leq n}  \mathcal{E}^{K,a}[G](\tau) \; \overline{\mathcal{E}}_{1}^{ \, a+\delta} \big[\underline{\alpha}^{\mathcal{I}^+} \big] \right|^{\frac{1}{2}} \int_{\tau=t}^n \frac{\dr \tau }{\langle \tau \rangle^{1+\delta}}   \leq  \frac{ \sqrt{2}}{\delta \, \langle t \rangle^{\delta}}   \left| \, \overline{\mathcal{E}}_{1}^{ \, a+\delta} \big[\underline{\alpha}^{\mathcal{I}^+} \big] \sup_{t \leq \tau \leq n}  \mathcal{E}^{K,a}[G](\tau) \right|^{\frac{1}{2}} .
\end{align*}
Next, for $\mathbf{I}_{\alpha}(t)$, we perform the change of variables $(u,\underline{u},\omega)=(\tau-|x|,\tau+|x|,x/|x|)$ and we apply the Cauchy-Schwarz inequality. We obtain
\begin{align*}
\mathbf{I}_{\alpha}(t) & \leq  \int_{u=-\infty}^n  \int_{\underline{u}=\max(2t-u,u)}^{2n-u} \int_{\mathbb{S}^2_\omega}  \frac{w_u \, \langle \underline{u} \rangle}{\langle u \rangle^{\frac{1+\delta}{2}} \, \langle \underline{u} \rangle^{\frac{1+\delta}{2}}} |\alpha (G)|(u,\underline{u},\omega) \frac{\langle u \rangle^{a+\delta}}{r} \big\|  \underline{\alpha}^{\mathcal{I}^+} \big\|_{1}(u,\omega)  \, r^2 \dr \mu_{\mathbb{S}^2_\omega} \dr \underline{u} \dr u  .
\end{align*}
Thus, by the Cauchy-Schwarz inequality in $u$, we get $\mathbf{I}_\alpha (t) \leq | \mathbf{I}^1_t \cdot \mathbf{I}^2_t |^{1/2}$, where,
\begin{align*}
\mathbf{I}^1_t & := \int_{u=-\infty}^n  \int_{\underline{u}=\max(2t-u,u)}^{2n-u} \int_{\mathbb{S}^2_\omega} \frac{1}{\langle u \rangle^{1+\delta}} |w_u|^2 \langle \underline{u} \rangle^{2} |\alpha (G)|^2(u,\underline{u},\omega) \, r^2 \dr \mu_{\mathbb{S}^2_\omega} \dr \underline{u} \dr u \\
& =\sqrt{2} \int_{u=-\infty}^n \frac{1}{\langle u \rangle^{1+\delta} } \int_{C_u^{t,n}} |w_u|^2  \langle \underline{u} \rangle^{2} |\alpha |^2 \dr \mu_{C_u}  \dr u  \leq  \, \sup_{u \in \R} \mathcal{F}^{K,a}[G]_t^n (u) \int_{u=-\infty}^{+\infty} \frac{\dr u}{\langle u \rangle^{1+\delta} } \leq \frac{4}{ \delta} \, \sup_{u \in \R} \mathcal{F}^{K,a}[G]_t^n (u)
\end{align*}
and
\begin{align*}
\mathbf{I}^2_t & := \int_{u=-\infty}^n  \int_{\underline{u}=\max(2t-u,u)}^{2n-u} \int_{\mathbb{S}^2_\omega} \frac{1}{\langle \underline{u} \rangle^{1+\delta}} \langle u \rangle^{2a+2\delta} \big\|  \underline{\alpha}^{\mathcal{I}^+} \big\|_{1}^2  (u,\omega)  \, \dr \mu_{\mathbb{S}^2_\omega} \dr \underline{u} \dr u \\
& =   \int_{\underline{u}=t}^{+ \infty} \frac{1}{\langle \underline{u} \rangle^{1+\delta}} \int_{u=2t-\underline{u}}^{\underline{u}} \int_{\mathbb{S}^2_\omega} \langle u \rangle^{2a+2\delta} \big\|  \underline{\alpha}^{\mathcal{I}^+} \big\|_{1}^2  (u,\omega)  \dr \mu_{\mathbb{S}^2_\omega} \dr u \dr \underline{u} \leq \overline{\mathcal{E}}^{\, a+\delta}_{1} \big[\underline{\alpha}^{\mathcal{I}^+} \big] \int_{\underline{u}=t}^{+ \infty} \frac{\dr \underline{u}}{\langle \underline{u} \rangle^{1+\delta}} \lesssim \frac{ \overline{\mathcal{E}}^{\, a+\delta}_{1} \big[\underline{\alpha}^{\mathcal{I}^+} \big]}{ \delta \, \langle t \rangle^{\delta}}  .
\end{align*}
\end{proof}

We are now able to bound the remainder part $(R^n)_{n \in \mathbb{N}^*}$ and to prove that it is a Cauchy sequence.
\begin{Pro}\label{ProRT}
There exists a constant $C$ depending only on $\delta$ such that, for every $n \in \mathbb{N}^*$,
$$ \forall \,  t \in \R_+, \qquad \; \mathcal{E}^{K,a}[R^n](t)+ \sup_{u \in \R} \, \mathcal{F}^{K,a}[R^n]_t^n(u) \leq C \, \langle t \rangle^{-\delta} \, \overline{\mathcal{E}}^{ \, a+\delta}_{1} \big[ \underline{\alpha}^{\mathcal{I}^+} \big].$$
Moreover, for all $1 \leq n_1 \leq n_2$, 
$$ \sup_{0 \leq t \leq n_1} \mathcal{E}^{K,a}[R^{n_1}-R^{n_2}](t) \leq C \, n_1^{-\delta} \, \overline{\mathcal{E}}^{ \, a+\delta}_{1} \big[ \underline{\alpha}^{\mathcal{I}^+} \big].$$
\end{Pro}
\begin{proof}
Fix $n \in \mathbb{N}^*$ and apply the energy estimate of Proposition \ref{ProenergyforscatMax} to $R^n$, which vanishes as well as its derivatives at $t=n$. We get, for all $t \in [0,n]$,
\begin{align*}
 \mathcal{E}^{K,a}[R^n](t) +\sup_{u \in \R} \,  \mathcal{F}^{K,a}[R^n]_t^n(u) \leq 8  \int_{\tau=t}^n \int_{\R^3_x} | \max (\tau-|x|,1) |^{2a}  \big| K^{\mu} R^n_{\mu \nu} J^{\mathrm{err},\nu} +K^{\mu} {}^* \!R^n_{\mu \nu} \widetilde{J}^{\mathrm{err},\nu } \big|(\tau,x) \dr \tau \dr x. 
\end{align*}
Note now that 
$$\mathcal{E}^{K,a} \big[{}^* \! R^n \big]=\mathcal{E}^{K,a} \big[R^n \big], \qquad \qquad \mathcal{F}^{K,a} \big[{}^* \! R^n \big]=  \mathcal{F}^{K,a} \big[ R^n \big].$$
Indeed, according to \cite[Equation $(3.9)$]{CK} or by direct computations, for any $2$-form $G$,
$$ \rho ( {}^* \! G ) = \sigma (G), \qquad \sigma ({}^* \! G)= -\rho (G)   , \qquad \underline{\alpha} ({}^* \! G)_{e_A}= \varepsilon^{AB} \underline{\alpha} (G)_{e_B}, \qquad \alpha ({}^* \! G)_{e_A}= \varepsilon^{BA} \alpha (G)_{e_B}, \qquad A \in \{ \theta , \varphi \} .$$
Apply Lemma \ref{Lemforenergy} for $G=R^n$ as well as $G={}^*R^n$. It yields
\begin{align*}
\forall \, t \in [T,n], \qquad  \mathcal{E}^{K,a}[R^n](t) +\sup_{u \in \R}\mathcal{F}^{K,a}[R^n]_t^n(u) \lesssim \frac{\delta^{-1}}{\langle t \rangle^{\frac{\delta}{2}}} \Big| \overline{\mathcal{E}}^{\, a+\delta}_{1}\big[\underline{\alpha}^{\mathcal{I}^+} \big] \Big|^{\frac{1}{2}} \left|  \mathcal{E}^{K,a}[R^n](t) +\sup_{u \in \R}\mathcal{F}^{K,a}[R^n]_t^n(u) \right|^{\frac{1}{2}} \! ,
 \end{align*}
from which we obtain the stated estimate for $R^n$ on $[0,n]$, and then on $\R_+$ since $R^n=0$ for $t \geq n$.

Consider now $1 \leq n_1 < n_2$ and introduce $G:=R^{n_2}-R^{n_1}$, which satisfies
$$\nabla^{\mu} G_{\mu \nu}= \left( \chi(t/n_1)-\chi(t/n_2) \right) J_\nu^{\mathrm{err}}, \qquad \nabla^{\mu} {}^* \! G_{\mu \nu}= \left( \chi(t/n_1)-\chi(t/n_2) \right) \widetilde{J}^{\mathrm{err}}_\nu .$$
As $\chi(s)=1$ for all $s \leq 1/4$, the source term vanishes for $t \leq n_1/4$, so that, by Proposition \ref{ProenergyforscatMax},
\begin{align*}
\sup_{ 0 \leq t \leq n_1} \mathcal{E}^{K,a}[G](t)+ \sup_{u \in \R} \mathcal{F}^{K,a}[G]_0^{n_1}&(u) \leq \mathcal{E}^{K,a}[G](n_1) \\
&+ \int_{\tau=\frac{n_1}{4}}^{n_1}\int_{\R^3_x} | \max (\tau-|x|,1) |^{2a}\big| \, \overline{K}_0^{\mu}G_{\mu \nu} J^{\mathrm{err},\nu}+ \overline{K}_0^{\mu} {}^* \! G_{\mu \nu} \widetilde{J}^{\mathrm{err},\nu}\big| (\tau,x) \dr x \dr \tau ,
\end{align*}
As $R^{n_1}(n_1,\cdot)=0$, we have $\mathcal{E}^{K,a}[G](n_1)= \mathcal{E}^{K,a}[R^{n_2}](n_1)\lesssim  n_1^{-\delta} \, \overline{\mathcal{E}}^{\, a+\delta}_{1}[\underline{\alpha}^{\mathcal{I}^+}]$. Moreover, by Lemma \ref{Lemforenergy}, applied to $G$ and ${}^* \! G$, the second term on the right hand side of the previous inequality is bounded above by
$$\left\langle \frac{n_1}{4} \right\rangle^{-\frac{\delta}{2}} \big| \overline{\mathcal{E}}^{\, a+\delta}_{1}\big[\underline{\alpha}^{\mathcal{I}^+}\big] \big|^{\frac{1}{2}} \Big| \sup_{n_1/4 \leq \tau \leq n_1}  \mathcal{E}^{K,a}[G](\tau) + \sup_{u \in \R} \mathcal{F}^{K,a} [G]_{n_1/4}^{n_1}(u) \Big|^{\frac{1}{2}}.$$
It implies the energy estimate for $G=R^{n_1}-R^{n_2}$.
\end{proof}

We finally study the sequence of approximate solutions $(F^n)_{n \geq T}$.
\begin{Pro}\label{ProCauchy}
There exists a constant $C$, depending only on $\delta$, such that the following estimates hold. For any $n \geq T$ and $T \leq n_1 < n_2$, we have
$$\sup_{T \leq  t \leq n} \mathcal{E}^{K,a}[F^n](t) \leq C \, \overline{\mathcal{E}}^{\, a+ \delta}_{1} \big[ \underline{\alpha}^{\mathcal{I}^+} \big]
+ B_{\mathrm{source}}, \qquad \quad \sup_{T \leq t \leq n_1} \mathcal{E}^{K,a}[ F^{n_2}-F^{n_1}](t) \leq \frac{C}{  n_1^\delta} \,  \overline{\mathcal{E}}^{\, a+ \delta}_{1} \big[ \underline{\alpha}^{\mathcal{I}^+} \big]+ \frac{2}{  n_1^\delta} B_{\mathrm{source}}.$$
\end{Pro}
\begin{proof}
As $F^n=F^{\mathrm{lead}}+R^n+S^n$, the two parts of the result ensue from Corollary \ref{CorFlead}, the previous Proposition \ref{ProRT} and the assumptions on $S(J,n)$.
\end{proof}
 
We are now able to conclude this section.

\begin{refproof}[Proof of Proposition \ref{ProscatMax}.] Consider $T_0 \geq T$ and let $(\mathbb{B},\sup_{[T,T_0]} \mathcal{E}^{K,a}[\cdot])$ be the subset of the locally integrable $2$-form $G$ defined on $[T,T_0]\times \R^3_x$ and such that $\sup_{T \leq \tau \leq T_0}\mathcal{E}^{K,a}[G](\tau)<+\infty$. Then,
\begin{itemize}
\item $F^n$ is a solution to the Maxwell equations with source term $J$ on $[T,T_0]$, for all $n \geq 4T_0$. 
\item By Proposition \ref{ProCauchy}, $(F^n)_{n \geq 4T_0}$ is bounded by $C \, \overline{\mathcal{E}}^{\, a+ \delta}_{1} \big[ \underline{\alpha}^{\mathcal{I}^+} \big]+ B_{\mathrm{source}}$ in $\mathbb{B}$. 
\item Still by Proposition \ref{ProCauchy}, $(F^n)_{n \geq 4T_0}$ is a Cauchy sequence in $\mathbb{B}$.
\end{itemize}
Consequently, $(F^n)_{n \geq 4T_0}$ converges in $\mathbb{B}$ to a $2$-form $F^{\infty}_{T_0}$, defined on $[T,T_0]\times \R^3_x$ and such that
\begin{equation*}
 \nabla^{\mu} F^{\infty}_{T_0, \, \mu \nu}=J_\nu, \qquad \nabla^{\mu} {}^* \! F^{\infty}_{T_0, \, \mu \nu}=0  .
\end{equation*}
Since for any $p \in \mathbb{N}$ and any $T_0 \geq T$, $F^{\infty}_{T_0+p}=F^{\infty}_{T_0}$ on $[T,T_0]$, we can define a $2$-form $F$ on $[T,+\infty[ \times \R^3_x$ by $F(t,x):=F^{\infty}_{  \lceil t \rceil}(t,x)$. The field $F$ is then a weak solution to the Maxwell equations, with source term $J$, which satisfies, according to Proposition \ref{ProCauchy},
$$ \sup_{t \geq T}\mathcal{E}^{K,a}[F](t) \leq C \, \overline{\mathcal{E}}^{\, a+\delta}_{1} \big[ \underline{\alpha}^{\mathcal{I}^+} \big]+ B_{\mathrm{source}}, \qquad \; \forall \, n \in \mathbb{N}^*, \quad   \sup_{t \geq T}\mathcal{E}^{K,a}[F-F^n](t) \leq \frac{C}{  n^\delta} \,  \overline{\mathcal{E}}^{\, a+ \delta}_{1} \big[ \underline{\alpha}^{\mathcal{I}^+} \big]+ \frac{2}{  n^\delta} B_{\mathrm{source}} .$$
 Finally, as $F^n = F^{\mathrm{lead}}$ on $[n,+\infty[ \times \R^3$, we get that $\underline{\alpha}^{\mathcal{I}^+}$ is the radiation field of $F$ in the sense of Definition \ref{Defrad}. Finally, the uniqueness of the solution $F$ is given by Theorem \ref{Thscat}, stated thereafter.
\end{refproof}

\subsection{Wave operators for the vacuum Maxwell equations}
We apply the results of this section, for a vanishing source term and an initial time $T=0$, in order to construct wave operators for
\begin{equation}\label{eq:Maxvac}
 \nabla^\mu F_{\mu \nu} =0, \qquad \qquad \nabla^\mu {}^* \! F_{\mu \nu} =0,
\end{equation}
between weighted $L^2$ Hilbert spaces. For this, we refine the results of Proposition \ref{Thscatforsmooth} when the source term $J$ vanishes. Recall the space of radiation fields $\mathcal{S}^K_{\mathcal{I}^+}$, introduced in Definition \ref{Defscatstae}.
\begin{Pro}\label{Provacforw}
Let $F$ be a solution to the vacuum Maxwell equations such that $\mathcal{E}^K_3[F](0)<+\infty$. Then, $F$ has a radiation field along future null infinity $\underline{\alpha}^{\mathcal{I}^+}$ which belongs to $\mathcal{S}_{\mathcal{I}^+}^K$ and such that
$$ \big\| \underline{\alpha}^{\mathcal{I}^+} \big\|^2_{\mathcal{S}^K_{\mathcal{I}^+}} := \int_{\R_u} \int_{\mathbb{S}^2_{\omega}} \langle u \rangle^2 \big|\underline{\alpha}^{\mathcal{I}^+} \big|^2(u,\omega)+4\left(\big|\rho^{\mathcal{I}^+}\big|^2+\big|\sigma^{\mathcal{I}^+}\big|^2\right)(u,\omega) \dr \mu_{\mathbb{S}^2_\omega} \dr u = \mathcal{E}^K[F](0) .$$   
\end{Pro}
\begin{proof}
Remark first that the conclusions of Proposition \ref{Thscatforsmooth} hold. In particular, $\underline{\alpha}^{\mathcal{I}^+} $, $\rho^{\mathcal{I}^+}$ as well as $\sigma^{\mathcal{I}^+}$ exist and $\underline{\alpha}^{\mathcal{I}^+} \! \in \mathcal{S}_{\mathcal{I}^+}^K$. By a density argument, it suffices to prove the conservation law for smooth and compactly supported initial data. Assume then that $F(0,\cdot)$ is supported in $\{|x| \leq R\}$ for a certain $R>0$. Applying the divergence theorem to $\mathbb{T}[F]_{\mu \nu} K^\nu$ in the region $\{(\tau,x) \in \R_+ \times \R^3 \, | \, \tau+|x| \leq \underline{u} \}$, for all $\underline{u} \geq R$, we get
\begin{align*}
\mathbf{F}(\underline{u}):=& \int_{|u| \leq \underline{u}} \int_{\mathbb{S}^2_\omega} \langle u \rangle^2 \, |r \underline{\alpha} (F)|^2(u,\underline{u},\omega)+ \langle \underline{u} \rangle^2 \left[|r \rho (F) |^2+| r \sigma(F)|^2\right](u,\underline{u},\omega) \dr \mu_{\mathbb{S}^2_\omega} \dr u \\
 = &\int_{|x| \leq \underline{u} }   \langle x \rangle^ 2 \,\left[ |\alpha (F)|^2+ |\underline{\alpha} (F)|^2+2(|\rho (F) |^2+|\sigma(F)|^2)\right]\!(0,x)  \,  \mathrm{d} x = \mathcal{E}^K[F](0).
\end{align*}
As $\Box F_{\mu \nu}=0$, we then deduce that $F$, as well as $\underline{\alpha}^{\mathcal{I}^+}$, $\rho^{\mathcal{I}^+}$ and $\sigma^{\mathcal{I}^+}$, are supported in $\{ |u|=|t-r|\leq R \}$. Hence, using that the domain of integration is compact, $\underline{u}=2r+u$ and the estimates \eqref{eq:decayrhosig}, we obtain
$$ \mathbf{F}(\underline{u}) = \int_{|u| \leq R} \int_{\mathbb{S}^2_\omega} 4 \big[|r^2 \rho (F)|^2+|r^2 \sigma (F)|^2\big](u,\underline{u},\omega)+\langle u \rangle^2 |r \underline{\alpha} (F)|^2(u,\underline{u},\omega) \mathrm{d} \mu_{\mathbb{S}^2_\omega} \mathrm{d} u + O(\underline{u}^{-1}). $$
It remains to use \eqref{eq:rhoetsigma}--\eqref{eq:underlinealpha}.

\end{proof}
Note then that $\big(\mathcal{S}^K_{\mathcal{I}^+}, \| \cdot \|_{\mathcal{S}^K_{\mathcal{I}^+}} \big)$ is a Hilbert space. We define now the space of the initial data.
\begin{Def}
Let $\mathcal{S}^K_{\{t=0\}}$ be the set of the $2$-forms $F_0$, defined on $\R^{1+3}$, independent of the time variable $t$ and such that
$$ \|F_0\|_{\mathcal{S}^K_{\{t=0 \}}} := \mathcal{E}^K[F_0] <+\infty,$$
where the norm $\mathcal{E}^K [ \cdot ]$ is defined in \eqref{keva:defenergy}. Then, $\big(\mathcal{S}^K_{\{t=0\}}, \| \cdot \|_{\mathcal{S}^K_{\{t=0\}}} \big)$ is a Hilbert space.
\end{Def}
We now construct an isometry between the initial data and the scattering states of \eqref{eq:Maxvac}. In fact, we already constructed a similar result in \cite{scat}, for different Hilbert spaces. For completeness and since we will require it once in this paper, we recall it here. For this, we introduce energy spaces associated to the multiplier $\partial_t$.
\begin{Def}
Let $\mathcal{S}^{\partial_t}_{\{t=0\}}$ be the set of the $2$-forms $F_0$, defined on $\R^{1+3}$ and independent of $t$, such that
$$ \|F_0\|_{\mathcal{S}^{\partial_t}_{\{t=0 \}}}^2 := \int_{\R^3_x} \big| \alpha(F_0) \big|^2(x)+\big| \underline{\alpha}(F_0) \big|^2(x)+2\big| \rho(F_0) \big|^2(x)+2\big| \sigma(F_0) \big|^2(x) \dr x < +\infty.$$
Let further, for $N \in \mathbb{N}$, $\mathcal{S}^{\partial_t,N}_{\{t=0 \}} \subset \mathcal{S}^{\partial_t}_{\{t=0 \}}$ be the set of the $2$-forms $F_0$ on $\R^{1+3}$ independent of $t$ verifying
$$\|F_0\|^2_{\mathcal{S}^{\partial_t,N}_{\{t=0 \}}} := \sum_{|\gamma| \leq N}  \|\mathcal{L}_{Z^\gamma}(F_0)(0,\cdot) \|^2_{\mathcal{S}^{\partial_t}_{\{t=0 \}}} <+\infty.$$
Consider $\mathcal{S}^{\partial_t}_{\mathcal{I}^+}$, the set of the $1$-forms $\underline{\alpha}^{\mathcal{I}^+}$ on $\R_u \times \mathbb{S}^2$ which are tangential to the $2$-spheres and such that
$$ \big\| \underline{\alpha}^{\mathcal{I}^+} \big\|^2_{\mathcal{S}^{\partial_t}_{ \mathcal{I}^+ }} := \int_{\R_u} \int_{\mathbb{S}^2_\omega} \big| \underline{\alpha}^{\mathcal{I}^+} \big|^2(u,\omega) \dr \mu_{\mathbb{S}^2_\omega} \dr u <+\infty .$$
Finally, for $N \in \mathbb{N}$, $\mathcal{S}^{\partial_t,N}_{\mathcal{I}^+}$ denotes the sets of the $1$-forms $\underline{\alpha}^{\mathcal{I}^+} \in \mathcal{S}^K_{\mathcal{I}^+}$ satisfying
$$ \big\| \underline{\alpha}^{\mathcal{I}^+} \big\|^2_{\mathcal{S}^{\partial_t,N}_{ \mathcal{I}^+ }} :=  \sum_{|\gamma| \leq N} \big\| \underline{\alpha}^{\mathcal{I}^+}_{\, \gamma} \big\|^2_{\mathcal{S}^{\partial_t}_{\mathcal{I}^+}}<+\infty .$$
Then, for any $N \in \mathbb{N}$, $\big(\mathcal{S}^{\partial_t,N}_{\{t=0 \}},\|\cdot \|_{\mathcal{S}^{\partial_t,N}_{\{t=0 \}}} \big)$ and $\big(\mathcal{S}^{\partial_t,N}_{\mathcal{I}^+}, \| \cdot \|_{\mathcal{S}^{\partial_t,N}_{ \mathcal{I}^+ }} \big)$ are Hilbert spaces.
\end{Def}
\begin{Th}\label{Thscat}
The forward map
\begin{array}[t]{lrcl}
&\mathscr{F}^+ : & \mathcal{S}^{\partial_t}_{\{t=0\}} \cap C_c^{\infty}& \longrightarrow  \mathcal{S}^{\partial_t}_{\mathcal{I}^+} \\
 &   & F_0 & \longmapsto  \lim_{\underline{u} \to +\infty} r\underline{\alpha}(F)(u,\underline{u},\omega), 
\end{array}

\vspace{1.5mm}

\noindent where $F$ is the unique solution to the vacuum Maxwell equations \eqref{eq:Maxvac} such that $F(0,\cdot)=F_0$, is well-defined and preserves the norm $\|F_0\|_{\mathcal{S}^K_{\{t=0 \}}} =\| \mathscr{F}^+(F) \|_{\mathcal{S}^K_{\mathcal{I}^+}}$. Moreover,
\begin{itemize}
\item this map can be uniquely extended in a bijective isometry $\mathscr{F}^+ : \mathcal{S}^{\partial_{t}}_{\{t=0\}} \rightarrow \mathcal{S}^{\partial_t}_{\mathcal{I}^+}$.
\item For $N \in \mathbb{N}$, $\mathscr{F}^+ \vert_{\mathcal{S}^{\partial_{t},N}_{\{t=0\}}}$ is a bijective isometry between the spaces of higher regularity $\mathcal{S}^{\partial_{t},N}_{\{t=0\}}$ and $\mathcal{S}^{\partial_{t},N}_{\mathcal{I}^+}$.
\item $\mathscr{F}^+ \vert_{\mathcal{S}^{K}_{\{t=0\}}}\!:\mathcal{S}^{K}_{\{t=0\}}\! \rightarrow \mathcal{S}^{K}_{\mathcal{I}^+}$ is a bijective isometry between Hilbert spaces of more strongly decaying data.
\end{itemize}
\end{Th}
\begin{Rq}
Similar considerations as those presented in Section \ref{Subsecsoncstr} allow one to construct various scattering maps for the vacuum Maxwell equations \eqref{eq:Maxvac}. 
\end{Rq}
\begin{proof}
All but the last part of the statement are given by \cite[Theorem $7.6$ and Corollary $7.15$]{scat}. According to Proposition \ref{Provacforw}, $\mathscr{F}^+\vert_{\mathcal{S}^K_{\{t=0\}} \cap C_c^{\infty}}$ extends to an injective isometry from $\mathcal{S}^K_{\{t=0\}}$ to $\mathcal{S}^K_{\mathcal{I}^+}$. Then, by Proposition \ref{ProscatMax}, applied for $J=0$ and $T=0$, we get that any smooth and compactly supported radiation field is in the image of $\mathscr{F}^+$, implying that the map is onto.
\end{proof}

\section{The asymptotic Maxwell equations}\label{SecMaxasymp}

We fix, for all this section, $f_\infty : \R^3_x \times \R^3_v \to \R$ a sufficiently regular function. Motivated by Corollary \ref{Corlinbound}, we introduce the asymptotic $4$-current density generated by a constant in time distribution function.
\begin{Def}\label{DefJasymp}
For any $h \in L^1( \R^3_z \times \R^3_v , \R)$, we define its asymptotic current $J^{\mathrm{asymp}}[h]$ as 
$$  \forall \, |x|<t, \quad J^{\mathrm{asymp}}_\nu[h](t,x) =\frac{1}{t^3} \frac{x_\nu}{t} \int_{\R^3_z} \big[ |v^0|^5 h \big] \bigg(z, \frac{\widecheck{\; x \;}}{t} \bigg)\dr z, \qquad \qquad \forall \, |x| \geq t>0, \quad J^{\mathrm{asymp}}_\nu[h](t,x)=0. $$
\end{Def}
\begin{Rq}
If $h \in C^0(\R^3_z \times \R^3_v,\R)$ and  $\int_z \langle v \rangle^5 h(z,v) \dr z \to 0$ as $|v| \to +\infty$, then $J^{\mathrm{asymp}}[f_\infty]$ vanishes and is continous at any point of the light cone $t=|x|$.
\end{Rq}
As we shall see, such a current is in fact generated by a singular solution to the linear Vlasov equation. This section is devoted to the study of $F^{\mathrm{asymp}}[f_\infty]$, the unique solution to the asymptotic Maxwell equations associated to $f_\infty$, with initial time $t_0 >0$,
\begin{equation}\label{eq:defasympF}
\begin{cases} \nabla^\mu F^{\mathrm{asymp}}_{\mu \nu}[f_\infty] = J^{\mathrm{asymp}}_\nu[f_\infty],  \\[3pt]  \nabla^\mu {}^* \! F^{\mathrm{asymp}}_{\mu \nu}[f_\infty] =0,  \\[3pt]
\nabla \times E(t_0,\cdot) = \nabla \times B(t_0,\cdot)=0, 
 \end{cases} 
 \end{equation}
where $E(t_0,\cdot)$ and $B(t_0,\cdot)$ are the initial electric and magnetic field,
 $$E^i(t_0,\cdot):=F^{\mathrm{asymp}}_{0i}[f_\infty](t_0,\cdot), \qquad B^i(t_0,\cdot):= -\varepsilon^{ijk} F^{\mathrm{asymp}}_{jk}[f_\infty] (t_0, \cdot).$$
Recall from Section \ref{Subsecinidata} that the initial assumptions imply that $B(t_0,\cdot)=0$ and $E(t_0,\cdot)=\nabla \phi$, where $\phi$ is solution to the Poisson equation $\Delta \phi =-J_0^{\mathrm{asymp}}[f_\infty](t_0, \cdot )$. 

On the top of the estimates derived in this section, we will control the asymptotic electromagnetic field $\mathbb{F}[f_\infty]$ and the correction coefficients. In the next sections, we will always take $t_0=1$. The wider range $0 < t_0 \leq 2$ will however be allowed in this section in order to prove the following statement, which is reminiscent of the case of the Vlasov-Poisson system \cite{scattPoiss}. 
\begin{Pro}\label{Probehavior}
Let $(t,v) \! \in \R_+^* \! \times \R^3_v$. Then, $F^{\mathrm{asymp}}[f_\infty](t,t\widehat{v}) \to t^{-2}\mathbb{F}[f_\infty](v)$ as the initial time $t_0 \to 0$. 
\end{Pro}

\subsection{Singular solution to the linear Vlasov equation}\label{Subsecfsing}

Motivated by Corollary \ref{Corlinbound}, we would like to interpret the leading order term $J^{\mathrm{asymp}}_0[f_\infty]$ in the asymptotic expansion of $J(f_1)_0=-\int_v f_1 \dr v$ as the velocity average of a (singular) solution to the linear Vlasov equation. This formalism will lead to considerable simplifications, for the commutation of \eqref{eq:defasympF} and in particular for the derivation of the Glassey-Strauss decomposition of $\nabla F^{\mathrm{asymp}}[f_\infty]$. We assume here that $f_\infty \in H^1(\R^3_z \times \R^3_v)$ goes fast enough to $0$ as $|z|+|v| \to \infty$.
\begin{Def}\label{Defsingsolu}
Let $f^{\mathrm{sing}}[f_\infty] \in \mathcal{D}'\big(\R_+^* \times \R^3_x \times \R^3_v\big)$ be the distribution
$$ f^{\mathrm{sing}}[f_\infty] (t,x,v) := \int_{\R^3_z} f_\infty(z,v) \dr z \, \delta (x-t\widehat{v}),$$
where $\delta$ is the Dirac delta distribution in $\R^3$. We define its velocity averages as
$$ \int_{\R^3_v} w(v) f^{\mathrm{sing}}[f_\infty] \dr v (t,x) := \frac{\mathds{1}_{t>|x|}}{t^3} \int_{\R^3_z} \Big[w(v) |v^0|^5 f_\infty \Big] \bigg( z,\frac{\widecheck{\; x \;}}{t} \bigg) \dr z,$$ 
for any sufficiently regular function $w : \R^3_v \to \R$.
\end{Def}
\begin{Rq}
The velocity average with the weight function $w(v)=\widehat{v}_\nu$ is equal to $J^{\mathrm{asymp}}_\nu[f_\infty]$.
\end{Rq}
We now investigate certain basic properties of this distribution, which explain in particular that we abusively use the term velocity average of $f^{\mathrm{sing}}[f_\infty]$. Note first that if $(\psi_n)_{n \in \mathbb{N}^*}$ is a mollifier in $\R^3$ and $$h_n(t,x,v):=\int_{\R^3_z} f_\infty(z,v) \dr z \, \psi_n (x-t\widehat{v}),$$ then $(h_n)_{n \in \mathbb{N}^*}$ converges to $f^{\mathrm{sing}}[f_\infty]$ in $\mathcal{D}'\big(\R_+^* \times \R^3_x \times \R^3_v\big)$.
\begin{Lem}\label{Lembasicpropsing}
The distribution $f^{\mathrm{sing}}[f_\infty]$ verifies
$$ f^{\mathrm{sing}}[f_\infty] (t,x,v) = \frac{\mathds{1}_{t>|x|}}{t^3} \int_{\R^3_z}  |v^0|^5 f_\infty  ( z,v) \dr z \, \delta \bigg( v= \frac{\widecheck{\; x \;}}{t} \bigg)   $$ 
and is a singular solution to the linear Vlasov equation, $\T_0(f^{\mathrm{sing}}[f_\infty])=0$. Furthermore,
$$ \int_{\R^3_v}w h_n \dr v    \xrightarrow[n \to + \infty]{} \int_{\R^3_v}w f^{\mathrm{sing}}[f_\infty] \dr v  \qquad \text{in $L^\infty_{\mathrm{loc}} \big(\R_+^*, L^2(\R^3_x) \big)$,}$$
for any $w \in C^1(\R^3_v ,\R)$ such that $w|v^0|^5 \int_z f_\infty (z,\cdot) \dr z \in H^1( \R^3_v)$.
\end{Lem}
\begin{proof}
The alternative expression for $f^{\mathrm{sing}}[f_\infty]$ follows from the composition rule for the Dirac delta function as well as Lemma \ref{cdv}. As $\T_0(v)=\T_0(x-t\widehat{v})=0$, we have $\T_0(h_n)=0$, which implies that $f^{\mathrm{sing}}[f_\infty]$ is solution to the linear Vlasov equation in the space of the distributions. Finally, the convergence of the velocity averages holds since $(\psi_n)_{n \geq 1}$ is an approximation to the identity.
\end{proof}
We are now interested in computing $\widehat{Z} f^{\mathrm{sing}}[f_\infty]$.
\begin{Lem}\label{Lemderivsing}
For any homogeneous vector field $\widehat{Z} \in \widehat{\mathbb{P}}_S \setminus \{ \partial_t, \partial_{x^1}, \partial_{x^2}, \partial_{x^3} \}$, we have
$$ \widehat{Z} f^{\mathrm{sing}}[f_\infty] = f^{\mathrm{sing}} \big[ \widehat{Z}_\infty f_\infty \big].$$
\end{Lem}
\begin{proof}
Let $\psi \in C^1(\R^3,\R)$, $1 \leq i < j \leq 3$ and $1 \leq k \leq 3$. Then, we have
\begin{align*}
  \widehat{\Omega}_{0k} \big( \psi(x-t\widehat{v}) \big)&=-(x^k-t\widehat{v}^k)\widehat{v} \cdot \nabla \psi (x-t\widehat{v}) , \qquad \qquad \widehat{\Omega}_{ij}\big( \psi(x-t\widehat{v}) \big)= \big[ \Omega_{ij} \psi \big](x-t\widehat{v})  , \\
  S \big( \psi(x-t\widehat{v}) \big)& = (x-t\widehat{v}) \cdot \nabla \psi (x-t\widehat{v}).
  \end{align*}
  Consider now $\phi \in H^1(\R_+^* \times \R^3_x \times \R^3_v)$ a compactly supported function and recall that the Dirac delta function belongs to $H^{-1}$. We denote by $\langle \cdot | \cdot \rangle_{L^2_{t,x,v}}$ the inner product of $L^2(\R_+^* \times \R^3_x \times \R^3_v)$ and by $\langle \cdot | \cdot \rangle$ the duality bracket for $H^1(\R_+^* \times \R^3_x \times \R^3_v)$. Performing integration by parts in $x$, we have
  $$ \Big\langle  \widehat{\Omega}_{0k} \big( \psi_n(x-t\widehat{v}) \big)  \,  \Big| \,   \phi \Big\rangle_{L^2_{t,x,v}}= \Big\langle \psi_n(x-t\widehat{v}) \, \Big| \, (x^k-t\widehat{v}^k)\widehat{v} \cdot \nabla_x  \phi +\widehat{v}^k \phi   \Big\rangle_{L^2_{t,x,v}}  \xrightarrow[n \to + \infty]{} \Big\langle \delta (x-t\widehat{v}) \, \Big| \, \widehat{v}^k \phi   \Big\rangle,$$
since $x^k-t\widehat{v}^k$ vanishes on the support of $\delta (x-t\widehat{v})$. Next, by integration by parts in $z$, we have 
\begin{equation}\label{eq:eq:model}
  \widehat{\Omega}_{0k} \int_{\R^3_z} f_\infty (z,v) \dr z = \int_{\R^3_z}v^0 \partial_{v^k} f_\infty (z,v ) \dr z =\int_{\R^3_z}\widehat{\Omega}^\infty_{0k} f_\infty (z,v ) \dr z-\int_{\R^3_z}\widehat{v}^k f_\infty (z,v ) \dr z,
  \end{equation}
for almost all $(t,v) \in \R \times \R^3_v$. Consequently, by continuity of the distributional partial derivatives, we obtain
$$ \widehat{\Omega}_{0k} f^{\mathrm{sing}}[f_\infty](t,x,v)=\int_{\R^3_z}\widehat{\Omega}^\infty_{0k} f_\infty (z,v ) \dr z \, \delta (x-t\widehat{v})=f^{\mathrm{sing}} \big[ \widehat{\Omega}^\infty_{0k} f_\infty \big](t,x,v).$$
 Similarly, there holds
\begin{align*}
\Big\langle  \widehat{\Omega}_{ij} \big( \psi_n(x-t\widehat{v}) \big)  \,  \Big| \,   \phi \Big\rangle_{L^2_{t,x,v}}&= \Big\langle \psi_n(x-t\widehat{v}) \, \Big| \, (x^i-t\widehat{v}^i)\partial_{x^j}\phi -(x^j-t\widehat{v}^j)\partial_{x^i}\phi   \Big\rangle_{L^2_{t,x,v}}  \xrightarrow[n \to + \infty]{} 0, \\
\Big\langle  S \big( \psi_n(x-t\widehat{v}) \big)  \,  \Big| \,   \phi \Big\rangle_{L^2_{t,x,v}}&= \Big\langle \psi_n(x-t\widehat{v}) \, \Big| \, -(x-t\widehat{v}) \cdot \nabla_x \phi  -3 \phi \Big\rangle_{L^2_{t,x,v}}  \xrightarrow[n \to + \infty]{} \Big\langle \delta (x-t\widehat{v}) \, \Big| \, -3 \phi   \Big\rangle.
\end{align*}
Then, the result ensues from $\widehat{S}=S+3$ as well as, by integration by parts in $z$,
\begin{equation}\label{eq:model2}
 \widehat{\Omega}_{ij} \int_{\R^3_z} f_\infty (z,v) \dr z = \int_{\R^3_z}\widehat{\Omega}_{ij} f_\infty (z,v ) \dr z, \qquad S \int_{\R^3_z} f_\infty (z,v) \dr z=0 = \int_{\R^3_z} \widehat{S} f_\infty (z,v) \dr z.
 \end{equation}
\end{proof}
The cases of the translations are slightly more complicated. It turns out that that, as in Section \ref{Subsubsecderivlin}, it is in fact easier to deal first with $Y_{\mu \nu}=\widehat{v}_\mu \partial_{x^\nu}-\widehat{v}_\nu \partial_{x^\mu}$.
\begin{Lem}
For any $1 \leq k \leq 3$ and $1 \leq i < j \leq 3$, there holds
$$ Y_{0k} f^{\mathrm{sing}}[f_\infty]\! =  \frac{1}{t} \! \int_{\R^3_z} \! f_\infty(z,v) \dr z \, v^0 \partial_{v^k} \! \Big[ \delta (x-t\widehat{v}) \Big], \quad  Y_{ij} f^{\mathrm{sing}}[f_\infty]\! = \! \frac{-1}{t} \! \int_{\R^3_z} \! f_\infty(z,v) \dr z \, \big(v^i \partial_{v^j}-v^j \partial_{v^i} \big) \Big[ \delta (x-t\widehat{v}) \Big].$$
\end{Lem}
\begin{proof}
Recall the relations
\begin{align*}
 -tY_{0k}&=t\left(\partial_{x^k}+\widehat{v}_k \partial_t \right)=\Omega_{0k}-(x^k-t\widehat{v}^k) \partial_t=-v^0 \partial_{v^k}+\widehat{\Omega}_{0k}-\widehat{v}^k- \partial_t \big[ (x^k-t\widehat{v}^k) \cdot \big],  \\
 tY_{ij}&=t\left(\widehat{v}_i\partial_{x^j}-\widehat{v}_j \partial_{x^i} \right) =-\big(v^i \partial_{v^j}-v^j \partial_{v^i} \big)+ \widehat{\Omega}_{ij}-(x^i-t\widehat{v}^i) \partial_{x^j}+(x^j-t\widehat{v}^j) \partial_{x^i}.
\end{align*}
Next, we claim that
$$Y_{\mu \nu} f^{\mathrm{sing}}[f_\infty]=\int_{\R^3_z} f_\infty (z,v) \dr z Y_{\mu \nu}(\delta (x-t\widehat{v})), \qquad \qquad \widehat{\Omega}_{0k}\big[\delta (x-t\widehat{v}) \big]-\widehat{v}^k\delta (x-t\widehat{v})=\widehat{\Omega}_{ij}\big[\delta (x-t\widehat{v}) \big]=0.$$
The first equality is straightforward since the spatial average of $f_\infty$ is independent of $(t,x)$. The other ones were derived during the proof of Lemma \ref{Lemderivsing}. The result then ensues from
$$ \Big\langle  \partial_{x^\nu} \big[(x^\ell-t\widehat{v}^\ell) \psi_n(x-t\widehat{v}) \big]  \,  \Big| \,   \phi \Big\rangle_{L^2_{t,x,v}}= \Big\langle \psi_n(x-t\widehat{v}) \, \Big| \, (x^\ell-t\widehat{v}^\ell)\partial_{x^\nu}  \phi  \Big\rangle_{L^2_{t,x,v}}  \xrightarrow[n \to + \infty]{} 0,$$
which holds for any $0 \leq \nu \leq 3$, $\ell \neq \nu$ and all compactly supported function $\phi \in H^1(\R_+^* \times \R^3_x \times \R^3_v)$.
\end{proof}
Using the relations
$$ \partial_t= |v^0|^2 \big( \T_0+ \widehat{v}^i Y_{0i} \big), \qquad \qquad \partial_{x^k}=|v^0|^2 \big( -Y_{0k}-\widehat{v}^k \T_0+\widehat{v}^i Y_{ki} \big) $$
as well as $\T_0 f^{\mathrm{sing}}[f_\infty]=0 $, we finally obtain the following result.
\begin{Lem}\label{Lemsingderiv}
For any $1 \leq k \leq 3$, there holds
\begin{align*}
\partial_{t} f^{\mathrm{sing}}[f_\infty](t,x,v)& =  \frac{1}{t}\int_{\R^3_z}|v^0|^2 f_\infty(z,v) \dr z \, v^i \partial_{v^i} \Big[ \delta (x-t\widehat{v}) \Big], \\
\partial_{x^k} f^{\mathrm{sing}}[f_\infty](t,x,v) &=   -\frac{1}{t}\int_{\R^3_z} |v^0|^2f_\infty(z,v) \dr z \, \big(v^0\partial_{v^k}-\widehat{v}^iv_i\partial_{v^k}+\widehat{v}^i v_k \partial_{v^i} \big) \Big[ \delta (x-t\widehat{v}) \Big].
\end{align*}
\end{Lem}
In order to consistently define averages in $v$ for these quantities, we compute, by performing integration by parts in $z$ as in \eqref{eq:eq:model}--\eqref{eq:model2}, 
\begin{align*}
\int_{\R^3_z}\partial_{v^i}\big( v^i|v^0|^2 f_\infty \big) \dr z &= \int_{\R^3_z}\widehat{\Omega}^\infty_{0i}\big( v^iv^0 f_\infty \big) \dr z, \\
 \int_{\R^3_z}\partial_{v^k} \big( |v^0|^3f_\infty \big)-\partial_{v^k}\big(  \widehat{v}^iv_i |v^0|^2 f_\infty \big)+ \partial_{v^i}  \big(  \widehat{v}^iv_k |v^0|^2 f_\infty \big) \dr z & = \int_{\R^3_z} \widehat{\Omega}^\infty_{0k}\big( |v^0|^2 f_\infty \big)+ \widehat{\Omega}_{ki} \big( v^i v^0 f_\infty \big) \dr z.
 \end{align*}
 Since we will in particular be interested in their $4$-current density, we further remark that, for $1 \leq \ell \leq 3$,
 $$ |v^0|^2v^i \partial_{v^i}  \big( \widehat{v}^\ell \big) = \widehat{v}^\ell , \qquad  |v^0|^2\big(v^0\partial_{v^k}-\widehat{v}^iv_i\partial_{v^k}+\widehat{v}^i v_k \partial_{v^i} \big)(\widehat{v}^\ell)=   \delta_k^\ell $$
 and we recall the operators $ D_{x^k} : h \mapsto \widehat{\Omega}_{0k}^\infty \big( |v^0|^2 h \big)+\widehat{\Omega}_{ki}\big( v^i v^0 h \big)$ as well as $D_t : h \mapsto -\widehat{\Omega}_{0i}^\infty \big( v^iv^0 h \big)-h$.
 \begin{Def}\label{Defcurrentderiv}
 We define the $4$-current densities
 $$ J \big( \partial_{t} f^{\mathrm{sing}}[f_\infty] \big)_\nu(t,x) = \frac{\mathds{1}_{t >|x|}}{t^4}\int_{\R^3_z}  \Big[ \widehat{v}_\nu |v^0|^5 D_t f_\infty  \Big]\bigg(z, \frac{\widecheck{ \; x \;}}{t} \bigg) \dr z  - \delta_\nu^0 \, \frac{\mathds{1}_{t >|x|}}{t^4}\int_{\R^3_z}\Big[ |v^0|^5 f_\infty \Big] \bigg(z, \frac{\widecheck{ \; x \;}}{t} \bigg) \dr z $$
 as well as, for $1 \leq k \leq 3$,
 $$ J \big( \partial_{x^k} f^{\mathrm{sing}}[f_\infty] \big)_\nu(t,x) = \frac{\mathds{1}_{t >|x|}}{t^4}\int_{\R^3_z} \Big[ \widehat{v}_\nu |v^0|^5 D_{x^k} f_\infty  \Big] \bigg( z, \frac{\widecheck{ \; x \;}}{t} \bigg) \dr z+\delta_\nu^k \, \frac{\mathds{1}_{t >|x|}}{t^4}\int_{\R^3_z}\Big[ |v^0|^5 f_\infty \Big] \bigg(z, \frac{\widecheck{ \; x \;}}{t} \bigg) \dr z . $$
 \end{Def}
 \begin{Rq}
 As $\widehat{v}_\nu=x_\nu/t$ if $v=\widecheck{x/t}$, the first term on the right hand side of both equations can be rewritten using the functional $t^{-1}J^{\mathrm{asymp}}_\nu [ \cdot ]$.
 \end{Rq}
 \begin{Rq}\label{Rqmollif}
 Since $(\psi_n)_{n \geq 1}$ is a mollifier, we have, for any $\widehat{Z} \in \K$ and all $t \in \R_+^*$,
 $$ J \big( \widehat{Z} h_n \big)_\nu \xrightarrow[n \to + \infty]{} J \big( \widehat{Z} f^{\mathrm{sing}}[f_\infty] \big)_\nu  \qquad \qquad \text{in $ L^2(\R^3_x) $}.$$
 \end{Rq}

Thus, according to Corollaries \ref{Corlinbound} and \ref{CorlinexpanderivativJ}, for any $\widehat{Z} \in \K$, $J(\widehat{Z} f^{\mathrm{sing}}[f_\infty])$ is the leading order term in the asymptotic expansion of $J(f_1)$.

\subsection{Commutation of the system} The results of this section suggest that the derivatives of $F^{\mathrm{asymp}}[f_\infty]$ are a good approximation, near timelike infinity, that is along timelike straight lines $t \mapsto (t,t\widehat{v})$, of the derivatives of the electromagnetic field $\mathcal{L}_{Z^\gamma}F$ that we will construct during the proof of Theorem \ref{Theo1}. 

We start by computing the derivatives of the current density $J^{\mathrm{asymp}}[f_\infty]$.
\begin{Lem}\label{Lemfirstasymp0}
Let $h : \R^3_x \times  \R^3_v \to \R$ be a sufficiently regular function. Then, we have
\begin{alignat*}{2}
 \mathcal{L}_{\Omega_{0k}} \big( J^{\mathrm{asymp}}[h] \big)&=J^{\mathrm{asymp}} \big[ \widehat{\Omega}_{0k}^\infty h \big] , \qquad \qquad \qquad \qquad &&1 \leq k \leq 3, \\ \mathcal{L}_{\Omega_{ij}} \big( J^{\mathrm{asymp}}[h] \big)&=J^{\mathrm{asymp}} \big[ \widehat{\Omega}_{ij} h \big] , \qquad \qquad \qquad \qquad &&1 \leq i < j \leq 3, \\
  \mathcal{L}_{S} ( J^{\mathrm{asymp}}[h] )&=J^{\mathrm{asymp}} \big[\widehat{S}h \big]-2J^{\mathrm{asymp}}[h].&& 
 \end{alignat*}
 For the translations $\partial_{x^\lambda}$, $0 \leq \lambda \leq 3$, there holds
$$
  \mathcal{L}_{\partial_{x^\lambda}} \big( J^{\mathrm{asymp}}[h] \big)_\nu = t^{-1}J^{\mathrm{asymp}}_\nu \big[D_{x^\lambda} h \big]-(-1)^{\delta_0^\lambda} \, \delta_\nu^\lambda \,t^{-1} J^{\mathrm{asymp}}_0[h],
  $$
  where the operators $D_{x^\lambda}$ are defined in \eqref{defDop}. Furthermore, $J^{\mathrm{asymp}}[h]$ is divergence free.
 \end{Lem}
 \begin{Rq}
Note that $J^{\mathrm{asymp}}[\partial^\infty_t h]=J^{\mathrm{asymp}}[\partial_{x^k} h]=0$ reflects that $\int_v \partial_{x^\lambda} f_1 \dr v$ decays faster than $t^{-3}$. 
 \end{Rq}
 \begin{proof}
 We apply, for any $n \in \mathbb{N}^*$, Lemma \ref{LemCom} in order to relate $\mathcal{L}_Z J(h_n)$ to $J(\widehat{Z}h_n)$. Recall further that $J(h_n)$ is divergence free. The result is obtained, in view of Remark \ref{Rqmollif}, by letting $n \to +\infty$ and by applying Lemmata \ref{Lemderivsing} and \ref{Lemsingderiv}.
 \end{proof}
 
We are now able to state the higher order commutation formula for both the Maxwell and the asymptotic Maxwell equations. In order to simplify the presentation, we will only consider operators of the form $Z^\gamma = \partial_{t,x}^\kappa Z^\beta$, where $Z^\beta$ is only composed by homogeneous vector fields, $|\beta|=\beta_H$. This reduction is possible in view of the next identity.
\begin{Lem}
Let $Z^\gamma \in \mathbb{K}^{|\gamma|}$ containng $\gamma_T$ translation $\partial_{x^\nu}$. There exists $C^\gamma_{\kappa, \beta} \in \mathbb{Z}$ such that
$$ Z^\gamma = \sum_{|\kappa|=\gamma_T} \sum_{|\beta| \leq |\gamma|-\gamma_T} C^\gamma_{\kappa , \beta} \, \partial_{t,x}^\kappa Z^\beta .$$
\end{Lem}
\begin{proof}
It suffices to use that, for all $0 \leq \mu \leq 3$ and $Z \in \mathbb{K}$, we have $[\partial_{x^\mu},Z]=0$ or there exists $0 \leq \nu \leq 3$ such that $[\partial_{x^\mu},Z]=\pm \partial_{x^\nu}$.
\end{proof}
\begin{Pro}\label{ProComMax}
Let $f : [T,+\infty[ \times \R^3_x \times \R^3_v \to \R$ and $F$ be a solution to the Maxwell equations
$$ \nabla^\mu F_{\mu \nu} = J(f)_\nu, \qquad \qquad \nabla^\mu {}^* \! F_{\mu \nu}=0 .$$
Assume that they are both sufficiently regular and consider $Z^\gamma \in \mathbb{K}^{|\gamma|}$.
\begin{itemize}
\item If $Z^\gamma$ is only composed by homogeneous vector fields $\Omega_{0k}$, $\Omega_{ij}$ and $S$, then
\begin{alignat*}{2}
 \nabla^\mu \mathcal{L}_{Z^\gamma}(F)_{\mu \nu}&=J \big( \widehat{Z}^\gamma f \big)_\nu , \qquad && \qquad \qquad \, \nabla^\mu {}^* \! \mathcal{L}_{Z^\gamma}(F)_{\mu \nu} =0, \\
\nabla^\mu \mathcal{L}_{Z^\gamma} \big( F^{\mathrm{asymp}}[f_\infty] \big)_{\mu \nu}&=J^{\mathrm{asymp}}_\nu \big[ \widehat{Z}^\gamma_\infty f_\infty \big]  , \qquad \qquad \qquad && \nabla^\mu {}^* \! \mathcal{L}_{Z^\gamma}\big(F^{\mathrm{asymp}}[f_\infty] \big)_{\mu \nu} =0.
\end{alignat*}
 Moreover, if $Z^\gamma$ contains the scaling vector field $S$, then $J^{\mathrm{asymp}} \big[ \widehat{Z}^\gamma_\infty f_\infty \big]=0$.
\item If $Z^\gamma = \partial_{t,x}^\kappa Z^\beta$, with $|\kappa|=\gamma_T \geq 1$, then
$$ \nabla^\mu \mathcal{L}_{Z^\gamma}(F)_{\mu \nu}= J \big( \partial_{t,x}^\kappa \widehat{Z}^\beta f \big)_\nu, \qquad \qquad \nabla^\mu {}^* \! \mathcal{L}_{Z^\gamma}(F)_{\mu \nu} =0,$$
and there exists $C^{\gamma,\lambda}_{\kappa,\beta,\nu} \in \mathbb{N}$ such that
\begin{align*}
 \nabla^\mu \mathcal{L}_{Z^\gamma} \big( F^{\mathrm{asymp}} [f_\infty] \big)_{\mu \nu}&=\frac{1}{t^{|\kappa|}}J^{\mathrm{asymp}}_\nu \big[ D_{t,x}^\kappa \widehat{Z}^\beta_\infty f \big]+\sum_{0 \leq \lambda\leq 3} \sum_{|\xi|<p} \sum_{|\beta| \leq |\gamma|-|\xi|} C^{\gamma,\lambda}_{\kappa,\beta , \nu} \, \frac{1}{t^{|\kappa|}} J^{\mathrm{asymp}}_\lambda \big[ D_{t,x}^\xi \widehat{Z}^\beta_\infty f \big] , \\
 \nabla^\mu {}^* \! \mathcal{L}_{Z^\gamma}\big(F^{\mathrm{asymp}}[f_\infty] \big)_{\mu \nu} &=0.
 \end{align*}
\end{itemize}
\end{Pro}
\begin{Rq}
In view of Corollary \ref{Corlinbound}, if $Z^\gamma$ is only composed by homogeneous vector fields, it suggests that the asymptotic behavior of $\mathcal{L}_{Z^\gamma}F$ is captured by $\mathcal{L}_{Z^\gamma}F^{\mathrm{asymp}}[f_\infty]$ if $f$ satisfies modified scattering to the density function $f_\infty$.

In fact, the same observation holds when $Z^\gamma$ is composed by $p \geq 1$ translations but we will not require it for the proof of Theorem \ref{Theo1}. One can easily check it for the case $p=1$ by computing the coefficients $C^{\gamma,\lambda}_{\kappa,\beta , \nu}$ using Lemma \ref{Lemfirstasymp0} and by comparing the commutation formulas with Corollary \ref{CorlinexpanderivativJ}. 
\end{Rq}
\begin{proof}
These relations are obtained by iterating Lemmata \ref{LemComMaxwell}, \ref{LemCom} and \ref{Lemfirstasymp0}. If $Z^\gamma_\infty$ is only composed by homogeneous vector field and contains at least once $S$, there exists $|\beta|=|\gamma|-1$ such that $\widehat{Z}_\infty^\gamma=\widehat{S} \widehat{Z}^\beta_\infty$. Indeed, $[S,\widehat{\Omega}_{0k}^\infty]=[S,\widehat{\Omega}_{ij}]=0$. Then, by integration by parts in $z$,
$$ \int_{\R^3_z}  \widehat{S} \widehat{Z}^\beta_\infty f_\infty (z,\cdot) \dr z= \int_{\R^3_z}  z^i \partial_{z^i} \widehat{Z}^\beta_\infty f_\infty (z,\cdot) \dr z+3\int_{\R^3_z}  \widehat{Z}^\beta_\infty f_\infty (z,\cdot) \dr z =0.$$

\end{proof}

\subsection{Estimates for the solution to the asymptotic Maxwell equations}

This subsection is devoted to the study of the large time behavior of $F^{\mathrm{asymp}}[f_\infty]$ and its derivatives. For this, it will be convenient to introduce, for any $N_0 \in \mathbb{N}$,
\begin{equation*}
\hspace{-5mm} Q:= \int_{\R^3_z} \int_{\R^3_v} f_\infty(z,v) \dr v \dr z , \qquad \overline{\mathbb{E}}_{N_0}[f_\infty]:= \sup_{|\kappa_z|+|\kappa_v| \leq N_0}  \, \sup_{ \R^3_z \times \R^3_v}  \langle x \rangle^{4+|\kappa_x|} \, \langle v \rangle^{7+3|\kappa_v|} \big|  \partial_z^{\kappa_z} \partial_v^{\kappa_v}f_\infty \big|(z,v) .
 \end{equation*}
One needs first to estimate the initial data $F^{\mathrm{asymp}}[f_\infty](t_0,\cdot)$.
\begin{Pro}\label{Proinidataasymp}
Recall that $\overline{F}(t,x)= \overline{\chi}(t-|x|)\frac{Qx_i}{4\pi|x|^3} \dr t \wedge \dr x^i$. Then, for any $|\gamma| \leq N_0$, 
 $$ \forall \, |\kappa| \leq 1, \qquad \sup_{\R^3_x} \, \langle x \rangle^{3+|\kappa|}  \left|\nabla_{t,x}^\kappa \mathcal{L}_{Z^\gamma}\big( F^{\mathrm{asymp}}[f_\infty]-\overline{F}) \right|\!(t_0,x) \lesssim_{t_0}  \overline{\mathbb{E}}_{N_0+1}[f_\infty].$$
 For any $|\gamma| \leq N_0+1$, we have
 $$ \forall \, |\kappa| \leq 1, \qquad \bigg|\int_{\R^3_x}\langle x \rangle^{\frac{5}{2}+2|\kappa|}  \left|\nabla_{t,x}^\kappa \mathcal{L}_{Z^\gamma}\big( F^{\mathrm{asymp}}[f_\infty]-\overline{F}) \right|^2\!(t_0,x) \dr x \bigg|^{\frac{1}{2}} \lesssim_{t_0}  \overline{\mathbb{E}}_{N_0+1}[f_\infty].$$
\end{Pro}
\begin{proof}
Recall that $J^{\mathrm{asymp}}[f_\infty](t_0, \cdot)$ is supported in $\{|x| \leq t_0 \}$. Remark that
$$ -\int_{\R^3_x} J^{\mathrm{asymp}}_0(t_0,x) \dr x=\int_{|x|<t_0} \int_{\R^3_z} \bigg\langle \frac{\widecheck{\;x \;}}{t_0} \bigg\rangle^5 f_\infty \bigg( z,\frac{\widecheck{\;x\;}}{t_0} \bigg) \frac{\dr x}{t_0^3} = Q,$$
where we used Lemma \ref{cdv} in order to perform the change of variables $x=t_0\widehat{v}$. Note also the bound
$$ \langle x \rangle^p\left| J^{\mathrm{asymp}} [h](t_0, x) \right| \leq \mathds{1}_{|x| <t_0}\frac{\langle t_0 \rangle^p}{t_0^3} \int_{\R^3_z} \frac{\dr z}{\langle z \rangle^4} \, \sup_{(z,v) \in \R^3_z \times \R^3_v} \langle z \rangle^4 \, \langle v \rangle^5 |h(z,v)| \lesssim_{t_0}  \overline{\mathbb{E}}_0[h] \mathds{1}_{|x| <t_0}, \qquad \quad p \in \mathbb{N},$$
which holds for any function $h \in L^\infty (\R^3_z \times \R^3_v)$ decaying fast enough. The result is now a direct consequence of Propositions \ref{Proinidata} and \ref{ProComMax}.
\end{proof}

We now state our main result concerning $F^{\mathrm{asymp}}[f_\infty]$. For this, we recall $\mathbb{F}[\cdot]$ introduced in \eqref{kev:defasympelec}--\eqref{kev:defasympmag} (see also Definition \ref{DefFasympelectromagn} below).
\begin{Pro}\label{Proasymp}
Let $N_0 \in \mathbb{N}$, $Z^\gamma \in \mathbb{K}^{|\gamma|}$ be a differential operator of order $|\gamma| \leq N_0$ and assume $1 \leq t_0 \leq 2$.
\begin{itemize}
\item Denoting by $0 \leq \gamma_T \leq |\gamma|$ the number of translations $\partial_{x^\lambda}$ composing $Z^\gamma$, we have
$$ \forall \, (t,x) \in [t_0,+\infty[ \times \R^3, \qquad \left| \mathcal{L}_{Z^\gamma} F^{\mathrm{asymp}}[f_\infty] \right|(t,x) \lesssim \langle t+|x| \rangle^{-1} \, \langle t-|x| \rangle^{-1-\gamma_T} \, \overline{\mathbb{E}}_{N_0+1}[f_\infty] .$$
\item For all $(t,x) \in [t_0,+\infty[ \times \R^3$ such that $|x| <t$, there holds
$$ \left| \mathcal{L}_{Z^\gamma}\big( F^{\mathrm{asymp}}[f_\infty] \big)(t,x) - \frac{1}{t^2}\mathbb{F}\big[ \widehat{Z}_\infty^\gamma f_\infty \big] \bigg( \frac{\widecheck{ \; x \;}}{t} \bigg)  \right| \lesssim \frac{\overline{E}_{N_0+1}\big[|v^0|^2f_\infty \big]}{\langle t+|x| \rangle \, \langle t-|x| \rangle^2}.$$
Finally, in the exterior of the light cone, there holds
$$ \forall \, |x| \geq t \geq t_0, \qquad  \left| \mathcal{L}_{Z^\gamma}\big( F^{\mathrm{asymp}}[f_\infty] \big)\right|(t,x) \lesssim  \frac{Q}{\langle t+|x| \rangle^2} \, \mathds{1}_{|x|-t \geq 1} + \frac{\overline{E}_{N_0+1}[f_\infty]}{\langle t+|x| \rangle \, \langle t-|x| \rangle^2}  .$$
\end{itemize}
\end{Pro}
\begin{Rq}\label{Rqscalingbetterdecay}
If $Z^\gamma$ is only composed by homogeneous vector fields, $S$, $\Omega_{0k}$ and $\Omega_{ij}$, and contains at least once the scaling vector field $S$, then we have the improved estimate
$$ \forall \, (t,x) \in \R_+ \times \R^3, \qquad \left| \mathcal{L}_{Z^\gamma} F^{\mathrm{asymp}}[f_\infty] \right|(t,x) \lesssim \langle t+|x| \rangle^{-1} \, \langle t-|x| \rangle^{-2}  \, \overline{\mathbb{E}}_{N_0+1}[f_\infty] .$$
If $\widehat{Z}^\gamma_\infty$ contains at least one translation ($\partial_t^\infty$ or $\partial_{z^k}$) or $\widehat{S}$, then $\mathbb{F}\big[ \widehat{Z}_\infty^\gamma f_\infty \big]=0$. If $\gamma_T \geq 1$, we could prove a more precise statement implying in particular that there exists $v \mapsto G^\gamma(v)$ such that $\mathcal{L}_{Z^\gamma}\big( F^{\mathrm{asymp}}[f_\infty] \big)(t,t\widehat{v})\sim t^{-2-\gamma_T} G^\gamma(v)$.
\end{Rq}
\begin{Rq}
$L^2$ estimates could be proved for $\mathcal{L}_{Z^\gamma} F^{\mathrm{asymp}}[f_\infty]$ up to order $N+1$. 
\end{Rq}

Since we would like to derive very precise informations on the asymptotic behavior of $F^{\mathrm{asymp}}[f_\infty]$, in particular along timelike straight lines $t \mapsto (t,x+t\widehat{v})$, we will exploit the wave equation satisfied by each component $F^{\mathrm{asymp}}_{\mu \nu}[f_\infty]$ in order to work with representation formula. 
\begin{Def}\label{Defdecompasymp}
For any $Z^\gamma \in \mathbb{K}^{|\gamma|}$, with $|\gamma| \leq N_0$, we consider the decomposition 
$$\mathcal{L}_{Z^\gamma}F^{\mathrm{asymp}}[f_\infty]=F^{\gamma,\mathrm{hom}}+F^{\gamma,\mathrm{inh}},$$
 where, for any $0 \leq \mu < \nu \leq 3$,
$$\Box F_{\mu\nu}^{\gamma,\mathrm{hom}} =0,  \qquad F^{\gamma, \mathrm{hom}}_{\mu \nu} (t_0, \cdot)= \mathcal{L}_{Z^\gamma}(F^{\mathrm{asymp}}[f_\infty])_{\mu \nu} (t_0, \cdot), \quad \; \partial_t F^{\gamma , \mathrm{hom}}_{\mu \nu} (t_0, \cdot) = \partial_t\mathcal{L}_{Z^\gamma}(F^{\mathrm{asymp}}[f_\infty])_{\mu \nu} (t_0, \cdot)$$
and
$$\Box F_{\mu\nu}^{\gamma,\mathrm{inh}} =\partial_{x^\nu} J^{\gamma}_\mu- \partial_{x^\mu}J^{\gamma}_\nu , \qquad \qquad F^{\gamma, \mathrm{inh}}_{\mu \nu} (t_0, \cdot)= 0, \qquad \partial_t F^{\gamma , \mathrm{hom}}_{\mu \nu} (t_0, \cdot) = 0,$$
where the source term $J^{\gamma}_\nu :=\nabla^{\lambda} \mathcal{L}_{Z^\gamma}(F^{\mathrm{asymp}} [f_\infty])_{\lambda \nu}$ is given by the commutation formula of Proposition \ref{ProComMax}. 
\end{Def}
\begin{Rq}
$F^{\gamma,\mathrm{hom}}$ is well-defined for $|\gamma| = N_0+1$ and we will see that it is the case of $F^{\gamma,\mathrm{inh}}$ as well.
\end{Rq}

\subsubsection{Control of the source term} 

We will naturally be lead to estimate several times the asymptotic electromagnetic current. We start by controlling it pointwise.

\begin{Lem}\label{LemasympJpoint}
Let $h : \R^3_x \times  \R^3_v \to \R$ be a sufficiently regular function. Then, for all $(t,x) \in \R_+^* \times \R^3$,
$$ \big| J^{\mathrm{asymp}}_{\underline{L}} \big|(t,x) \leq \frac{2}{t^3} \overline{\mathbb{E}}_0[h] \, \mathds{1}_{t >|x|}, \qquad \big| J^{\mathrm{asymp}}_{L} \big|(t,x) \leq \frac{t-|x|}{t^4} \overline{\mathbb{E}}_0[h] \, \mathds{1}_{t >|x|}, \qquad \big| J^{\mathrm{asymp}}_{e_A} \big|(t,x) =0, $$
where $A \in \{ \theta , \, \varphi \}$. Moreover, we have improved estimates near the light cone,
$$ \forall (t,x) \in \R_+^* \times \R^3, \qquad \big| J^{\mathrm{asymp}} \big|(t,x) \leq \frac{(t-|x|)^p}{t^{3+p}} \overline{\mathbb{E}}_0\big[ |v^0|^{2p}h \big] \, \mathds{1}_{t >|x|}, \qquad \qquad p \in \R_+.$$
\end{Lem}
\begin{proof}
As $J^{\mathrm{asymp}}[h](t,x)=0$ for $|x| \geq t$, we fix $t>|x|$. Then, remark that
$$ J^{\mathrm{asymp}}[h](t,x) =\frac{1}{t^3} \int_{\R^3_z} \Big[ |v^0|^5 h \Big] \bigg( z,\frac{\widecheck{\; x \;}}{t} \bigg) \dr z \, \Gamma (t,x), \qquad \Gamma (t,x):= -\dr t +\frac{x_i}{t} \dr x^i.$$
The first three inequalities then follow from
$$  \Gamma_{\underline{L}}(t,x)=-\frac{t+|x|}{t}, \qquad  \qquad \Gamma_L(t,x)=-\frac{t-|x|}{t}, \qquad \qquad \Gamma_{e_A}(t,x)=0, \quad A \in \{ \theta , \, \varphi \}.$$ 
For the last one, use that $2|v^0|^2=2t^2 (t^2-|x|^2)^{-1} \geq t(t-|x|)^{-1}$ if $v=\frac{\widecheck{\;x\;}}{t}$.
\end{proof}

We will also be interested in the conservation of the charge.

\begin{Lem}\label{Lemchargeasymp}
Let $h : \R^3_x \times  \R^3_v \to \R$ be a sufficiently regular function. Then, for all $\tau >0$,
\begin{align*}
 \int_{\R^3_x} J_0^{\mathrm{asymp}}[h](\tau,x) \dr x &=\int_{u=0}^{2\tau} \int_{\mathbb{S}_\omega^2} J_{\underline{L}}^{\mathrm{asymp}}[h]\Big(\frac{\tau+u}{2}, \frac{\tau-u}{2} \omega \Big) \frac{r^2}{2} \dr \mu_{\mathbb{S}^2_\omega} \dr u = -\int_{\R^3_x} \int_{\R^3_v} h(x,v) \dr v \dr x.
\end{align*}
\end{Lem}
\begin{proof}
Let $t>0$ and recall that $J^{\mathrm{asymp}}[f_\infty](t, \cdot)$ is supported in $\{|x| \leq t \}$. Remark that
$$ -\int_{\R^3_x} J^{\mathrm{asymp}}_0[h](t,x) \dr x=\int_{|x|<t_0} \int_{\R^3_z} \bigg\langle \frac{\widecheck{ \; x \; }}{t} \bigg\rangle^5 h \bigg( z,\frac{\widecheck{ \; x \;}}{t} \bigg) \frac{\dr x}{t^3} = \int_{\R^3_v} \int_{\R^3_z} h(z,v) \dr v \dr x,$$
where we used Lemma \ref{cdv} in order to perform the change of variables $x=t\widehat{v}$. Fix now $\tau > 0$ and recall from Lemma \ref{Lemfirstasymp0} that $J^{\mathrm{asymp}}[h]$ is divergence free. Consequently, the divergence theorem, applied in the region $\{ t+|x| \leq 2\tau, \, t\geq \tau \}$, provides
$$ \int_{|x| \leq \tau} J^{\mathrm{asymp}}_0[h](\tau,x) \dr x - \int_{u=0}^{2\tau}  \int_{\mathbb{S}_\omega^2} J_{\underline{L}}^{\mathrm{asymp}}[h]\Big(\frac{\tau+u}{2}, \frac{\tau-u}{2} \omega \Big) \frac{r^2}{2} \dr \mu_{\mathbb{S}^2_\omega} \dr u =0.$$
It remains to use again that $J^{\mathrm{asymp}}_0[h](\tau,x)=0$ for all $|x| \geq \tau$.
\end{proof}

\subsubsection{Estimates for integrals on spheres or cones} We are interested here in controlling quantities appearing in the representation formula for a solution to the wave equation with a source term and initial data decaying fast enough.

\begin{Lem}\label{estiWeiYang}
For any function $h \in L^\infty(\R^3)$ and all $(t,x) \in \R_+^* \times \R^3$,
\begin{equation*}
\int_{\mathbb{S}^2_\omega} |h|(x+t \omega) \dr \mu_{\mathbb{S}^2_\omega}  \leq \left\{ 
	\begin{array}{ll}
        12\pi \, \langle t+|x| \rangle^{-1}\, \langle t-|x|\rangle^{-2} \sup_{y \in \R^3} \, \langle y \rangle^{3} |h(y)| , \\
        8\pi  \, t^{-1} \langle t+|x| \rangle^{-1}\, \langle t-|x|\rangle^{-2} \sup_{y \in \R^3} \, \langle y \rangle^{4} |h(y)|.
    \end{array} 
\right. 
\end{equation*}
\end{Lem}
\begin{proof}
For all $(t,x) \in \R_+ \times \R^3$, we have
\begin{equation}\label{eq:boundintr}
 \forall \, \omega \in \mathbb{S}^2, \qquad |x+t\omega| \geq |t-|x||,
 \end{equation}
 so that
\begin{equation*}
\forall \, (t,x) \in \R_+^* \times \R^3, \qquad \int_{\mathbb{S}^2_\omega} |h|(x+t \omega) \dr \mu_{\mathbb{S}^2_\omega} \leq \frac{1}{\langle t-|x| \rangle}\int_{\mathbb{S}^2_\omega} |\langle \cdot \rangle \, h|(x+t \omega) \dr \mu_{\mathbb{S}^2_\omega} .
\end{equation*}
Then, apply \cite[Lemma~$4.1$]{WeiYang} to the function $y \mapsto \langle y \rangle \, h(y)$.
\end{proof}

We are now able to estimate solutions to the free wave equation. 

\begin{Lem}\label{LemforKirchpoint}
Let $K_0 \geq 0$ and $\phi : \R_+ \times \R^3 \to \R$ be a solution to $\Box \phi =0$ such that $\langle x \rangle^3 |\phi (0,x)|+\langle x \rangle^4 |\nabla_{t,x} \phi|(0,x) \leq K_0$. Then,
$$\forall \, (t,x) \in \R_+ \times \R^3, \qquad  |\phi|(t,x) \leq 3 K_0 \, \langle t+|x| \rangle^{-1} \, \langle t-|x| \rangle^{-2}. $$
\end{Lem}
\begin{proof}
The result follows from Lemma \ref{estiWeiYang} and Kirchoff's formula
\begin{equation}\label{Kirchof:eq}
 4 \pi \phi(t,x) =  \int_{\mathbb{S}^2_\omega} \phi(0,x+t\omega)\dr \mu_{\mathbb{S}^2_\omega}+t \int_{\mathbb{S}^2_\omega}\omega \cdot \nabla_x \phi(0,x+t\omega)+ \partial_t \phi(0,x+t\omega) \dr \mu_{\mathbb{S}^2_\omega}.
\end{equation}
\end{proof}

An analogous $L^2$ estimate hold.
\begin{Lem}\label{LemforKirchL2}
Let $\phi : \R_+ \times \R^3 \to \R$ be a sufficiently regular solution to $\Box \phi =0$. Then,
$$ \forall \, t \in \R_+, \qquad \int_{\R^3_x} \langle t-|x| \rangle^{\frac{5}{2}} |\phi (t,x)|^2 \dr x \lesssim \int_{\R^3_x} \langle x \rangle^{\frac{5}{2}} |\phi(0,x) |^2+\langle x \rangle^{\frac{9}{2}} |\nabla_{t,x}\phi(0,x) |^2 \dr x.$$
\end{Lem}
\begin{proof}
By Kirchoff's formula \eqref{Kirchof:eq}, \eqref{eq:boundintr} and the Cauchy-Schwarz inequality,
$$ \langle t-|x| \rangle^{\frac{5}{2}} |\phi (t,x) |^2 \! \lesssim \!\int_{\mathbb{S}^2_\omega}\langle x+t\omega \rangle^{\frac{5}{2}} | \phi (0,x+t\omega)|^2 \mathrm{d}\mu_{\mathbb{S}^2_\omega} + \int_{\mathbb{S}^2_\omega}\langle x+t\omega\rangle^{\frac{9}{2}}|\nabla_{t,x} \phi (0,x+t\omega)|^2  \mathrm{d}\mu_{\mathbb{S}^2_\omega} \int_{\mathbb{S}^2_\omega} \! \frac{\langle t-|x|\rangle^2 \, t^2\mathrm{d}\mu_{\mathbb{S}^2_\omega}}{\langle x+t\omega \rangle^4} .$$
By Lemma \ref{estiWeiYang}, the second factor of the last term in the right hand side is bounded by $8\pi$. It then remains to integrate over $\R^3_x$, to apply Fubini's theorem and to perform the change of variables $y(x)=x+t\omega$.
\end{proof}

In order to deal with solutions to the inhomogeneous wave equation, we will use the next result.

\begin{Lem}\label{Lemint}
Consider, for any $(t,x) \in \R_+ \times \R^3$, $a \geq 4$ and $b \geq 7/2$,
$$\mathcal{K}^1_{a}(t,x) := \int_{|y-x| \leq t} \frac{ \dr y}{\langle t-|y-x|+|y| \rangle^a  \, |y-x|},  \quad \qquad \mathcal{K}^2_{b}(t,x) := \int_{|y-x| \leq t}\mathds{1}_{t-|y-x| \geq |y|} \frac{\langle t-|y-x|-|y| \rangle^{\frac{1}{2}} \, \dr y}{\langle t-|y-x| \rangle^b  \, |y-x|^2}.$$
Then, we have
$$ \mathcal{K}^1_a(t,x) \lesssim \langle t+|x| \rangle^{-1} \, \langle t-|x| \rangle^{-a+3}, \qquad \qquad \mathcal{K}^{2}_b(t,x) \lesssim   \langle t \rangle^{-2} \langle t-|x|\rangle^{-b+\frac{7}{2}}  \, \mathds{1}_{t \geq |x|}.$$
\end{Lem}
\begin{proof}
Fix $(t,x) \in \R_+ \times \R^3$ and remark that
\begin{equation*}
t-|y-x|+|y| \geq t-|y|-|x|+|y|=t-|x|, \qquad t-|y-x|+|y| \geq |y|\geq |x|-|y-x| \geq |x|-t
\end{equation*}
on the domain of integration, so that $\mathcal{K}^p_{a}(t,x) \leq \langle t-|x| \rangle^{-a+4} \, \mathcal{K}^p_{4}(t,x)$. It then suffices to treat the case $a=4$ and, by continuity, $x \neq 0$. According to \cite[Lemma~$6.5.2$]{Glassey},
$$
  \mathcal{K}^{1}_{4}(t,x) =   \frac{2 \pi}{|x|} \int_{\tau=0}^t \int_{\lambda=||x|-t+\tau|}^{|x|+t-\tau}  \frac{\lambda \, \dr \lambda \dr \tau }{\langle \tau+\lambda \rangle^4} \leq  \frac{2 \pi}{|x|} \int_{\tau=0}^{t} \int_{\lambda=||x|-t+\tau|}^{|x|+t-\tau}  \frac{ \dr \lambda \dr \tau }{\langle \tau+\lambda \rangle^3}  .
$$
We perform the change of variables $\underline{u}=\tau+\lambda$ and $u=\tau-\lambda$. On the domain of integration, we have $||x|-t| \leq \underline{u}\leq t+|x|$ and $ u \leq ||x|-t|$. Since $2\tau \geq 0$, we have further $u\geq -\underline{u}$, so that
$$
 \mathcal{K}^{1}_{4}(t,x)    \leq \frac{ \pi}{|x|} \int_{\underline{u}=||x|-t|}^{t+|x|} \int_{u= -\underline{u}}^{||x|-t|}  \frac{\dr u \dr \underline{u}}{\langle \underline{u} \rangle^3}   \leq \frac{2\pi}{|x|} \int_{\underline{u}=||x|-t|}^{t+|x|}  \frac{\dr \underline{u}}{\langle \underline{u} \rangle^2} \leq \frac{4 \pi(1+t+|x|-1-|t-|x||)}{|x|(1+t+|x|)(1+|t-|x||)} .
 $$ 
Since $t+|x|-|t-|x|| =2 \min (t,|x|)$, this implies the result for $\mathcal{K}^{1}_{4}(t,x)$. Next, since $t-|y-x|-|y| \leq t-|x|$, we have $\mathcal{K}^{2}_{b}(t,x) =0$ if $|x| \geq t$. Otherwise $t >|x|$ and as previously, it suffices to treat the case $b=7/2$. We have
$$  \mathcal{K}^{2}_{\frac{7}{2}}(t,x) \leq \langle t-|x| \rangle^{\frac{1}{2}} \, \mathbf{K}_{[0,\frac{t}{2}]}+ \langle t-|x| \rangle^{\frac{1}{2}} \, \mathbf{K}_{[\frac{t}{2},t]}, \qquad \qquad \mathbf{K}_I:= \int_{|y-x| \in I }\mathds{1}_{t-|y-x| \geq |y|} \frac{ \, \dr y}{\langle t-|y-x| \rangle^{\frac{7}{2}}  \, |y-x|^2}.$$
On the domain of integration of $\mathbf{K}_{[0,\frac{t}{2}]}$, we have $t-|y-x| \geq t/2$, so that
$$
\langle t-|x| \rangle^{\frac{1}{2}} \, \mathbf{K}_{[0,\frac{t}{2}]} \leq \frac{\langle t-|x| \rangle^{\frac{1}{2}} }{2 \langle t \rangle^{\frac{7}{2}}} \int_{r=0}^{\frac{t}{2}} \dr r = \frac{t \, \langle t-|x| \rangle^{\frac{1}{2}} }{4 \langle t \rangle^{\frac{7}{2}}} \leq \frac{1}{4} \langle t\rangle^{-2}.
$$
Applying again \cite[Lemma~$6.5.2$]{Glassey}, we have
$$ \mathbf{K}_{[\frac{t}{2},t]} = \frac{2 \pi}{|x|}\int_{\tau=0}^{\frac{t}{2}} \int_{\lambda=||x|-t+\tau|}^{|x|+t-\tau} \mathds{1}_{\tau \geq \lambda} \frac{ \lambda \, \dr \lambda \dr \tau}{ \langle \tau \rangle^{\frac{7}{2}}(t-\tau)} \leq \frac{4 \pi}{|x|t}\int_{\tau=0}^{\frac{t}{2}}  \frac{  (\tau-\min(\tau,||x|-t+\tau|)) \dr \tau}{ \langle  \tau \rangle^{\frac{5}{2}}}.$$
Note that $ 0 \leq 2\tau \leq t-|x|$ implies $||x|-t+\tau| \geq \tau$. Since $t \geq |x|$, we then have, for all $0 \leq \tau \leq t/2$,
$$ \tau-\min(\tau,||x|-t+\tau|) \leq \min ( t-|x|,2\tau-t+|x|) \leq \min ( t-|x|,|x|).$$
We then deduce
$$ \langle t-|x| \rangle^{\frac{1}{2}} \, \mathbf{K}_{[\frac{t}{2},t]} \lesssim \frac{ \langle t-|x| \rangle^{\frac{1}{2}} \,  \min(t-|x|,|x|)}{|x|t}\int_{\tau=\frac{t-|x|}{2}}^{\frac{t}{2}}  \frac{   \dr \tau}{ \langle  \tau \rangle^{\frac{5}{2}}} \lesssim \frac{\langle t-|x| \rangle^{\frac{1}{2}} \,\min(t-|x|,|x|)}{|x|\langle t \rangle \,  \langle t-|x| \rangle^{\frac{3}{2}}} \lesssim \frac{1}{\langle t \rangle^2}.$$
\end{proof}

\subsubsection{Study of the homogeneous term}\label{subsubsechompar}

In view of the initial decay of $F^{\gamma, \mathrm{hom}}$, an adaptation of Lemma \ref{LemforKirchpoint} would give $\langle t+|x|\rangle |F^{\gamma,\mathrm{hom}}|(t,x) \lesssim  \langle t-|x| \rangle^{-1}$ (see for instance \cite[Proposition~$2.21$]{scat}). However,
\begin{itemize}
\item this estimate merely provides the decay rate $\langle u \rangle^{-1}$ for the radiation field of $F^{\gamma,\mathrm{hom}}$. In fact, we will prove that it decays at least as $\langle u \rangle^{-2}$.
\item Moreover, it does not allow us to prove that $F^{\gamma,\mathrm{hom}}$ does not contribute to the asymptotic electromagnetic field $\mathbb{F}[f_\infty]$. This can only hold if $t^2F^{\gamma,\mathrm{hom}}$ goes to $0$ along timelike straight lines.
\end{itemize}
Our goal is to improve this naive estimate. The idea consists in isolating the leading order term in the asymptotic expansion of $F^{\gamma,\mathrm{hom}}(t_0,\cdot)$, which is $\frac{Qx_i}{4\pi|x|^3} \dr t \wedge \dr x^i$ when $|\gamma|=0$. We address these issues by exploiting the solution to the homogeneous wave equation $\widetilde{F}$, introduced in Section \ref{SubsecPurecharge} and estimated in Proposition \ref{Propurechargetilde}. Indeed,
\begin{itemize}
\item $F^{\gamma,\mathrm{hom}}-\mathcal{L}_{Z^\gamma} \widetilde{F}$ enjoys stronger decay properties than $F^{\gamma,\mathrm{hom}}$.
\item $r|\mathcal{L}_{Z^\gamma}  \widetilde{F}(r+u,r\omega)| \lesssim \epsilon \, \mathds{1}_{-1 \leq u \leq -2}+O(r^{-1})$, so that the radiation field of $\mathcal{L}_{Z^\gamma}  \widetilde{F}$, which can in fact be explicitly computed, is then compactly supported.
\item $\mathcal{L}_{Z^\gamma}  \widetilde{F}$ vanishes in the interior of the light cone, so that $t^2 \mathcal{L}_{Z^\gamma}  \widetilde{F}(t,x+t\widehat{v})=0$ for $t$ large enough. 
\end{itemize}
Let us then prove stronger decay estimates for $F^{\gamma, \mathrm{hom}}$.
\begin{Pro}\label{Prohompart}
We have, for any $|\gamma| \leq N_0$ and all $(t,x) \in [t_0,+ \infty [ \times \R^3$,
$$ \left| F^{\gamma, \mathrm{hom}} \right|(t,x) \lesssim_{t_0}   \langle t+|x| \rangle^{-2} \, \mathds{1}_{|x|-t \geq 1} \, Q+ \langle t+|x| \rangle^{-1} \, \langle t-|x| \rangle^{-2}  \, \overline{E}_{N_0+1}[f_\infty] .$$
\end{Pro}
\begin{Rq}\label{Rqhompart}
We will use once, in Section \ref{Subsecnullproasymp} that, if $|\gamma| \leq N_0-1$ and $Z \in \mathbb{K}$, 
$$ \forall \, t \geq |x|+2, \qquad  \left|\mathcal{L}_{Z}  F^{\gamma, \mathrm{hom}} \right|(t,x) \lesssim_{t_0}   \langle t+|x| \rangle^{-1} \, \langle t-|x| \rangle^{-2}  \, \overline{E}_{N_0+1}[f_\infty] .$$
\end{Rq}
\begin{proof}
Note that $\Box(F^{\gamma , \mathrm{hom}}_{\mu \nu}-\mathcal{L}_{Z^\gamma}(\widetilde{F})_{\mu \nu})=0$ by Proposition \ref{Propurecharge}. Next, according to Proposition \ref{Proinidataasymp},
$$ \sup_{|\kappa| \leq N_0+1-|\gamma|} \, \sup_{x \in \R^3} \langle x \rangle^{3+|\kappa|} \big|\nabla_{t,x}^\kappa \big( F^{\gamma , \mathrm{hom}}_{\mu \nu}-\mathcal{L}_{Z^\gamma}(\widetilde{F})_{\mu \nu} \big) \big|(t_0,x) \lesssim_{t_0}  Q+\overline{\mathbb{E}}_{N_0+1}[f_\infty]  \lesssim_{t_0} \overline{\mathbb{E}}_{N_0+1}[f_\infty],$$
where the higher order time derivatives are controlled by exploiting the homogeneous wave equation. Applying Lemma \ref{LemforKirchpoint}, allowing us to estimate solutions to $\Box \phi =0$, we get
\begin{equation}\label{eq:estihomFtilde}
\forall \, (t,x) \in [t_0,+ \infty [ \times \R^3, \qquad \big| F^{\gamma, \mathrm{hom}}-\mathcal{L}_{Z^\gamma}(\widetilde{F}) \big|(t,x) \lesssim_{t_0}    \langle t+|x| \rangle^{-1} \, \langle t-|x| \rangle^{-2}  \, \overline{E}_{N_0+1}[f_\infty] .
\end{equation}
It remains to apply Propositions \ref{Propurecharge} and \ref{Propurechargetilde}. For the statement of Remark \ref{Rqhompart}, use that for any $2$-form $H$ and any $Z  \in \mathbb{K}$,
$$ |\mathcal{L}_{Z}(H)_{\mu \nu}| \lesssim \sup_{0 \leq \mu < \nu \leq 3}|Z(H)_{\mu \nu}|+|H_{\mu \nu}|, \qquad [\Box, Z]=2\delta_{Z}^S \Box, \qquad |Z H_{\mu \nu}| \lesssim \langle x \rangle |\nabla_{t,x} H|$$
and apply Lemma \ref{LemforKirchpoint} to $\mathcal{L}_Z(F^{\gamma , \mathrm{hom}}-\mathcal{L}_{Z^\gamma}(\widetilde{F}))_{\mu \nu}$. The result ensues from Propositions \ref{Propurecharge} and \ref{Propurechargetilde}.
\end{proof}

We need a refined estimate in order to prove Remark \ref{Rqscalingbetterdecay}.
\begin{Pro}\label{ProhompartS}
Let $Z^\gamma \in \mathbb{K}^{|\gamma|}$, $|\gamma| \leq N_0$, composed only by homogenous vector fields and containing at least once the scaling $S$. Then,
$$ \forall \, (t,x) \in [t_0,+ \infty [ \times \R^3, \qquad  \left| F^{\gamma, \mathrm{hom}} \right|(t,x) \lesssim_{t_0}    \langle t+|x| \rangle^{-1} \, \langle t-|x| \rangle^{-2}  \, \overline{E}_{N_0+1}[f_\infty] .$$
\end{Pro}
\begin{proof}
Since $[S, \Omega_{\lambda k}]=0$ for any $0 \leq \lambda < k \leq 3$, $\mathcal{L}_{Z^\gamma}\widetilde{F}$ is in that case supported in $\{|t-|x|| \leq 2\}$ according to Propositions \ref{Propurecharge} and \ref{Propurechargetilde}. Thus, $|\mathcal{L}_{Z^\gamma}\widetilde{F}|(t,x) \lesssim Q \langle t+|x| \rangle^{-1} \mathds{1}_{|t-|x|| \leq 2}$ and it remains to use \eqref{eq:estihomFtilde}.
\end{proof}

Finally, we apply Proposition \ref{Proinidataasymp} together with Lemma \ref{LemforKirchL2} in order to derive an $L^2$ estimate.
\begin{Pro}\label{ProestihomL2}
For any $|\gamma| \leq N_0+1$,
$$ \forall \, t \in \R_+, \qquad \bigg|\int_{\R^3_x} \langle t-|x| \rangle^{\frac{5}{2}} \big| F^{\gamma, \mathrm{hom}}-\mathcal{L}_{Z^\gamma} \widetilde{F} \big|^2(t,x) \dr x \bigg|^{\frac{1}{2}} \lesssim \overline{\mathbb{E}}_{N_0+1} [f_\infty].$$
\end{Pro}

\subsubsection{Naive estimate of the inhomogeneous term} We start by a computation which will allow us to rewrite the source term of the wave equation satisfied by $F^{\gamma, \mathrm{inh}}$.
\begin{Lem}\label{Lemsource}
Let $h : \R_z \times \R^3_v \to \R$ a sufficiently regular function and $1 \leq i < j \leq 3$, $1 \leq k \leq 3$. Then,
\begin{align*}
 \partial_{x^j} J^{\mathrm{asymp}}_i \big[ h \big]- \partial_{x^i}J^{\mathrm{asymp}}_j \big[ h \big] &= \frac{1}{t^4}  \int_{\R^3_v}\Big[ |v^0|^5  \widehat{\Omega}_{ij} h \Big] \bigg( z,\frac{\widecheck{\; x \;}}{t} \bigg) \dr z, \\
  \partial_{x^k} J^{\mathrm{asymp}}_0 \big[ h \big]- \partial_{t}J^{\mathrm{asymp}}_k \big[ h \big] &= -\frac{1}{t^4}  \int_{\R^3_v} \Big[ |v^0|^5 \widehat{\Omega}_{0k}^\infty h \Big] \bigg( z,\frac{\widecheck{\; x \;}}{t} \bigg) \dr z.
  \end{align*}
\end{Lem}
\begin{proof}
Recall that $J^{\mathrm{asymp}}_\ell [h](t,x)=-\frac{x_\ell}{t} J^{\mathrm{asymp}}_0 [h](t,x)$. 
\begin{itemize}
\item As $\partial_{x^i}(x_j /t)=\partial_{x^j}(x_j /t)=0$ and $\Omega_{ij}=x^i \partial_{x^j}-x^j \partial_{x^i}$,
$$ \partial_{x^j} J^{\mathrm{asymp}}_i \big[ h \big]- \partial_{x^i}J^{\mathrm{asymp}}_j \big[ h \big] = - t^{-1} \, \Omega_{i j}  \big( J^{\mathrm{asymp}}_0 [h] \big).$$
\item Since $\partial_t (x_k/t)=-x_k/t^2$ and $\Omega_{0k}=t \partial_{x^k}+x^k \partial_t$, one gets
$$ \partial_{x^k} J^{\mathrm{asymp}}_0 \big[ h \big]- \partial_{t} J^{\mathrm{asymp}}_k \big[ h \big] = \frac{1}{t}\Omega_{0k} \big( J^{\mathrm{asymp}}_0 \big[ h \big] \big) -\frac{ x_k }{t^2}J^{\mathrm{asymp}}_0 \big[ h \big]=  \frac{1}{t}\Omega_{0k} \big( J^{\mathrm{asymp}}_0 \big[ h \big] \big) +\frac{ 1 }{t}J^{\mathrm{asymp}}_k \big[ h \big].    $$
\end{itemize}
The result follows from Lemma \ref{Lemfirstasymp0} and $\mathcal{L}_Z(J)_0=Z(J_0)+\partial_t Z^\mu J_\mu$, for any $1$-form $J$ and vector field $Z$.
\end{proof}

We are now able to derive a general estimate, which will allow us, together with Proposition \ref{Prohompart}, to prove the first part of Proposition \ref{Proasymp}.
\begin{Pro}\label{Proinhpart0}
Let $Z^\gamma \in \mathbb{K}^{|\gamma|}$ with $|\gamma| \leq N_0$. We have,
$$ \forall \, (t,x) \in [t_0 , +\infty[ \times \R^3, \qquad |F^{\gamma, \mathrm{inh}}|(t,x) \lesssim_{t_0} \langle t+|x| \rangle^{-1} \, \langle t-|x| \rangle^{-1-\gamma_T} \overline{\mathbb{E}}_{N_0+1}[f_\infty].$$
If $Z^\gamma$ is only composed by homogeneous vector fields and contains at least once the scaling $S$, then $F^{\gamma, \mathrm{inh}}=0$.
\end{Pro}
\begin{proof}
Assume first that $\gamma_T=0$. Combining the commutation formula of Proposition \ref{ProComMax} with the previous Lemma \ref{Lemsource}, one has
$$ \left| \Box F^{\gamma , \mathrm{inh}} \right|(t,x) \lesssim_{t_0} \frac{1}{t^{4}} \sup_{|\beta| = |\gamma|+1} \int_{\R^3_v}\Big[ |v^0|^5  \widehat{Z}_{\infty}^\beta h \Big] \bigg( z,\frac{\widecheck{\; x \;}}{t} \bigg) \dr z \lesssim_{t_0} \frac{1}{\langle t+|x|\rangle^{4}} \overline{\mathbb{E}}_{|\gamma|+1}[f_\infty] ,$$
for all $t \in [t_0,+\infty [$ and $|x| < t$. Since $\Box F^{\gamma,\mathrm{inh}}$ is supported in the interior of the light cone $|x| \leq t$, the estimate is then provided by Lemma \ref{Lemint}. The last part of the lemma is clear since $\Box F^{\gamma , \mathrm{inh}}=0$ in that case according to Proposition \ref{ProComMax}. If $\gamma_T \geq 1$, then we have similarly
$$ \left| \Box F^{\gamma , \mathrm{inh}} \right|(t,x) \lesssim_{t_0} \frac{1}{t^{4+\gamma_T}} \sup_{|\kappa| = \gamma_T+1} \, \sup_{|\beta| \leq |\gamma| - \gamma_T} \int_{\R^3_v}\Big[ |v^0|^5 D^\kappa_{t,x} \widehat{Z}_{\infty}^\beta h \Big] \bigg( z,\frac{\widecheck{\; x \;}}{t} \bigg) \dr z \lesssim_{t_0} \frac{1}{\langle t+|x|\rangle^{4+\gamma_T}} \overline{\mathbb{E}}_{|\gamma|+1}[f_\infty] $$
and we apply again Lemma \ref{Lemint}.
\end{proof}

\subsubsection{Gain of regularity} We now prove a more precise estimate for $F^{\gamma, \mathrm{inh}}$ in the case where either $|\gamma|=0$ or $Z^\gamma$ is only composed by Lorentz boosts $\Omega_{0k}$ and rotational vector fields $\Omega_{ij}$. The representation formula for the solution to the wave equation provides an explicit expression for $F^{\gamma , \mathrm{inh}}$. According to the commutation formula of Proposition \ref{ProComMax} and Lemma \ref{Lemsource}, we have, for all $(t,x) \in [t_0,+\infty[ \times \R^3$,
$$ F^{\gamma , \mathrm{inh}}_{\mu \nu} (t,x)= \mathfrak{s}_{\mu \nu}\int_{|y-x| \leq t-t_0}  \int_{\R^3_z} \Big[|v^0|^5 \Omega^\infty_{\mu \nu} \widehat{Z}_\infty^\gamma f_\infty \Big] \bigg( z,\frac{\widecheck{\; \; \; y \; \;\;}}{t-|y-x|} \bigg) \dr z \frac{\dr y}{(t-|y-x|)^4|y-x|}    , $$
for any $0 \leq \mu < \nu \leq 3$, where $\mathfrak{s}_{\mu \nu}=1$ if $\mu \neq 0$ and $\mathfrak{s}_{0 \nu}=-1$. This formula is not satisfying since we would like to express $F^{\gamma,\mathrm{inh}}$ as a functional of $\widehat{Z}^\gamma_\infty f_\infty$ instead of a functional of its derivatives. Our strategy relies on the Glassey-Strauss decomposition of the electromagnetic field, \cite[Theorem~$3$]{GlStrauss}. In order to fit into the framework of Glassey-Strauss, we introduce the following integral kernels.
\begin{Def}\label{DefKernel} 
Let, for all $(\omega,v) \in \mathbb{S}^2 \times \R^3_v$, $\mathbf{w}^{\mathrm{T}}_{\mu\nu}(\omega,v)$ be the antisymmetric valued matrix defined by
\begin{equation*}
\mathbf{w}^{\mathrm{T}}_{0k}(\omega,v)=-\mathbf{w}^{\mathrm{T}}_{k0}(\omega,v):= \frac{\omega_k+\widehat{v}_k}{|v^0|^2(1+\omega \cdot \widehat{v})^2}, \qquad \mathbf{w}^{\mathrm{T}}_{ij}(\omega,v)  := \frac{\omega_i\widehat{v}_j-\omega_j\widehat{v}_i}{|v^0|^2(1+\omega \cdot \widehat{v})^2}, \qquad 1 \leq i, \, j, \, k\leq 3,
\end{equation*}
where $\omega_i:=x_i/|x|$ if $x \in \R^3 \setminus \{ 0 \}$ satisfies $\omega=x/|x|$. Let further, for any $1 \leq \ell \leq 3$,
$$   \mathbf{a}^\ell_{0k}(\omega, z)=- \mathbf{a}^\ell_{k0}(\omega, z) := \delta^\ell_k-\frac{\omega_k+\widehat{v}_k}{1+\omega \cdot z}\widehat{v}^\ell, \qquad \mathbf{a}^\ell_{ij}(\omega, v) := \widehat{v}_j \, \delta^\ell_i-\widehat{v}_i \, \delta^\ell_j+\frac{\widehat{v}_i \omega_j- \widehat{v}_j\omega_i}{1+\omega \cdot \widehat{v}} \widehat{v}^\ell, \qquad \mathbf{a}^\ell_{00}(\omega, v):=0.$$
\end{Def}
\begin{Rq}
For any $\mu \in \llbracket 0 , 3 \rrbracket$, $\mathbf{w}^{\mathrm{T}}_{\mu\mu}=0$ since $\mathbf{w}^{\mathrm{T}}$ is antisymmetric. Note also $1+\omega \cdot \widehat{v} =2v^L>0$.
\end{Rq}
The only informations that we will need on these functions are contained in the next result.
\begin{Lem}\label{Lemkernel}
For all $(\omega, v) \in \mathbb{S}^2 \times \R^3_v$ and any $0 \leq \mu, \, \nu \leq 3$, $|\mathbf{w}^{\mathrm{T}}_{\mu\nu}(\omega,v)| \leq 4\langle v \rangle$ and $|\mathbf{a}^\ell_{\mu \nu}(\omega,v)| \leq 4\langle v \rangle$.
\end{Lem}
\begin{proof}
We refer to the first two inequalities of \cite[Corollary~$5.5$]{scat}.
\end{proof}
In order to lighten the presentation of the next result, we will use the convention that $\widehat{Z}^\gamma_\infty f_\infty(z,\widecheck{u})=0$ if $|u| \geq 1$.
\begin{Pro}\label{GSasymp}
Let $Z^\gamma \in \mathbb{K}^{|\gamma|}$, $|\gamma| \leq N_0$, be only composed by homogeneous vector fields, $\gamma_T=0$. Then, 
$$ \forall \, (t,x) \in [t_0,+\infty[ \times \R^3, \qquad F^{\gamma , \mathrm{inh}}_{\mu \nu}(t,x;t_0) = F^{\gamma , \mathrm{T}}_{\mu \nu}(t,x;t_0)-F^{\gamma , \mathrm{sph}}_{\mu \nu}(t,x;t_0), $$
where, for any $0 \leq \mu , \, \nu \leq 3$, 
\begin{itemize}
\item $F^{\gamma,\mathrm{T}}=F^{\gamma, \mathrm{T},[0,t-t_0]}$ and, for any $I \subset [0,t]$,
\begin{align*}
& F^{\gamma , \mathrm{T},I}_{\mu \nu}(t,x;t_0)\\
  & \qquad \; : = \int_{|y-x| \in I}\mathbf{w}^{\mathrm{T}}_{\mu \nu}\bigg(\frac{y-x}{|y-x|}, \frac{\widecheck{ \qquad y \qquad}}{t-|y-x|} \bigg) \int_{\R^3_z} \Big[  |v^0|^5 \widehat{Z}_\infty^\gamma f_\infty \Big] \bigg(z, \frac{\widecheck{ \qquad y \qquad}}{t-|y-x|} \bigg) \dr z\frac{\dr y}{(t-|y-x|)^3|y-x|^2} .
 \end{align*}
\item $F^{\gamma , \mathrm{sph}}_{\mu \nu}(t,x;t_0)$ is given as an integral over $\mathbb{S}^2$,
$$ F^{\gamma , \mathrm{sph}}_{\mu \nu}(t,x;t_0)=\frac{t-t_0}{t_0^3}\int_{\mathbb{S}^2_\sigma} \sigma_k \, \mathbf{a}^k_{\mu \nu} \left(\sigma , \frac{\widecheck{ \; x+(t-t_0)\sigma \; }}{t_0} \right) \int_{\R^3_z}  \Big[|v^0|^5\widehat{Z}^\gamma_\infty f_\infty \Big] \bigg(z, \frac{\widecheck{ \; x+(t-t_0)\sigma \; }}{t_0} \bigg) \dr z \dr \mu_{\mathbb{S}^2_\sigma}.$$
 \end{itemize}
\end{Pro}
\begin{Rq}
As a consequence, $F^{\gamma, \mathrm{inh}}$ is well-defined for $|\gamma| = N_0+1$ as well.
\end{Rq}
\begin{proof}
Let $|\gamma| \leq N_0$, $(\psi_n)_{n \geq 1}$ be a mollifier and $H^n$ be a solution to
$$\nabla^\mu H^n_{\mu \nu} = J(h_n)_\nu, \quad \nabla^\mu {}^* \! H^n_{\mu \nu} =0, \qquad \qquad h_n(t,x,v):=\int_{\R^3_z} \widehat{Z}^\gamma_\infty f_\infty(z,v) \dr z \, \psi_n (x-t\widehat{v}).$$
Let further $H^{n,\mathrm{inh}}$ be the unique solution to
$$ \Box H^{n,\mathrm{inh}}_{\mu \nu} = \Box H^n_{\mu \nu}= \partial_{x^\nu} J(h_n)_\mu- \partial_{x^\mu} J(h_n)_\nu, \qquad \qquad \qquad H^{n,\mathrm{inh}}_{\mu \nu}(t_0,\cdot)=\partial_t H^{n,\mathrm{inh}}_{\mu \nu}(t_0,\cdot)=0.$$
 According to \cite[Theorem~$3$]{GlStrauss} (see also \cite[Proposition~$5.3$]{scat}), we have the decomposition
$$H^{n,\mathrm{inh}}_{\mu \nu}=-H^{\mathrm{sph}}_{\mu \nu}+H^{\mathrm{T}}_{\mu \nu}+H^S_{\mu \nu},
$$
where
\begin{itemize}
\item $H^{\mathrm{sph}}_{\mu \nu}$ is purely determined by the initial data since
\begin{align*}
 H^{\mathrm{sph}}_{\mu \nu}(t,x) &:= (t-t_0)\int_{\mathbb{S}^2_\sigma} \int_{\R^3_v}\sigma_k \, \mathbf{a}^k_{\mu \nu} \left(\sigma , v \right)   h_n(t_0,x+(t-t_0)\sigma,v) \dr v \dr \mu_{\mathbb{S}^2_\sigma} .
\end{align*} 
\item $H^{\mathrm{T}}_{\mu \nu}$ is given by
$$ H^{\mathrm{T}}_{\mu \nu}(t,x) := \int_{|y-x| \leq t-t_0} \int_{\R^3_v}\mathbf{w}^{\mathrm{T}}_{\mu \nu}\bigg(\frac{y-x}{|y-x|}, v \bigg)  h_n(t-|t-x|,y,v) \dr v\frac{\dr y}{|y-x|^2}.$$
\item $H^S_{\mu \nu}$ is a functional of $\T_0(h_n)$. Since $h_n$ is solution to the linear Vlasov equation, this term vanishes.
\end{itemize}
By Lemma \ref{Lembasicpropsing}, $H^{\mathrm{sph}}_{\mu \nu}$ and $H^{\mathrm{T}}_{\mu \nu}$ converge respectively, as $n \to \infty$, to $F^{\gamma,\mathrm{sph}}_{\mu \nu}(\cdot,\cdot \, ;t_0)$ and $F^{\gamma,\mathrm{T}}_{\mu \nu}(\cdot,\cdot \, ;t_0)$.
\end{proof}

We are now able to derive the following estimates.
\begin{Pro}\label{estimatesinh}
Let $Z^\gamma \in \mathbb{K}^{|\gamma|}$, with $|\gamma| \leq N_0+1$, only composed by homogeneous vector fields. Then,
\begin{align*}
 \forall \, (t,x) \in [t_0, +\infty[ \times \R^3, \qquad \qquad  \left| F^{\gamma , \mathrm{T}} \right|(t,x;t_0) &\lesssim  t^{-2} \, \mathds{1}_{t >|x|} \, \overline{\mathbb{E}}_{0}\big[\widehat{Z}^\gamma_\infty f_\infty \big], \\
  \left| F^{\gamma, \mathrm{sph}} \right|(t,x;t_0) &\lesssim \frac{1}{t_0 \, t} \, \mathds{1}_{0 \leq t-|x| \leq 2t_0} \, \overline{\mathbb{E}}_{0}\big[\widehat{Z}^\gamma_\infty f_\infty \big].
  \end{align*}
\end{Pro}
\begin{proof}
Recall from Proposition \ref{GSasymp} the definition of $F^{\gamma , \mathrm{T}}$ and $F^{\gamma , \mathrm{sph}}$ and remark that they vanish in the exterior of the light cone. Indeed, if $|x| \geq t$, then the domain of integration is included in $\{(\tau,y) \in \R_+ \times \R^3 \; | \; |y| \geq \tau \}$, where the integrands vanish. We then fix $t \geq t_0$ and $|x|<t$. According to Lemma \ref{Lemkernel} and by support considerations,
$$
 \left| F^{\gamma, \mathrm{sph}} \right|(t,x;t_0) \lesssim \frac{t-t_0}{t_0^3}\int_{\mathbb{S}^2_\sigma} \mathds{1}_{\big|\frac{x}{t_0}+\frac{t-t_0}{t_0}\sigma \big| <1} \, \dr \mu_{\mathbb{S}^2_\sigma}  \sup_{\R^3_z \times \R^3_v} \, \langle z \rangle^4 \langle v \rangle^6  \big|\widehat{Z}^\gamma_\infty f_\infty |(z,v).
$$
Since the right hand side vanishes if $|t-t_0-|x|| \geq t_0$, we get using the estimate of Lemma \ref{estiWeiYang} for integrals over $\mathbb{S}^2$,
$$
\frac{t-t_0}{t_0^3}\int_{\mathbb{S}^2_\sigma} \mathds{1}_{\big|\frac{x}{t_0}+\frac{t-t_0}{t_0}\sigma \big| <1}  \dr \mu_{\mathbb{S}^2_\omega} \lesssim 
        \frac{t-t_0}{t_0^3}  \frac{t_0}{t-t_0} \left\langle \frac{t-t_0+|x|}{t_0} \right\rangle^{-1} \mathds{1}_{t-|x| \leq 2t_0}  \lesssim \frac{1}{t_0 \, t}\mathds{1}_{t-|x| \leq 2t_0}.
        $$
Using again Lemma \ref{Lemkernel} as well as $\sqrt{2}v^0=\sqrt{2}t (t^2-|x|^2)^{-1/2} \geq \sqrt{t}(t-|x|)^{-1/2}$ if $\widehat{v}=x/t$, there holds
$$
\left| F^{\gamma , \mathrm{T}} \right|(t,x;t_0)  \lesssim \mathcal{K}(t,x;t_0)  \sup_{\R^3_z \times \R^3_v} \,  \langle z \rangle^4 \langle v \rangle^7 \big|\widehat{Z}^\gamma_\infty f_\infty |(z,v) ,
$$     
with
$$  \mathcal{K}(t,x;t_0):=\int_{|y-x| \leq t-t_0} \mathds{1}_{ t-|y-x|>|y|} \frac{  ( t-|y-x| -|y| )^{\frac{1}{2}} \, \dr y }{(t-|y-x|)^{\frac{7}{2}} |y-x|^2} .$$
Performing the change of variables $y'=y/t_0$, we get
$$ \mathcal{K}(t,x;t_0) \lesssim  \frac{1}{t_0^2}\int_{\big|y'-\frac{x}{t_0}\big| \leq \frac{t-t_0}{t_0} }\mathds{1}_{ \frac{t}{t_0}-|y'-\frac{x}{t_0}|>|y'|} \frac{\big(1+\frac{t-t_0}{t_0}-\big|y'- \frac{x}{t_0} \big|-|y'|\big)^{1/2}}{ \big(1+\frac{t-t_0}{t_0}-\big|y'- \frac{x}{t_0} \big|\big)^{7/2}}\frac{\dr y'}{\big| y'- \frac{x}{t_0} \big|^2}\lesssim \frac{1}{t_0^2} \mathcal{K}^{2}_{7/2} \!\left( \frac{t-t_0}{t_0}, \frac{x}{t_0} \right) ,$$
where $\mathcal{K}^{2}_{7/2}$ is defined and estimated in Lemma \ref{Lemint}. To conclude the proof, note that for $t \geq t_0$, we have $t_0 \langle (t-t_0)/t_0 \rangle \gtrsim t$.
\end{proof}

\subsubsection{Asymptotic electromagnetic field}

We are now ready to determine the asymptotic expansion of $F^{\gamma,\mathrm{inh}}$. To write the leading order term in a synthetic way, we introduce the following quantity, which turns out to be an alternative expression for \eqref{kev:defasympelec}--\eqref{kev:defasympmag}.

\begin{Def}\label{DefFasympelectromagn}
We define, for any sufficiently regular function $h : \R^3_z \times \R^3_v \to \R$,
$$ \mathbb{F}_{\mu \nu}\big[ h \big]:v \mapsto \int_{\substack{|y| \leq 1 \\ |y+\widehat{v}| <1-|y|  }} \mathbf{w}^T_{\mu \nu} \! \bigg( \frac{y}{|y|}, \frac{\widecheck{  y+\widehat{v}  }}{1-|y|}\bigg) \int_{\R^3_z} \Big[ |v^0|^5 h \Big] \! \bigg( z,\frac{\widecheck{  y+\widehat{v}  }}{1-|y|} \bigg) \dr z\frac{\dr y}{(1-|y|)^3 |y|^2}.$$
In fact, by assuming more regularity on $h$, we have
$$ \mathbb{F}_{\mu \nu}\big[ h \big](v) = \mathfrak{s}_{\mu\nu} \!\int_{\substack{|y| \leq 1 \\ |y+\widehat{v}| <1-|y|  }} \! \int_{\R^3_z} \Big[ |v^0|^5 \widehat{\Omega}^\infty_{\mu \nu}h \Big] \! \bigg( z,\frac{\widecheck{  y+\widehat{v}  }}{1-|y|} \bigg) \dr z \frac{\dr y}{(1-|y|)^4 |y|}, \quad \mathfrak{s}_{\mu \nu}\! = \left\{ 
	\begin{array}{ll}
        -1 \!& \mbox{if $\mu=0$, $1 \leq \nu \leq 3$,} \\
        1  \!& \mbox{if $1 \leq \mu < \nu \leq 3$}.
    \end{array} 
\right.$$
\end{Def}
\begin{Rq}\label{Rqdomaint}
It is important to observe that the domain of integration is included in $\{0 \leq |y| \leq \frac{1+|\widehat{v}|}{2}\}$, so that $1-|y| \geq \frac{1}{4 \langle v \rangle^2 }$. Indeed, if $|y| \geq \frac{1+|\widehat{v}|}{2}$, we have
$$ |y+\widehat{v}| \geq |y|-1+1-|\widehat{v}| \geq \frac{1-|\widehat{v}|}{2} \geq 1-|y|.$$
\end{Rq}

\begin{Pro}\label{Proinhom}
Let $Z^{\gamma} \in \mathbb{K}^{|\gamma|} $, $|\gamma| \leq N+1$, composed only by homogeneous vector fields. Then,
\begin{itemize}
\item for all $(t,x) \in [t_0,+\infty[ \times \R^3$ such that $|x| \geq t$, $F^{\gamma , \mathrm{inh}}(t,x)=0$.
\item There exists a remainder $2$-form $F^{\gamma,\mathrm{rem}}$ such that, for all $(t,x) \in [t_0,+\infty[ \times \R^3$ verifying $|x| <  t$, 
$$F^{\gamma , \mathrm{inh}}(t,x;t_0)= \frac{1}{t^2}\mathbb{F}\big[ \widehat{Z}^\gamma_\infty f_\infty \big]\! \left( \frac{\widecheck{ \; x \;}}{t} \right)-F^{\gamma , \mathrm{rem}}_{\mu \nu}(t,x;t_0).$$
Furthermore, the remainder verifies
$$  \left| F^{\gamma , \mathrm{rem}}_{\mu \nu}(t,x;t_0) \right| \lesssim \frac{1}{t_0 \, t}  \mathds{1}_{t-|x| \leq 2t_0} \, \overline{\mathbb{E}}_0\big[ \widehat{Z}^\gamma_\infty f_\infty \big].$$
\item For all $v \in \R^3_v$, we have $\big| \mathbb{F} \big[\widehat{Z}_\infty^\gamma f_\infty \big](v)\big| \lesssim  \overline{\mathbb{E}}_0\big[ \widehat{Z}^\gamma_\infty f_\infty \big]$.
\end{itemize}
\end{Pro}
\begin{proof}
Let $0 \leq \mu < \nu \leq 3$ and $(t,x) \in [t_0,+\infty[ \times \R^3$. The result for the case $|x| \geq t$ follows from 
$$F^{\gamma , \mathrm{inh}}(t_0,\cdot \, ;t_0)= \partial_t F^{\gamma , \mathrm{inh}}(t_0,\cdot \, ;t_0)=0, \qquad \qquad \forall \, |x'| \geq t' >0, \quad \Box F^{\gamma , \mathrm{inh}}(t',x' ;t_0)=0. $$ 
Assume now that $|x|<t$ and let, with the shorthand $t_y:=t-|t-x|$,
$$ F^{\gamma , \mathrm{rem}}_{\mu \nu}(t,x;t_0)  := \int_{t-t_0 <|y-x| \leq t} \mathbf{w}^{\mathrm{T}}_{\mu \nu}\bigg(\frac{y-x}{|y-x|}, \frac{\widecheck{ \; y \;}}{t_y} \bigg) \int_{\R^3_z} \Big[  |v^0|^5 \widehat{Z}_\infty^\gamma f_\infty \Big] \bigg(z, \frac{\widecheck{\; y \;}}{t_y} \bigg) \dr z\frac{\dr y}{t_y^3|y-x|^2} +F^{\gamma,\mathrm{sph}}_{\mu \nu}(t,x;t_0) .
$$
Then, the stated decomposition of $F^{\gamma, \mathrm{inh}}$ follows from Proposition \ref{GSasymp} as well as the change of variables $y'=t^{-1}(y-x)$. Next, remark that if $2t_0 < t-|x|$, then $F^{\gamma , \mathrm{rem}}_{\mu \nu}(t,x;t_0)=0$. Indeed, in that case, 
$$ |y| \geq |y-x|-|x| \geq t-t_0-|x| >t_0 \geq t-|y-x|   , \qquad \qquad F^{\gamma,\mathrm{sph}}_{\mu \nu}(t,x;t_0)=0  .$$
For the pointwise estimate, fix $v \in \R^3_v$ and remark that $t-|t\widehat{v}|> 2$ for $t \geq 4\langle v \rangle^{2}$. Hence,
$$\forall \, t \geq 4 \langle v \rangle^2, \qquad \qquad F^{\gamma , \mathrm{rem}}(t,t\widehat{v};1)=0,  \qquad t^2 F^{\gamma , \mathrm{inh}}(t,t\widehat{v};1)=\mathbb{F}[\Omega_v^\gamma f_\infty](v) .$$
It remains to apply Proposition \ref{estimatesinh} in order to get, for all $ t \geq 4 \langle v \rangle^2$,
$$ |t^2 F^{\gamma , \mathrm{inh}}(t,t\widehat{v};1)|=|t^2 F^{\gamma , \mathrm{T}}(t,t\widehat{v};1)| \lesssim  \overline{\mathbb{E}}_0\big[ \widehat{Z}^\gamma_\infty f_\infty \big].$$
Finally, we derive the estimate for the remainder on its support using Proposition \ref{estimatesinh},
$$t|F^{\gamma , \mathrm{rem}}|(t,x;t_0)\lesssim t|F^{\gamma , \mathrm{inh}}|(t,x)+t^{-1} \big\| \mathbb{F} \big[ \widehat{Z}^\gamma_\infty f_\infty \big] \big\|_{L^\infty_v} \lesssim t_0^{-1}\overline{\mathbb{E}}_0 \big[  \widehat{Z}^\gamma_\infty f_\infty \big].$$
\end{proof}
Note that for all $(t,x) \in \R^*_+ \times \R^3$ such that $|x|<t$, we have $F^{\gamma , \mathrm{rem}}_{\mu \nu}(t,x;t_0)\to 0$, as $t_0 \to 0$. Furthermore, $J^{\mathrm{asymp}}[f_\infty](t_0,\cdot)$ converges to the Dirac $\delta$ function as $t_0 \to 0$. Consequently, for all $|x|<t$, $F^{\gamma,\mathrm{hom}}(t,x) \to 0$ as $t_0 \to 0$, so that Proposition \ref{Probehavior} holds.

We assume, for the rest of this section, that $1 \leq t_0 \leq 2$.

\subsubsection{Proof of Proposition \ref{Proasymp}} Let $Z^\gamma \in \mathbb{K}^{|\gamma|}$ with $|\gamma| \leq N_0$.
\begin{itemize}
\item We have $| \mathcal{L}_{Z^\gamma} F^{\mathrm{asymp}}[f_\infty] |(t,x) \lesssim \langle t+|x| \rangle^{-1} \, \langle t-|x| \rangle^{-1} \, \overline{\mathbb{E}}_{N_0+1}[f_\infty]$ according to Proposition \ref{Prohompart} and \ref{Proinhpart0}. The improved estimate, in the case $\gamma_T \geq 1$, then follows from Lemma \ref{improderiv}.
\item If $Z^\gamma$ is only composed by homogeneous vector fields and contains the scaling $S$, the estimate stated in Remark \ref{Rqscalingbetterdecay} follows from Propositions \ref{ProhompartS} and \ref{Proinhpart0}.
\item Finally, for the last part of the statement, remark first that by integration by parts in $z$, $\mathbb{F}\big[ \widehat{Z}_\infty^\gamma f_\infty \big]=0$ if $\widehat{Z}_\infty^\gamma$ contains at least one translation or $\widehat{S}$. Thus, there is nothing more to prove in these cases. Otherwise $Z^\gamma$ is only composed by boosts and rotations $\Omega_{\lambda k}$ and one gets the result by combining Propositions \ref{Prohompart} and \ref{Proinhom}.
\end{itemize}

\subsection{Estimates for the asymptotic electromagnetic field and the correction coefficients}

In order to study algebraic relatons between $\mathbb{F}[\widehat{Z}^\gamma_\infty f_\infty]$, $|\gamma| \leq N_0+1$, we will make use of
\begin{equation}\label{eq:corconvasymp}
\forall \, (t,v) \in [t_0,+\infty[ \times \R^3_v, \qquad  \big|t^2 \mathcal{L}_{Z^\gamma}\big(F^{\mathrm{asymp}}[f_\infty] \big)(t,t\widehat{v})-\mathbb{F} \big[ \widehat{Z}^\gamma_\infty f_\infty \big](v) \big| \lesssim_{t_0} \langle t \rangle^{-\frac{1}{2}} \, \langle v \rangle^3 \, \overline{\mathbb{E}}_{N_0+1}[f_\infty],
\end{equation}
for any $|\gamma| \leq N_0$. If $t-t|\widehat{v}| \geq 2t_0$, we use $2t-2t|\widehat{v}| \geq t\,\langle v \rangle^{-2}$ as well as the precise estimate for $\mathcal{L}_{Z^\gamma}F^{\mathrm{asymp}}[f_\infty]$, for the interior of the light cone, given by Proposition \ref{Proasymp}. Otherwise $t-t|\widehat{v}| \leq 4$, so that $t\,\langle v \rangle^{-2} \leq 8$ and
we apply Propositions \ref{Proasymp} and \ref{Proinhom} in order to get
$$\big|t^2 \mathcal{L}_{Z^\gamma}\big(F^{\mathrm{asymp}}[f_\infty] \big)(t,t\widehat{v})-\mathbb{F}^\gamma [f_\infty](v) \big| \leq t^2\big| \mathcal{L}_{Z^\gamma}F^{\mathrm{asymp}}[f_\infty] \big|(t,t\widehat{v})+\big|\mathbb{F} \big[ \widehat{Z}^\gamma_\infty f_\infty \big] \big|(v)\lesssim \langle t \rangle \, \overline{\mathbb{E}}_{N_0+1}[f_\infty].$$

We are now able to relate $\mathbb{F}[\widehat{Z}^\gamma_\infty f_\infty]$ to derivatives of $\mathbb{F}[f_\infty]$. In fact, in the context of the Vlasov-Maxwell system, $\mathbb{F}[f_\infty]$ will only appear through the asymptotic Lorentz force $\widehat{v}^\mu {\mathbb{F}_{\mu}}^j[f_\infty] \partial_{x^j}$. For this reason, we also study this quantity.
\begin{Pro}\label{Proasympelec}
Let $h \in L^\infty( \R^3_z \times \R^3_v , \R)$ such that $\overline{\mathbb{E}}_{1}[ h] <+\infty$. The asymptotic electromagnetic field $\mathbb{F}[h]$ verifies the following properties.
\begin{enumerate}
\item Algebraic relations between the derivatives of $\mathbb{F}[h]$ and $\mathbb{F}[\widehat{Z}_\infty f_\infty]$, for $\widehat{Z}_\infty \in \K^\infty$. There holds
\begin{alignat*}{2}
\mathbb{F}_{\mu \nu}\big[ \widehat{\Omega}_{0k}^\infty h \big]&:=v^0\partial_{v^k}\mathbb{F}_{\mu \nu}[h]-2\widehat{v}^k \, \mathbb{F}_{\mu \nu}[h]+\delta_\mu^0 \, \mathbb{F}_{k \nu}[h]+\delta_\mu^k \, \mathbb{F}_{0 \nu}[h]+\delta_\nu^0 \, \mathbb{F}_{\mu k}[h]+\delta_\nu^k \, \mathbb{F}_{\mu 0}[h], \quad &&1 \leq k \leq 3, \\
\mathbb{F}_{\mu \nu}\big[ \widehat{\Omega}_{jk} h \big]&:=\big(v^j \partial_{v^k}-v^k \partial_{v^j}\big)\mathbb{F}_{\mu \nu}[h]+\delta_\mu^j \, \mathbb{F}_{k \nu}[h]- \delta_\mu^k \, \mathbb{F}_{j \nu}[h]+\delta_\nu^j \, \mathbb{F}_{\mu k}[h]- \delta_\nu^k \, \mathbb{F}_{\mu j}[h] , \quad \;  1 \leq && \, j<k \leq 3.
\end{alignat*}
Otherwise $\widehat{Z}_\infty \in \{ \partial_t^\infty, \, \partial_{z^1}, \, \partial_{z^2}, \, \partial_{z^3}, \, S \}$ and $\mathbb{F}\big[ \widehat{Z}_\infty h \big]=0$.
\item Algebraic relations between the derivatives of $\widehat{v}^\mu\mathbb{F}_{\mu i}[h]$ and $\widehat{v}^\mu \mathbb{F}_{\mu \ell}[\widehat{Z}_\infty f_\infty]$, for $\widehat{Z}_\infty \in \K^\infty$ and $1 \leq i \leq 3$,
\begin{alignat*}{2}
\frac{\widehat{v}^\mu}{v^0}\mathbb{F}_{\mu i}\big[ \widehat{\Omega}^\infty_{0k} h \big](v)&:=v^0\partial_{v^k} \! \left( \frac{\widehat{v}^\mu}{v^0} \mathbb{F}_{\mu i}[h](v) \right)-\delta^k_i \, \widehat{v}^\ell  \frac{\widehat{v}^\mu}{v^0} \mathbb{F}_{\mu \ell}[h](v) , \\
\frac{\widehat{v}^\mu}{v^0}\mathbb{F}_{\mu i}\big[ \widehat{\Omega}_{ik} h \big](v) &:= \big(v^j \partial_{v^k}-v^k \partial_{v^j}\big)\! \left(\frac{\widehat{v}^\mu}{v^0} \mathbb{F}_{\mu i}[h](v) \! \right)+\delta_i^j \, \frac{\widehat{v}^\mu}{v^0} \mathbb{F}_{\mu k}[h](v)- \delta_i^k \, \frac{\widehat{v}^\mu}{v^0} \mathbb{F}_{\mu j}[h](v).
\end{alignat*}
\item Uniform bound for $\mathbb{F}[h]$,
$$ \forall \, v \in \R^3_v, \qquad \big| \mathbb{F} [h] \big|(v) \lesssim  \sup_{(z,p) \in \R^3 \times \R^3} \langle z \rangle^4 \, \langle p  \rangle^7 \big| h (z,p) \big|.$$ 
\end{enumerate}
\end{Pro}
\begin{proof}
By a density argument, we can reduce the proof to the case $\overline{\mathbb{E}}_{2}[ h] <+\infty$. For any sufficiently regular $2$-form $G$, we have
\begin{align*}
 v^0 \partial_{v^k}\!\left(G_{\mu \nu}(t,t\widehat{v})\right) & = t\left(\delta_k^i-\widehat{v}^k \widehat{v}^i \right) \partial_{x^i} \! \left( G_{\mu \nu} \right)(t,t\widehat{v})   =  \left(\Omega_{0k} G_{\mu \nu} \right)\!(t,t\widehat{v})-t\widehat{v}^k\partial_t \! \left(G_{\mu \nu} \right)\!(t,t\widehat{v})-t\widehat{v}^k\widehat{v}^i\partial_{x^i} \! \left( G_{\mu \nu} \right)\!(t,t\widehat{v}) \\
 &= \left(\Omega_{0k} G_{\mu \nu} \right)(t,t\widehat{v})-\widehat{v}^k \left( SG_{\mu \nu} \right)(t,t\widehat{v})
 \end{align*}
as well as
$$
\mathcal{L}_{S}(G)_{\mu \nu} = S \left(  G_{\mu \nu}\right)+2G_{\mu \nu}, \qquad \mathcal{L}_{\Omega_{0k}}(G)_{\mu \nu} = \Omega_{0k} \left(  G_{\mu \nu} \right)+\delta_\mu^0  G_{k \nu}+ \delta_\mu^k G_{0 \nu}+\delta_\nu^0  G_{\mu k}+ \delta_\nu^k G_{\mu 0} .
$$
Applying these relations for $G= F^{\mathrm{asymp}}[h]$ and using the convergence estimate \eqref{eq:corconvasymp} for $f_\infty=h$, we get the first identity as $\mathbb{F}[\widehat{S}h]=0$. The case of the angular derivatives can be treated similarly since
$$ (v^j \partial_{v^k}-v^k \partial_{v^j})\left(G_{\mu \nu}(t,t\widehat{v})\right)  =  \left(\Omega_{jk} G_{\mu \nu} \right)(t,t\widehat{v}), \qquad \mathcal{L}_{\Omega_{jk}}(G)_{\mu \nu} = \Omega_{jk} \left(  G_{\mu \nu} \right)+\delta_\mu^j  G_{k \nu}- \delta_\mu^k G_{j \nu}+\delta_\nu^j  G_{\mu k}- \delta_\nu^k G_{\mu j}.$$
Next, since $\widehat{v}^\mu /v^0=v^\mu/|v^0|^2$, we have
\begin{align*}
v^0\partial_{v^k}\! \left(\frac{\widehat{v}^\mu}{v^0} \mathbb{F}_{\mu i}[h](v) \! \right) & =\widehat{v}^\mu \partial_{v^k}\mathbb{F}_{\mu i}[h](v)-2\widehat{v}^k\frac{\widehat{v}^\mu}{v^0} \mathbb{F}_{\mu i}[h](v)+\frac{\widehat{v}^k}{v^0}\mathbb{F}_{0 i}[h](v)+\frac{1}{v^0}\mathbb{F}_{k i}[h](v) , \\
(v^j \partial_{v^k}-v^k \partial_{v^j})\! \left(\frac{\widehat{v}^\mu}{v^0} \mathbb{F}_{\mu i}[h](v) \! \right) & =\frac{\widehat{v}^\mu}{v^0} (v^j \partial_{v^k}-v^k \partial_{v^j})\mathbb{F}_{\mu i}[h](v)+\frac{\widehat{v}^j}{v^0}\mathbb{F}_{k i}[h](v)-\frac{\widehat{v}^k}{v^0}\mathbb{F}_{j i}[h](v).
\end{align*}
The relation for the asymptotic Lorentz force then ensues from the first part of the statement as well as $\widehat{v}^\mu G_{\mu 0}=-\widehat{v}^\ell \widehat{v}^\mu G_{\mu \ell}$, which holds for any antisymmetric matrix $G_{\mu \nu}$. Finally, the pointwise estimates are given by Proposition \ref{Proinhom}.
\end{proof}

We then deduce, using the expression of the Lie derivative \eqref{keva:defLie}, the following identity.
\begin{Cor}\label{CorLieasympelec}
View as the $2$-form $(t,x) \mapsto \mathbb{F}[h](x)$, the asymptotic electromagnetic field verifies
$$ \mathcal{L}_{\Omega_{ij}} \mathbb{F}[h]=\mathbb{F} \big[ \widehat{\Omega}_{ij}h \big], \qquad \qquad 1 \leq i < j \leq 3.$$
\end{Cor}
Let us mention, even if it will not be useful in this article, that as for the asymptotic Lorentz force, the structure of the modified coefficients $\C_{t,v}^i$ is also preserved by differentiation (see \cite[Lemma~$6.33$]{scat}).

We end this subsection by proving a convergence result for $t^2\mathcal{L}_{Z^\gamma}F^{\mathrm{asymp}}[f_\infty]$ along $t \mapsto (t,x+t\widehat{v})$. For this, we will make use of the following lemma, which will allow us to transform $t-r$ decay into time decay along timelike straight lines.  
\begin{Lem}\label{Lemxvt}
Let $(x,v) \in \R^3_x \times \R^3_v$. Then,
$$ \forall \, 0 \leq t \leq 16 \, \langle x \rangle \, \langle v \rangle^2, \quad 1 \leq 16 \frac{\langle x \rangle \, \langle v \rangle^2}{\langle t \rangle}, \qquad \qquad \forall \, t \geq  16 \,\langle x \rangle|v^0|^2, \qquad t-|x+t\widehat{v}| \geq \frac{t}{4 \langle v \rangle^2} \geq 2t_0.$$
\end{Lem}
\begin{proof} The first inequality is straightforward. For the second one, use $1 \leq t_0 \leq 2$ as well as
\begin{equation*}
 \forall \, t \geq 4|x|\, \langle v \rangle^2, \qquad t \geq \frac{2|x|}{1-|\widehat{v}|},  \qquad  t-|x+t\widehat{v}| \geq t-\frac{1-|\widehat{v}|}{2}t-|\widehat{v}|t = t-\frac{1+|\widehat{v}|}{2}t \geq \frac{t}{4 \langle v \rangle^2}.
\end{equation*}
\end{proof}
\begin{Cor}\label{Corconvasymp}
Let $|\gamma| \leq N_0$. Then, for all $(t,x,v) \in [t_0,+\infty[ \times \R^3_x \times \R^3_v$,
$$  \big|t^2 \mathcal{L}_{Z^\gamma}\big(F^{\mathrm{asymp}}[f_\infty] \big)(t,x+t\widehat{v})-\mathbb{F} \big[ \widehat{Z}^\gamma_\infty f_\infty](v) \big| \lesssim \langle t \rangle^{-\frac{1}{2}} \, \langle x \rangle^{\frac{3}{2}} \, \langle v \rangle^3 \, \overline{\mathbb{E}}_{N_0+1}[f_\infty].$$
\end{Cor}
\begin{Rq}
By loosing more powers of $|x|$ and $|v|$, we could improve the rate of convergence to $t^{-1}$.
\end{Rq}
\begin{proof}
If $Z^\gamma$ is composed by at least one translation $\partial_{x^\lambda}$, so that $\gamma_T \geq 1$, or the scaling $S$, the result follows from $\mathbb{F} \big[ \widehat{Z}^\gamma_\infty f_\infty]=0$, the decay estimates provided by Proposition \ref{Proasymp} and the previous Lemma \ref{Lemxvt}. Otherwise, $Z^\gamma$ is only composed by boosts and rotations $\Omega_{\lambda k}$. 
\begin{itemize}
\item If $0 \leq t \leq 16 \langle x \rangle \, \langle v \rangle^2$, we perform the crude estimate 
$$ \big|t^2 \mathcal{L}_{Z^\gamma}\big(F^{\mathrm{asymp}}[f_\infty] \big)(t,x+t\widehat{v})-\mathbb{F} \big[ \widehat{Z}_\infty^\gamma f_\infty \big](v) \big| \lesssim  (t+1)\overline{\mathbb{E}}_{N_0+1}[f_\infty],$$
obtained by controlling the first term through the first pointwise decay estimate of Proposition \ref{Proasymp} and the second one by Proposition \ref{Proinhom}. It remains to apply Lemma \ref{Lemxvt}. 
\item Otherwise, using again Lemma \ref{Lemxvt}, we have $t-|x+t\widehat{v}| \geq \frac{t}{4 \langle v \rangle^2} \geq 2t_0$, so that the precise estimate for $\mathcal{L}_{Z^\gamma}F^{\mathrm{asymp}}[f_\infty]$, for the interior of the light cone, given by Proposition \ref{Proasymp} implies
$$ \bigg|t^2 \mathcal{L}_{Z^\gamma}\big( F^{\mathrm{asymp}}[f_\infty] \big)(t,x+t\widehat{v}) - \mathbb{F}\big[ \Omega_v^\gamma f_\infty \big] \bigg( \frac{\widecheck{  x+t\widehat{v} }}{t} \bigg)  \bigg| \lesssim \frac{\overline{E}_{N_0+1}[f_\infty] \, \langle v \rangle^3}{t^{\frac{1}{2}}}.$$
Note now that for all $|x|<t$,
\begin{equation*}
\frac{\widecheck{ \; x \;}}{t} = \frac{x}{\sqrt{t^2-|x|^2}}, \qquad \left\langle \frac{\widecheck{ \; x \;}}{t} \right\rangle = \frac{t}{\sqrt{t^2-|x|^2}}, 
\end{equation*}
so that, for any $\psi \in H^1_{\mathrm{loc}}(\R^3_v,\R)$ and $1 \leq \ell \leq 3$,
$$
\partial_{x^\ell} \bigg( \psi \bigg( \frac{\widecheck{\; x \;}}{t} \bigg) \bigg) =\frac{1}{\sqrt{t^2-|x|^2}} \big[ \partial_{v^\ell} \psi \big] \bigg( \frac{\widecheck{\; x \;}}{t} \bigg) + \frac{x^\ell x}{\sqrt{t^2-|x|^2}^{\,3}} \cdot \big[ \nabla_v \psi \big] \bigg( \frac{\widecheck{\; x \;}}{t} \bigg). $$
We then deduce
$$ \bigg| \nabla_x \bigg[ \mathbb{F}\big[ \widehat{Z}_\infty^\gamma f_\infty \big] \bigg( \frac{\widecheck{\; x \;}}{t} \bigg) \bigg] \bigg| \lesssim \frac{1}{t} \left( 1+\frac{|x|^2}{t^2-|x|^2} \right) \bigg[v^0 \nabla_v \mathbb{F}\big[  \widehat{Z}_\infty^\gamma f_\infty \big]  \bigg]\bigg( \frac{\widecheck{\; x \;}}{t} \bigg) \leq  \frac{1}{t-|x|}  \bigg[v^0 \nabla_v \mathbb{F}\big[  \widehat{Z}_\infty^\gamma f_\infty \big]  \bigg]\bigg( \frac{\widecheck{\; x \;}}{t} \bigg) .$$
Applying the divergence theorem and using the algebraic relations between the derivatives of $\mathbb{F}[f_\infty]$ provided by Proposition \ref{Proasympelec}, we get
\begin{equation}\label{kevatalenn:Fasymp}
\bigg| \mathbb{F}\big[ \widehat{Z}_\infty^\gamma f_\infty \big] \bigg( \frac{\widecheck{  x+t\widehat{v} }}{t} \bigg)-\mathbb{F}\big[ \widehat{Z}_\infty^\gamma f_\infty \big] (v)\bigg| \leq \frac{|x|}{t} \sup_{|\xi| \leq |\gamma|+1} \sup_{\tau \in [0,1]} \frac{1}{t-|\tau x+t\widehat{v}|}  \Big| \mathbb{F} \big[\widehat{Z}_\infty^\xi f_\infty \big] \Big| \bigg( \frac{\widecheck{  \tau x+t\widehat{v} }}{t} \bigg).
\end{equation}
Since $\|\mathbb{F}[\widehat{Z}_\infty^\xi f_\infty]\|_{L^\infty_v} \lesssim \overline{\mathbb{E}}_{N_0+1}[f_\infty]$ by Proposition \ref{Proasympelec}, it remains to remark 
$$t-|\tau x+t\widehat{v}| \geq \tau(t-| x+t\widehat{v}|)+(1-\tau)(t-t|\widehat{v}|)\gtrsim \tau t \, \langle v \rangle^{-2}+(1-\tau)t \, \langle v \rangle^{-2} =t \, \langle v \rangle^{-2}.$$
\end{itemize}
\end{proof}

In fact, in the nonlinear problem, we will only be interested in the rate of convergence of the Lorentz force, which turns out to be stronger (see Proposition \ref{ProconvasympLor} below).

\subsection{Study of the null properties of $F^{\mathrm{asymp}}[f_\infty]$ and of its radiation field}\label{Subsecnullproasymp} In view of the weak decay estimates verified by $F^{\mathrm{asymp}}[f_\infty]$ and by comparison with a generic solution to the vacuum Maxwell equations initially decaying as $r^{-2}$, one should, from a quick look, expect
\begin{itemize}
\item $F^{\mathrm{asymp}}[f_\infty]$ to admit a radiation field along $\mathcal{I}^+$ decaying merely as $\langle u \rangle^{-1}$.
\item Moreover, it is not clear if $r^2 \rho (F^{\mathrm{asymp}}[f_\infty])$ and $r^2 \sigma (F^{\mathrm{asymp}}[f_\infty])$ converge or not along null rays $r \mapsto (r+u,r\omega)$.
\end{itemize} 
In fact, a more thorough analysis will allow us to obtain stronger results than expected. We first establish improved estimates for the good null components of the solution to the asymptotic Maxwell equations \eqref{eq:defasympF}. Recall from \eqref{defFoverlin} the definition of $\overline{F}$.

\begin{Pro}\label{Pronullcompoasymp}
Let $|\gamma| \leq N_0$ and $|\xi| \leq N_0-1$. Then, for all $(t,x) \in [t_0+\infty[ \times \R^3$,
$$\big| \underline{\alpha} \big(\mathcal{L}_{Z^{\gamma}} F^{\mathrm{asymp}}[f_\infty]\big)\big|(t,x)  \lesssim \frac{\overline{\mathbb{E}}_{N_0+1}[f_\infty] }{\langle t+|x| \rangle^2}+\frac{\overline{\mathbb{E}}_{N_0+1}[f_\infty] }{\langle t+|x| \rangle \, \langle t-|x| \rangle^2}.
$$
and
$$
\big|\alpha \big(\mathcal{L}_{Z^{\xi}} F^{\mathrm{asymp}}[f_\infty]\big)\big|(t,x)+\big|\rho \big(\mathcal{L}_{Z^{\xi}}  F^{\mathrm{asymp}}[f_\infty]\big)\big|(t,x)+\big|\sigma \big(\mathcal{L}_{Z^{\xi}}  F^{\mathrm{asymp}}[f_\infty]\big)\big| (t,x)  \lesssim \frac{\overline{\mathbb{E}}_{N_0+1}[f_\infty] }{\langle t+|x| \rangle^2}
$$
Moreover, we have improved estimates in the exterior of the light cone,
$$ \forall \, |x| \geq t, \qquad \Big|\rho \Big(\mathcal{L}_{Z^{\xi}} \big( F^{\mathrm{asymp}}[f_\infty]-\overline{F} \,\big)\Big)\Big|(t,x)+\Big|\sigma \Big(\mathcal{L}_{Z^{\xi}} \big( F^{\mathrm{asymp}}[f_\infty]-\overline{F} \, \big)\Big)\Big|(t,x)  \lesssim \frac{\overline{\mathbb{E}}_{N_0+1}[f_\infty] }{\langle t+|x| \rangle^2 \, \langle t-|x| \rangle} .$$
Finally, if $|\beta| \leq N_0-2$, the best null component verifies
$$ \forall \, (t,x) \in [t_0+\infty[ \times \R^3 , \qquad \big|\alpha \big(\mathcal{L}_{Z^{\beta}}  F^{\mathrm{asymp}}[f_\infty]\big) \big|(t,x) \lesssim  \frac{\langle t-|x| \rangle +\log (\langle t \rangle ) }{\langle t+|x| \rangle^3 }   \, \overline{\mathbb{E}}_{N_0+1}[f_\infty].$$
\end{Pro}
\begin{proof}
Fix $(t,x) \in [t_0 , + \infty[ \times \R^3$ and remark that the result is a consequence of Proposition \ref{Proasymp} if $|x| \leq t/2+1$. Indeed, in this case, $\overline{F}(t,x)=0$ and $\langle t-|x| \rangle \gtrsim \langle t+|x| \rangle$. Assume then that $|x| \geq t/2+1$ and remark
\begin{align}
\hspace{-4mm} |\mathcal{L}_{Z^\gamma}\overline{F}|(t,x) & \lesssim Q|x|^{-2-\gamma_T} \mathds{1}_{|x| \geq t+1} , \label{eq:overlineF00} \\
\hspace{-4mm} \big| \mathcal{L}_{Z^\gamma} \big(F^{\mathrm{asymp}}[f_\infty] -\overline{F} \, \big) \big|(t,x)& \lesssim \langle t-|x| \rangle^{-2}\mathds{1}_{t >|x|}  \, \overline{\mathbb{E}}_{N_0+1}[f_\infty]+ \langle t+|x| \rangle^{-1} \, \langle t-|x| \rangle^{-2} \, \overline{\mathbb{E}}_{N_0+1}[f_\infty]. \nonumber
\end{align}
Indeed, the first estimate is given by Proposition \ref{Prorad} and the second one by \eqref{eq:estihomFtilde}, Proposition \ref{Propurechargetilde} as well as Proposition \ref{estimatesinh}. It implies the estimate for the null component $\underline{\alpha}$. Consider now $|\xi| \leq N_0-1$ and let us denote by $\zeta$ any of the null components $\alpha$, $\rho$ or $\sigma$. Then, using
\begin{itemize}
\item The null Maxwell equations \eqref{null1} and \eqref{null3}--\eqref{null4}, 
\item Lemma \ref{LemComgoodderiv}, allowing us to extract $r$ decay from the angular derivatives $\slashed{\nabla}$,
\item Lemma \ref{LemasympJpoint} and Proposition \ref{Propurecharge}, giving that the source term of the Maxwell equations verified by $\mathcal{L}_{Z^\xi} F^{\mathrm{asymp}}[f_\infty]$ (respectively $\mathcal{L}_{Z^\xi} \overline{F}$) decays as $ t^{-3}\overline{E}_{N_0-1}[f_\infty] \mathds{1}_{t >|x|} $ (respectively as $Q|x|^{-2} \mathds{1}_{|t-|x|| \leq 2}$),
\end{itemize}
we get, as $|x| \geq t/2+1$, so that $|x| \gtrsim \langle t+|x| \rangle$,
\begin{align*}
 \Big|\nabla_{\underline{L}}\zeta \Big( \mathcal{L}_{Z^\xi} \! \big(F^{\mathrm{asymp}}[f_\infty]-\overline{F} \big) \Big) \Big|(t,x) &\lesssim \frac{\overline{E}_{N_0-1}[f_\infty]}{\langle t+|x| \rangle^3}\mathds{1}_{t >|x|}+\frac{Q \, \mathds{1}_{|t-|x|| \leq 2}}{\langle t+|x| \rangle^2} \! + \! \sum_{|\beta| \leq N_0}\frac{ \big| \mathcal{L}_{Z^\beta} \! \big(F^{\mathrm{asymp}}[f_\infty]-\overline{F} \big) \big|(t,x)}{|x|},\\
 & \lesssim \frac{\overline{\mathbb{E}}_{N_0+1}[f_\infty] }{\langle t+|x| \rangle^{3} }\mathds{1}_{t >|x|}  + \frac{\overline{\mathbb{E}}_{N_0+1}[f_\infty] }{\langle t+|x| \rangle^{2} \, \langle t-|x| \rangle^{2}}. 
\end{align*}
Consider now
\begin{equation}\label{defphifornull}
 \phi (u,\underline{u}) := \zeta \Big( \mathcal{L}_{Z^\xi}  \big(F^{\mathrm{asymp}}[f_\infty]-\overline{F} \big)\Big)  \left( \frac{\underline{u}+u}{2}, \frac{\underline{u}-u}{2} \omega \right), \qquad \qquad |u| \leq \underline{u}, 
 \end{equation}
and remark that, in order to integrate in $u$, along a null ray of constant $t+r$, 
$$\left| \nabla_{\partial_u} \phi \right|(u,\underline{u}) = \frac{1}{2}\left| \nabla_{\underline{L}} \zeta \Big( \mathcal{L}_{Z^\xi} \big(F^{\mathrm{asymp}}[f_\infty]-\overline{F} \big) \Big) \right|\left( \frac{\underline{u}+u}{2}, \frac{\underline{u}-u}{2} \omega \right) \lesssim  \frac{\overline{\mathbb{E}}_{N_0+1}[f_\infty] }{\langle \underline{u} \rangle^3} \mathds{1}_{u >0} + \frac{\overline{\mathbb{E}}_{N_0+1}[f_\infty] }{\langle \underline{u} \rangle^{2} \, \langle u \rangle^{2}}.$$
Fix $t \geq t_0$, $r>0$, $\omega \in \mathbb{S}^2$ and write
$$  \left| \zeta \Big( \mathcal{L}_{Z^\xi} \big(F^{\mathrm{asymp}}[f_\infty]-\overline{F} \big) \Big) \right| (t,r\omega)=|\phi|(t-r,t+r)= \bigg|\phi(2t_0-t-r,t+r)+\int_{u=2t_0-t-r}^{t-r} \nabla_{\partial_u} \phi(u,t+r) \dr u \bigg|,$$
where the initial term is controlled by Proposition \ref{Proinidataasymp} as
$$ |\phi(2t_0-t-r,t+r) |= \left| \zeta \Big( \mathcal{L}_{Z^\xi} \big(F^{\mathrm{asymp}}[f_\infty]-\overline{F} \big) \Big) \right|\big(t_0,(t+r-t_0)\omega \big) \lesssim \langle t+r \rangle^{-3} \, \overline{\mathbb{E}}_{N_0+1}[f_\infty].$$  
For the integral, we have 
\begin{align*}
 \int_{u=2t_0-t-r}^{t-r} |\nabla_{\partial_u} \phi(u,t+r)| \dr u & \lesssim  \frac{\overline{\mathbb{E}}_{N_0+1}[f_\infty] }{\langle t+r \rangle^3} \int_{u=0}^{\max(0,t-r)} \dr u+  \frac{\overline{\mathbb{E}}_{N_0+1}[f_\infty] }{\langle t+r \rangle^2} \int_{u=2t_0-t-r}^{t-r} \frac{\dr u}{\langle u \rangle^2} \\[2pt]
 & \lesssim \left\{ 
	\begin{array}{ll}
        \langle t+r \rangle^{-2} \, \overline{\mathbb{E}}_{N_0+1}[f_\infty], \qquad & \text{if $t \geq r$}, \\
        \langle t+r \rangle^{-2}  \langle t-r \rangle^{-1}\, \overline{\mathbb{E}}_{N_0+1}[f_\infty], \qquad & \text{if $t \leq r$}.
    \end{array} 
\right.
 \end{align*}
Combining these last computations with \eqref{eq:overlineF00} yields all the stated estimates for the good null components but the last one on $\alpha$. Consider then $|\beta| \leq N_0-2$ and, as before, we use the null Maxwell equation \eqref{null1}, Lemma \ref{LemComgoodderiv} as well as Lemma \ref{LemasympJpoint} and Proposition \ref{Propurecharge}. The differences allowing to a better estimate, compared to the case of $\rho$ and $\sigma$, are that
\begin{itemize}
\item in \eqref{null1}, the bad derivative $\underline{L}$ of the null component $\alpha$ is expressed in terms of good derivatives $\slashed{\nabla}$ of the good components $\alpha$, $\rho$ and $\sigma$ as well as the angular components of the electromagnetic current.
\item Moreover, the angular components of $J^{\mathrm{asymp}}[h]$ vanish and the angular component of the source term of the Maxwell equations satisfied by $\mathcal{L}_{Z^\beta} \overline{F}$ decays as  $Q|x|^{-3} \mathds{1}_{1 \leq |x|-t \leq 2}$.
\end{itemize}
Using the estimates that we have just proved and $|x| \geq t/2+1$, we get
\begin{align*}
 \Big|\nabla_{\underline{L}}  \alpha \Big( \mathcal{L}_{Z^\beta} \! \big(F^{\mathrm{asymp}}[f_\infty]-\overline{F} \, \big) \Big) \Big|(t,x) &\lesssim \frac{Q}{|x|^3}\mathds{1}_{|t-|x|| \leq 2}+ \sum_{|\xi| \leq N_0-1} \sum_{\zeta \in \{ \alpha ,\rho , \sigma \}}\frac{\big|\zeta \big(\mathcal{L}_{Z^{\xi}} ( F^{\mathrm{asymp}}[f_\infty]-\overline{F} \, )\big)\big|(t,x)}{|x|},\\
 & \lesssim \frac{\overline{\mathbb{E}}_{N_0+1}[f_\infty] }{\langle t+|x| \rangle^{3} } \mathds{1}_{t > |x|} + \frac{\overline{\mathbb{E}}_{N_0+1}[f_\infty] }{\langle t+|x| \rangle^{3} \, \langle t-|x| \rangle}. 
\end{align*}
It allows us to improve the estimate for $\nabla_{\partial_u} \phi$, in the case $\xi=\beta$ and $\zeta=\alpha$, where the function $\phi$ is introduced in \eqref{defphifornull}. As a consequence, we have
\begin{align*}
 \int_{u=2t_0-t-r}^{t-r} |\nabla_{\partial_u} \phi(u,t+r)| \dr u & \lesssim  \frac{\overline{\mathbb{E}}_{N_0+1}[f_\infty] }{\langle t+r \rangle^3} \int_{u=0}^{\max(0,t-r)} \dr u+  \frac{\overline{\mathbb{E}}_{N_0+1}[f_\infty] }{\langle t+r \rangle^3} \int_{u=2t_0-t-r}^{t-r} \frac{\dr u}{\langle u \rangle} \\[2pt]
 & \lesssim \left\{ 
	\begin{array}{ll}
        \langle t+r \rangle^{-3} \, (\langle t-r \rangle +\log(\langle t+r \rangle)) \, \overline{\mathbb{E}}_{N_0+1}[f_\infty], \qquad & \text{if $t \geq r$}, \\
        \langle t+r \rangle^{-3}  \log(\langle t+r \rangle)\, \overline{\mathbb{E}}_{N_0+1}[f_\infty], \qquad & \text{if $t \leq r$},
    \end{array} 
\right.
 \end{align*}
from which we deduce an improved estimate for $\phi$. Finally, in order to conclude the proof, recall from Proposition \ref{Propurecharge} that $|\alpha (\mathcal{L}_{Z^\beta} \overline{F})|(t,x) \lesssim Q\langle t+r \rangle^{-3} \, \langle t-r \rangle$.
\end{proof}

These estimates, together with Lemma \ref{improderiv}, allow us to control the weighted energy of $\nabla_{t,x} F^{\mathrm{asymp}}[f_\infty]$.
\begin{Cor}\label{CorweightedcontrolFasymp}
Let $|\gamma| \leq N-1$. If $\gamma_T \geq 1$, we have $\mathcal{E}^K \big[ \mathcal{L}_{Z^\gamma} F^{\mathrm{asymp}}[f_\infty] \big](t) \lesssim \overline{\mathbb{E}}_{N_0+1}[f_\infty]$ for all $t \geq t_0$.
\end{Cor}

We now prove a result which will be useful in order to rewrite the spherical divergence of $\underline{\alpha} (F^{\mathrm{asymp}}[f_\infty])$.
\begin{Lem}\label{Appendixdivergence}
Recall that $\Omega_{ij}$ are vector fields on $\mathbb{S}^2$. For any $1$-form $\beta :  T\mathbb{S}^2  \to  \R$ on the sphere $\mathbb{S}^2$,
$$ \slashed{\nabla} \cdot \beta = \sum_{1 \leq i < j \leq 3} \mathcal{L}_{\Omega_{ij}}(\beta)_{\Omega_{ij}}.$$
\end{Lem}
\begin{proof}
In a well-chosen spherical coordinate system $(\theta , \varphi)$, we have
\begin{equation*}
\Omega_{12} = \partial_{\varphi}, \qquad \Omega_{23} = - \sin \varphi \, \partial_{\theta}-\cot \theta \cos \varphi \, \partial_{\varphi} , \qquad \Omega_{31} =  \cos \varphi \, \partial_{\theta}-\cot \theta \sin \varphi \, \partial_{\varphi}  .
 \end{equation*}
 Recall that,
 \begin{equation}\label{eq:divspheric}
  \forall \, (\theta , \varphi)\in \, ]0, \pi[ \times ]0, 2 \pi [ \,, \qquad  \slashed{\nabla} \cdot \beta (\theta,\varphi)= \partial_{\theta} \big( \beta_{e_\theta} \big)+\cot (\theta) \beta_\theta+\frac{1}{\sin (\theta )} \partial_{\varphi} \big( \beta_{e_\varphi} \big).
  \end{equation}
 Remark now that
 $$ \forall \, 1 \leq i < j \leq 3 , \qquad \mathcal{L}_{\Omega_{ij}}(\beta)_{\Omega_{ij}}= \Omega_{ij} \big( \beta_{ \, \Omega_{ij}} \big)-\beta_{[\Omega_{ij},\Omega_{ij}]}= \Omega_{ij} \big( \beta_{\, \Omega_{ij}} \big).$$
As $(e_\theta , e_\varphi)=(\partial_\theta, \partial_\varphi / \sin \theta  )$, we have further
\begin{align*}
 \Omega_{12} \big( \beta_{\Omega_{12}} \big) \! & = \sin (\theta) \partial_{\varphi} \big( \beta_{e_\varphi} \big), \\
\Omega_{23} \big( \beta_{\Omega_{23}} \big) \! & = \sin^2 ( \varphi ) \partial_{\theta} \big( \beta_{e_\theta} \big)+\sin \varphi\cos \varphi \, \partial_{\theta} \big( \cos \theta \, \beta_{e_\varphi} \big)+\cot \theta \cos \varphi \, \partial_\varphi \big( \sin \varphi \, \beta_{e_\theta} \big)+\cos \theta \cos \varphi \, \partial_\varphi \big( \cos \varphi \, \beta_{e_\varphi} \big), \\
  \Omega_{13} \big( \beta_{\Omega_{13}} \big)\!  & = \cos^2 (\varphi) \partial_{\theta} \big( \beta_{e_\theta} \big) -\cos \varphi \sin \varphi \, \partial_{\theta} \big( \cos \theta \, \beta_{e_\varphi} \big) - \cot \theta \sin \varphi \, \partial_\varphi \big( \cos \varphi \, \beta_{e_\theta} \big)+\cos \theta \sin \varphi \, \partial_\varphi \big( \sin \varphi \, \beta_{e_\varphi} \big).
  \end{align*}
Finally, one can check that the sum of the last three terms is equal to the right hand side of \eqref{eq:divspheric}.
\end{proof}

We are now able to prove convergence results for the null components of $F^{\mathrm{asymp}}[f_\infty]$.

\begin{Pro}\label{Proradasymp}
Let $|\gamma| \leq N_0+1$. Then, the following properties hold.
\begin{enumerate}
\item $\mathcal{L}_{Z^\gamma}F^{\mathrm{asymp}}[f_\infty]$ has a radiation field $\underline{\alpha}^{\mathrm{asymp}}_\gamma [f_\infty]$ along future null infinity $\mathcal{I}^+$ and
\begin{align*}
\int_{\R_u} \int_{\mathbb{S}^2_\omega} \langle u \rangle^{\frac{5}{2}} \left| \underline{\alpha}^{\mathrm{asymp}}_\gamma [f_\infty] \right|^2\!(u,\omega) \dr \mu_{\mathbb{S}^2_\omega} \dr u & \lesssim \overline{\mathbb{E}}_{N_0+1}\big[ f_\infty \big] , \qquad |\gamma| \leq N_0+1, \\
  \sup_{(u,\omega) \in \R_u \times \mathbb{S}^2_\omega} \langle u \rangle^2 \left| \underline{\alpha}^{\mathrm{asymp}}_\gamma [f_\infty] \right|(u,\omega) & \lesssim \overline{\mathbb{E}}_{N_0+1}\big[  f_\infty \big] , \qquad |\gamma| \leq N_0.
\end{align*}
Moreover, for any $|\xi| \leq N_0-2$, we have, for all $\underline{u} \geq t_0$,
\begin{align*}
\sup_{(u,\omega) \in \R_u \times \mathbb{S}^2_\omega} \left| \big( r \underline{\alpha}(\mathcal{L}_{Z^\xi} F^{\mathrm{asymp}}[f_\infty]) \big)(u, \underline{u},\omega)-\underline{\alpha}^{\mathrm{asymp}}_\xi[f_\infty](u,\omega) \right| & \lesssim \overline{\mathbb{E}}_{N_0+1}\big[  f_\infty \big] \, \langle \underline{u} \rangle^{- 1}.
\end{align*}
\item For any $|\beta| \leq N_0-3$, there exists $\rho^{\mathrm{asymp}}_\beta [f_\infty], \, \sigma^{\mathrm{asymp}}_\beta [f_\infty] \in C^{0}\!\cap L^\infty(\R_u \times \mathbb{S}^2_\omega)$ such that, for all $\underline{u} \geq t_0$, 
\begin{align*}
\forall \, |u| \leq \underline{u}, \, \omega \in \mathbb{S}^2\!,  \qquad \! &  \left| \big( r^2 \rho(\mathcal{L}_{Z^\beta} F^{\mathrm{asymp}}[f_\infty]) \big)(u, \underline{u},\omega)-\rho^{\mathrm{asymp}}_\beta[f_\infty](u,\omega) \right| \! \lesssim \frac{\langle u \rangle + \log(\langle \underline{u} \rangle )}{ \langle \underline{u} \rangle} \, \overline{\mathbb{E}}_{N_0+1}\big[  f_\infty \big], \\ 
\! &\left| \big( r^2 \sigma(\mathcal{L}_{Z^\beta} F^{\mathrm{asymp}}[f_\infty]) \big)(u, \underline{u},\omega)-\sigma^{\mathrm{asymp}}_\beta [f_\infty](u,\omega) \right| \! \lesssim \frac{\langle u \rangle + \log(\langle \underline{u} \rangle) }{ \langle \underline{u} \rangle} \, \overline{\mathbb{E}}_{N_0+1}\big[  f_\infty \big].
\end{align*}
\item If $|\gamma| =0$, we drop the subscript $\gamma$ on $\underline{\alpha}^{\mathrm{asymp}}_\gamma[f_\infty]$, $\rho^{\mathrm{asymp}}_\gamma[f_\infty]$ and $\sigma^{\mathrm{asymp}}_\gamma[f_\infty]$. Their regularity are
$$ \underline{\alpha}^{\mathrm{asymp}}[f_\infty] \in H^{N_0+1}(\R_u \times \mathbb{S}^2_\omega), \qquad  \rho^{\mathrm{asymp}}[f_\infty] , \, \sigma^{\mathrm{asymp}}[f_\infty] \in  H^{N_0}_{\mathrm{loc}} \cap C^{N_0-1}_b(\R_u \times \mathbb{S}^2_\omega).$$
Moreover, if $|\gamma| \geq 1$, $\rho^{\mathrm{asymp}}_\gamma[f_\infty]$ (respectively $\sigma^{\mathrm{asymp}}_\gamma[f_\infty]$) can be computed in terms of derivatives of $\underline{\alpha}^{\mathrm{asymp}}[f_\infty]$ and $\rho^{\mathrm{asymp}}[f_\infty]$ (respectively $\sigma^{\mathrm{asymp}}[f_\infty]$).
\item For any $|\gamma| \leq N_0-3$, 
$$ 2\partial_u \rho^{\mathrm{asymp}}_\gamma [f_\infty] = -\slashed{\nabla} \cdot \underline{\alpha}^{\mathrm{asymp}}_\gamma [f_\infty], \qquad \qquad 2\partial_u \sigma^{\mathrm{asymp}}_\gamma [f_\infty] = -\slashed{\nabla} \times \underline{\alpha}^{\mathrm{asymp}}_\gamma [f_\infty] .$$
\item The following asymptotic constraint equations hold. For all $\omega \in \mathbb{S}^2_\omega$,
\begin{align*}
    \int_{\R_u}  \slashed{\nabla} \cdot \underline{\alpha}^{\mathrm{asymp}} [f_\infty] (u,\omega) \dr u =& \frac{1}{2 \pi} \int_{\R^3_x} \int_{\R^3_v} f_\infty(x,v) - f_\infty(x,|v| \omega)\dr v \dr x \\ &- \int_{\tau=0}^{+\infty}  \frac{2 \tau^2}{\langle \tau \rangle + \tau} \, \slashed{\nabla} \cdot  \underline{\alpha} \big( \mathbb{F} [ f_\infty ] \big) (\tau \omega) \frac{\dr \tau}{\langle \tau \rangle^3} ,\\[3pt] 
     \int_{\R_u} \slashed{\nabla} \times \underline{\alpha}^{\mathrm{asymp}}[f_\infty] (u,\omega) \dr u =& - \int_{\tau=0}^{+\infty}  \frac{2 \tau^2}{\langle \tau \rangle + \tau} \, \slashed{\nabla} \times  \underline{\alpha} \big( \mathbb{F} [ f_\infty ] \big) (\tau \omega) \frac{\dr \tau}{\langle \tau \rangle^3} .
   \end{align*}
   \item The functions $\rho^{\mathrm{asymp}}[f_\infty]$ and $\sigma^{\mathrm{asymp}}[f_\infty]$ admit limit as $u \to \pm \infty$. For all $\omega \in \mathbb{S}^2_\omega$,
$$
 \lim_{u \to -\infty} \rho^{\mathrm{asymp}}[f_\infty](u,\omega) =\frac{Q}{4\pi}, \qquad \qquad \lim_{u \to -\infty} \sigma^{\mathrm{asymp}}[f_\infty](u,\omega) =0.
 $$
 The limit as $u \to +\infty$ can then be computed using $4.$ and $5.$.
\end{enumerate}
\end{Pro}
\begin{Rq}
One can use the constraint equations to derive similar relations for $\underline{\alpha}^{\mathrm{asymp}}_\gamma [f_\infty]$.
\end{Rq}
\begin{proof}
We focus first on the points $1$ and $2$. Let $|\gamma| \leq N_0+1$ and $J^\gamma$ be the source term, computed in Proposition \ref{ProComMax}, in the Maxwell equations satisfied by $ \mathcal{L}_{Z^\gamma}F^{\mathrm{asymp}}[f_\infty]$. Then, according to Proposition \ref{Proinidataasymp} and Lemma \ref{LemasympJpoint},
$$ \big\| \mathcal{L}_{Z^\gamma} F^{\mathrm{asymp}} (t_0 , \cdot ) \big\|_{L^2(\R^3_x)} + \sup_{t \geq t_0} \, \langle t \rangle^{\frac{3}{2}} \big\|  J^\gamma(t,\cdot) \big\|_{L^2(\R^3_x)} \lesssim \overline{\mathbb{E}}_{N_0+1}[f_\infty] .$$
We then deduce from Proposition \ref{blackboxscat} the existence of the radiation field $\underline{\alpha}^{\mathrm{asymp}}_\gamma [f_\infty]$ for $ \mathcal{L}_{Z^\gamma}F^{\mathrm{asymp}}[f_\infty]$. Fix $|\xi| \leq N_0-2$ and recall from Lemma \ref{LemasympJpoint} that $J^{\mathrm{asymp}}_{e_A}[ \cdot ]=0$ for any angular component $A \in \{ \theta , \, \varphi \}$. Hence, applying first Corollary \ref{Corgoodnull} and then the previous Proposition \ref{Pronullcompoasymp}, we get 
$$ \left|\nabla_L \big( r \underline{\alpha} (\mathcal{L}_{Z^\xi} F^{\mathrm{asymp}}[f_\infty] ) \big) \right|(t,x) \lesssim \sum_{|\kappa| \leq N_0-1}\! \Big(\big|\rho (\mathcal{L}_{Z^\kappa} F^{\mathrm{asymp}}[f_\infty] )\big|+\big|\sigma (\mathcal{L}_{Z^\kappa} F^{\mathrm{asymp}}[f_\infty] )\big|\Big)(t,x) \lesssim \frac{\overline{\mathbb{E}}_{N_0+1}[f_\infty]}{\langle t+|x| \rangle^2} $$
as well as
\begin{align*}
&\left|\nabla_L \big( r^2 \rho (\mathcal{L}_{Z^\beta} F^{\mathrm{asymp}}[f_\infty] ) \big) \right|(t,x)+\left|\nabla_L \big( r^2 \sigma (\mathcal{L}_{Z^\beta} F^{\mathrm{asymp}}[f_\infty] ) \big) \right|(t,x) \\
& \qquad \qquad \qquad \lesssim |x|^2|J^\beta_L|(t,x)+\sum_{|\kappa| \leq N_0-2} |x| \big| \alpha (\mathcal{L}_{Z^\kappa} F^{\mathrm{asymp}}[f_\infty] ) \big|(t,x) \lesssim  \frac{\langle t-|x| \rangle+\log(\langle t+|x|\rangle)}{\langle t+|x| \rangle^2} \overline{\mathbb{E}}_{N_0+1}[f_\infty].
\end{align*}
Recall that $L=\partial_t+\partial_r=2\partial_{\underline{u}}$ in the null coordinate system $(u,\underline{u},\theta , \varphi)$.
\begin{itemize}
\item The uniform convergence for $r  \underline{\alpha} (\mathcal{L}_{Z^\xi} F^{\mathrm{asymp}}[f_\infty] )$ is then obtained by integrating in $\tau$ along $(\tau+u,\tau\omega)$. In other words, by integrating in $\underline{u}$, along $t-r=u$, for fixed $\omega \in \mathbb{S}^2$.
\item The same procedure provides the existence of the limits $\rho^{\mathrm{asymp}}_\beta [f_\infty]$ and $\sigma^{\mathrm{asymp}}_\beta [f_\infty]$, as well as the stated rate of convergence. Moreover, these two functions belong to $L^\infty_{u,\omega}$ since $ r^2 \rho (\mathcal{L}_{Z^\beta} F^{\mathrm{asymp}}[f_\infty] ) $ and $r^2 \sigma (\mathcal{L}_{Z^\beta} F^{\mathrm{asymp}}[f_\infty] ) $ are uniformly bounded according to Proposition \ref{Pronullcompoasymp}.
\end{itemize}
It remains us to justify the bounds for $\underline{\alpha}_\gamma^{\mathrm{asymp}}[f_\infty]$. We start with the pointwise one, so that we assume $|\gamma| \leq N_0$. According to Proposition \ref{Pronullcompoasymp}, we have, for any $U >0$ and all $r \geq U^2$,
$$ \sup_{|u| \leq U} \sup_{\omega \in \mathbb{S}^2_\omega} \langle u \rangle^2   \big|r\underline{\alpha} (\mathcal{L}_{Z^\gamma} F^{\mathrm{asymp}}[f_\infty] )  \big|(r+u,r \omega) \lesssim \sup_{|u| \leq U} \big(\langle u \rangle^2 \, \langle r \rangle^{-1}+1  \big) \overline{\mathbb{E}}_{N_0+1}[f_\infty] \leq 2 \overline{\mathbb{E}}_{N_0+1}[f_\infty].$$
If $|\gamma| \leq N_0-2$, it implies the result. Otherwise, by Banach-Alaoglu theorem and the uniqueness of the limit in $\mathcal{D}'(\R_u \times \mathbb{S}^2_\omega)$, there exists an increasing sequence $(r_n)_{n \geq 0}$ of integers such that 
   $$ (u,\omega) \! \mapsto \langle u \rangle^2  r\underline{\alpha} (\mathcal{L}_{Z^\gamma} F^{\mathrm{asymp}}[f_\infty] )  (r_n+u,r_n \omega)  \xrightharpoonup[n \to + \infty]{*} (u,\omega) \! \mapsto \langle u \rangle^2  \underline{\alpha}^{\mathrm{asymp}}_\gamma [f_\infty]   (u, \omega), \quad  \text{in $L^\infty([-U,U] \times \mathbb{S}^2_\omega)$,}$$
 and, by the uniform boundedness principle, the weak-$^*$ limit is bounded by $ \overline{\mathbb{E}}_{N_0+1}[f_\infty]$. We proceed similarly for the $L^2_{u,\omega}$ bound. Let $|\gamma| \leq N_0+1$ and recall that $\mathcal{L}_{Z^\gamma}F^{\mathrm{asymp}}[f_\infty]=F^{\gamma,\mathrm{hom}}+F^{\gamma, \mathrm{inh}}$. Combining the pointwise decay estimate for $F^{\gamma,\mathrm{inh}}$, given by Proposition \ref{Proinhom}, and the $L^2_x$ estimate for $F^{\gamma , \mathrm{hom}}$, provided by Proposition \ref{ProestihomL2}, one has, for any $U >0$ and all $n \geq U^2$,
 $$ \int_{t=n}^{n+2U}\int_{|t-|x|| \leq U} \langle t-|x| \rangle^{\frac{5}{2}} \left| \mathcal{L}_{Z^\gamma}F^{\mathrm{asymp}}[f_\infty] \right|^2(t,x) \dr x \dr t \lesssim \overline{\mathbb{E}}_{N_0+1}[f_\infty] \int_{t=n}^{n+2U} \frac{\langle U \rangle^{\frac{7}{2}}}{  \langle t \rangle^2}+1  \dr t  \lesssim U \overline{\mathbb{E}}_{N_0+1}[f_\infty].$$
 Consequently, there exists a sequence $(\underline{u}_n)_{n \geq U^2}$ of advanced times, such that, for all $n \geq U^2$,
 $$  2n+U \leq \underline{u}_n \leq 2n+3U, \qquad \int_{|u|\leq U} \int_{\mathbb{S}^2_\omega} \langle u \rangle^{\frac{5}{2}} \left| r^2 \mathcal{L}_{Z^\gamma}F^{\mathrm{asymp}}[f_\infty] \right|^2 \left(\frac{\underline{u}_n+u}{2}, \frac{\underline{u}_n-u}{2}\omega \right) \dr \mu_{\mathbb{S}^2_\omega} \dr u \lesssim \overline{\mathbb{E}}_{N_0+1}[f_\infty].$$
The bound on $\underline{\alpha}^{\mathrm{asymp}}_\gamma[f_\infty]$ then follows from Banach-Alaoglu theorem, applied this time in $L^2([-U,U] \times \mathbb{S}^2)$, and the uniform boundedness principle.

We now deal with the last four parts of the statement. Fix $|\gamma| \leq N_0-3$, introduce $G=\mathcal{L}_{Z^\gamma}F^{\mathrm{asymp}}[f_\infty]$ in orter to lighten the notations and let us use the null cordinates $(u,\underline{u},\omega)$, where $u=t-r$ and $\underline{u}=t+r$. The null Maxwell equations \eqref{null3}--\eqref{null4} provide
\begin{equation}\label{eq:pourlasuite}
\underline{L} \big( r^2  \rho ( G) \big)  + r\slashed{\nabla} \cdot r\underline{\alpha}( G)  = r^2 J^\gamma_{\underline{L}} , \qquad \qquad   \underline{L} \big( r^2 \sigma ( G) \big) + r  \slashed{\nabla} \times r\underline{\alpha}( G)  =0.
\end{equation}
In view of the commutation formula of Proposition \ref{ProComMax} and Lemma \ref{LemasympJpoint}, we have
\begin{equation*}
\forall \, \omega \in \mathbb{S}^2_\omega, \qquad  \big| r^2 J^\gamma_{\underline{L}}\big|(u,\underline{u},\omega) \lesssim  \overline{\mathbb{E}}_{N_0-3}[f_\infty] \, \langle \underline{u} \rangle^{-1} .
\end{equation*} 
As $(re_\theta,re_\varphi)=(\partial_\theta,\partial_\varphi/\sin (\theta))$ and $\underline{L}=2\partial_u$, we have in $\mathcal{D}'(\R_u \times \mathbb{S}^2_\omega )$,
\begin{alignat*}{2}
 \underline{L} \big( r^2  \rho ( G) \big)(\cdot, \underline{u} , \cdot) &\xrightharpoonup[\underline{u} \to +\infty]{} 2 \partial_u \rho^{\mathrm{asymp}}_\gamma [f_\infty], \qquad \qquad  \qquad \quad \underline{L} \big( r^2  \sigma ( G) \big)(\cdot, \underline{u} , \cdot) &&\xrightharpoonup[\underline{u} \to +\infty]{} 2 \partial_u \sigma^{\mathrm{asymp}}_\gamma [f_\infty], \\
 r\slashed{\nabla} \cdot r\underline{\alpha}( G) (\cdot, \underline{u} , \cdot) &\xrightharpoonup[\underline{u} \to +\infty]{} \slashed{\nabla} \cdot \underline{\alpha}^{\mathrm{asymp}}_\gamma [f_\infty] , \qquad  \qquad \qquad  \;  r \slashed{\nabla} \times r\underline{\alpha}( G) (\cdot, \underline{u} , \cdot) &&\xrightharpoonup[\underline{u} \to +\infty]{} \slashed{\nabla} \times \underline{\alpha}^{\mathrm{asymp}}_\gamma [f_\infty].
\end{alignat*}
We then deduce that, in $\mathcal{D}'(\R_u \times \mathbb{S}^2_\omega )$ and for any $|\gamma| \leq N_0-3$,
\begin{equation}\label{eq:asympequationnull}
 2 \partial_u \rho^{\mathrm{asymp}}_\gamma [f_\infty] +\slashed{\nabla} \cdot \underline{\alpha}^{\mathrm{asymp}}_\gamma [f_\infty] =0, \qquad \qquad  2 \partial_u \sigma^{\mathrm{asymp}}_\gamma [f_\infty]+\slashed{\nabla} \times \underline{\alpha}^{\mathrm{asymp}}_\gamma [f_\infty]=0.
\end{equation}
We are now able to deduce several informations from these two equations.
\begin{itemize}
\item Recall from Proposition \ref{blackboxscat} that $\underline{\alpha}^{\mathrm{asymp}}_\gamma [f_\infty]$ can be expressed in terms of derivatives of order less than $|\gamma|$ of $\underline{\alpha}^{\mathrm{asymp}}[f_\infty]$, so that $\underline{\alpha}^{\mathrm{asymp}}[f_\infty] \in H^{N_0+1}(\R_u \times \mathbb{S}^2_\omega)$.
\item Let $u \leq -2$ and $\omega \in \mathbb{S}^2_\omega$. Recall from Proposition \ref{Prorad} that $r^2 \rho (\mathcal{L}_{Z^\gamma}\overline{F})(r+u,r\omega)$ and $r^2 \sigma (\mathcal{L}_{Z^\gamma}\overline{F})(r+u,r\omega)$ converge, as $r \to + \infty$, and that the limit is independent of $u \leq -2$. In the particular case $|\gamma|=0$, these two quantities are constants, respectively equal to $Q/4\pi$ and $0$. Consequently, the improved estimates for the exterior of the light cone given by Proposition \ref{Pronullcompoasymp} imply
$$ \rho^{\mathrm{asymp}}_\gamma  [f_\infty](u,\omega) = \Phi^\rho_\gamma(\omega)+O_{u \to -\infty}\big(|u|^{-1} \big), 
 \qquad  \qquad \sigma^{\mathrm{asymp}}_\gamma  [f_\infty](u,\omega) = \Phi^\sigma_\gamma(\omega)+O_{u \to -\infty}\big(|u|^{-1} \big),$$
 where $\Phi^{\rho}_\gamma$ and $\Phi^\sigma_\gamma$ are smooth function on $\mathbb{S}^2$.
 \item We then get, from \eqref{eq:asympequationnull} for the particular case $|\gamma|=0$, that for all $(u,\omega) \in \R_u \times \mathbb{S}^2_\omega$,
\begin{align*}
 2\rho^{\mathrm{asymp}}  [f_\infty](u,\omega)&=\frac{Q}{2 \pi}-\int_{\tau=-\infty}^u \slashed{\nabla} \cdot \underline{\alpha}^{\mathrm{asymp}} [f_\infty] (\tau,\omega) \dr \tau, \\
  2 \sigma^{\mathrm{asymp}} [f_\infty](u,\omega)&=-\int_{\tau= -\infty}^u \slashed{\nabla}\times \underline{\alpha}^{\mathrm{asymp}} [f_\infty] (\tau,\omega) \dr \tau.
  \end{align*}
In view of the $H^{N_0+1}(\R_u \times \mathbb{S}^2_\omega)$ regularity of $\underline{\alpha}^{\mathrm{asymp}}[f_\infty]$ as well as its decay properties, we obtain from these two equations that $\rho^{\mathrm{asymp}} [f_\infty] , \, \sigma^{\mathrm{asymp}} [f_\infty] \in H^{N_0}_{\mathrm{loc}} \cap C^{N_0-1}_b(\R_u \times \mathbb{S}^2_\omega)$.
\item We also get from \eqref{eq:asympequationnull} that $\rho^{\mathrm{asymp}}_\gamma  [f_\infty](u,\omega)$ and $\sigma^{\mathrm{asymp}}_\gamma  [f_\infty](u,\omega)$ both converge, as $u \to +\infty$.
\end{itemize}
In order to conclude the proof, it then remains us to prove the constraint equations verified by the spherical divergence and the spherical curl of $\underline{\alpha}^{\mathrm{asymp}} [f_\infty]$. From now on, $G=F^{\mathrm{asymp}}[f_\infty]$. Fix further $\underline{u} \geq 2$, $\omega \in \mathbb{S}^2_\omega$ and integrate \eqref{eq:pourlasuite} in $u$, between $-\underline{u}+4$ and $\underline{u}$. As $r=0$ for $u=\underline{u}$ and $(t,r)=(2,\underline{u}-2)$ for $u=\underline{u}-4$, we get
\begin{align}
 &-2(\underline{u}-2)^2 \rho (G) \big(2,(\underline{u}-2)\omega \big)+\int_{u=-\underline{u}+4}^{\underline{u}} \Big( r \slashed{\nabla} \cdot r \underline{\alpha} (G) \Big) (u,\underline{u},\omega) \dr u =  \int_{u=-\underline{u}+4}^{\underline{u}} \big( r^2 J^{\mathrm{asymp}}_{\underline{L}}[f_\infty] \big) (u,\underline{u},\omega) \dr u , \nonumber \\
& -2(\underline{u}-2)^2 \sigma (G) \big(2,(\underline{u}-2)\omega \big)+\int_{u=-\underline{u}+4}^{\underline{u}} \Big( r \slashed{\nabla} \times r \underline{\alpha} (G)\Big) (u,\underline{u},\omega) \dr u =0. \label{eq:constraintbis}
 \end{align}
The main difficulty is that the assumptions of the dominated convergence theorem are not satisfied for the integrals involving $\underline{\alpha}(G)$. In fact, we will see that the conclusion of this theorem does not hold neither. Using the pointwise estimates of Proposition \ref{Pronullcompoasymp} for the exterior of the light cone, one has
$$ \lim_{r \to + \infty} r^2 \rho (G) (2,r\omega) = \lim_{r \to +\infty} r^2 \rho (\overline{F})(2,r\omega)=\frac{Q}{4\pi}, \qquad \lim_{r \to + \infty} r^2 \sigma (G) (2,r\omega) = \lim_{r \to +\infty} r^2 \sigma (\overline{F})(2,r\omega)=0.$$
Next, applying Lemma \ref{Lemchargeasymp} to the function $h(z,v):= f_\infty (z,|v|\omega)$, we get
\begin{align*}
 \int_{u=-\underline{u}+4}^{\underline{u}} \big( r^2 J^{\mathrm{asymp}}_{\underline{L}}[f_\infty] \big) (u,\underline{u},\omega) \dr u &= \frac{2}{4\pi}  \int_{u=-\underline{u}+4}^{\underline{u}}\int_{\mathbb{S}^2_{\bar{\omega}}} J^{\mathrm{asymp}}_{\underline{L}}[h] \left( \frac{\underline{u}+u}{2},\frac{\underline{u}-u}{2}\bar{\omega} \right) \frac{r^2}{2} \dr \mu_{\mathbb{S}^2_{\bar{\omega}}} \dr u \\
 & = -\frac{1}{2\pi}  \int_{\R^3_z} \int_{\R^3_v}  f_\infty (z,|v|\omega) \dr z \dr v .
 \end{align*}
It remains us to treat one term in both equations of \eqref{eq:constraintbis}. We start with the one in the equation satisfied by $\rho(G)$ and we will deal with the other term by duality. For this, rewrite the spherical divergence by applying Lemma \ref{Appendixdivergence} and by using that $\mathcal{L}_{\Omega_{ij}}$ commute with the null decomposition (see Lemma \ref{LemComgoodderiv}). For all $t \geq 2$, $r >0$ and $\omega \in \mathbb{S}^2_\omega$, 
 \begin{align}
  \nonumber  \Big( r \slashed{\nabla} \cdot r \underline{\alpha} (G) \Big) (u,\underline{u},\omega) &=  \sum_{1 \leq i < j \leq 3} r\underline{\alpha} (\mathcal{L}_{\Omega_{ij}} G)_{\Omega_{ij}^{\mathrm{resc}}}(u,\underline{u},\omega)   , \qquad \Omega_{ij}^{\mathrm{resc}} = \omega^i \partial_{x^j}-\omega^j \partial_{x^i} =(\omega^i \omega_j^A-\omega^j \omega_i^A)e_A  , \\
    \slashed{\nabla} \cdot \underline{\alpha}^{\mathrm{asymp}}[f_\infty] (u,\omega) & =  \sum_{1 \leq i < j \leq 3} \mathcal{L}_{\Omega_{ij}} \big(\underline{\alpha}^{\mathrm{asymp}}[f_\infty] \big)_{\Omega_{ij}}(u,\omega)   , \qquad \Omega_{ij}  =(\omega^i \omega_j^A-\omega^j \omega_i^A)e_A . \label{eq:alphanulldiv}
\end{align}
 Fix now $1 \leq i < j \leq 3$ and let $|\kappa_{ij}|=1$ such that $\Omega_{ij}=Z^{\kappa_{ij}}$. Recall from Definition \ref{Defdecompasymp} the decomposition $\mathcal{L}_{\Omega_{ij}}G=\mathcal{L}_{\Omega_{ij}}F^{\mathrm{asymp}}[f_\infty]=F^{\kappa_{ij},\mathrm{inh}}+F^{\kappa_{ij},\mathrm{hom}}$ and from Proposition \ref{GSasymp} $F^{\kappa_{ij},\mathrm{inh}}=F^{\kappa_{ij},\mathrm{T}}-F^{\kappa_{ij},\mathrm{sph}}$. Then, 
 $$ \underline{\alpha} \big(\mathcal{L}_{\Omega_{ij}}G \big)=\underline{\alpha} \big(F^{\kappa_{ij},\mathrm{T}})+\underline{\alpha}(F^{\kappa_{ij},\mathrm{sph}}+F^{\kappa_{ij}, \mathrm{hom}} \big).$$
As $|F^{\kappa_{ij} ,T}|(t,x) \lesssim \overline{\mathbb{E}}_{N_0+1}[f_\infty] \, \langle t+|x| \rangle^{-2}$ according to Proposition \ref{estimatesinh}, we get
$$ \lim_{r \to +\infty} r \underline{\alpha} \big(F^{\kappa_{ij},\mathrm{sph}}+F^{ \kappa_{ij},\mathrm{hom}} \big)(r+u,r\omega)= \underline{\alpha}^{\mathrm{asymp}}_{\kappa_{ij}}[f_\infty](u,\omega)= \mathcal{L}_{\Omega_{ij}}\underline{\alpha}^{\mathrm{asymp}}[f_\infty](u,\omega),$$
where we used Proposition \ref{blackboxscat} in the last step. Recall from Proposition \ref{Propurecharge} that $\mathcal{L}_{\Omega_{ij}} \overline{F}=0$. Hence, by the pointwise decay estimates given by \eqref{eq:estihomFtilde} and Proposition \ref{estimatesinh},
$$ r \big| \underline{\alpha} \big( F^{\kappa_{ij},\mathrm{sph}}+F^{ \kappa_{ij},\mathrm{hom}} \big) \big|(r+u,r\omega) \lesssim r \big|F^{\kappa_{ij},\mathrm{sph}} \big|(r+u,r\omega)+r\big|F^{\kappa_{ij}, \mathrm{hom}}- \mathcal{L}_{Z^{\kappa_{ij}}}\overline{F}\big|(r+u,r\omega) \lesssim \overline{\mathbb{E}}_{N_0+1}[f_\infty] \, \langle u \rangle^{-2}.$$
Applying the dominated convergence theorem as well as \eqref{eq:alphanulldiv}, we then get
$$ \lim_{\underline{u} \to + \infty} \sum_{1 \leq i < j \leq 3} \int_{u=-\underline{u}+4}^{\underline{u}}  r\underline{\alpha} \big( F^{\kappa_{ij},\mathrm{sph}}+F^{ \kappa_{ij},\mathrm{hom}} \big)_{\Omega_{ij}^{\mathrm{resc}}}(u,\underline{u},\omega) \dr u =  \int_{\R_u} \slashed{\nabla} \cdot \underline{\alpha}^{\mathrm{asymp}}[f_\infty] (u,\omega) \dr u.$$
It remains us to prove that
\begin{equation}\label{eq:evitarfin}
\lim_{ \underline{u} \to + \infty} \sum_{1 \leq i < j \leq 3} \int_{u=-\underline{u}+4}^{\underline{u}}  r \underline{\alpha} \big(F^{\kappa_{ij}, \mathrm{T}} \big)_{\Omega_{ij}^{\mathrm{resc}}}(u,\underline{u},\omega) \dr u = \int_{\tau=0}^{+\infty}  \frac{2 \tau^2}{\langle \tau \rangle + \tau} \, \slashed{\nabla} \cdot  \underline{\alpha} \big( \mathbb{F} [ f_\infty ] \big) (\tau \omega) \frac{\dr \tau}{\langle \tau \rangle^3}.
\end{equation}
Note first that we can restrict the domain of integration of the integral in the left hand side to $\{ 4 \leq u \leq \underline{u} \}$, since the remainder converge to $0$. Indeed, we have $|F^{\kappa_{ij} ,T}|(t,x) \lesssim \overline{\mathbb{E}}_{N_0+1}[f_\infty] \, \langle t+|x| \rangle^{-2} \mathds{1}_{t >|x|}$ by Proposition \ref{estimatesinh}. Consequently, as $1 \leq t_0 \leq 2$, we have from Proposition \ref{Proinhom} that
$$ \forall \, \omega \in \mathbb{S}^2_\omega, \; 4 \leq u \leq \underline{u}, \qquad  F^{\kappa_{ij}, \mathrm{T}}_{\mu \nu}(u,\underline{u},\omega)= \frac{4}{(\underline{u}+u)^2} \mathbb{F}_{\mu \nu} \big[ \, \widehat{\Omega}_{ij} f_\infty \big] \bigg( \frac{\widecheck{ \, (\underline{u}-u)\omega \,}}{\underline{u}+u} \bigg),$$
where we used that $u=t-r$ and $\underline{u}=t+r$. Moreover, in view of Corollary \ref{CorLieasympelec}, Lemma \ref{LemComgoodderiv} and Lemma \ref{Appendixdivergence},
$$ \underline{\alpha} \big( \mathbb{F}\big[\widehat{\Omega}_{ij} f_\infty \big] \big) = \underline{\alpha} \big( \mathcal{L}_{\Omega_{ij}} \mathbb{F}[ f_\infty ] \big)= \mathcal{L}_{\Omega_{ij}}\underline{\alpha} \big( \mathbb{F} [ f_\infty ] \big), \qquad \qquad \sum_{1 \leq i < j \leq 3} \underline{\alpha} \left( \mathbb{F}\big[\widehat{\Omega}_{ij} f_\infty \big] \right)_{\Omega_{ij}^{\mathrm{resc}}} = r \slashed{\nabla} \cdot \underline{\alpha} \big( \mathbb{F}[f_\infty] \big).$$
We then deduce that, using the notation $[r h](y)=|y|h(y)$, 
$$ \sum_{1 \leq i < j \leq 3}  \int_{u=4}^{\underline{u}}  r \underline{\alpha} (F^{\kappa_{ij}, \mathrm{T}})_{\Omega_{ij}^{\mathrm{resc}}}(u,\underline{u},\omega) \dr u = 2\int_{u=4}^{\underline{u}}  \frac{\underline{u}-u}{(\underline{u}+u)^2}\Big[ r \slashed{\nabla} \cdot  \underline{\alpha}  \big( \mathbb{F} [ f_\infty ] \big) \Big] \!\bigg( \frac{\widecheck{ \, (\underline{u}-u)\omega \,}}{\underline{u}+u} \bigg) \dr u .$$
As, on the domain of integration, $t^2 F^{\kappa_{ij},T}(t,x)$ is a function of $x/t$, we perform the contractor change of variables $p=u/\underline{u}$. It allows us to integrate over a subset of $\{t+r=1\}$ instead of $\{t+r=\underline{u}\}$ and to obtain 
\begin{equation*}
\lim_{ \underline{u} \to + \infty} \sum_{1 \leq i < j \leq 3} \int_{u=-\underline{u}+4}^{\underline{u}}  r \underline{\alpha} \big(F^{\kappa_{ij}, \mathrm{T}} \big)_{\Omega_{ij}^{\mathrm{resc}}}(u,\underline{u},\omega) \dr u = 2\int_{p=0}^{1}  \frac{1-p}{(1+p)^2} \Big[ r \slashed{\nabla} \cdot  \underline{\alpha}  \big( \mathbb{F} [ f_\infty ] \big) \Big] \!\bigg( \frac{\widecheck{ \, (1-p)\omega \,}}{1+p} \bigg) \dr \tau.
\end{equation*}
We finally obtain \eqref{eq:evitarfin} by performing the following two change of variables,
$$q=\frac{1-p}{1+p}, \quad \dr q = \frac{-2}{(1+p)^2} \dr p, \quad 1-p=\frac{2q}{1+q}, \qquad \qquad \tau= \frac{q}{\sqrt{1-q^2}}, \quad \frac{\dr \tau}{\langle \tau \rangle^3} = \dr q, \quad \frac{2q}{1+q}=\frac{2\tau}{\langle \tau \rangle+\tau} .$$
In order to deal with the spherical curl of $\underline{\alpha} (G)$, we relate it to the divergence of its Hodge dual. According to \cite[Equations~$(3.9)$--$(3.9')$]{CK}, for any $2$-form $H_{\mu \nu}$, we have
\begin{equation}\label{Hodgedualalpha}
{}^* \! \underline{\alpha}(H)_{e_A}:= {\varepsilon_A}^B \underline{\alpha} (H)_{e_B}, \qquad {}^* \! \underline{\alpha}(H)_{e_A}= \underline{\alpha} \big( {}^* \! H \big)_{e_A}, \qquad \qquad A \in \{\theta , \varphi \}. 
\end{equation}
It allows us to deduce that
\begin{equation}\label{Hodgedualalpha2}
\slashed{\nabla} \times \underline{\alpha}(G) = \varepsilon^{AB} \slashed{\nabla}_{e_A} \underline{\alpha} (G)_{e_B} = \slashed{\nabla}^{e_A} \underline{\alpha} \big( {}^* \! G \big)_{e_A} = \slashed{\nabla} \cdot \underline{\alpha} \big( {}^* \! G \big)
 \end{equation}
and we can apply the previous proof to ${}^*G={}^*F^{\mathrm{asymp}}[f_\infty]$. Indeed, the properties used on $ F^{\kappa_{ij},\mathrm{T}}$, $F^{\kappa_{ij}, \mathrm{sph}}$, $F^{\kappa_{ij}, \mathrm{hom}}$ and $\mathcal{L}_{\Omega_{ij}}\overline{F}$ hold as well for their Hodge dual. Consequently, using \eqref{Hodgedualalpha}--\eqref{Hodgedualalpha2}, we obtain
\begin{align*}
 \lim_{\underline{u} \to + \infty} \sum_{1 \leq i < j \leq 3} \int_{u=-\underline{u}+4}^{\underline{u}} r \underline{\alpha} \big({}^* \! F^{\kappa_{ij},\mathrm{sph}}\!+ {}^* \!F^{ \kappa_{ij},\mathrm{hom}} \big)_{\Omega_{ij}^{\mathrm{resc}}}(u,\underline{u},\omega) \dr u &= \int_{\R_u} \slashed{\nabla} \cdot {}^* \underline{\alpha}^{\mathrm{asymp}}[f_\infty] (u,\omega) \dr u , \\
\lim_{ \underline{u} \to + \infty}  \sum_{1 \leq i < j \leq 3} \int_{u=-\underline{u}+4}^{\underline{u}}  r \underline{\alpha} \big( {}^* \! F^{\kappa_{ij}, \mathrm{T}} \big)_{\Omega_{ij}^{\mathrm{resc}}}(u,\underline{u},\omega) \dr u &=  \int_{\tau=0}^{+\infty}  \frac{2 \tau^2}{\langle \tau \rangle + \tau} \, \slashed{\nabla} \cdot  \underline{\alpha} \big( {}^* \mathbb{F} [ f_\infty ] \big) (\tau \omega) \frac{\dr \tau}{\langle \tau \rangle^3}
\end{align*}
and
$$ \slashed{\nabla} \cdot {}^* \underline{\alpha}^{\mathrm{asymp}}[f_\infty]= \slashed{\nabla} \times \underline{\alpha}^{\mathrm{asymp}}[f_\infty] , \qquad \qquad \slashed{\nabla} \cdot \underline{\alpha} \big( {}^*  \mathbb{F} [ f_\infty] \big)  = \slashed{\nabla} \times \underline{\alpha} \big(   \mathbb{F} [ f_\infty] \big)  .$$
\end{proof}

We now improve the convergence estimate of Corollary \ref{Corconvasymp} for the Lorentz force $\widehat{v}^\mu F^{\mathrm{asymp}}_{\mu j}[f_\infty]$. It will allow us to slightly relax the velocity decay for $f_\infty$ and to simplify the analysis of Section \ref{SecVlasov}. For this, our proof will require to exploit improved decay estimates for the null components of $F^{\mathrm{hom}}$ in the interior of the light cone. 
\begin{Lem}\label{LemnullFhom}
For any $|\gamma| \leq N_0-1$, we have, for all $t \geq |x|+2$,
$$|\underline{\alpha} \big( F^{\gamma, \mathrm{hom}} \big) \big|(t,x) \lesssim \frac{\overline{\mathbb{E}}_{N_0+1}[f_\infty]}{t \, \langle t-|x| \rangle^2 }, \quad \qquad  \big( |\alpha \big( F^{\gamma, \mathrm{hom}} \big) \big|+|\rho \big( F^{\gamma, \mathrm{hom}} \big) \big|+|\sigma \big( F^{\gamma, \mathrm{hom}} \big) \big| \big) (t,x) \lesssim \frac{\overline{\mathbb{E}}_{N_0+1}[f_\infty]}{t^{2} \, \langle t-|x| \rangle }.$$
\end{Lem}
\begin{proof}
For the the null component $\underline{\alpha}$, it suffices to apply Proposition \ref{Prohompart}. The key idea in order to improve the estimates satisfied by the components $\alpha$, $\rho$ and $\sigma$ consists in exploiting certain of the Maxwell equations, written in null coordinates. The problem is that $F^{\gamma, \mathrm{hom}}$ is not solution solution to the Maxwell equations. However, one can prove that $F^{\gamma, \mathrm{hom}}$ is solution to the vacuum Maxwell equations in the interior of a light cone. Fix then $|\gamma| \leq N_0-1$ and remark that, since $\Box F^{\gamma , \mathrm{hom}}_{\mu \nu}=0$,
$$ \Box \,  \nabla^\mu F^{\gamma , \mathrm{hom}}_{\mu \nu} =0, \qquad \Box \,  \nabla^\mu {}^* \! F^{\gamma, \mathrm{hom}}_{\mu \nu} = 0, \qquad \nu \in \llbracket 0 , 3 \rrbracket.$$
By definition, we have further, for any $ \nu \in \llbracket 0 , 3 \rrbracket$,
\begin{alignat*}{2}
& \nabla^\mu F^{\gamma , \mathrm{hom}}_{\mu \nu}(t_0,\cdot) = \nabla^\mu \mathcal{L}_{Z^\gamma}( F)_{\mu \nu}(t_0, \cdot)=J^\gamma_\nu (t_0,\cdot), \qquad  && \; \; \;  \nabla^\mu {}^* \! F^{\gamma , \mathrm{hom}}_{\mu \nu}(t_0,\cdot) = \nabla^\mu {}^* \! \mathcal{L}_{Z^\gamma}( F)_{\mu \nu}(t_0, \cdot)=0 , \\
 \partial_t & \nabla^\mu F^{\gamma , \mathrm{hom}}_{\mu \nu}(t_0,\cdot) = \partial_t \nabla^\mu \mathcal{L}_{Z^\gamma}( F)_{\mu \nu}(t_0, \cdot)= \partial_t J^\gamma_\nu (t_0,\cdot), \qquad  && \partial_t \nabla^\mu {}^* \! F^{\gamma , \mathrm{hom}}_{\mu \nu}(t_0,\cdot) = \partial_t \nabla^\mu {}^* \! \mathcal{L}_{Z^\gamma}( F)_{\mu \nu}(t_0, \cdot)=0.
\end{alignat*}
It implies that $\nabla^\mu {}^* \! F^{\gamma, \mathrm{hom}}_{\mu \nu}=0$ on $[t_0,+ \infty [ \times \R^3$. Exploiting the wave equation, we get that all the derivatives $\partial_{t,x}^\kappa$, up to second order $|\kappa| \le 2$, of $\nabla^\mu  F^{\gamma, \mathrm{hom}}_{\mu \nu}(t_0, \cdot)$ are supported in $\{ |x| \leq t_0 \}$ since this property holds true for $J^\gamma (t_0,\cdot)$. As $t_0 \leq 2$, we get $\nabla^\mu  F^{\gamma, \mathrm{hom}}_{\mu \nu}=0$ on $\{ t \geq |x|+2 \}$ by Huygens' principle. 

Fix now $t \geq |x|+2$. If $t \geq 2|x|$, $t-|x| \geq t/2$ and it suffices to apply Proposition \ref{Prohompart}. Otherwise $t \leq 2|x|$, so that $|x| \gtrsim 1+t+|x|$ and let $\zeta \in \{ \alpha, \rho , \sigma \}$. Using the Maxwell equation \eqref{null1}, \eqref{null3} or \eqref{null4} and then the decay estimates given by Proposition \ref{Prohompart} as well as Remark \ref{Rqhompart},
$$ \big| \nabla_{\underline{L}} \zeta \big( F^{\gamma, \mathrm{hom}} \big) |(t,x) \lesssim |x|^{-1} \sup_{|\xi| \leq 1}  \big| \mathcal{L}_{Z^\xi} F^{\gamma, \mathrm{hom}} \big|(t,x) \lesssim \langle t+|x| \rangle^{-2} \langle t-|x| \rangle^{-2} \overline{\mathbb{E}}_{N_0+1}[f_\infty] .$$
In the same spirit as in the proof of Proposition \ref{Pronullcompoasymp}, consider $\omega=x/x$ and, for all $2 \leq u \leq \underline{u}$,
$$ \phi(u,\underline{u}):= \zeta \big( F^{\gamma, \mathrm{hom}} \big) \Big( \frac{\underline{u}+u}{2}, \frac{\underline{u}-u}{2}\omega \Big), \qquad |\phi|(u,\underline{u}) \lesssim \frac{\overline{\mathbb{E}}_{N_0+1}[f_\infty]}{\langle \underline{u} \rangle \, \langle u \rangle^{2} }, \qquad |\partial_u \phi|(u,\underline{u}) \lesssim \frac{\overline{\mathbb{E}}_{N_0+1}[f_\infty]}{\langle \underline{u} \rangle^{2} \langle u \rangle^{2} }.$$
Integrating in $u$, along the segment $u \mapsto (u, t+|x|)$ with $t-|x| \leq u \leq t+|x|$, yields
$$ \big|\zeta \big( F^{\gamma, \mathrm{hom}} \big) \big|(t,x) =|\phi|(t-|x|,t+|x|) \leq \big| \zeta \big( F^{\gamma, \mathrm{hom}} \big) \big|(t+|x|,0)+\langle t+|x| \rangle^{-2} \int_{u=t-|x|}^{t+|x|} \frac{\overline{\mathbb{E}}_{N_0+1}[f_\infty]}{ \langle u \rangle^{2} } \dr u,$$
which implies the result.
\end{proof}

We now investigate the null structure of the Lorentz force.

\begin{Lem}\label{nullstruct}
Let $H$ be a $2$-form defined on $\R_+ \times \R^3$. Then, for any $0 \leq \nu \leq 3$ and all $(t,x) \in \R_+ \times \R^3$,
$$ \left| \widehat{v}^\mu H_{\mu \nu}(t,x) \right| \lesssim \left( |\alpha (H)| +  |\rho(H)|+|\sigma (H)| \right)(t,x)+\big(\widehat{v}^{\underline{L}}(x)+|\widehat{\slashed{v}}(x)| \big) |\underline{\alpha} (H)|(t,x)  .$$
\end{Lem}
\begin{proof}
Since $H$ is a $2$-form and $\widehat{v}^0=1$, we have $\widehat{v}^\mu H_{\mu 0}=-\widehat{v}^\mu \widehat{v}^j  H_{\mu j} $ and it suffices to treat the case $\nu \in \llbracket 1,3 \rrbracket$. In order to lighten the notations, we drop in this proof the dependence in $(t,x)$ of the quantities that we will consider. We start by expanding $\widehat{v}^\mu H_{\mu j}$, for $1 \leq j \leq 3$, according to the null frame $(\underline{L},L,e_\theta, e_\varphi)$. We get
\begin{align}
 \left| \widehat{v}^\mu H_{\mu j} \right| &= \left|\widehat{v}^L H_{L j}+\widehat{v}^{\underline{L}}H_{\underline{L} j}+\widehat{v}^{e_\theta} H_{e_\theta j}+\widehat{v}^{e_\varphi} H_{e_\varphi j} \right| .
\end{align}
As $\widehat{v}^L \leq 1$ and $2\partial_{x^j}= \omega_j L-\omega_j \underline{L}+2\omega_j^{e_A} e_A$, with $\omega_j(x):=x_j/|x|$ and $|\omega_j^{e_A}| \leq 1$, we obtain
\begin{equation}\label{eq:toremember00}
\left|\widehat{v}^\mu H_{\mu j}\right| \lesssim |\alpha( H)|+\left( \widehat{v}^{\underline{L}}+|\omega_j| \right)|\rho (H)|+|\slashed{v}||\sigma (H)|+ \big( \widehat{v}^{\underline{L}}+|\widehat{\slashed{v}}||\omega_j| \big)| \underline{\alpha}(H)|.
\end{equation}
It remains to use $\widehat{v}^{\underline{L}}+|\omega_j| +|\widehat{\slashed{v}}| \lesssim 1$.
\end{proof}
\begin{Rq}\label{Rqimprofullnullstruct}
We will not need it in this article, but $(\delta_i^j-\widehat{v}^j \widehat{v}_i)\widehat{v}^\mu H_{\mu j}(t,x)$ enjoys better null properties than $\widehat{v}^\mu H_{\mu \nu}$. For this use \eqref{eq:toremember00} and observe that
\begin{align*}
|(\delta_i^j-\widehat{v}^j \widehat{v}_i)\omega_j|&=|\omega_i-\widehat{v}_i+\widehat{v}_i(1-\widehat{v}^j \omega_j)| \leq |\omega_i-\widehat{v}_i|+2 \widehat{v}^{\underline{L}}  \leq |\omega-\widehat{v}|+2 \widehat{v}^{\underline{L}} \\
 & \leq \left( |\omega|^2+|\widehat{v}|^2-2 \widehat{v} \cdot \omega \right)^{\frac{1}{2}}+2\widehat{v}^{\underline{L}} \leq \sqrt{2} \left|1-\widehat{v} \cdot \omega \right|^{\frac{1}{2}}+2\widehat{v}^{\underline{L}} = \sqrt{2} \left|v^{\underline{L}} \right|^{\frac{1}{2}}+2\widehat{v}^{\underline{L}} \leq (2+\sqrt{2})|\widehat{v}^{\underline{L}}|^{\frac{1}{2}}.
\end{align*}
\end{Rq}

We end this section by a convergence estimate for the Lorentz force.
\begin{Pro}\label{ProconvasympLor}
Let $|\gamma| \leq N_0-1$ and $1 \leq j \leq 3$. Then, for all $(t,x,v) \in [2,+\infty[ \times \R^3_x \times \R^3_v$,
$$  \Big|t^2 \widehat{v}^\mu \mathcal{L}_{Z^\gamma}\big(F^{\mathrm{asymp}}[f_\infty] \big)_{\mu j}(t,x+t\widehat{v})-\widehat{v}^\mu \mathbb{F}_{\mu j}  \big[ \widehat{Z}^\gamma_\infty f_\infty \big](v) \Big| \lesssim \bigg( \frac{ |\widehat{v}^{\underline{L}}(x+t\widehat{v}) |^{\frac{1}{2}} }{ \langle t-|x+t\widehat{v}| \rangle}  +\frac{1}{t} \bigg) \langle x \rangle \, \langle v \rangle \, \overline{\mathbb{E}}_{N_0+1}[f_\infty].$$
Note that the right hand side is bounded by $t^{-\frac{1}{2}}\, \langle x \rangle^{\frac{3}{2}} \, \langle v \rangle \, \overline{\mathbb{E}}_{N_0+1}[f_\infty]$ according to Lemma \ref{gainv}.
\end{Pro}
\begin{proof}
Start by applying Lemma \ref{nullstruct} to $G:=\mathcal{L}_{Z^\gamma}(F^{\mathrm{asymp}}[f_\infty])$ in order to get
$$ t^2 \big| \widehat{v}^\mu G_{\mu j}\big|(t,x+t\widehat{v}) \lesssim t^2 \!\big( |\alpha (G)| \!+\!  |\rho(G)| \! + \!|\sigma (G)| \big)(t,x+t\widehat{v})+t^2\big(\widehat{v}^{\underline{L}}(x+t\widehat{v})+|\widehat{\slashed{v}}(x+t\widehat{v})| \big) |\underline{\alpha} (G)|(t,x+t \widehat{v}) .$$
If $Z^\gamma$ contains at least one translation $\partial_{x^\lambda}$, that is $\gamma_T \geq 1$, the pointwise decay estimates for the null components of $G$, given by Proposition \ref{Pronullcompoasymp} and Lemma \ref{improderiv}, provide
$$ t^2 \big| \widehat{v}^\mu G_{\mu j}\big|(t,x+t\widehat{v}) \lesssim \overline{\mathbb{E}}_{N_0+1}[f_\infty] \left[ \langle t-|x+t\widehat{v}| \rangle^{-1}+t \, \langle t-|x+t\widehat{v}| \rangle^{-2}  \big(\widehat{v}^{\underline{L}}(x+t\widehat{v})+|\widehat{\slashed{v}}(x+t\widehat{v})| \big) \right].$$
It remains to use $\mathbb{F} \big[ \widehat{Z}^\gamma_\infty f_\infty \big] =0$ and to apply Lemma \ref{gainv}, which gives
\begin{equation}\label{eq:gainvv}
1 \lesssim \langle v \rangle \big| \widehat{v}^{\underline{L}} \big|^{1/2}, \qquad t \big( \widehat{v}^{\underline{L}}(x+t\widehat{v})+|\widehat{\slashed{v}}(x+t\widehat{v})| \big) \lesssim \langle x \rangle \, \max \big(1 , t- |x-t\widehat{v} | \big).
\end{equation}
Assume now $Z^\gamma$ is only composed by homogeneous vector fields. We treat first the case $|x+t\widehat{v}| \geq t-2$. Then, as $\big|\mathbb{F} \big[ \widehat{Z}^\gamma_\infty f_\infty \big] \big|(v) \lesssim   \overline{\mathbb{E}}_{N_0+1}[f_\infty]$ by Proposition \ref{Proasympelec}, the decay estimates of Proposition \ref{Pronullcompoasymp} yield
$$
t^2 \big| \widehat{v}^\mu G_{\mu j}\big|(t,x+t\widehat{v})+\big|\widehat{v}^\mu \mathbb{F}_{\mu j} \big[ \widehat{Z}^\gamma_\infty f_\infty \big](v) \big| \lesssim \overline{\mathbb{E}}_{N_0+1}[f_\infty] \left[1+t \, \langle t-|x+t \widehat{v} | \rangle^{-1}  \big(\widehat{v}^{\underline{L}}(x+t\widehat{v})+|\widehat{\slashed{v}}(x+t\widehat{v})| \big) \right] .$$
The result follows from \eqref{eq:gainvv} as well as $|x| \geq |x+t\widehat{v}|-t \geq -2$. Finally, if $|x+t\widehat{v}| \leq t-2$, since $G=F^{\gamma, \mathrm{hom}}+F^{\gamma, \mathrm{inh}}$, we have, according to Proposition \ref{Proinhom},
$$ t^2 \widehat{v}^\mu G_{\mu j}(t,x+t\widehat{v})-\widehat{v}^\mu \mathbb{F}_{\mu j}\big[ \widehat{Z}^\gamma_\infty f_\infty \big](v)=t^2 \widehat{v}^\mu F^{\gamma, \mathrm{hom}}_{\mu j}(t,x+t\widehat{v})+\widehat{v}^\mu \mathbb{F}_{\mu j}\big[ \widehat{Z}^\gamma_\infty f_\infty \big] \bigg( \frac{\widecheck{ x+t\widehat{v}}}{t} \bigg)-\widehat{v}^\mu \mathbb{F}_{\mu j}\big[ \widehat{Z}^\gamma_\infty f_\infty \big](v) .$$
We can bound the last two terms on the right hand side by $\langle x \rangle \langle v \rangle \, t^{-1} \overline{\mathbb{E}}_{N_0+1}[f_\infty]$ by interpolation. Indeed, they are both controlled by $\overline{\mathbb{E}}_{N_0+1}[f_\infty]$ and their difference is bounded by $\langle x \rangle \, t^{-1} \min(1,\langle v \rangle^2 t^{-1})  \overline{\mathbb{E}}_{N_0+1}[f_\infty]$ according to \eqref{kevatalenn:Fasymp}. It then remains us to control sufficiently well $t^2 \widehat{v}^\mu F^{\gamma, \mathrm{hom}}_{ \mu j} (t,x+t \widehat{v})$. For this, we apply Lemma \ref{nullstruct} in order to get
\begin{align*}
 t^2 \big| \widehat{v}^\mu F^{\gamma , \mathrm{hom}}_{\mu j}\big|(t,x+t\widehat{v}) & \lesssim t^2 \!\big( |\alpha (F^{\gamma , \mathrm{hom}})| \!+\!  |\rho(F^{\gamma , \mathrm{hom}})| \! + \!|\sigma (F^{\gamma , \mathrm{hom}})| \big)(t,x+t\widehat{v}) \\
 & \quad +t^2\big(\widehat{v}^{\underline{L}}(x+t\widehat{v})+|\widehat{\slashed{v}}(x+t\widehat{v})| \big) \big|\underline{\alpha} (F^{\gamma , \mathrm{hom}}) \big|(t,x+t \widehat{v}) .
 \end{align*}
We then obtain from Lemma \ref{LemnullFhom} and \eqref{eq:gainvv} that
$$ t^2 \big| \widehat{v}^\mu F^{\gamma , \mathrm{hom}}_{\mu j}\big|(t,x+t\widehat{v}) \lesssim \overline{\mathbb{E}}_{N_0+1}[f_\infty] \, \langle t-|x+t\widehat{v}| \rangle^{-1} \langle x \rangle \lesssim  \big| v^{\underline{L}}(x+t\widehat{v}) \big|^{\frac{1}{2}} \langle t-|x+t\widehat{v}| \rangle^{-1} \langle x \rangle \, \langle v \rangle \, \overline{\mathbb{E}}_{N_0+1}[f_\infty] .$$
\end{proof}
\begin{Rq}
Let us mention that a refined estimate could provide an even stronger convergence estimate for $t^2(\delta_i^j-\widehat{v}_i \widehat{v}^j) \widehat{v}^\mu \mathcal{L}_{Z^\gamma}\big(F^{\mathrm{asymp}}[f_\infty] \big)_{\mu j}(t,x+t\widehat{v})$. In particular, exploiting Remark \ref{Rqimprofullnullstruct}, we could prove an estimate uniform in $v$. For this, it would also be important to prove that $\langle v \rangle^2 \, |\alpha (\mathbb{F}[f_\infty])|(v) \lesssim \overline{\mathbb{E}}_{N_0+1}[f_\infty]$ by exploiting the strong decay estimate satisfied by $\alpha (F^{\mathrm{asymp}}[f_\infty])$.
\end{Rq}

\section{Set up of the Picard iteration and strategy of the proof}\label{SecPicard} 
We fix, for the rest of this paper, an initial time $T \geq 2$ and two constants $\Lambda \geq 1$, $\epsilon >0$ such that $\Lambda \geq \epsilon$. We al so consider $N \geq 8$, as well as a distribution function $f_\infty : \R^3_z \times \R^3_v \to \R$ and a radiation field $\underline{\alpha}^\infty$ satisfying the assumptions of Theorem \ref{Theo1}. We assume further 
\begin{equation}\label{eq:extracondi}
\overline{\mathbb{E}}_{N+1}[f_\infty]= \sup_{|\kappa_z|+|\kappa_v| \leq N+1}  \, \sup_{ \R^3_z \times \R^3_v}  \langle z \rangle^{4+|\kappa_z|} \, \langle v \rangle^{7+3|\kappa_v|} \big|  \partial_z^{\kappa_z} \partial_v^{\kappa_v}f_\infty \big|(z,v)\leq \sqrt{\epsilon},
\end{equation}
so that the result of Section \ref{SecMaxasymp} holds for the order of regularity $N_0=N$. The field $F^{\mathrm{asymp}}[f_\infty]$ will denote the solution to \eqref{eq:defasympF} with initial time $t_0=1$. This extra assumption is automatically verified if $N = 8$ but requires more decay in $v$ for $f_\infty$ when $N \geq 9$. It is only significant when $N$ is large. It will allow us to work with a much more convenient functional space and to simplify two parts of the proof. We detail how to deal with the general case in Section \ref{Secweakerassump}.

For convenience, and since a smallness assumption will be required on this quantity, we further impose $\epsilon \log^{-1}(T) \leq 1$. In particular, we then have according to Proposition \ref{Proasympelec} that $|\C_{t,v}| \lesssim \sqrt{\epsilon} \log(t) \lesssim \log^2(t)$, where $\C_{t,v}$ is introduced in \eqref{defXC}.

\subsection{Sequence of approximate solutions}

The proof of Theorem \ref{Theo1} is based on a Picard iteration. More precisely, we will prove that the following sequence $(f_n,F^n)_{n \geq 1}$ is well-defined.
\begin{itemize}
\item We consider first
$$  f_1(t,x,v):=f_\infty(t,x-t\widehat{v},v),$$
so that $f_1$ is the solution of the linear Vlasov equation $\T_{0}(f_1)=\widehat{v}^\mu \partial_{x^\mu}f_1 =0$ satisfying $f_1(0,\cdot  , \cdot)=f_\infty$. 
\item Assume that we have constructed the distribution function $f_n :[T,+\infty[ \times \R^3_x \times \R^3_v \to \R_+$, for $n \in \mathbb{N}^*$. We define $F^{n}$ as the unique solution to
\begin{equation}\label{Cauchy1}
 \nabla^\mu F^{n}_{\mu \nu} = J(f_n)_\nu, \qquad \nabla^\mu {}^* \! F^{n}_{\mu \nu} =0, \qquad \qquad \lim_{r \to + \infty} r\underline{\alpha}(F^{n})(r+u,r\omega)=\underline{\alpha}^{\infty}(u,\omega) .
 \end{equation}
\item Suppose that we have constructed the electromagnetic field $F^n$, defined on $[T,+\infty[ \times \R^3$, for $n \in \mathbb{N}^*$. Then, the distribution function $f_{n+1}$ is defined as the unique function satisfying
\begin{equation}\label{Cauchy2}
 \T_{F^n} (f_{n+1}) =0, \qquad  \qquad \lim_{t \to +\infty} f_{n+1}\big(t,x+t\widehat{v}+\C_{t,v},v \big)=f_\infty (z,v).
 \end{equation}
\end{itemize}
For this, we will have to justify that \eqref{Cauchy1} and \eqref{Cauchy2} both admit a unique classical solution. Then, we will prove that $(f_n,F^n)_{n \geq 1}$ is a Cauchy sequence in a well-chosen Banach space.
\begin{Rq}
$(f_1,F^1)$ is a solution to the linearised system \eqref{VM1lin}--\eqref{VM3lin} around the trivial solution.
\end{Rq}

\subsection{Functional spaces}

For $n \geq 1$, the function $f_n$ will belong to the following set.
\begin{Def}\label{DeffuncspaVla}
Let $D\geq 1$ and $\mathbb{V}^{D,\epsilon}_N$ be the set of the distribution functions $f : [T,+\infty[ \times \R^3_x \times \R^3_v \to \R$ such that $g(t,z,v):=f(t,x+t\widehat{v}+\C_{t,v},v)$ verifies the bound
$$\mathbf{E}_N^{7,13}[g]:=  \sup_{|\beta| \leq N} \sup_{t \geq T} \int_{\R^3_z} \int_{\R^3_v} \langle z \rangle^{14+2N-2\beta_H}\, \langle v \rangle^{26}  \big| \widehat{Z}^\beta_\infty g (t,z,v) \big|^2 \dr z \dr v  \leq D \epsilon .$$
\end{Def}

For the electromagnetic fields $F^n$, we will make use of the norms introduced in \eqref{keva:defenergy2}. 
\begin{Def}\label{DefsetMax}
Let $D \geq 1$ and $\mathbb{M}_N^{D,\Lambda}$ be the set of the $2$-forms $F$ defined on $[T,+\infty[ \times \R^3$ and verifying
$$ \sup_{t \geq T} \mathcal{E}^K_{N-1} \big[ F-F^{\mathrm{asymp}}[f_\infty] \big](t)+ 
 \, \sup_{|\gamma| = N-1} \, \sup_{t \geq T} \mathcal{E}^{K,1} \big[ \nabla_{t,x} \mathcal{L}_{Z^\gamma} ( F-F^{\mathrm{asymp}}[f_\infty] ) \big](t)  \leq D \Lambda .$$
\end{Def}

We now derive the properties that we will use for the electromagnetic fields $F \in \mathbb{M}^{D,\Lambda}_N$. For this, we will make use of the following result, suggesting that the velocity average of a distribution function satisfying modified scattering still decays as $t^{-3}$.
\begin{Lem}\label{decayparticucase}
There exists an absolute constant $\delta>0$ such that, if $\sqrt{\epsilon} T^{-1} \log(T) \leq \delta$, then
$$ \forall \, (t,x) \in [T,+\infty[ \times \R^3, \qquad  \int_{\R^3_v} \frac{\dr v}{ \langle x-t\widehat{v}-\C_{t,v} \rangle^{3} \, \langle v \rangle^5}  \lesssim \frac{\log(t)}{\langle t+|x|\rangle^3}.$$
\end{Lem}
\begin{proof}
Let us justify that we can perform the change of variables $y=x-t\psi_t(v)$, where $\psi_t(v):=\widehat{v}+t^{-1}\C_{t,v}$.
Recall now from Lemma \ref{cdv} that the Jacobian determinant of $v \mapsto \widehat{v}$ is equal to $-\langle v \rangle^{-5}$. Moreover, as $|\mathbb{F}_{\mu \nu}[f_\infty](v)|+|v^0 \partial_{v^k} \mathbb{F}_{\mu \nu}[f_\infty](v) | \lesssim \sqrt{\epsilon}$ according to Proposition \ref{Proasympelec}, we have
$$ |\C_{t,v}| \lesssim \sqrt{\epsilon} \log(t) \langle v \rangle^{-1} , \qquad \qquad | \partial_{v^k} \C^i_{t,v} | = \log(t) \Big| \partial_{v^k} \Big( \, \frac{\delta^j_i-\widehat{v}_i \widehat{v}^j}{v^0}  {\mathbb{F}_{\mu j}} [f_\infty](v) \Big) \Big| \lesssim \sqrt{\epsilon} \log(t) \langle v \rangle^{-2} .$$
The Jacobian determinant of $v \mapsto t^{-1}\C_{t,v}$ is then bounded by $\epsilon^{\frac{3}{2}} t^{-3} \log^3(t) \langle v \rangle^{-6}$. Hence, there exists $\delta >0$ such that, if $\sqrt{\epsilon} T^{-1} \log(T) \leq \delta$, then 
$$ \forall \, t \geq T, \qquad |\det \psi_t (v)| \geq \langle v \rangle^{-5}/2, \qquad |\psi_t(v)|\leq 2 . $$
 We then deduce that
$$ \forall \, t \geq T, \qquad t^3 \,\int_{\R^3_v} \frac{\dr v}{ \langle x-t\widehat{v}-\C_{t,v} \rangle^{3} \, \langle v \rangle^5} \leq \int_{|y-x|\leq 2t} \frac{\dr y}{ \langle y \rangle^3} \leq 4\log(t).$$
It concludes the proof if $|x| \leq 3t$. Otherwise use that we have $|x-t\widehat{v}-\C_{t,v}| \leq |x|-t-\sqrt{\epsilon} \log(t) \leq |x|/3$.
\end{proof}

\begin{Pro}\label{PropropMaxfield}
Let $F \in \mathbb{M}_N^{D,\epsilon}$. Then, the following pointwise decay estimates hold for any $|\gamma| \leq N-3$,
\begin{alignat}{2}
 & \hspace{-2mm} \forall \, (t,x) \in [T,+\infty[ \times \R^3, \qquad \quad \qquad &&\left| \mathcal{L}_{Z^\gamma}F \right| (t,x)  \lesssim  \sqrt{\Lambda} \langle t+|x| \rangle^{-1} \langle t-|x| \rangle^{-1}, \label{eq:Mpoint} \\[3pt]
 & \hspace{-2mm} \forall \, (t,x) \in [T,+\infty[ \times \R^3, \quad \qquad \qquad &&\big(|\alpha (\mathcal{L}_{Z^{\gamma}} F)|+|\rho (\mathcal{L}_{Z^{\gamma}} F)|+|\sigma (\mathcal{L}_{Z^{\gamma}} F)| \big)(t,x)  \lesssim  \sqrt{\Lambda} \, \langle t+|x| \rangle^{-2} , \label{eq:Mpointnull} 
 \end{alignat}
 as well as, for all $(t,x,v) \in [T, + \infty[ \times \R^3_x \times \R^3_v$ and any $1 \leq j \leq 3$,
\begin{equation}
\left| t^2\widehat{v}^\mu \mathcal{L}_{Z^\gamma} (F)_{\mu j}(t,x+t\widehat{v})- \widehat{v}^\mu  \mathbb{F}_{\mu j}\big[ \widehat{Z}_\infty^\gamma f_\infty \big] (v)\right| \lesssim \sqrt{\Lambda} \bigg( \, \frac{ |v^{\underline{L}}(x+t\widehat{v}) |^{\frac{1}{2}}}{ \langle t-|x+t\widehat{v}| \rangle^{\frac{1}{2}}} +\frac{1}{t} \bigg) \langle x \rangle \, \langle v \rangle . \label{eq:Mpointconv}
\end{equation}
Moreover, we have for any $N-3 \leq |\xi| \leq N-1$ and $N-2 \leq |\gamma| \leq N-1$,
\begin{alignat}{2}
& \hspace{-7mm} \forall \, t \geq T, \qquad  &&\int_{\R^3_x} \max (t-|x|,1)\, \langle t-|x| \rangle^2 \big| \nabla_{t,x} \mathcal{L}_{Z^\xi} F \big|^2(t,x) \mathrm{d} x  \lesssim  \Lambda \log^2(t), \label{eq:ML2}  \\
& \hspace{-7mm} \forall \, t \geq T, \qquad  &&  \int_{\R^3_x} \int_{\R^3_v}  \left| t^2\widehat{v}^\mu \mathcal{L}_{Z^\gamma} (F)_{\mu j}(t,x) - \widehat{v}^\mu \mathbb{F}_{\mu j}\big[ \widehat{Z}_\infty^\gamma f_\infty \big] (v)  \right|^2 \frac{\dr v \mathrm{d} x}{\langle x-t\widehat{v} -\C_{t,v} \rangle^{7} \langle v \rangle^{13}}  \lesssim  \Lambda \frac{\log^9(t)}{t} .  \label{eq:ML2conv}
\end{alignat}
\end{Pro}
\begin{proof}
We recall, since we will often use it here, that $\epsilon \leq \Lambda$ and $\Lambda \geq 1$. By applying Proposition \ref{decayMaxell} to $H:=\mathcal{L}_{Z^\gamma}(F-F^{\mathrm{asymp}}[f_\infty])$, for $|\gamma| \leq N-3$, we obtain that 
\begin{align*}
|\underline{\alpha}(H)|(t,x) & \lesssim \langle t+|x| \rangle^{-1} \, \langle t-|x| \rangle^{-\frac{3}{2}}   \left| \mathcal{E}^K_{N-1}\big[F-F^{\mathrm{asymp}}[f_\infty] \big](t) \right|^{\frac{1}{2}} ,\\
|\alpha ( H) |(t,x)+ |\rho(H)|(t,x)+ |\sigma (H)|(t,x) & \lesssim  \langle t+|x| \rangle^{-2} \,  \langle t-|x| \rangle^{-\frac{1}{2}}  \left| \mathcal{E}^K_{N-1}\big[F-F^{\mathrm{asymp}}[f_\infty] \big](t) \right|^{\frac{1}{2}} ,
 \end{align*}
for all $(t,x) \in [T,+ \infty[ \times \R^3$. Together with Proposition \ref{Pronullcompoasymp}, where the the null components of $\mathcal{L}_{Z^\gamma}F^{\mathrm{asymp}}[f_\infty]$ are estimated, it yields \eqref{eq:Mpoint}--\eqref{eq:Mpointnull}. Recall now from Lemma \ref{gainv} that $\langle t+|x| \rangle ( \widehat{v}^{\underline{L}}+|\slashed{\widehat{v}}|)\lesssim \langle t- |x| \rangle \, \langle x-t \widehat{v} \rangle$ and $1 \lesssim \langle v \rangle | \widehat{v}^{\underline{L}} |^{1/2}$. By Lemma \ref{nullstruct}, we then get
$$ t^2 \big| \widehat{v}^\mu H_{\mu j} \big|(t,x) \lesssim \big| v^{\underline{L}}(x) \big|^{\frac{1}{2}} \langle v \rangle \,t^2 \big( |\alpha ( H) |+ |\rho(H)|+ |\sigma (H)|\big)(t,x) +\big| v^{\underline{L}}(x) \big|^{\frac{1}{2}} \langle x-t\widehat{v} \rangle \,  \langle v \rangle \frac{t^2\langle t-|x| \rangle }{\langle t+|x| \rangle} |\underline{\alpha}(H)|(t,x)  .$$
Using the estimates for the null components of $H$, we then deduce that
$$ t^2 \big| \widehat{v}^\mu H_{\mu j} \big|(t,x+t\widehat{v}) \lesssim  \sqrt{\Lambda} \big| v^{\underline{L}}(x+t\widehat{v}) \big|^{\frac{1}{2}} \langle x \rangle \, \langle v \rangle \,\langle t-|x-t\widehat{v}| \rangle^{-\frac{1}{2}}  ,$$
so that \eqref{eq:Mpointconv} follows from the corresponding estimate for $\mathcal{L}_{Z^\gamma} F^{\mathrm{asymp}}[f_\infty]$, given by Proposition \ref{ProconvasympLor}, and $\epsilon \leq \Lambda$. Next we fix $N-3 \leq |\xi| \leq N-1$ and we remark that Proposition \ref{Proasymp} implies 
$$\int_{\R^3_x}  \frac{ \langle t-|x| \rangle^3}{\log^2(1+\langle t-|x| \rangle) } \big| \nabla_{t,x} \mathcal{L}_{Z^\xi} F^{\mathrm{asymp}}[f_\infty] \big|^2(t,x) \mathrm{d} x \lesssim \int_{\R^3_x} \frac{ \epsilon \,  \mathrm{d} x }{ \langle t+|x| \rangle^2 \langle t-|x| \rangle \, \log^2(1+\langle t-|x| \rangle)} \lesssim \epsilon .$$
If $|\xi|=N-1$, we then obtain \eqref{eq:ML2} by using the $L^2_x$ estimate for the top order derivatives of $F \in \mathbb{M}^{D,\Lambda}_N$. Otherwise, applying Lemma \ref{improderiv}, in order to gain $t-r$ decay through $\nabla_{t,x}$, we get
\begin{equation}\label{eq:technicalloworder}
 \int_{\R^3_x}  \langle t-|x| \rangle^{4} \big| \nabla_{t,x} \mathcal{L}_{Z^\xi} \big(F-F^{\mathrm{asymp}}[f_\infty] \big) \big|^2(t,x) \mathrm{d} x  \lesssim \mathcal{E}^K_{N-1} \big[ F-F^{\mathrm{asymp}}[f_\infty] \big](t) \leq D \Lambda .
 \end{equation}
Finally, using first Proposition \ref{ProconvasympLor} and then $|\C_{t,v}| \lesssim \sqrt{\epsilon} \log(t)\lesssim \log^2(t)$ as well as the change of variables $y=x-t\widehat{v} - \C_{t,v}$,
\begin{align*}
& \int_{\R^3_x} \int_{\R^3_v}  \left| t^2\widehat{v}^\mu \mathcal{L}_{Z^\gamma} (F^{\mathrm{asymp}}[f_\infty])_{\mu j}(t,x) - \widehat{v}^\mu \mathbb{F}_{\mu j}\big[ \widehat{Z}_\infty^\gamma f_\infty \big] (v)  \right|^2 \frac{\dr v \mathrm{d} x}{\langle x-t\widehat{v}-\C_{t,v} \rangle^{7} \langle v \rangle^{13}} \\
&  \qquad \qquad \qquad \qquad  \lesssim \int_{\R^3_x} \int_{\R^3_v}   \frac{\epsilon \, \langle x-t\widehat{v} \rangle^3 \dr v \mathrm{d} x}{t \, \langle x-t\widehat{v}-\C_{t,v} \rangle^7 \langle v \rangle^{11}}  \lesssim \frac{\epsilon \, \langle  \log^2(t) \rangle^3 }{t} \int_{\R^3_x} \int_{\R^3_v}   \frac{\dr v \mathrm{d} x}{ \langle x-t\widehat{v}-\C_{t,v} \rangle^4 \langle v \rangle^{11}} \lesssim \frac{ \epsilon \log^6(t)}{ t}.
\end{align*}
Let $G:= F- F^{\mathrm{asymp}}[f_\infty]$ and remark that Lemma \ref{gainv} implies
$$ \int_{\R^3_x} \int_{\R^3_v}   \frac{|t^2  \mathcal{L}_{Z^\gamma} G  |^2(t,x) \dr v \mathrm{d} x}{\langle x-t\widehat{v}-\C_{t,v} \rangle^{7} \langle v \rangle^{13}} \lesssim  \int_{\R^3_x} \int_{\R^3_v} \big| \max (t-|x|,1)\big|^4 \left|  \mathcal{L}_{Z^\gamma} G  \right|^2(t,x) \frac{\langle x-t \widehat{v} \rangle^4 \,  \dr v \mathrm{d} x}{\langle x-t\widehat{v}-\C_{t,v} \rangle^{7} \langle v \rangle^{5}}  .$$
Using again $|\C_{t,v}| \lesssim \log^2(t)$ and applying Lemma \ref{decayparticucase}, one obtains
 $$ \int_{\R^3_x} \int_{\R^3_v}   \frac{|t^2  \mathcal{L}_{Z^\gamma} G  |^2(t,x) \dr v \mathrm{d} x}{\langle x-t\widehat{v}-\C_{t,v} \rangle^{7} \langle v \rangle^{13}} \lesssim  \log^9(t) \int_{\R^3_x}  \big| \max(t-|x|,1) \big|^4 \left|  \mathcal{L}_{Z^\gamma} G  \right|^2(t,x) \frac{\dr x}{\langle t+|x| \rangle^3} .$$
If $|\gamma| \leq N-1$, the right hand side is bounded by
$$\log^9(t) t^{-1} \mathcal{E}^K_{N-1}\big[F-F^{\mathrm{asymp}}[f_\infty] \big](t) \leq D \Lambda \log^9(t) t^{-1}.$$  
Otherwise, there exists $|\kappa|=N-1$ such that $Z^\gamma=ZZ^\kappa$. Then $|\mathcal{L}_{Z^\gamma}G| \lesssim \langle t+r \rangle |\nabla_{t,x}  \mathcal{L}_{Z^\kappa}G|$ and the right hand side of the previous inequality is this time bounded above by
$$  \log^9(t) t^{-1} \mathcal{E}^{K,1} \big[ \nabla_{t,x} \mathcal{L}_{Z^\kappa} \big(F-F^{\mathrm{asymp}}[f_\infty] \big) \big](t) \leq D \Lambda \log^9(t) t^{-1}.$$  
\end{proof}

Let us finally prove an inequality which will be used several times in the next section.
\begin{Lem}\label{gainvadapted}
For all $(t,z,v) \in [T,+\infty[ \times \R^3_z \times \R^3_v$, we have
$$ \langle t+|\XX_\C| \rangle \lesssim \log^2(t) \langle t-|\XX_\C| \rangle \, \langle v \rangle^2 \, \langle z \rangle .$$
\end{Lem}
\begin{proof}
Apply Lemma \ref{gainv} for $x=\XX_\C(t,z,v)$ and use $|\C_{t,v}|\lesssim  \log^2(t)$, so that $\langle z-\C_{t,v} \rangle \lesssim \langle z \rangle \, \log^2(t)$.
\end{proof}
\subsection{The Picard iteration and strategy of the proof}

The main part of the remainder of this article will be devoted to the proof of the following proposition.

\begin{Pro}\label{Picard}
There exists $D \geq 1$ and $B>0$, both depending only $N$, and an absolute constant $\overline{\varepsilon} >0$, such that the following properties hold. If $B_\Lambda \epsilon \, \log^{-1}(T) \leq \overline{\varepsilon}$, where $B_\Lambda := \exp \big( \exp \big( B \sqrt{\Lambda} \big) \big)$, then
\begin{enumerate}
\item $F^1$ is well defined and belongs to $\mathbb{M}_N^{D,\Lambda}$.
\item For any $F \in \mathbb{M}_N^{D,\Lambda}$, there exists a unique function $f : [T,+\infty[ \times \R^3_x \times \R^3_v \to \R$ such that 
$$\T_F(f)=0, \qquad \lim_{t \to + \infty} f(t,x+t\widehat{v}+\C_{t,v},v) = f_\infty (x,v), \qquad f \in \mathbb{V}^{B_\Lambda,\epsilon}_{N}.$$ 
Moreover, there exists a unique classical solution $F^{\mathrm{new}}$ to the Maxwell equations with asymptotic data
\begin{equation}\label{kevatalenn:trede}
\nabla^\mu G_{\mu \nu} = J(f)_\nu, \qquad \nabla^\mu {}^* \! G_{\mu \nu} =0, \qquad \lim_{r \to + \infty} r \underline{\alpha} (G)(r+u,r\omega)=\underline{\alpha}^{\infty} (u, \omega ),
\end{equation}
and $F^{\mathrm{new}} \in \mathbb{M}^{D,\Lambda}_N$.
\end{enumerate}
\end{Pro}

This result allows us to uniquely define the sequence $(f_n , F_n)_{n \geq 1} \in  ( \mathbb{V}^{B_\Lambda,\epsilon}_{N} \times \mathbb{M}^{D,\Lambda}_N \, )^{\mathbb{N}^*}$. Section \ref{SecVlasov} will be devoted to the analysis of the Vlasov equation and Section \ref{SecMax} to the asymptotic Cauchy problem \eqref{kevatalenn:trede} as well as the construction of $F^1$. Then, in Section \ref{SecCauchy}, we prove that $(f_n , F_n)_{n \geq 1}$ is a Cauchy sequence in $L^\infty \big( [T,+ \infty[ , L^2(\R^3_x \times \R^3_v) \times L^2(\R^3_x) \big)$, allowing us to conclude the proof of Theorem \ref{Theo1} if the additional assumption \eqref{eq:extracondi} holds. Finally, in Section \ref{Secweakerassump}, we treat the general case. In this more general setting, $\mathbb{M}^{D,\Lambda}_N$ is not necessarily well-defined so we need to consider a slightly different functional space. Two steps are then slightly more technical, the proof of Proposition \ref{PropropMaxfield} and part of the analysis of \eqref{kevatalenn:trede}.

\section{The Vlasov equation with asymptotic data}\label{SecVlasov}
We fix, for all this section as well as the next one, Section \ref{SecMax}, $F \in \mathbb{M}_N^{D,\Lambda}$. When there is no risk of confusion, we will simply write $\XX_\C$ instead of $\XX_\C(t,z,v)=z+t\widehat{v}+\C_{t,v}$, where the correction term is given in \eqref{defXC}.
\subsection{Preparatories}

We start by proving results which will be useful in order to state the first and the higher order commutation formula. We will then be able to propagate regularity on the solution. In order to simplify the presentation of certain computations, we introduce the following notation.

\begin{Def}\label{DefDelta}
Let $|\gamma| \leq N$, $1 \leq i \leq 3$ and consider
$$ \Delta_{Z^\gamma }^i(t,z,v):= \frac{\widehat{v}^\mu}{v^0} \Big( t^2 \mathcal{L}_{Z^{\gamma}}(F)_{\mu j}\big(t,\XX_\C(t,z,v)\big)-  \mathbb{F}_{\mu j}\big[ \widehat{Z}_\infty^\gamma f_\infty \big](v) \Big).$$
When $|\gamma|=0$, we simply write $\Delta^i$.
\end{Def}
\begin{Rq}\label{DeltatranslaS}
Since $\mathbb{F}_{\mu \nu}[\widehat{Z}_\infty^\xi f_\infty](v) =0$ if the multi-index $\xi$ is such that $Z^\xi$ contains at least one translation $\partial_{x^\lambda}$ or the scaling vector field $S$, we have, for any $|\gamma| \leq N-1$ and $\lambda \in \llbracket 0 , 3 \rrbracket$,
$$ \Delta_{\partial_{x^\lambda} Z^\gamma}^i (t,z,v) = t^2\frac{\widehat{v}^\mu}{v^0}    {\mathcal{L}_{\partial_{x^\lambda} Z^\gamma}(F)_{\mu}}^j\big(t,\XX_\C\big), \qquad \qquad \Delta_{S Z^\gamma}^i (t,z,v) = t^2\frac{\widehat{v}^\mu}{v^0}   {\mathcal{L}_{S Z^\gamma}(F)_{\mu}}^j\big(t,\XX_\C\big).$$
\end{Rq}

The following results will be useful in order to commute the Vlasov equation. 

\begin{Lem}\label{LemCom1}
Let $G$ be a sufficiently regular $2$-form defined on $[T,+\infty[ \times \R^3$. Then, denoting $\XX_\C(t,z,v)$ by $\XX_\C$, we have for all $(t,z,v) \in [T,+ \infty[ \times \R^3_z \times \R^3_v$,
\begin{align*}
 \partial_{z^k} \big[ v^\mu G_{\mu \ell} (t,\XX_\C)  \big] &=v^\mu \mathcal{L}_{\partial_{x^k} }(G)_{\mu \ell} (t,\XX_\C) , \\
  \partial_t^\infty \big[ v^\mu G_{\mu \ell} (t,\XX_\C)  \big]& =v^\mu \mathcal{L}_{\partial_{t} }(G)_{\mu \ell} (t,\XX_\C)+ \partial_t \big( \C_{t,v}^n \big) v^\mu \mathcal{L}_{\partial_{x^n}}(G)_{\mu \ell}(t,\XX_\C), \\
  S \big[ v^\mu G_{\mu \ell} (t,\XX_\C)  \big] &= \Big[ v^\mu \mathcal{L}_{S }(G)_{\mu \ell}-2v^\mu G_{\mu \ell}  +t\partial_t \big( \C_{t,v}^n \big) v^\mu \mathcal{L}_{\partial_{x^n}}(G)_{\mu \ell}\Big](t,\XX_\C), \\
  \widehat{\Omega}_{ij}  \big[ v^\mu G_{\mu \ell} (t,\XX_\C)  \big] &= \left[v^\mu \mathcal{L}_{\Omega_{ij} }(G)_{\mu \ell} -\delta_\ell^i  v^\mu G_{\mu j}+ \delta_\ell^j v^\mu G_{\mu i}\right] \!(t,\XX_\C) \\
 &  \quad + \left[\C^j_{t,v} v^\mu\mathcal{L}_{\partial_{x^i}}(G)_{\mu \ell}-\C^i_{t,v} v^\mu\mathcal{L}_{\partial_{x^j}}(G)_{\mu \ell}  +\widehat{\Omega}_{ij}\big( \C^n_{t,v} \big) v^\mu \mathcal{L}_{\partial_{x^n}}(G)_{\mu \ell}\right]\! (t,\XX_\C),\\
  \widehat{\Omega}_{0k}^\infty  \big[ v^\mu G_{\mu \ell} (t,\XX_\C)  \big] &= \left[ v^\mu \mathcal{L}_{\Omega_{0k} }(G)_{\mu \ell} +\delta_\ell^k \widehat{v}^n v^\mu G_{\mu n}-\C^k_{t,v} v^\mu \mathcal{L}_{\partial_t}(G)_{\mu \ell} +\widehat{\Omega}_{0k}^\infty \big( \C^n_{t,v} \big)v^\mu \mathcal{L}_{\partial_{x^n}}(G)_{\mu \ell} \right] \!(t,\XX_\C),
\end{align*}
for any $1 \leq i < j \leq 3$ and $1 \leq k, \, \ell \leq 3$. 
\end{Lem}
\begin{proof}
Since $\partial_{z^i} \XX_\C^j(t,z,v) =\delta_{i}^j$ and $\mathcal{L}_{\partial_{x^\nu}}=\nabla_{\partial_{x^\nu}}$, we have
$$ \partial_{z^k} \big[ G_{\mu \ell} \big(t,\XX_\C\big) \big] = \big[ \partial_{x^k} G_{\mu \ell} \big] \big(t,\XX_\C \big)= \mathcal{L}_{\partial_{x^k}}(G)_{\mu \ell} \big(t,\XX_\C \big).$$
We get similarly the relation for $\partial_t^\infty = \partial_t-\widehat{v} \cdot \nabla_z$ since
$$ \partial_{t} \big[ G_{\mu \ell} \big(t,\XX_\C\big) \big] = \big[ \partial_{t} G_{\mu \ell} \big] \big(t,\XX_\C \big)+\widehat{v}^n \big[ \partial_{x^n} G_{\mu \ell} \big] \big(t,\XX_\C \big)+\partial_t \C^n_{t,v}\big[ \partial_{x^n} G_{\mu \ell} \big] \big(t,\XX_\C \big) .$$
Applying the relations of Lemma \ref{Lemrelftog} to the function $f(t,x,v)=G_{\mu \nu}(t,x)$, we get
\begin{align*}
S \big[ G_{\mu \ell} (t,\XX_\C ) \big] \! &= \! \big[ S G_{\mu \ell} \big] (t,\XX_\C )-\big( \C^\ell_{t,v}-t\partial_t \C_{t,v}^n \big)\big[\partial_{x^n}  G_{\mu \ell} \big] (t,\XX_\C ) , \\
\widehat{\Omega}_{ij} \big[ G_{\mu \ell} (t,\XX_\C ) \big] \! &= \!\big[ \Omega_{ij} G_{\mu \ell} \big] (t,\XX_\C )-\C^i_{t,v} \big[ \partial_{x^j} G_{\mu \ell} \big](t,\XX_\C)+\C^j_{t,v} \big[ \partial_{x^i} G_{\mu \ell} \big](t,\XX_\C) + \widehat{\Omega}_{ij}\big( \C_{t,v}^n \big) \big[\partial_{x^n}  G_{\mu \ell} \big](t,\XX_\C )  , \\
\widehat{\Omega}_{0k}^\infty \big[ G_{\mu \ell} (t,\XX_\C ) \big] \!&= \! \big[ \Omega_{0k} G_{\mu \ell} \big] (t,\XX_\C )-\C^k_{t,v} \big[ \partial_t G_{\mu \ell} \big] (t,\XX_\C )+\widehat{\Omega}_{0k}^\infty \big( \C_{t,v}^n \big)\big[ \partial_{x^n} G_{\mu \ell} \big] (t,\XX_\C ).
\end{align*}
To conclude the proof for the case of the scaling vector field $S$, we use
$$
\mathcal{L}_{S}(G)_{\mu \ell} = S \! \left(  G_{\mu \ell}\right)+2G_{\mu \ell}, \qquad  \qquad S(v^\mu)=0.
$$
For the rotations, remark that
$$\mathcal{L}_{\Omega_{ij}}(G)_{\mu \ell} = \Omega_{ij} \!\left( G_{\mu \ell} \right)+\delta_\mu^i  G_{j \ell}- \delta_\mu^j G_{i \ell}+\delta_\ell^i  G_{\mu j}- \delta_\ell^j G_{\mu i}, \qquad \big( v^i \partial_{v^j}-v^j \partial_{v^i} \big) (v^\mu) G_{\mu \ell}=v^i G_{j \ell}-v^j G_{i \ell}.$$
Finally, for the Lorentz boosts, use
$$
 \mathcal{L}_{\Omega_{0k}}(G)_{\mu \ell} = \Omega_{0k} \! \left(  G_{\mu \ell} \right)+\delta_\mu^0  G_{k \ell}+ \delta_\mu^k G_{0 \ell}+\delta_\ell^0  G_{\mu k}+ \delta_\ell^k G_{\mu 0} , \qquad v^0 \partial_{v^k} (v^\mu) G_{\mu \ell}= v^k G_{0\ell}+v^0G_{k\ell}
$$
and $\widehat{v}^\mu G_{\mu 0}=-\widehat{v}^\mu G_{\mu n}\widehat{v}^n$ since $G$ is a $2$-form and $\widehat{v}^0=1$. 
\end{proof}

\begin{Lem}\label{LemCom2}
Let $|\gamma| \leq N-1$, $1 \leq i < j \leq 3$ and $1 \leq k, \, \ell \leq 3$. We have, for all $(t,z,v) \in [T,+ \infty[ \times \R^3_z \times \R^3_v$, 
\begin{align*}
 \partial_{z^k} \Big[ t^{-1} \Delta^\ell_{Z^\gamma}(t,z,v)  \Big] &=t^{-1} \Delta^\ell_{\partial_{x^k}Z^\gamma}(t,z,v) ,  \\
 \partial_{t}^\infty \Big[ t^{-1} \Delta^\ell_{Z^\gamma}(t,z,v)  \Big] &=t^{-1} \Delta^\ell_{\partial_{t}Z^\gamma}(t,z,v)-t^{-2}\Delta^\ell_{Z^\gamma}(t,z,v) + 2\frac{\widehat{v}^\mu}{v^0}  \mathcal{L}_{Z^{\gamma}}(F)_{\mu \ell} (t,\XX_\C ) \\
 & \quad +t\partial_t \big( \C_{t,v}^n \big) t^{-2} \Delta^\ell_{\partial_{x^n} Z^\gamma}(t,z,v) ,  \\
  S \Big[ t^{-1} \Delta^\ell_{Z^\gamma}(t,z,v)  \Big] &=t^{-1} \Delta^\ell_{SZ^\gamma}(t,z,v)-t^{-1} \Delta^\ell_{Z^\gamma}(t,z,v) +t\partial_t \big( \C_{t,v}^n \big) t^{-1} \Delta^\ell_{\partial_{x^n} Z^\gamma}(t,z,v), \\
  \widehat{\Omega}_{ij}  \Big[ t^{-1} \Delta^\ell_{Z^\gamma}(t,z,v)  \Big] &= t^{-1}\Delta^\ell_{ \Omega_{ij} Z^\gamma}(t,z,v)-\delta_n^i  t^{-1} \Delta^j_{Z^\gamma}(t,z,v)+ \delta_n^j t^{-1} \Delta^i_{Z^\gamma}(t,z,v) \\
  & \quad  +\C^j_{t,v} t^{-1} \Delta^\ell_{ \partial_{x^i}Z^\gamma}(t,z,v) - \C^i_{t,v} t^{-1} \Delta^\ell_{ \partial_{x^j}Z^\gamma}(t,z,v) +\widehat{\Omega}_{ij}\big( \C^n_{t,v} \big)t^{-1} \Delta^\ell_{ \partial_{x^n}Z^\gamma}(t,z,v),\\
  \widehat{\Omega}_{0k}^\infty  \Big[ t^{-1}  \Delta^\ell_{Z^\gamma}(t,z,v)  \Big] &=t^{-1} \Delta^\ell_{ \Omega_{0k}Z^\gamma}(t,z,v)-\widehat{v}^kt^{-1} \Delta^\ell_{Z^\gamma}(t,z,v)+\delta_\ell^k \widehat{v}_n t^{-1} \Delta^n_{Z^\gamma}(t,z,v)+z^k t^{-2} \Delta^\ell_{ Z^\gamma}(t,z,v) \\
  & \quad -\C^k_{t,v}t^{-1} \Delta^\ell_{ \partial_t Z^\gamma}(t,z,v)+\widehat{\Omega}_{0k}^\infty \big( \C^n_{t,v} \big)t^{-1} \Delta^\ell_{ \partial_{x^n}Z^\gamma}(t,z,v) .
\end{align*}
\end{Lem}
\begin{proof}
For the translations $\partial_{z^k}$, the scaling vector field $S$ and the rotations $\widehat{\Omega}_{ij}$, it suffices to apply the previous Lemma \ref{LemCom1} to $G= \mathcal{L}_{Z^\gamma}F$, Proposition \ref{Proasympelec} and to use 
$$ \partial_{z^k}(t)=\widehat{\Omega}_{ij}(t)=0, \qquad S(t)=t, \qquad \partial_{z^k}(v^0)=S(v^0)=\widehat{\Omega}_{ij}(v^0)=0.$$   
For $\partial_t^\infty=\partial_t-\widehat{v} \cdot \nabla_z$ and the Lorentz boosts, use Lemma \ref{LemCom1}, Proposition \ref{Proasympelec} as well as
$$t^{-1}\partial_t^\infty \! \big( t^2 |v^0|^{-2} \big) \!=2|v^0|^{-2}, \quad \partial_t^\infty \! \big( t^{-1} \big)=-t^{-2}, \quad t^{-1}\widehat{\Omega}^\infty_{0k}\big( t^2 |v^0|^{-2} \big)=2z^k|v^0|^{-2}, \quad \widehat{\Omega}^\infty_{0k} \big( t^{-1} \big)=-z^k t^{-2}-\widehat{v}^kt^{-1}.$$
\end{proof}
As $[\partial_{z^k},\partial_{z^i}]=[\partial_{z^k},v^0 \partial_{v^i}]=0$, we get, by applying Lemmata \ref{LemCom1} and \ref{LemCom2}, for $|\gamma|=0$, the first order commutation formula for the $z$ derivatives.
\begin{Lem}\label{Comz}
We have, for any $k \in \llbracket 1 , 3 \rrbracket$,
$$[\T_F^\infty , \partial_{z^k}] = \frac{\delta_j^i-\widehat{v}^i \widehat{v}_j}{t v^0} \Delta_{\partial_{x^k}}^j(t,z,v)\partial_{z^i}-\frac{\widehat{v}^\mu}{v^0} {\mathcal{L}_{\partial_{x^k}}(F)_{\mu}}^j(t,\XX_\C)\left( v^0\partial_{v^j}- v^0\partial_{v^j} (\C^i_{t,v}) \partial_{z^i} \right)\!.$$
\end{Lem}
The cases of the other derivatives are more involved to compute. For this reason, we will directly state the higher order commutation formula. Moreover, as $v^0 \partial_{v^i}$ are not part of our set of commutators, we need to rewrite them as
\begin{equation}\label{eq:expressionvderiv}
 v^0 \partial_{v^k}=\widehat{\Omega}^\infty_{0k} - z^k \partial_t^\infty - \widehat{v}^k  S+\widehat{v}^k z^i \partial_{z^i} .
 \end{equation}

\subsection{Commutation formula}
It will be convenient to introduce the schematic notation $P_{p,q}(\C)$ in order to denote any quantity of the form
$$ P_{p,q}(\C):= \prod_{1 \leq k \leq p}  \big( v^0\partial_{v}\big)^{\kappa_k} (t\partial_t)^{n_k}  \! \left( \C^{i_k}\right), \qquad (p,q) \in \mathbb{N}^2, \quad 1 \leq i_k \leq 3, \quad  q=\sum_{1 \leq k \leq p} |\kappa_k|+n_k . $$
By convention, we set $P_{0,q}(\C)=1$ for $p=0$. Note that these quantities does not generate strong divergent terms since, in view of Proposition \ref{Proasympelec} and $\epsilon \log^{-1}(T) \leq 1$, we have 
\begin{equation}\label{estiPolyC}
\hspace{-15mm} \forall \, (t,v) \in [T,+\infty[\times \R^3_v, \qquad \qquad\big|P_{p,q}(\C) \big|(t,v) \lesssim \epsilon^{\frac{p}{2}} |v^0|^{-p}  \log^p( t) \leq \log^{2p}( t).
\end{equation}
In order to state the commutation formula in a concise form, we will use the notation $\partial_{z^0}\!=\partial_t^\infty=\partial_t-\widehat{v} \cdot \nabla_x$. Recall further that, for $\widehat{Z}^\beta_\infty \in \widehat{\mathbb{P}}_S^\infty$, $\beta_H$ denotes the number of homogeneous vector fields, that is $\widehat{\Omega}_{ij}$, $\widehat{\Omega}^\infty_{0k}$ and $S$, composing $\widehat{Z}^\beta$. The quantity $\beta_T=|\beta|-\beta_H$ is the number of translations $\partial_{z^\mu}$.
\begin{Pro}\label{Comgen}
Let $\beta$ be a multi-index such that $|\beta| \leq N$. Then, $[\T^\infty_F,\widehat{Z}_\infty^\beta]$ can be written as a linear combination, with coefficients which are polynomials in $(\widehat{v}^1,\widehat{v}^2,\widehat{v}^3)$, of the following five families of error terms.
\begin{alignat}{2}
&\hspace{-1mm} \bullet \; \,  \widehat{Z}_\infty^\kappa \T_F^\infty,  \tag{Err-0} \label{T0} \\ \noalign{\vskip3pt}
& \hspace{-1mm} \bullet \; \,  R \!\left(z \right) P_{p,q}(\C) \, t^{-1-m}\Delta^i_{Z^\gamma}(t,z,v) \, \partial_{z^\nu}  \widehat{Z}_\infty^\kappa, \qquad &&  \nu \in \llbracket 0 , 3 \rrbracket,    \tag{Err-1} \label{T1} ,  \\ \noalign{\vskip3pt}
& \hspace{-1mm} \bullet \; \,  R \!\left(z \right) P_{p,q}(\C) \, t^{-m} \, \frac{\widehat{v}^\mu}{v^0} \mathcal{L}_{Z^\gamma}(F)_{\mu i}(t,\XX_\C) \, \widehat{Z}_\infty \widehat{Z}_\infty^\kappa, \qquad \qquad  &&  \widehat{Z}_\infty \in \widehat{\mathbb{P}}_S^\infty,  \tag{Err-2} \label{T2} 
\end{alignat}
where $i \in \llbracket 1,3 \rrbracket$, $m \in \mathbb{N}$, $R$ is a polynomial in $z=(z^j)_{1 \leq j \leq 3}$ of degree $\deg R$, 
$$\deg R +\kappa_H \leq \beta_H, \qquad \qquad \deg R \leq m, \qquad  \qquad |\gamma|+|\kappa | \leq |\beta | , \qquad \qquad |\kappa | \leq |\beta |-1 , \qquad \qquad  p, \, q \leq |\beta|+1 . $$
Moreover, for the terms \eqref{T1}--\eqref{T2}, we have the following extra condition. If $\deg R +\kappa_H = \beta_H$, then
\begin{itemize}
\item either $m \geq 1$
\item or $ \gamma_T \geq 1$ and $p=0$.
\end{itemize}
\end{Pro}
\begin{Rq}
The extra condition provides more decay in $t-|\XX_\C|$, if $\gamma_T \geq 1$, or in $t$, if $m \geq 1$.
\end{Rq}
\begin{proof} 
The proof is based on an induction and in order to avoid any confusion, we will say that the error terms \eqref{T0}--\eqref{T2} written in the statement of the proposition are associated to the derivative $\widehat{Z}_\infty^\beta$. The starting point consists in writting the Vlasov operator as follows
\begin{equation}\label{eq:Vlasovbis}
\hspace{-5mm} \T_F^\infty=\partial_t-\frac{\delta_j^i-\widehat{v}_j \widehat{v}^i}{v^0} t^{-1} \Delta^j \partial_{z^i}+\frac{\widehat{v}^\mu}{v^0} {F_{\mu}}^j(t,\XX_\C)\left( \widehat{\Omega}^\infty_{0j} - z_j \partial_t^\infty - \widehat{v}_j  S+ \widehat{v}_j z^i \partial_{z^i}-v^0 \partial_{v^j} \C^i(t,v) \partial_{z^i} \right),
\end{equation}
where we rewrite the $v$-derivatives using \eqref{eq:expressionvderiv}. We will further need the following relations,
\begin{align}
\nonumber &[\partial_{z^k},\partial_{z^i}]=0, \qquad [\partial_{z^k},\partial_t^\infty]=0, \qquad [\partial_{z^k},S]=\partial_{z^k}, \qquad  [ \partial_{z^k},\widehat{\Omega}^\infty_{0k} ]=\partial_t^\infty, \qquad [\partial_{z^k} , \widehat{\Omega}_{ij}]= \delta_k^i \partial_{z^j}-\delta^j_k \partial_{z^i}, \\
& [\partial_t^\infty,\partial_{z^i}]=0, \qquad \, [\partial_t^\infty,\partial_t^\infty]=0, \qquad \, [\partial_t^\infty,S]=\partial_t^\infty, \qquad \,  [ \partial_{t}^\infty,\widehat{\Omega}^\infty_{0k} ]=\partial_{z^k}, \qquad  [\partial_t^\infty , \widehat{\Omega}_{ij}]= 0. \label{eq:commutePinf}
\end{align}
We will also make use of
\begin{alignat*}{2}
&\partial_t^\infty (t)=1, \qquad &&\partial_{z^k}(t)=0, \qquad \quad  S(t)=t, \qquad \quad  \widehat{\Omega}_{ij}(t)=0, \qquad \qquad \qquad \; \; \; \widehat{\Omega}^\infty_{0k}(t)=z^k+t\widehat{v}^k , \\
&\partial_t^\infty (z^\ell)=-\widehat{v}^\ell, \qquad &&\partial_{z^k}(z^\ell)=\delta_k^\ell, \qquad  S(z^\ell)=z^\ell, \qquad \widehat{\Omega}_{ij}(z^\ell)=\delta^\ell_j z^i-\delta_i^\ell z^j, \qquad \, \widehat{\Omega}^\infty_{0k}(z^\ell)=-z^k \widehat{v}^\ell ,\\
&\partial_t^\infty (v^\mu)=0, \qquad &&\partial_{z^k}(v^\mu)=0, \qquad \; S(v^\mu)=0, \qquad \; \widehat{\Omega}_{ij}(v^\mu)=\delta^\ell_j v^i-\delta_i^\ell v^j, \qquad \widehat{\Omega}^\infty_{0k}(v^\mu)=\delta_k^\mu v^0+\delta_0^\mu v^k.
\end{alignat*}
Next, given a quantity of the form $P_{p,q}(\C)$, we have the schematic relations
\begin{align*}
& \partial_{z^k}P_{p,q}(\C)=0, \qquad \partial^\infty_t P_{p,q}(\C)=\partial_t P_{p,q}(\C)=t^{-1}P_{p,q+1}(\C), \qquad  S P_{p,q}(\C)\! =t\partial_t P_{p,q}(\C)=P_{p,q+1}(\C), \\ 
&\widehat{\Omega}_{ij} P_{p,q}(\C)=\big(v^i \partial_{v^j}-v^j \partial_{v^i}\big)P_{p,q}(\C)=\widehat{v}^iP_{p,q+1}(\C)+\widehat{v}^jP_{p,q+1}(\C) , \\
& \widehat{\Omega}_{0k}^\infty P_{p,q}(\C) = v^0 \partial_{v^k}P_{p,q}(\C)+z^k \partial_t P_{p,q}(\C)+\widehat{v}^k t \partial_t P_{p,q}(\C)=P_{p,q+1}(\C)+z^k t^{-1}P_{p,q+1}(\C)+\widehat{v}^k P_{p,q+1}(\C).
\end{align*}
Finally, in order to commute with the scaling vector field $S$ or $\widehat{\Omega}^\infty_{0k}$, we will use
$$
[\partial_t, t \partial_t]=\partial_t=\T_F^\infty +\frac{\delta_j^i-\widehat{v}_j \widehat{v}^i}{v^0} t^{-1} \Delta^j \partial_{z^i}-\frac{\widehat{v}^\mu}{v^0} {F_{\mu}}^j(t,\XX_\C)\left( \widehat{\Omega}^\infty_{0j} - z_j \partial_t^\infty - \widehat{v}_j  S+ \widehat{v}_j z^i \partial_{z^i}-v^0 \partial_{v^j} \C^i(t,v) \partial_{z^i} \right).
$$
The first order commutation formula then follows from these relations as well as Lemmata \ref{LemCom1}--\ref{LemCom2}.

Let $n \in \mathbb{N}^*$ such that the statement holds for any multi-index $|\xi| =n$ and consider $\widehat{Z}_\infty \in \widehat{\mathbb{P}}_S^\infty$ as well as $|\beta| =n$. We use then the general relation
$$[ \T^\infty_F, \widehat{Z}_\infty\widehat{Z}_\infty^\beta ]=  [ \T^\infty_F, \widehat{Z}_\infty] \widehat{Z}_\infty^\beta + \widehat{Z}_\infty[ \T^\infty_F, \widehat{Z}_\infty^\beta ]. $$
Appealing to the first order commutation formula, we can write $[ \T^\infty_F, \widehat{Z}_\infty] \widehat{Z}_\infty^\beta$ as a linear combination of terms of type \eqref{T1}--\eqref{T2}, corresponding to the derivative $\widehat{Z}_\infty\widehat{Z}_\infty^\beta$, and, if $\widehat{Z}_\infty=S$ or $\widehat{Z}_\infty = \widehat{\Omega}^\infty_{0k}$, $\T_F^\infty \widehat{Z}_\infty^\beta$. We deal with the last term by writting $\T_F^\infty \widehat{Z}_\infty^\beta=\widehat{Z}_\infty^\beta \T_F^\infty+[\T_F^\infty , \widehat{Z}_\infty^\beta]$ and then by using the induction hypothesis. The other term, $\widehat{Z}_\infty[ \T^\infty_F, \widehat{Z}_\infty^\beta ]$, can be treated by combining the induction hypothesis with the following property.

Let $0 \leq \lambda \leq 2$ and $\mathbf{E}^\beta_\lambda$ be an error term of the form (Err-$\lambda$), associated to $\widehat{Z}_\infty^\beta$. Then, 
\begin{itemize}
\item if $(\widehat{Z}_\infty, \lambda) \neq (\partial_t^\infty,1)$, $\widehat{Z}_\infty \mathbf{E}^\beta_\lambda$ is a linear combination of terms of type (Err-$\lambda$) associated to $\widehat{Z}_\infty \widehat{Z}_\infty^\beta$.
\item $\partial_t^\infty \mathbf{E}^\beta_1$ is a linear combination of terms of type \eqref{T1}--\eqref{T2}, associated to $\partial_t^\infty \widehat{Z}_\infty^\beta$. 
\end{itemize}
This can be observed by combining again the relations written in the beginning of the proof with Lemmata \ref{LemCom1}--\ref{LemCom2}. The crucial part consists in observing that the extra condition is preserved. For this, introduce $|\xi|=|\beta|+1$ such that $\widehat{Z}_\infty^\xi= \widehat{Z}_\infty \widehat{Z}_\infty^\beta$ and make the next observations. Once we apply $\widehat{Z}_\infty$ to an error term $E^\beta_\lambda$, $\lambda \in \{1,2\}$, we obtain error terms $E^\xi_j$, $j \in \{1,2 \}$, such that
\begin{itemize}
\item if $\widehat{Z}_\infty \in \{ \partial_t^\infty , \partial_{z^1}, \, \partial_{z^2}, \, \partial_{z^3} \}$, then $\xi_H=\beta_H$ and $\deg R$ as well as $\kappa_H$ does not strictly increase. The quantities $\gamma_T$ and $m$ does not strictly decrease. Moreover, according to Lemmata \ref{LemCom1}--\ref{LemCom2}, if $p$ strictly increases, then $m$ increases by $1$.
\item Otherwise $\widehat{Z}_\infty$ is an homogeneous vector field and $\xi_H=\beta_H+1$. As before, none of the quantities $\gamma_T$ and $m$ strictly decrease. Moreover, $\deg R+\kappa_H$ increases at most by $1$, in which case either $p$ does not stricly increase or $m$ increases by $1$.
\end{itemize}
\end{proof}

It is for the purpose of exploiting the hierarchies related to the condition $\deg_z \! R+\kappa_H \leq \beta_H$ that we will propagate weighted norms such as the ones of Definition \ref{DeffuncspaVla}. For this we will apply the next result, which is a consequence of Proposition \ref{Comgen}.
\begin{Cor}\label{CorCom}
Let $(M_z,M_v) \in \R^2$ and $h$ be a sufficiently regular solution to $\T_F^\infty (h)=0$. Then, for any $|\beta| \leq N$, we can bound $\big[\T_F, \langle v \rangle^{M_v} \, \langle z \rangle^{M_z-\beta_H} \widehat{Z}_\infty^\beta \big](h)$ by a linear combination of the following terms,
\begin{flalign}\label{T1bis}
\hspace{-3mm}& \;   \bullet \; \, \frac{\log^{2N+2}(t)}{ \langle z \rangle \, v^0 \, t}   \left|  \widehat{v}^\mu\Big( t^2   \mathcal{L}_{Z^\gamma}(F)_{\mu i}(t,\XX_\C)-  \mathbb{F}_{\mu i}\big[\widehat{Z}_\infty^\gamma f_\infty \big](v)  \Big) \right| \left| \langle v \rangle^{M_v} \, \langle z \rangle^{M_z-\kappa_H} \, \widehat{Z}_\infty^\kappa h \right|,   \tag{T-1}  \\ \noalign{\vskip2pt}
\hspace{-3mm}& \; \bullet \; \, \frac{\langle t+|z| \rangle}{v^0}\Big|  \widehat{v}^\mu  \mathcal{L}_{Z^\gamma}(F)_{\mu i}(t,\XX_\C) \Big| \left|\langle v \rangle^{M_v} \, \langle z \rangle^{M_z-\kappa_H} \, \widehat{Z}_\infty^\kappa h \right| , \qquad \qquad \qquad \qquad \quad Z^\gamma=\partial_{x^\nu}Z^\xi \tag{T-2} \label{T2bis}  \\ \noalign{\vskip2pt}
\hspace{-3mm}& \; \bullet \; \log^{2N+2}(t) \frac{\langle t+|z| \rangle}{v^0 \, t} \Big|   \mathcal{L}_{Z^\gamma} F (t,\XX_\C) \Big| \left|  \langle v \rangle^{M_v} \, \langle z \rangle^{M_z-\kappa_H} \, \widehat{Z}_\infty^\kappa h \right|,  \tag{T-3} \label{T3bis}  \\ \noalign{\vskip2pt}
\hspace{-3mm}& \; \bullet \; \,   t^{-\frac{7}{4}}  \left|  \langle v \rangle^{M_v} \, \langle z \rangle^{M_z-\kappa_H} \, \widehat{Z}_\infty^\kappa h \right|,  \tag{T-4} \label{T4bis}  
\end{flalign}
where $1 \leq i \leq 3$, and $|\gamma|+|\kappa| \leq |\beta|+1$ and $|\gamma|, \, |\kappa| \leq |\beta|$.
\end{Cor}
\begin{proof}
Let us first reduce the proof to the case $M_v=M_z=0$. Note that
\begin{align*}
 \big| \big[\T^\infty_F, \langle v \rangle^{M_v} \, \langle z \rangle^{M_z-\beta_H} \, \widehat{Z}_\infty^\beta \big](h) \big| &\leq  (|M_z|+N) \big|z_i\T_F^\infty(z^i)\big| \langle v \rangle^{M_v} \, \langle z \rangle^{M_z-\beta_H-2}\big|\widehat{Z}_\infty^\beta h \big| \\ \noalign{\vskip2pt}
 & \quad +|M_v| \big| \T_F^\infty(\langle v \rangle) \big| \langle v \rangle^{M_v-1} \, \langle z \rangle^{M_z-\beta_H}\big|\widehat{Z}_\infty^\beta h \big| \\ \noalign{\vskip2pt}
 & \quad  + \langle v \rangle^{M_v} \, \langle z \rangle^{M_v} \, \langle z \rangle^{-\beta_H}\big| \big[\T^\infty_F, \widehat{Z}_\infty^\beta \big](h) \big|.
 \end{align*}
Recall now from \eqref{estiPolyC} that $v^0|\partial_{v^j} \C^i_{t,v}| \lesssim \sqrt{\epsilon} \log(t) $ and remark, in view of the expression \eqref{eq:Vlasov} of $\T_F^\infty$, that
 \begin{align*}
\T_F^\infty (\langle v \rangle) &= \widehat{v}_i \T_F^\infty(v^i)= \widehat{v}^\mu \widehat{v}^i  F_{\mu i}(t,\XX_\C)  , \\
\T_F^\infty (  z^i  )  & =   -\frac{\delta_j^i-\widehat{v}_j \widehat{v}^i}{tv^0} \widehat{v}^\mu \Big( t^2  {F_{\mu}}^j(t,\XX_\C)-{\mathbb{F}_{\mu}}^j[f_\infty](v) \Big)-\frac{\widehat{v}^\mu}{v^0} {F_{\mu}}^j(t,\XX_\C)v^0 \partial_{v^j} \C^i_{t,v}   .
\end{align*}
The first term on the right hand side of the commutator estimate can then be bounded by quantities of type \eqref{T1bis} and \eqref{T3bis}. We can control the second one by terms of the form \eqref{T3bis}. Next, in order to deal with the last term on the right hand side, it suffices to prove the result for $M_v=M_z=0$. For this, we apply the previous Proposition \ref{Comgen}. As $\T^\infty_F(h)=0$, the terms \eqref{T0} vanish. In view of the estimate \eqref{estiPolyC} of the correction coefficients, the error terms \eqref{T1}--\eqref{T2} can be bounded by quantities of the form
\begin{alignat*}{2}
\mathbf{A}_1&:= \log^{2p}(t) \, t^{-1-m}|\Delta^i_{Z^\gamma}|(t,z,v) \langle z \rangle^{-\beta_H + \deg R} \, \big|\partial_{z^\nu} \widehat{Z}^\kappa h \big| , \qquad && \nu \in \llbracket 0,3 \rrbracket,  \\ \noalign{\vskip2pt}
\mathbf{A}_2&:= \log^{2p}(t) \, t^{-m} \left| \frac{\widehat{v}^\mu}{v^0} \mathcal{L}_{Z^\gamma}(F)_{\mu i}(t,\XX_\C) \right| \langle z \rangle^{-\beta_H + \deg R}  \, \big| \widehat{Z}_\infty \widehat{Z}^\kappa h \big|, \qquad \qquad && \widehat{Z}_\infty \in \widehat{\mathbb{P}}_S^\infty, 
\end{alignat*}
where $$\deg R +\kappa_H \leq \beta_H, \qquad i \in \llbracket 1,3 \rrbracket, \qquad m \in \mathbb{N}, \qquad p \leq |\beta|+1,   \qquad  |\gamma|+|\kappa | \leq |\beta | , \qquad |\kappa | \leq |\beta |-1  $$
and $\deg R +\kappa_H =\beta_H$ implies 
\begin{itemize}
\item $m \geq 1$
\item or $\gamma_T \geq 1$ and $p=0$.
\end{itemize}
Note that, for any $Z \in \mathbb{K}$ and $0 \leq \nu \leq 3$, we have $[\partial_{x^\nu},Z]=0$ or $[\partial_{x^\nu},Z]=\pm \partial_{x^\lambda}$ for a certain $0 \leq \lambda \leq 3$. Hence, if $\gamma_T \geq 1$, we can assume without loss of generality that $Z^\gamma = \partial_{x^\nu} Z^\zeta$.

In order to deal with $\mathbf{A}_2$, we stress that there exists $|\xi|=|\kappa|+1$ such that $\widehat{Z}_\infty^\xi = \widehat{Z}_\infty \widehat{Z}_\infty^\kappa$ and $\xi_H \leq  \kappa_H+1$.
\begin{itemize}
\item If $  - \beta_H+\deg  R \leq -\kappa_H-1 \leq -\xi_H$, $\mathbf{A}_2$ is less than a term of type \eqref{T3bis}.
\item Otherwise $  - \beta_H+\deg  R =-\kappa_H$, so that $\langle z \rangle^{-\beta_H + \deg R} \leq \langle z \rangle \, \langle z \rangle^{-\xi_H }$. If $p \geq 1$, we have $m \geq 1$ and $\mathbf{A}_2$ can be bounded by \eqref{T3bis}. Finally, for the case $p=0$, we necessarily have $\gamma_T \geq 1$ and $\mathbf{A}_2$ can be controlled by quantities of the form \eqref{T2bis}.
\end{itemize}
For $\mathbf{A}_1$, we introduce $|\xi|=|\kappa|+1$ verifying $\widehat{Z}_\infty^\xi = \partial_{z^\nu} \widehat{Z}_\infty^\kappa$, so that $\xi_H =  \kappa_H$. Then,
\begin{itemize}
\item if $\deg R +\kappa_H  \leq \beta_H-1$, we have $\langle z \rangle^{-\beta_H + \deg R} \leq  \langle z \rangle^{-1} \, \langle z \rangle^{-\xi_H}$ and $\mathbf{A}_1$ is bounded above by \eqref{T1bis}.
\item If $m \geq 1$, we have
$$ \mathbf{A}_1 \lesssim \frac{\log^{2p}(t) }{v^0t^2} \left(t^2 \big| \mathcal{L}_{Z^\gamma}F \big|(t,\XX_\C)+\big|\mathbb{F}\big[ \widehat{Z}_\infty^\gamma f_\infty \big] \big|(v)  \right)\langle z \rangle^{-\xi_H} \, \big|\widehat{Z}^\xi h \big|.$$
Since $\big|\mathbb{F}\big[ \widehat{Z}_\infty^\gamma f_\infty \big] \big|(v) \lesssim \sqrt{\epsilon} \lesssim \log(t)$ by Proposition \ref{Proasympelec} and $\log^{2p+1}(t) \lesssim t^{1/4}$, $\mathbf{A}_1$ is controlled by terms of the form \eqref{T3bis}--\eqref{T4bis}.
\item Otherwise  $\deg R +\kappa_H = \beta_H$ and $m=0$, so that $\gamma_T \geq 1$ and $p=0$. We then have
$$\mathbb{F}\big[ \widehat{Z}_\infty^\gamma f_\infty \big]=0, \qquad t^{-1}|\Delta_{Z^\gamma}^i|(t,z,v)\lesssim \sup_{1 \leq k \leq 3} \, \frac{t}{v^0} \big| \widehat{v}^\mu\mathcal{L}_{Z^\gamma}(F)_{\mu k}(t,\XX_\C) \big|, \qquad \gamma_T \geq 1,$$ 
and $\mathbf{A}_1$ is bounded by terms of type \eqref{T2bis}.
\end{itemize}
\end{proof}

One can check in the commutation formula of Corollary \ref{CorCom} that the error terms containing a high order derivative of $F$ carry a low order derivative of $h$. Consequently, we can expect to close the energy estimates by bounding first a well-chosen weighted $L^2$ norm of the lower order derivatives of $h$ and by controlling then a weaker norm of the higher order derivatives $\widehat{Z}^\beta_\infty h$, $|\beta| \geq N-2$.

More concretely, in the terms \eqref{T1bis}--\eqref{T3bis}, when $|\gamma| \geq N-2$ the distribution function is allowed to absorb more powers of $\langle z \rangle$ and $v^0$. Since we have a weaker control of $\mathcal{L}_{Z^\gamma}F$ at the top order $|\gamma|=N$, we unify the treatment of \eqref{T2bis}--\eqref{T3bis} in the following way.
\begin{Rq}\label{RqCom}
If $|\gamma| \geq N-2$, the terms \eqref{T2bis}--\eqref{T3bis} are bounded by a sum of quantities of the form
$$ \log^{2N+4}(t)  \frac{ \langle t+|\XX_\C| \rangle \, \langle z \rangle \, \langle v \rangle}{t} \langle t-|\XX_\C| \rangle\big|  \nabla_{t,x}  \mathcal{L}_{Z^\xi} F\big|(t,\XX_\C) \left|\langle v \rangle^{M_v} \, \langle z \rangle^{M_z-\kappa_H} \, \widehat{Z}_\infty^\kappa h \right|, $$ 
where $|\xi|+|\kappa| \leq N$ and $ N-3 \leq |\xi| \leq N-1$.
\end{Rq}
\begin{proof}
Remark that if $Z^\gamma=ZZ^\xi$, then $|\mathcal{L}_{Z^\gamma}F| \lesssim \langle t+r \rangle |\nabla_{t,x}  \mathcal{L}_{Z^\xi}F|$. It then remains us to observe
$$ \frac{\langle t+|z| \rangle}{v^0} + \frac{\langle t+|z| \rangle \, \langle t+|\XX_\C| \rangle}{tv^0} \lesssim \frac{ \langle t+|\XX_\C| \rangle^2}{tv^0} \lesssim \frac{ \langle t+|\XX_\C| \rangle}{t}  \langle t-|\XX_\C| \rangle \, \langle z \rangle \, \langle v \rangle \log^2(t).$$
The first inequality ensues from $|z| \lesssim |\XX_\C|+t+|\C_{t,v}| \lesssim t+|\XX_\C|$, where we used $|\C_{t,v}| \lesssim \epsilon \log(t) \lesssim \log^2(t)$. For the second one, it suffices to apply Lemma \ref{gainv} for $x=\XX_\C=z+t\widehat{v}+\C_{t,v}$.
\end{proof}

Finally, we end this subsection by writing an other consequence of Proposition \ref{Comgen}. We will use it in order to prove a strong convergence estimate for the spatial average of $\widehat{Z}^\beta f$ and derive the asymptotic behavior of the source term of the Maxwell equations. The idea is that we can get rid of the worst error terms \eqref{T1bis} in the previous commutation formula by performing integration by parts in $z$.
\begin{Cor}\label{Corestipartialtgg}
Let $h$ be a sufficiently regular solution to $\T_F^\infty (h)=0$ and $|\beta| \leq N-1$. Then,
$$  \partial_t \int_{\R^3_z} \widehat{Z}^\beta_\infty h(t,z,v) \dr z$$
can be bounded by a linear combination of terms of the form
 \begin{flalign}
& \; \bullet \; \, \log^{2N}(t) \frac{1}{v^0} \int_{\R^3_z}t   \Big| \nabla_{t,x}  \mathcal{L}_{Z^\gamma} F  \Big|(t,\XX_\C) \left|   \widehat{Z}_\infty^\kappa h \right|(t,z,v) \dr z,  \tag{SA1} \label{T3bisbis}  \\
& \; \bullet \; \, \log^{2N}(t) \frac{1}{v^0} \int_{\R^3_z}  \Big|   \mathcal{L}_{Z^\gamma} F  \Big|(t,\XX_\C) \left|   \langle z \rangle^{N-\kappa_H} \, \widehat{Z}_\infty^\kappa h \right|(t,z,v) \dr z,  \tag{SA2} \label{T2bisbis}  \\
& \; \bullet \; \, \log^{2N+1}(t) \frac{1}{v^0 t^2} \int_{\R^3_z}    \langle z \rangle^{N-\kappa_H} \, \left| \widehat{Z}_\infty^\kappa h \right|(t,z,v) \dr z,  \tag{SA3} \label{T4bisbis}  
\end{flalign}
where $|\gamma|+|\kappa| \leq N$ and $|\gamma| \leq N-1$.
\end{Cor}
\begin{proof}
We use \eqref{eq:Vlasovbis} in order to write
\begin{equation*}
   \partial_t \widehat{Z}^\beta_\infty  h = \T_F^\infty \big( \widehat{Z}^\beta_\infty  h \big) +\frac{\delta_j^i-\widehat{v}_j \widehat{v}^i}{v^0t}  \Delta^j \partial_{z^i}\widehat{Z}^\beta_\infty  h - \frac{\widehat{v}^\mu}{v^0} {F_{\mu}}^j(t,\XX_\C)\! \left( \widehat{\Omega}^\infty_{0j} \! - z_j \partial_t^\infty \! - \widehat{v}_j  S+ \widehat{v}_j z^i \partial_{z^i}-v^0 \partial_{v^j} \C^i_{t,v} \partial_{z^i} \! \right) \! \big( \widehat{Z}^\beta_\infty  h \big).
\end{equation*}
As $\T_F^\infty (h)=0$, we can then express $\partial_t \int_{\R^3_z} \widehat{Z}^\beta_\infty h(t,z,v) \dr z$ as a combination, with coefficients which are polynomials in $\widehat{v}$, of quantities of the form
\begin{alignat}{2}
& \hspace{-1mm} \bullet \; \, \int_{\R^3_z} R \!\left(z \right) P_{p,q}(\C) \, t^{-1-m}\Delta^i_{Z^\gamma}(t,z,v) \, \partial_{z^\nu}  \widehat{Z}_\infty^\xi h (t,z,v) \dr z, \qquad &&  \nu \in \llbracket 0 , 3 \rrbracket   \label{SAT1} ,  \\ \noalign{\vskip3pt}
& \hspace{-1mm} \bullet \; \,  \int_{\R^3_z} R \!\left(z \right) P_{p,q}(\C) \, t^{-m} \, \frac{\widehat{v}^\mu}{v^0} \mathcal{L}_{Z^\gamma}(F)_{\mu i}(t,\XX_\C) \, \widehat{Z}_\infty \widehat{Z}_\infty^\xi h (t,z,v) \dr z, \qquad \qquad  &&  \widehat{Z}_\infty \in \widehat{\mathbb{P}}_S^\infty,   \label{SAT2} 
\end{alignat}
where $i \in \llbracket 1,3 \rrbracket$, $m \in \mathbb{N}$, $R$ is a polynomial in $z=(z^j)_{1 \leq j \leq 3}$ of degree $\deg R$, 
$$\deg R +\xi_H \leq \beta_H, \qquad \qquad \deg R \leq m, \qquad  \qquad |\gamma|+|\xi | \leq |\beta | ,  \qquad \qquad  p, \, q \leq |\beta|+1 . $$
Note then that $\deg R \leq N-1 - \xi_H$ and recall from \eqref{estiPolyC} that $|P_{p,q}(\C)| \lesssim \log^{2p}(t)$. One can then control \eqref{SAT2} by \eqref{T2bisbis}, where $\kappa$ verifies $|\kappa|=|\xi|+1$ and $\xi_H \leq \kappa_H+1$. For \eqref{SAT1}, if $\nu=0$, we write
$$ \partial_{z^0}= \partial_t^\infty=\partial_t-\widehat{v} \cdot \nabla_z= t^{-1}S-t^{-1} z \cdot \nabla_z-\widehat{v} \cdot \nabla_z.$$
Hence, we are lead to bound
\begin{alignat}{2}
& \hspace{-9mm} \log^{2p}(t) \bigg|\int_{\R^3_z} t^{-1}\Delta^i_{Z^\gamma}(t,z,v) \, \partial_{z^j}  \widehat{Z}_\infty^\xi h (t,z,v) \dr z \bigg|, \qquad &&  j \in \llbracket 1 , 3 \rrbracket    \label{SAT1bis} ,  \\ \noalign{\vskip3pt}
& \hspace{-9mm}  \log^{2p}(t) \bigg|\int_{\R^3_z} R \!\left(z \right)  t^{-2}\Delta^i_{Z^\gamma}(t,z,v) \, \widehat{Z}_\infty^\kappa h (t,z,v) \dr z \bigg|, \qquad && \deg R +\kappa_H \leq \beta_H+1, \quad |\kappa| \leq |\beta|+1-|\gamma|. \label{SAT2bis}
\end{alignat}
By performing an integration by parts, we get
$$
\bigg|\int_{\R^3_z} t^{-1}\Delta^i_{Z^\gamma}(t,z,v) \, \partial_{z^j}  \widehat{Z}_\infty^\xi (t,z,v) \dr z \bigg|  \lesssim \int_{\R^3_z} t \big|\nabla_{\partial_{x^j}}\mathcal{L}_{Z^\gamma} F\big|(t,\XX_\C) \,  \big|\widehat{Z}_\infty^\xi h\big|(t,z,v) \dr z .$$
Next, in order to control \eqref{SAT2bis} by quantities of the form \eqref{T2bisbis}--\eqref{T4bisbis}, we simply write
$$ \big| R \!\left(z \right)  t^{-2}\Delta^i_{Z^\gamma}(t,z,v) \, \widehat{Z}_\infty^\kappa h (t,z,v) \big| \lesssim \Big( \big| \mathcal{L}_{Z^\gamma} F \big|(t,\XX_\C)+ t^{-2}\big| \mathbb{F} \big[ \widehat{Z}^\gamma_\infty f_\infty \big] \big| (v) \Big) \langle z \rangle^{N-\kappa_H}  \big| \widehat{Z}_\infty^\kappa h (t,z,v) \big|$$
and we use $\big| \mathbb{F} \big[ \widehat{Z}^\gamma_\infty f_\infty \big] \big| (v) \lesssim \sqrt{\epsilon} \lesssim \log(t)$.
\end{proof}

It allows to derive the following simpler estimate. For this, recall from Definition \ref{DeffuncspaVla} the energy $\mathbf{E}^{7,14}_N[\cdot]$.
\begin{Cor}\label{Corestipartialtg}
Consider $h$ a sufficiently regular solution to $\T_F^\infty (h)=0$. For any $|\beta| \leq N-1$ and almost all $(t,v) \in [T,+\infty[ \times \R^3_v$,
\begin{align*}
 \bigg| \partial_t \int_{\R^3_z} \widehat{Z}^\beta_\infty h(t,z,v) \dr z \bigg|  \lesssim \sum_{|\kappa| \leq N} & \, \frac{\sqrt{\Lambda} \log^{2N+4}(t)}{ t^2} \int_{\R^3_z}  \langle v \rangle^3  \langle z \rangle^{N+2-\kappa_H} \, \left| \widehat{Z}_\infty^\kappa g \right|(t,z,v) \dr z \\
 & + \log^{2N+2}(t) \big|  \mathbf{E}^{7,13}_{N}[h] \big|^{\frac{1}{2}} \int_{\R^3_z}  \frac{\langle t-|\XX_\C| \rangle}{ \langle z \rangle^4 \, \langle v \rangle^{12}} \left|\nabla_{t,x} \mathcal{L}_{Z^\kappa}F \right|(t,\XX_\C) \dr z  .
\end{align*}
\end{Cor}
\begin{proof}
This result is a consequence of Corollary \ref{Corestipartialtgg}, $\Lambda \geq 1$, and the following properties.
\begin{itemize}
\item Lemma \ref{gainvadapted} and the pointwise decay estimates \eqref{eq:Mpoint}, together with Lemma \ref{improderiv}, give, for any $|\gamma| \leq N-4$,
\begin{align*}
 t \big|\nabla_{t,x}\mathcal{L}_{Z^\gamma}  F \big|(t,\XX_\C) & \lesssim \sqrt{\Lambda} \langle t-|\XX_\C| \rangle^{-2} \lesssim \sqrt{\Lambda} \log^4(t)  t^{-2} \langle z \rangle^2 \, \langle v \rangle^4 , \\
  \big|\mathcal{L}_{Z^\gamma}  F \big|(t,\XX_\C) & \lesssim \sqrt{\Lambda} \, t^{-1} \, \langle t-|\XX_\C| \rangle^{-1} \lesssim \sqrt{\Lambda} \log^2(t) t^{-2} \langle z \rangle \, \langle v \rangle^2 .
  \end{align*}
  \item If $N-3 \leq \gamma \leq N-1$, we write $|\mathcal{L}_{Z^\gamma}F| \lesssim \langle t+r \rangle |\nabla_{t,x}\mathcal{L}_{Z^\kappa}F|$, where $\kappa$ is such that $Z^\gamma = ZZ^\kappa$. Moreover, by Lemma \ref{gainvadapted},
  $$ \langle t+|\XX_\C| \rangle \big|\nabla_{t,x}\mathcal{L}_{Z^\gamma}  F \big|(t,\XX_\C) \lesssim  \log^2(t) \, \langle z \rangle \, \langle v \rangle^2 \, \langle t-|\XX_\C| \rangle \big|\nabla_{t,x}\mathcal{L}_{Z^\gamma}  F \big|(t,\XX_\C) .$$
  \item Finally, according to the Sobolev embedding $  H^2_{z,v} \hookrightarrow L^\infty_{z,v}$, we have, for all $|\kappa| \leq 3 \leq N-5$,
$$  \langle z \rangle^{5+N-\kappa_H}  \langle v \rangle^{13}   \big| \widehat{Z}_\infty^\kappa h \big|(t,z,v) \lesssim \big|  \mathbf{E}^{5,13}_{N-3}[h] \big|^{\frac{1}{2}} \leq \big|  \mathbf{E}^{7,13}_{N}[h] \big|^{\frac{1}{2}}.$$
\end{itemize}
\end{proof}

\subsection{Estimates of the error terms}

We now perform estimates which will allow us to control the terms \eqref{T1bis}--\eqref{T3bis}. We start by considering the cases where the electromagnetic field is differentiated at most $N-3$ times, so that it can be estimated pointwise and we have access to improved estimates on its good null components.

\begin{Pro}\label{Proestloworder}
Let $i \in \llbracket 1 , 3 \rrbracket$, $|\gamma| \leq N-3$. Then, for all $(t,z,v) \in [T,+\infty[ \times \R^3_z \times \R^3_v$,
\begin{align*}
\frac{ \log^{2N+2}(t)}{v^0 \langle z \rangle \, t}   \Big|   t^2 \widehat{v}^\mu  \mathcal{L}_{Z^\gamma}(F)_{\mu i}(t,\XX_\C)- \widehat{v}^\mu \mathbb{F}_{\mu i} \big[ \widehat{Z}^\gamma_\infty f_\infty \big](v)  \Big| & \lesssim \frac{\sqrt{\Lambda} }{  t^{\frac{4}{3}}}+\frac{\sqrt{\Lambda} \, \widehat{v}^{\underline{L}}(\XX_\C)  }{ \langle t-|\XX_\C| \rangle^{\frac{4}{3}}}, \\
  \log^{2N+2}(t)\frac{\langle t+|z| \rangle}{v^0 \, t} \big|  \mathcal{L}_{Z^\gamma}(F)(t,\XX_\C) \big| & \lesssim \frac{\sqrt{\Lambda}}{ t^{\frac{7}{4}}}+\frac{\sqrt{\Lambda} \, \widehat{v}^{\underline{L}}(\XX_\C)}{ \langle t-|\XX_\C|\rangle^{2}}.
\end{align*}
If $0 \leq \nu \leq 3$ and $|\xi| \leq N-4$, there holds
$$ \frac{\langle t+|z| \rangle}{v^0}\Big|  \widehat{v}^\mu  \mathcal{L}_{\partial_{x^\nu}Z^\xi}(F)_{\mu i}\Big|(t,\XX_\C)   \lesssim \frac{\sqrt{\Lambda}}{ t^2}+\frac{\sqrt{\Lambda} \, \widehat{v}^{\underline{L}}(\XX_\C)}{ \langle t-|\XX_\C|\rangle^{2}}.$$
\end{Pro}
\begin{proof}
For the first estimate, start by using the convergence estimate \eqref{eq:Mpointconv} for $x=z+\C_{t,v}$. Since $\langle z+\C_{t,v} \rangle \lesssim \log^2(t) \,  \langle z \rangle$, we get 
\begin{align*}
  \log^{2N+2}(t)\frac{1}{v^0 \langle z \rangle \, t}   \Big|   t^2 \widehat{v}^\mu  \mathcal{L}_{Z^\gamma}(F)_{\mu i}(t,\XX_\C)- \widehat{v}^\mu \mathbb{F}_{\mu i} \big[ \widehat{Z}^\gamma_\infty f_\infty \big](v)  \Big| & \lesssim \sqrt{\Lambda}  \log^{2N+4}(t)  \bigg( \, \frac{ |v^{\underline{L}}(\XX_\C) |^{\frac{1}{2}}}{ t \, \langle t-|\XX_\C| \rangle^{\frac{1}{2}}} +\frac{1}{t^2} \bigg) \\
 & \lesssim \sqrt{\Lambda}  \big| \widehat{v}^{\underline{L}}(\XX_\C) \big|^{\frac{1}{2}} \,  t^{-\frac{5}{6}}  \langle t-|\XX_\C| \rangle^{-\frac{1}{2}}+\sqrt{\Lambda} \, t^{-\frac{7}{4}}.
 \end{align*}
If $|\XX_\C| \geq 2t$, it implies the stated estimate. Otherwise $t^{-1/6} \leq \langle t-|\XX_\C| \rangle^{-1/6}$ and it then suffices to use the inequality $2ab \leq a^2+b^2$. Next, using first 
$$ t+|z| \leq  2 t+|\XX_\C|+|\C_{t,v}| \lesssim t+|\XX_\C|, \qquad \qquad 1 \leq 2v^0 |\widehat{v}^{\underline{L}}|^{1/2},$$
which is given by Lemma \ref{gainv}, as well as the pointwise decay estimate \eqref{eq:Mpoint}, we obtain
$$\log^{2N+2}(t)\frac{\langle t+|z| \rangle}{v^0 \, t} \big|  \mathcal{L}_{Z^\gamma}(F)(t,\XX_\C) \big| \lesssim  \frac{\sqrt{\Lambda} \big| \widehat{v}^{\underline{L}}(\XX_\C) \big|^{\frac{1}{2}}\log^{2N+2}(t)}{t \, \langle t-|\XX_\C| \rangle} \leq \frac{\sqrt{\Lambda} \log^{4N+4}(t)}{t^2}+\frac{\Lambda \,\widehat{v}^{\underline{L}}(\XX_\C)}{  \langle t - |\XX_\C| \rangle^2} ,$$
which implies the second inequality. Finally, for the last one, we start by expanding $\widehat{v}^\mu \mathcal{L}_{\partial_{x^\nu}Z^\xi}(F)_{\mu i}$ according to the null frame $(\underline{L},L,e_\theta, e_\varphi)$. More precisely, Lemmata \ref{gainv} and \ref{nullstruct} imply
$$
 \left| \widehat{v}^\mu \mathcal{L}_{\partial_{x^\nu}Z^\xi}(F)_{\mu i} \right| \lesssim v^0 \big| \widehat{v}^{\underline{L}} \big|^{\frac{1}{2}} \big( |\alpha (\mathcal{L}_{\partial_{x^\nu}Z^\xi} F)|+|\rho (\mathcal{L}_{\partial_{x^\nu}Z^\xi} F)|+|\sigma (\mathcal{L}_{\partial_{x^\nu}Z^\xi} F)|\big) +v^0 \widehat{v}^{\underline{L}}|\underline{\alpha }( \mathcal{L}_{\partial_{x^\nu}Z^\xi} F)|. $$
Recall from Lemma \ref{improderiv} that we have improved decay estimates for the quantities on the right hand side since $\partial_{x^\nu}Z^\xi$ contains a translation. Hence, using the pointwise decay estimates \eqref{eq:Mpointnull} for the null components of $\mathcal{L}_{\partial_{x^\nu}Z^\gamma}F$, we get
$$ \frac{\langle t+|z| \rangle}{v^0}\Big|  \widehat{v}^\mu  \mathcal{L}_{\partial_{x^\nu}Z^\xi}(F)_{\mu i}(t,\XX_\C) \Big| \lesssim \frac{\sqrt{\Lambda} \big| v^{\underline{L}}(\XX_\C) \big|^{\frac{1}{2}}}{\langle t+|\XX_\C| \rangle \, \langle t-|\XX_\C| \rangle}+\frac{\sqrt{\Lambda}  v^{\underline{L}}(\XX_\C) }{ \langle t-|\XX_\C| \rangle^2}.$$
It remains to use again the inequality $2ab \leq a^2+b^2$.
\end{proof}

We now prove estimates which will allow us to deal with the error terms where the electromagnetic field is differentiated too many times in order to be controlled pointwise. For simplicity, even if the estimates that we have at our disposal on $F$ are worse at the top order, we will unify the analysis of the cases $N-2 \leq |\gamma| \leq N$. In fact, it will be convenient to use the coordinate system $(t,x,v)=(t,z-t\widehat{v}-\C_{t,v},v)$ in the energy estimates. 
\begin{Pro}\label{Proesttoporder}
Let $1 \leq i \leq 3$, $N-2 \leq |\gamma| \leq N$ and $N-3 \leq \xi \leq N-1$. Then, we have
\begin{align*}
  \int_{\R^3_x } \int_{\R^3_v} \frac{1}{ t^2}   \Big|t^2 \widehat{v}^\mu  \mathcal{L}_{Z^\gamma}(F)_{\mu i}(t,x)- \widehat{v}^\mu \mathbb{F}_{\mu i}\big[ \widehat{Z}^\gamma f_\infty \big](v)  \Big|^2 \frac{\dr v  \dr x}{ \langle x-t\widehat{v}-\C_{t,v} \rangle^7 \, \langle v \rangle^{13}} & \lesssim \frac{\Lambda \log^{9}(t)}{t^3}, \\ \noalign{\vskip3pt}
 \int_{\R^3_x } \int_{\R^3_v}\frac{ \langle t+|x| \rangle^2 \, }{t^2} \langle t-|x| \rangle^2\big|  \nabla_{t,x}  \mathcal{L}_{Z^\xi} F\big|^2(t,x) \frac{ \dr v \dr x}{ \langle x-t\widehat{v}-\C_{t,v} \rangle^4 \, \langle v \rangle^{7}} & \lesssim \frac{\Lambda \log^5(t)}{t^4} .
\end{align*}
\end{Pro}
\begin{proof}
The first estimate is a direct consequence of \eqref{eq:ML2conv}. For the second one, using first first Lemma \ref{gainv} and then $|\C_{t,v}| \lesssim \log^2(t)$ as well as Lemma \ref{decayparticucase}, we have
$$ \int_{\R^3_v}\frac{ \dr v}{ \langle x-t\widehat{v}-\C_{t_v} \rangle^4 \, \langle v \rangle^{7}} \lesssim \frac{\max (t-|x|,1) }{\langle t+|x| \rangle} \int_{\R^3_v}\frac{ \langle x-t\widehat{v} \rangle \, \dr v}{ \langle x-t\widehat{v}-\C_{t_v} \rangle^4 \, \langle v \rangle^{5}} \lesssim \frac{\max (t-|x|,1) }{\langle t+|x| \rangle^4} \log^3(t).$$
It remains to use the $L^2_x$ estimate \eqref{eq:ML2}. 
\end{proof}

\subsection{Well-posedness for the Vlasov equation with asymptotic data}

The purpose of this subsection is to solve the asymptotic Cauchy problem
\begin{equation}\label{Cauchy:asymp}
\begin{cases}
 \,  \T_F^\infty (h)(t,z,v) =0 , \qquad &(t,z,v) \in [T,+ \infty[ \times \R^3_z \times \R^3_v, \\
 \underset{t \to + \infty}{\lim} \, h(t,z,v) = h_\infty(z,v), \quad &(z,v) \in \R^3_z \times \R^3_v,
\end{cases}
\end{equation}
for a sufficiently regular asymptotic data $h_\infty : \R^3_z \times \R^3_v \to \R$. For this, in order to perform energy estimates and motivated by the commutation formula of Corollary \ref{CorCom}, we introduce (the square of) hierarchised weighted energy norms. We define
$$ \overline{h}(t,x,v) := h(t,x-t\widehat{v}-\C_{t,v},v), \qquad \qquad \T_F \big( \, \overline{h} \, \big) =\overline{ \T_F^\infty ( h)} =0,$$
which simply corresponds to the change of coordinates $(t,z,v)=(t,x+t\widehat{v}+\C_{t,v},v)$, so that $\overline{\XX_\C}(t,x,v)=x$. 
\begin{Def}
Let $M \in \mathbb{N}$, $(N_z,N_v) \in \R^2$ as well as $h :[T,+\infty[ \times \R^3_z \times \R^3_v \to \R$ and $h_\infty : \R^3_z \times \R^3_v \to \R$ be two sufficiently regular functions. Then, we consider
\begin{align*}
\mathbb{E}_M^{N_z,N_v}[h](t,u) := \sup_{|\kappa| \leq M} \, \overline{\mathbb{E}} \Big[ \;   \overline{ \langle z \rangle^{N_z+M-\kappa_H} \,\langle v \rangle^{N_v}  \big|\widehat{Z}^\kappa_\infty  h \big|} \; \Big] (t,u),  \qquad t \geq T, \, u \in \R,
\end{align*}
where $\overline{\mathbb{E}}[ \cdot ]$ is introduced in \eqref{defnormVlasovL2}, as well as
\begin{align*}
 \mathbb{E}_M^{N_v,N_z}[h_\infty] := \sup_{|\beta| \leq M} \int_{\R^3_z} \int_{\R^3_v}       \langle z \rangle^{2N_z+2M-2\beta_H} \, \langle v \rangle^{2N_v}  \big|\widehat{Z}^\beta_\infty h_\infty \big|^2(z,v) \dr v \dr z ,
\end{align*}
which is equivalent to $\sup_{|\kappa| \leq M} \,\|      \langle z \rangle^{N_z+N-|\kappa_v|} \, \langle v \rangle^{N_v+|\kappa_v|} \partial_v^{\kappa_v} \partial_z^{\kappa_z} h_\infty \|_{L^2_{z,v}}^2$.
\end{Def}

We are now able to prove the following existence and uniqueness result for \eqref{Cauchy:asymp}. We could consider less regular initial data but it would not be useful for the purpose of this paper.
\begin{Pro}\label{ProCauchpb}
Let $(N_z,N_v)\in \R^2$ and  $h_\infty : \R^3_z \times \R^3_v \to \R$ be such that $\mathbb{E}_N^{N_z,N_v}[h_\infty]< +\infty$. Then, there exists a unique global and classical solution $h : [T,+\infty[ \times \R^3_z \times \R^3_v \to \R$ to the asymptotic Cauchy problem \eqref{Cauchy:asymp}. Moreover, there exists a constant $B_N^{N_z,N_v} >0$, such that
\begin{equation}\label{eq:estiwellposed}
 \sup_{t \geq T} \, \sup_{u \in \R} \, \mathbb{E}_{N-3}^{N_z,N_v}[h](t,u)+ \sup_{t \geq T} \, \sup_{u \in \R} \, \mathbb{E}^{N_z-3,N_v-6}_{N}[h](t,u) \leq \mathbb{E}_{N}^{N_z,N_v}[h_\infty]\exp \left[ \exp \left( B_N^{N_z,N_v} \, \sqrt{\Lambda} \right) \right] .
 \end{equation}
\end{Pro}

The proof of Proposition \ref{ProCauchpb} is divided in three parts. We derive first a priori estimates for solutions to the Vlasov equation with asymptotic data. Then, we prove, by studying the characteristics of the operator $\T_F^\infty$, that for any $h_\infty \in C^1(\R^3_z \times \R^3_v)$, there exists a global $C^1$ solution to \eqref{Cauchy:asymp}. Finally, by combining these results and by a standard approximation argument, we are able to conclude the proof of Proposition \ref{ProCauchpb}.
\subsubsection{A priori estimates} We start by proving a two variables version of Gronwall's inequality.

\begin{Lem}\label{LemGronwall2variables}
Let $\Psi \geq 0$ and three continous function 
$$\psi_1 \in L^1( [T,+\infty[ , \R_+), \qquad \psi_2 \in L^1( \R_u , \R_+), \qquad H: [T,+ \infty[ \times \R_u \to \R_+.$$
Assume further that $ H(t_1,u) \geq H(t_2,u)$ for all $t_1 \leq t_2$ and $u \in \R$ as well as
$$ \forall \, (t,u) \in [T,+\infty[ \times \R_u, \qquad  H(t,u) \leq \Psi+\int_{\tau=t}^{+\infty} \psi_1(\tau) H(\tau,u) \dr s+\int_{U=u}^{+\infty} \psi_2(U) H(t,U) \dr U  .$$
Then,
$$ \forall \, (t,u) \in [T,+\infty[ \times \R_u, \qquad H(t,u) \leq \Psi \exp \left( \| \psi_1 \|_{L^1_t}  + e^{ \| \psi_1 \|_{L^1_t} } \| \psi_2\|_{L^1_u} \right).$$
\end{Lem}
\begin{proof}
Fix $u \in \R_u$. As $t \mapsto \int_{U=u}^{+\infty} \psi_2(U) H(t,U) \dr u$ is non-increasing, we have, by applying Gronwall's inequality to $t \mapsto H(t,u)$,
\begin{align*}
 \forall \, t \geq T, \qquad H(t,u) & \leq \left(\Psi+\int_{U=u}^{+\infty} \psi_2(U) H(t,U) \dr U \right) \exp \left( \int_{\tau=t}^{+ \infty} \psi_1(\tau) \dr \tau \right) \\
 & \leq \Psi e^{ \| \psi_1 \|_{L^1_t} }+\int_{U=u}^{+\infty} e^{ \| \psi_1 \|_{L^1_t} } \psi_2(\tau) H(t,U) \dr u .
 \end{align*}
Since this inequality holds for all $u \in \R_u$, we can this time fix $t \geq T$ and apply Gronwall's inequality to the function $u \mapsto H(t,u)$.
\end{proof}

We are now able to prove the next result.

\begin{Pro}\label{Proapriori}
Consider an asymptotic data $h_\infty : \R^3_z \times \R^3_v \to \R$ and $h : [T,+ \infty[ \times \R^3_z \times \R^3_v \to \R$. Assume that $h$ is a solution to \eqref{Cauchy:asymp} satisfying
$$ \mathbb{E}_{N}^{N_z,N_v}[h_\infty] < +\infty, \qquad \qquad \sup_{t \geq T} \, \sup_{u \in \R} \, \mathbb{E}_{N-3}^{N_z,N_v}[h](t,u)+ \sup_{t \geq T} \, \sup_{u \in \R} \, \mathbb{E}^{N_z-3,N_v-6}_{N}[h](t,u) < +\infty.$$
Then, $h$ verifies the estimate \eqref{eq:estiwellposed}.
\end{Pro}
\begin{proof}
In order to lighten the presentation, we introduce the weighted derivatives of $h$ that we have to bound in order to control the lower order energy norm $\mathbb{E}^{N_z,N_v}_{N-3}[h]$,
$$ h^\kappa_{\mathrm{low}} (t,z,v) := \langle v \rangle^{N_v} \, \langle z \rangle^{N_z+N - \kappa_H}  \widehat{Z}^\kappa_\infty h (t,z,v), \qquad \overline{h^\kappa_{\mathrm{low}}}(t,x,v) = h^\kappa_{\mathrm{low}} (t,x-t\widehat{v}-\C_{t,v},v), \qquad |\kappa| \leq N-3.$$
Let $t_0 \geq T$ and $u_0 \in \R$. According to the energy estimate of Proposition \ref{ProenergyL2Vlasov}, we have
$$  \mathbb{E}_{N-3}^{N_z,N_v}[h](t_0,u_0) \leq \mathbb{E}_{N}^{N_z,N_v}[h_\infty]+  \sup_{|\beta| \leq N-3} \, \int_{t=t_0}^{+ \infty} \int_{|x| \leq t-u_0}  \int_{\R^3_v}  \bigg| \,  \overline{  \T_F^\infty \big( |h^\beta_{\mathrm{low}} |^2 \big)^{\mathstrut} } \, \bigg|(t,x,v) \dr v \dr x \dr t.$$
Since $\T_F^\infty (h)=0$ and $|\beta| \leq N-3$, the integrands on the right hand side of the previous inequality can be estimated by applying the commutation formula of Corollary \ref{CorCom} and Proposition \ref{Proestloworder}. It yields, since $\overline{\XX_\C}(t,x,v)=x$, to
$$  \mathbb{E}_{N-3}^{N_z,N_v}[h](t_0,u_0) - \mathbb{E}_{N}^{N_z,N_v}[h_\infty] \lesssim  \sup_{|\beta| , \, |\kappa| \leq N-3} \,  \int_{t=t_0}^{+ \infty} \! \int_{|x| \leq t-u_0} \int_{\R^3_v} \bigg( \frac{\sqrt{\Lambda}}{t^{\frac{4}{3}}} + \frac{\sqrt{\Lambda} \widehat{v}^{\underline{L}} ( x) }{\langle t-|x| \rangle^{\frac{4}{3}}} \bigg) \Big|\overline{ h^\kappa_{\mathrm{low}} \, h^\beta_{\mathrm{low}} } \Big|(t,x,v) \dr v \dr x \dr t .$$
Since $2h ^\kappa_{\mathrm{low}} \, h^\beta_{\mathrm{low}} \leq |h^\kappa_{\mathrm{low}}|^2+|h^\beta_{\mathrm{low}}|^2$, we finally get
$$  \mathbb{E}_{N-3}^{N_z,N_v}[h](t_0,u_0) - \mathbb{E}_{N}^{N_z,N_v}[h_\infty] \lesssim \sup_{|\kappa| \leq N-3} \mathcal{T}_{t_0,u_0}\big[ |h^\kappa_{\mathrm{low}}|^2 \big]+\mathcal{U}_{t_0,u_0} \big[ |h^\kappa_{\mathrm{low}} |^2\big],$$
where
\begin{align*}
  \mathcal{T}_{t_0,u_0}\big[ | h^\kappa_{\mathrm{low}}|^2 \big]&:=    \int_{t=t_0}^{+ \infty} \! \int_{|x| \leq t-u_0} \int_{\R^3_v}  \frac{\sqrt{\Lambda}}{t^{\frac{4}{3}}} \Big|\overline{ h^\kappa_{\mathrm{low}}  } \Big|^2(t,x,v) \dr v \dr x \dr t  \leq \int_{t=t_0}^{+ \infty}  \frac{\sqrt{\Lambda}}{t^{\frac{4}{3}}} \, \mathbb{E}_{N-3}^{N_z,N_v} [h](t,u_0) \dr t, \\
 \mathcal{U}_{t_0,u_0}\big[ |h^\kappa_{\mathrm{low}}|^2 \big] &:=  \int_{t=t_0}^{+ \infty} \! \int_{|x| \leq t-u_0} \int_{\R^3_v}  \frac{\sqrt{\Lambda} \widehat{v}^{\underline{L}} ( x) }{\langle t-|x| \rangle^{\frac{4}{3}}}  \Big|\overline{h^\kappa_{\mathrm{low}}} \Big|^2(t,x,v) \dr v \dr x \dr t      .
\end{align*}
For $\mathcal{U}_{t_0,u_0}[| h^\kappa_{\mathrm{low}}|^2 ]$, we perform the change of variables $(u,\underline{u},\omega)=(t-|x|,t+|x|,x/|x|)$, so that
\begin{align*}
\mathcal{U}_{t_0,u_0}\big[ |h^\kappa_{\mathrm{low}}|^2 \big] &=  \int_{u \geq u_0}  \int_{\underline{u} \geq \max(u,2t_0-u)} \int_{\mathbb{S}^2_\omega} \int_{\R^3_v} \frac{\sqrt{\Lambda} \widehat{v}^{\underline{L}}(\omega)}{\langle u \rangle^{\frac{4}{3}}}   \Big|\overline{h^\kappa_{\mathrm{low}}} \Big|^2 \Big( \frac{\underline{u}+u}{2},\frac{\underline{u}-u}{2}\omega \Big) \dr v  \frac{r^2}{2} \dr \mu_{\mathbb{S}^2_\omega} \dr \underline{u} \dr u \\
& = \int_{u \geq u_0} \frac{\sqrt{\Lambda}}{\sqrt{2}\langle u \rangle^{\frac{4}{3}}} \int_{C_u^{t_0}}  \int_{\R^3_v} \widehat{v}^{\underline{L}}  \Big|\overline{h^\kappa_{\mathrm{low}}} \Big|^2 \dr v \dr \mu_{C_u} \dr u \leq  \int_{u \geq u_0} \frac{\sqrt{\Lambda} }{\sqrt{2}\langle u \rangle^{\frac{4}{3}}} \, \mathbb{E}_{N-3}^{N_z,N_v} [h](t_0,u) \dr u.
\end{align*}
We then deduce that, for all $(t_0,u_0) \in [T,+ \infty[ \times \R_u$,
 $$   \mathbb{E}_{N-3}^{N_z,N_v}[h](t_0,u_0) - \mathbb{E}_{N}^{N_z,N_v}[h_\infty] \lesssim \int_{t=t_0}^{+ \infty}  \frac{\sqrt{\Lambda}}{t^{\frac{4}{3}}} \, \mathbb{E}_{N-3}^{N_z,N_v} [h](t,u_0) \dr t+\int_{u \geq u_0} \frac{\sqrt{\Lambda} }{\langle u \rangle^{\frac{4}{3}}} \, \mathbb{E}_{N-3}^{N_z,N_v} [h](t_0,u) \dr u.$$
The estimate for $\mathbb{E}_{N-3}^{N_z,N_v}[h]$ then ensues from Lemma \ref{LemGronwall2variables}, applied to $H(t,u)=\sup_{\tau \geq t} \mathbb{E}^{N_z,N_v}_{N-3}[h](\tau,u)$. We now turn to the higher order energy norm. We then introduce the notation
 $$ h^\kappa_{\mathrm{high}}:= \langle v \rangle^{N_v-6} \, \langle z \rangle^{N_z-3+N - \kappa_H}  \widehat{Z}_\infty^\kappa h (t,z,v), \qquad \overline{h^\kappa_{\mathrm{high}}}(t,x,v) = h^\kappa_{\mathrm{high}} (t,x-t\widehat{v}-\C_{t,v},v), \qquad |\kappa| \leq N.$$
 As before, we fix $t_0 \geq T$, $u_0 \in \R$ and we apply the energy inequality of Proposition \ref{ProenergyL2Vlasov}. We claim that it yields to
 \begin{align*}
   \mathbb{E}_{N}^{N_z-3,N_v-6}[h](t_0,u_0) - \mathbb{E}_{N}^{N_z,N_v}[h_\infty] \lesssim \! \sum_{|\beta|\leq N} & \sum_{  |\kappa| \leq N}  \mathcal{T}_{t_0,u_0}\big[  h^\kappa_{\mathrm{high}} \, h^\beta_{\mathrm{high}} \big]+ \mathcal{U}_{t_0,u_0}\big[  h^\kappa_{\mathrm{high}} \, h^\beta_{\mathrm{high}} \big] \\
&   \!+\!\sum_{N-2 \leq |\gamma| \leq N} \, \sum_{|\kappa| \leq N+1-|\gamma|}  \mathcal{I}^{\, \gamma, \kappa,\beta}_{t_0,u_0} +\sum_{N-3 \leq |\xi| \leq N-1} \, \sum_{|\kappa| \leq N-|\xi|} \mathcal{J}^{ \,\xi, \kappa,\beta}_{t_0,u_0}  , 
   \end{align*}
where
   \begin{align*}
\mathcal{I}^{\, \gamma, \kappa,\beta}_{t_0,u_0} &:=  \int_{t=t_0}^{+ \infty} \! \int_{|x| \leq t-u_0} \int_{\R^3_v}   \Big|t^2 \widehat{v}^\mu  \mathcal{L}_{Z^\gamma}(F)_{\mu j}(t,x)- \widehat{v}^\mu \mathbb{F}_{\mu j}\big[ \widehat{Z}^\gamma_\infty f_\infty \big](v)  \Big| \, \Big|\overline{ h^\kappa_{\mathrm{high}} \, h^\beta_{\mathrm{high}} } \Big|(t,x,v) \frac{\log^{2N+2}(t)  \dr v \dr x \dr t}{\langle v \rangle \, \langle x-t\widehat{v}-\C_{t,v} \rangle \, t}  , \\
\mathcal{J}^{\, \xi, \kappa,\beta}_{t_0,u_0}& :=  \int_{t=t_0}^{+ \infty} \! \int_{|x| \leq t-u_0} \int_{\R^3_v}   \frac{ \log^{2N+4}(t) \, \langle t+|x| \rangle}{ t } \langle t-|x| \rangle \Big|  \nabla_{t,x} \mathcal{L}_{ Z^\xi}(F) \Big|(t,x)  \Big|\overline{ \langle z \rangle \, \langle v \rangle \, h^\kappa_{\mathrm{high}} \, h^\beta_{\mathrm{high}} } \Big|(t,x,v) \dr v \dr x \dr t.
\end{align*}
Indeed, we bound $\T_F^\infty( |h^\beta_{\mathrm{high}}|^2)$ by applying the commutation formula of Corollary \ref{CorCom}, for any $|\beta| \leq N$. 
\begin{itemize}
\item The error terms of type \eqref{T1bis}--\eqref{T3bis} give rise, when $N-2 \leq |\gamma| \leq N$ and according to Remark \ref{RqCom}, to $\mathcal{I}^{\, \gamma, \kappa,\beta}_{t_0,u_0}$ and $\mathcal{J}^{\, \xi, \kappa,\beta}_{t_0,u_0}$.
\item Otherwise $|\gamma| \leq N-3$ and they can be controlled by $\mathcal{T}_{t_0,u_0}[  h^\kappa_{\mathrm{high}} \, h^\beta_{\mathrm{high}} ]+ \mathcal{U}_{t_0,u_0} [  h^\kappa_{\mathrm{high}} \, h^\beta_{\mathrm{high}} ]$ according to Proposition \ref{Proestloworder}.
\item Finally the source terms \eqref{T4bis} are bounded by $\mathcal{T}_{t_0,u_0}[  h^\kappa_{\mathrm{high}} \, h^\beta_{\mathrm{high}} ]$.
\end{itemize}
Arguing as previously, we obtain
\begin{align*}
\sum_{|\beta| ,  |\kappa| \leq N}  \mathcal{T}_{t_0,u_0}\big[  h^\kappa_{\mathrm{high}} \, h^\beta_{\mathrm{high}} \big]& \lesssim \sum_{ |\kappa| \leq N}  \mathcal{T}_{t_0,u_0}\big[  |h^\kappa_{\mathrm{high}} |^2 \big]  \lesssim \int_{t = t_0}^{+\infty} \frac{\sqrt{\Lambda}}{t^{\frac{4}{3}}} \mathbb{E}^{N_z-5,N_v-11}_{N}[h](t,u_0) \dr t , \\
\sum_{|\beta| ,  |\kappa| \leq N}  \mathcal{U}_{t_0,u_0}\big[  h^\kappa_{\mathrm{high}} \, h^\beta_{\mathrm{high}} \big]& \lesssim \sum_{ |\kappa| \leq N}  \mathcal{U}_{t_0,u_0}\big[  |h^\kappa_{\mathrm{high}} |^2 \big]  \lesssim \int_{u \geq u_0} \frac{\sqrt{\Lambda} }{2\langle u \rangle^{\frac{4}{3}}} \, \mathbb{E}^{N_z-5,N_v-11}_{N}[h](t,u_0) \dr u.
\end{align*}
Fix now $|\beta| \leq N$, $|\kappa| \leq 3$, $N-2 \leq |\gamma| \leq N$ and $N-3 \leq |\xi| \leq N-1$. The Cauchy-Schwarz inequality in $(x,v)$ gives
\begin{align*}
\mathcal{I}^{\, \gamma, \kappa,\beta}_{t_0,u_0} \leq \int_{t=t_0}^{+ \infty} \bigg( \, &  \bigg|     \int_{\R^3_x } \int_{\R^3_v}  \Big|t^2 \widehat{v}^\mu  \mathcal{L}_{Z^\gamma}(F)_{\mu j}(t,x)- \widehat{v}^\mu \mathbb{F}_{\mu j} \big[ \widehat{Z}_\infty^\gamma f_\infty \big](v)  \Big|^2 \frac{\log^{4N+4}(t)\dr v \, \dr x}{ t^2 \langle x-t\widehat{v}-\C_{t,v} \rangle^7 \, \langle v \rangle^{13}} \bigg|^{\frac{1}{2}} \\
 & \times \bigg|  \int_{|x| \leq t-u_0} \int_{\R^3_v}  \Big|\overline{  \langle z \rangle^{\frac{5}{2}} \, \langle v \rangle^{\frac{11}{2}}  h^\kappa_{\mathrm{high}} \, h^\beta_{\mathrm{high}} } \Big|^2(t,x,v) \dr x \dr v \bigg|^{\frac{1}{2}} \, \bigg) \dr t
\end{align*}
as well as 
\begin{align*}
\mathcal{J}^{\, \xi, \kappa,\beta}_{t_0,u_0} \leq \int_{t=t_0}^{+ \infty} \bigg( &  \bigg|     \int_{\R^3_x } \int_{\R^3_v} \frac{\langle t+|x|\rangle^2}{t^2}\langle t-|x| \rangle^2  \Big|  \nabla_{t,x} \mathcal{L}_{ Z^\xi}(F) \Big|^2\!(t,x) \, \frac{\log^{4N+8}(t)\dr v \, \dr x}{  \langle x-t\widehat{v}-\C_{t,v} \rangle^4 \, \langle v \rangle^{7}} \bigg|^{\frac{1}{2}} \\
 & \times \bigg|  \int_{|x| \leq t-u_0} \int_{\R^3_v}  \Big|\overline{   \langle z \rangle^3 \, \langle v \rangle^{\frac{9}{2}} h^\kappa_{\mathrm{high}} \, h^\beta_{\mathrm{high}} } \Big|^2(t,x,v) \dr x \dr v \bigg|^{\frac{1}{2}} \bigg) \dr t.
 \end{align*}
Applying Proposition \ref{Proesttoporder}, we obtain, 
$$ \mathcal{I}^{\, \gamma, \kappa,\beta}_{t_0,u_0}+\mathcal{J}^{\, \gamma, \kappa,\beta}_{t_0,u_0} \lesssim \int_{t=t_0}^{+ \infty} \frac{\sqrt{\Lambda} \log^{2N+7}(t)}{t^{\frac{3}{2}}} \bigg|  \int_{|x| \leq t-u_0} \int_{\R^3_v}  \Big|\overline{  \langle z \rangle^3 \,  \langle v \rangle^6 h^\kappa_{\mathrm{high}} \, h^\beta_{\mathrm{high}} } \Big|^2(t,x,v) \dr x \dr v \bigg|^{\frac{1}{2}} \dr t.$$
Note, as $|\kappa| \leq 3 \leq N-3$, that $\langle v \rangle^{6} \, \langle z \rangle^3 h^\kappa_{\mathrm{high}}= h^\kappa_{\mathrm{low}}$. Since we have further $|\kappa| \leq 3 \leq N-5$, we can in fact use that $H^2(\R^3_z \times \R^3_v)$ is continously embedded in $L^\infty (\R^3_z \times \R^3_v)$ in order to get
\begin{align*}
 \int_{|x| \leq t-u_0} \int_{\R^3_v}  \Big|\overline{  \langle z \rangle^3 \, \langle v \rangle^{6}  h^\kappa_{\mathrm{high}} \, h^\beta_{\mathrm{high}} } \Big|^2(t,x,v) \dr x \dr v & \lesssim   \sup_{u \in \R}\mathbb{E}^{N_z,N_v}_{N-3}[h](t,u) \int_{|x| \leq t-u_0} \int_{\R^3_v}  \Big|\overline{ h^\beta_{\mathrm{high}} } \Big|^2(t,x,v) \dr x \dr v \\
 & \leq \sup_{u \in \R} \mathbb{E}^{N_z,N_v}_{N-3}[h](t,u) \, \mathbb{E}^{N_z-3,N_v-6}_{N}[h](t,u_0).
 \end{align*}
 By combining the previous estimates and using $2ab \leq a^2+b^2$ as well as $\log^{2N+7}(t)\lesssim t^{1/4}$, we finally get
 \begin{align*}
     \mathbb{E}_{N}^{N_z-3,N_v-6}[h](t_0,u_0) - \mathbb{E}_{N}^{N_z,N_v}[&h_\infty] \lesssim  \int_{t=t_0}^{+ \infty}  \frac{\sqrt{\Lambda}}{t^{\frac{5}{4}}} \, \sup_{\tau \geq T} \, \sup_{u \in  \R} \mathbb{E}_{N-3}^{N_z,N_v} [h](\tau,u) \dr t  \\
     & + \int_{t=t_0}^{+ \infty}  \frac{\sqrt{\Lambda}}{t^{\frac{5}{4}}} \, \mathbb{E}_{N}^{N_z-3,N_v-6} [h](t,u_0) \dr t  +\int_{u \geq u_0} \frac{\sqrt{\Lambda} }{\langle u \rangle^{\frac{4}{3}}} \, \mathbb{E}_{N}^{N_z-3,N_v-6} [h](t_0,u) \dr u.
     \end{align*}
According to the first part of the proof, there exists $\widetilde{D}=\widetilde{D}^{N_z,N_v}_N>0$ such that
$$ \int_{t=t_0}^{+ \infty}  \frac{\Lambda}{t^{\frac{5}{4}}} \, \sup_{\tau \geq T} \, \sup_{u \in  \R} \mathbb{E}_{N-3}^{N_z,N_v} [h](\tau,u) \dr t  \leq 4\Lambda \exp\left( e^{ \widetilde{D} \Lambda} \right) \mathbb{E}^{N_z,N_v}_N [h_\infty]  \leq 4 \exp\left( e^{\Lambda+ \widetilde{D} \Lambda} \right) \, \mathbb{E}^{N_z,N_v}_N [h_\infty].$$
In order to conclude the proof, it remains to apply Lemma \ref{LemGronwall2variables} to $H(t,u)=\sup_{\tau \geq t} \mathbb{E}_{N}^{N_z-3,N_v-6}[h](\tau,u)$, with $\Psi = \big(1+ C \exp \big(e^{\Lambda+\widetilde{D} \Lambda} \big) \big) \mathbb{E}^{N_z,N_v}_N[h_\infty]$, where $C>0$ depends only on $(N,N_z,N_v)$.
\end{proof}

\subsubsection{Study of the characteristics of $\T_F^\infty$}

The purpose of this section is to prove that the characteristics of the operator $\T^\infty_F$ are defined globally and converge as time approaches infinity. For this, it is convenient to introduce
$$
a^i_z:(t,z,v) \mapsto  -\frac{\delta_j^i-\widehat{v}_j \widehat{v}^i}{tv^0} \widehat{v}^\mu \left( t^2  {F_{\mu}}^j(t,\XX_\C)-{\mathbb{F}_{\mu}}^j[f_\infty](v) \right)-\frac{\widehat{v}^\mu}{v^0} {F_{\mu}}^j(t,\XX_\C) v^0 \partial_{v^j}\C^j_{t,v} = \T_F^\infty(z^i), 
$$
so that $\T_F^\infty = \partial_t+a^j_z(t,z,v)\partial_{z^j}+\widehat{v}^\mu {F_{\mu}}^j \big(t,\XX_\C (t,z,v) \big)\partial_{v^j}$.
\begin{Def}\label{defcharac}
For any $(t_0,z,v) \in [T,+\infty[ \times \R^3_z  \times \R^3_v$, we denote by $t \mapsto \big( \Z(t,t_0,z,v), \V(t,t_0,z,v) \big)$ the unique solution to the Cauchy problem
\begin{equation}\label{characdef}
  \dot{\Z}^i = a_z^i(t , \Z, \V), \qquad \dot{\V}^i= \widehat{\V}^\mu {F_{\mu}}^i \big(t,\XX_\C (t,\Z,\V) \big), \qquad \big(\Z (t_0,t_0,z,v),\V (t_0,t_0,z,v) \big)=(z,v).
  \end{equation}
A characteristic of the operator $\T_F^\infty$ will refer to any such function. 
\end{Def}
\begin{Rq}\label{Rqcharactxv}
Let $(t_0,z,v) \in [T,+\infty[ \times \R^3_z \times \R^3_v$, $\Z_t := \Z(t,t_0,z,v)$ and $\V_t :=\V(t,t_0,z,v)$. The characteristic, in the coordinate system $(t,x,v)$, of the Vlasov operator $\T_F$ which is equal to $(z+t_0\widehat{v}+\C_{t_0,v},v)$ at time $t_0$ is 
$$ t \mapsto (\X_t,\V_t), \qquad \qquad \X_t := \Z_t +t \widehat{\V}_t+ \mathscr{C}_{t,\V_t} = \XX_\C \big( t, \Z_t , \V_t \big). $$
In particular, $\dot{\X}_t=\widehat{\V}_t$.
\end{Rq}

Our goal then consists in proving that $(t_0,z,v) \mapsto (\Z,\V )(t',t_0,z,v)$ converges locally in $L^\infty_{t_0,z,v}$ as $t' \to +\infty$ and that the limit belongs to $C^1([T,+\infty[ \times \R^3_z \times \R^3_v)$. Our analysis will rely on the following result.
\begin{Lem}\label{Lemcompuint}
Let $(t_0,z,v) \in [T,+\infty[ \times \R^3_z \times \R^3_v$, $\Z_t := \Z(t,t_0,z,v)$ as well as $\V_t :=\V(t,t_0,z,v)$. Then,
$$ \mathcal{I}(t_0,z,v):= \int_{t=T}^{+\infty} \frac{\V^{\underline{L}}_t \big(\XX_\C(t,\Z_t,\V_t) \big)}{\langle t-|\XX_\C(t,\Z_t,\V_t)| \rangle^{\frac{4}{3}}} \leq \int_{u=-\infty}^{+\infty} \frac{\dr u}{2\langle u \rangle^{\frac{4}{3}}} \leq 5.$$
\end{Lem}
\begin{proof}
In order to bound the first integral by the second one, we perform the change of variables $u=t-|\XX_\C(t,\Z_t,\V_t)|$. In other words, we parameterise the characteristic by $u$ instead of $t$. Note that in the coordinate system $(u,x,v)$, the linear Vlasov operator reads $2 \widehat{v}^{\underline{L}} \partial_u+\widehat{v} \cdot \nabla_x$. In our case, in view of Remark \ref{Rqcharactxv}, we have 
 $$\partial_t (u)=1-\widehat{\V}_t\cdot \frac{\XX_\C(t,\Z_t,\V_t)}{|\XX_\C(t,\Z_t,\V_t)|}=2\widehat{\V}^{\underline{L}}_t\big(\XX_\C(t,\Z_t,\V_t) \big) .$$
 Then, the second integral can be easily estimated.
\end{proof}

We are now able to prove that the characteristics converge locally in $L^\infty_{z,v}$ as $t \to +\infty$.
\begin{Lem}\label{Lemcharacbase}
There exists an absolute constant $C>0$ such that the following statement holds. For all $(t_0,z,v) \in [T,+\infty[ \times \R^3_z \times \R^3_v$, there exists $\Z (\infty,t_0,z,v) \in \R^3_z$ and $\V(\infty ,t_0,z,v) \in \R^3_v$ verifying
$$ \lim_{t \to + \infty} \big( \Z, \V \big) (t,t_0,z,v) = \big( \Z, \V \big) (\infty , t_0 ,z ,v)   .$$
Moreover, we have
$$\forall \, T \leq t \leq \infty, \qquad \quad \big|\V (t,t_0,z,v) - v  \big| \leq  C \sqrt{\Lambda }    , \qquad \big| \Z (t,t_0,z,v) - z  \big| \leq e^{C \sqrt{\Lambda}}  \langle z \rangle $$
and $(\Z,\V)(\infty,t_0,z,v) \to (z,v)$ as $t_0 \to +\infty$, for all $(z,v) \in \R^3_z \times \R^3_v$.
\end{Lem}
\begin{proof}
Let $1 \leq i \leq 3$. Note first the inequality $|\langle y \rangle - \langle x \rangle | \leq |x-y|$, which holds for all $x, \, y \in \R^3$ and which will be used many times during this proof. Next, Proposition \ref{Proestloworder} provides
\begin{equation*}
\left| \mathbf{T}_F^\infty(v^i) \right|(t,z,v) = \left| \widehat{v}^\mu F_{\mu i}(t,\XX_\C)  \right|\lesssim \frac{\sqrt{\Lambda}}{ t^{\frac{7}{4}}}+ \frac{\sqrt{\Lambda}  v^{\underline{L}}(\XX_\C)}{ \langle t-|\XX_\C| \rangle^2}.
\end{equation*}
Fix $(t_0,z,v) \in [T,+\infty[ \times \R^3_z \times \R^3_v$ and let $\Z_t := \Z(t,t_0,z,v)$ as well as $\V_t :=\V(t,t_0,z,v)$. By Duhamel's principle we have, for all $t \geq  T$,
\begin{equation}\label{DuhamelVi}
 \V^i_t-\V^i_{t_0}=\V^i_t-v^i =  \int_{s =t_0}^t \T_F^\infty(v^i)(s,\Z_s,\V_s) \mathrm{d}s  .
\end{equation}
Moreover, for all $T \leq \tau_1 \leq \tau_2$,
$$
 \int_{s=\tau_1}^{\tau_2}\left| \T_F^\infty (v^i)(s,\Z_s,\V_s) \right| \mathrm{d}s \lesssim \sqrt{\Lambda}  \int_{s=\tau_1}^{+\infty} \frac{1}{s^{\frac{7}{4}}}+\frac{\widehat{\V}^{\underline{L}}(\XX_\C(s,\Z_s,\V_s)}{\langle s-|\XX_\C(s,\Z_s,\V_s)| \rangle^2} \mathrm{d}s \leq \sqrt{\Lambda}  \big( 2+ \, \mathcal{I}(t_0,z,v) \big) \lesssim \sqrt{\Lambda} ,
 $$
 where $\mathcal{I}(t_0,z,v)$ is defined and bounded by $5$ in Lemma \ref{Lemcompuint}. We then deduce from the last estimate and \eqref{DuhamelVi} that $\V^i_t$ converges to $\V^i_\infty$, as $t \to +\infty$, and then that $|\V^i_t-\V^i_{t_0}| \lesssim \sqrt{\Lambda}$ for all $t_0 \geq T$ and $T \leq t \leq \infty$. 

Next, we obtain from Proposition \ref{Proestloworder},
$$
 \left|\T_F^\infty( z^i )\right| = \bigg|-\frac{\delta_j^i-\widehat{v}_j \widehat{v}^i}{tv^0} \widehat{v}^\mu \left( t^2  {F_{\mu}}^j(t,\XX_\C)-{\mathbb{F}_{\mu}}^{j}[f_\infty](v) \right)-\frac{\widehat{v}^\mu}{v^0} {F_{\mu}}^j(t,\XX_\C)v^0 \partial_{v^j} \C^i_{t,v} \bigg| \lesssim \frac{\sqrt{\Lambda} }{  t^{\frac{4}{3}}}\langle z \rangle+\frac{\sqrt{\Lambda} \, \widehat{v}^{\underline{L}}(\XX_\C)  }{  \langle t-|\XX_\C| \rangle^{\frac{4}{3}}}\langle z \rangle.
$$
Fix again $(t_0,z,v) \in [T,+\infty[ \times \R^3_z \times \R^3_v$. By Dumamel's principle, we have, for all $t \ge T$,
\begin{equation}\label{DuhamelZi}
 \Z^i_t-z^i=  \int_{s =t_0}^t \T_F^\infty(z^i)(s,\Z_s,\V_s) \mathrm{d}s  , \qquad \qquad \langle \Z_t \rangle \leq \langle z \rangle + \int_{s = \min (t_0,t)}^{\max (t_0,t)} \big|\T_F^\infty \big( \langle z \rangle \big) \big|(s,\Z_s,\V_s) \mathrm{d}s .
\end{equation}
As $|\T_F (\langle z \rangle)| \leq \langle z \rangle^{-1} |z_i| |\T_F^\infty (z^i)|$, we get, for all $t \geq T$,
$$  \langle \Z_t \rangle - \langle z \rangle \lesssim \int_{s = \min (t_0,t)}^{\max (t_0,t)} \frac{\sqrt{\Lambda}}{s^{\frac{4}{3}}} \langle \Z_s \rangle +\frac{\sqrt{\Lambda} \, \widehat{v}^{\underline{L}}\big(\XX_\C(s,\Z_s,\V_s) \big) }{ \langle t-|\XX_\C(s,\Z_s,\V_s)|\rangle^{\frac{4}{3}} }  \, \langle \Z_s \rangle \, \mathrm{d}s  .$$
Applying first Gronwall's inequality and then Lemma \ref{Lemcompuint}, we obtain that there exists a constant $C_0>0$ such that
$$ \forall \, t \geq T, \qquad \langle \Z_t \rangle \leq \langle z \rangle \exp \big( 3C_0\sqrt{\Lambda} + C_0 \sqrt{\Lambda} \, \mathcal{I}(t_0,z,v) \big) \leq \langle z \rangle \exp \big( 8 C_0\sqrt{\Lambda} \, \big)  .$$
It implies, applying again Lemma \ref{Lemcompuint}, that for all $T \leq \tau_1 \leq \tau_2$,
$$
 \int_{s = \tau_1}^{\tau_2} \left|\T_F( z^i )\right| (s,\Z_s,\V_s) \dr s \lesssim \sqrt{\Lambda} \, e^{8C_0 \sqrt{\Lambda}} \langle z \rangle \left( \int_{s=\tau_1}^{+\infty} \frac{\dr s}{s^{\frac{4}{3}}}+\mathcal{I}(t_0,z,v) \right) \leq  8\sqrt{\Lambda} \, e^{8C_0 \sqrt{\Lambda}} \langle z \rangle  .$$
Thus, from the last estimate and \eqref{DuhamelZi}, we get that $\Z^i_t$ converges, as $t \to +\infty$, and that the estimate $|\Z^i_t-z^i| \lesssim \sqrt{\Lambda} \, e^{8C_0 \sqrt{\Lambda}} \langle z \rangle $ holds for all $t_0 \leq t \leq \infty$. Then, we bound the constant using $a e^c \leq e^{a+c}$. Finally, for the convergence of $(\Z,\V)(\infty,t_0,z,v)$ as $t_0 \to +\infty$, we use the relations \eqref{DuhamelVi}--\eqref{DuhamelZi} for $t=+\infty$ and that the right hand sides go to zero as $t_0 \to +\infty$.
\end{proof}

We now prove that the characteristics as well their first order derivatives converge as $t_0 \to + \infty$. Inverting the limit will provide the $C^1$ regularity of $(\Z,\V)(\infty,t_0,\cdot , \cdot)$. For this, we will make use of the next result.
\begin{Lem}\label{conv}
Let $d \geq 1$ and, for all $n \in \mathbb{N}$, $\psi_n \in C^1(\R^d,\R^d)$ be an invertible function such that
\begin{itemize}
\item $(\psi_n)$ converges pointwise to $\psi : \R^d \to \R^d$ and, for any compact $K \subset \R^d$, $\sup_n \sup_K |\nabla \psi_n | <+\infty$.
\item $(\psi_n^{-1})$ converges pointwise to $\overline{\psi} : \R^d \to \R^d$ as $n \to + \infty$.
\end{itemize}
Then, $\psi$ is invertible and $\overline{\psi}= \psi^{-1}$.
\end{Lem}
\begin{proof}
For all $x \in \R^d$ and $n \in \mathbb{N}$, we have
$$\big| \psi(\overline{\psi}(x))-x \big|= \big|\psi \big(\overline{\psi}(x) \big)- \psi_n \big( \psi^{-1}_n(x) \big) \big|= \big|\psi \big( \overline{\psi}(x) \big)- \psi_n \big( \overline{\psi}(x) \big)\big|+ \big|\psi_n \big( \overline{\psi}(x) \big)- \psi_n \big( \psi^{-1}_n(x) \big)\big| .$$
By assumption, $\psi \big( \overline{\psi}(x) \big)- \psi_n \big( \overline{\psi}(x) \big) \to 0$ as $n \to +\infty$. Since $(\psi_n^{-1}(x))_n$ is bounded and converges to $\overline{\psi} (x)$, the mean value theorem implies that the second term on the right hand side goes to $0$ as $n \to +\infty$.
\end{proof}
For any $t_ 0 \geq T$, since $(z,v) \mapsto (\Z,\V)(t_0,t_0, z, v)=(z,v)$, we have, for any $1 \leq i , \, j \leq 3 $,
\begin{align}\label{eq:inidata}
\hspace{-2mm} &  \partial_{z^j} \Z^{i}(t_0,t_0 , z ,v )= \partial_{v^j} \V^{i}(t_0,t_0 , z ,v ) = \delta^i_j, \qquad \qquad \partial_{v^j} \Z^{i}(t_0,t_0 , z ,v )=\partial_{z^j} \V^{i}(t_0,t_0 , z ,v )=0.
\end{align}

\begin{Pro}\label{Lemderivcharac}
Let $t \in [T,+ \infty[$. Then,
\begin{itemize}
\item $(z,v) \mapsto (\Z,\V)(t,t_0,z,v)$ converges to a $C^1(\R^3_z \times \R^3_v)$ function $(\Z,\V)(t,\infty,\cdot,\cdot)$, as $t_0 \to +\infty$, which is the inverse of $(\Z,\V)(\infty, t,\cdot,\cdot)$.
\item There exists a constant $D_\Lambda>0$ such that, for all $T \leq t_0 \leq + \infty$, $(z,v) \in \R^3_z \times \R^3_v$ and $1 \leq i, \, j \leq 3$,
\begin{align*}
  \big| \partial_{z^j} \Z^i (t,t_0,z,v)-\delta_j^i \big|+\big| \partial_{z^j} \V^i (t,t_0,z,v) \big| & \leq e^{D_\Lambda \langle v \rangle \, \langle z \rangle^3} , \\
   \langle v \rangle \, \big| \partial_{v^j} \Z^i (t,t_0,z,v) \big|+ \langle v \rangle \, \big| \partial_{v^j} \V^i (t,t_0,z,v) -\delta_j^i\big| & \leq e^{D_\Lambda \langle v \rangle \, \langle z \rangle^3} .
      \end{align*}
\end{itemize}
\end{Pro}
\begin{proof}
Let $\tau \geq T$, $1 \leq i \leq 3$ and $h= \Z^i(\tau, \cdot,\cdot , \cdot)$ or $h= \V^i(\tau, \cdot,\cdot , \cdot)$. By definition of the characteristics, $h$ is solution to the Vlasov equation $\T_F^\infty (h)=0$. The strategy is the following.
\begin{enumerate}
\item First, we prove that the $(z,v)$-derivatives of $h$ converge in $C^1_{\mathrm{loc}}(\R^3_z \times \R^3_v)$, as $t_0 \to + \infty$.
\item Then, according to Lemma \ref{Lemcharacbase}, $t_0 \mapsto (\Z,\V)(\tau,t_0,0,0)$ is bounded. Hence, there exists a sequence of time $(s_n)$, with $s_n \to + \infty$, such that $(\Z,\V)(\tau,s_n,0,0)$ converges as $n \to +\infty$. Together with $1.$, we get that $(\Z,\V)(\tau,s_n,z,v)$ converges for all $(z,v) \in \R^3_z \times \R^3_v$. Applying Lemma \ref{conv}, with $d=6$, to
$$ \psi_n = (\Z,\V)(\tau,s_n,\cdot,\cdot), \qquad \psi_n^{-1} =(\Z,\V)(s_n,\tau,\cdot, \cdot) ,$$
we obtain that $(z,v) \mapsto(\Z,\V)(\infty,\tau,z,v)$ is invertible and that its inverse, denoted $(\Z,\V)(\tau,\infty,z,v)$, is the limit of $(\Z,\V)(\tau,s_n,z,v)$. Finally, since $t_0 \mapsto (\Z,\V)(\tau,t_0,z,v)$ is bounded and has a unique subsequential limit, it converges to $(\Z,\V)(\tau,\infty,z,v)$ as $t_0 \to +\infty$.
\end{enumerate}
In order to prove the first step as well as the estimates for the derivatives of $h$, remark that in view of \eqref{eq:inidata} and
$$ v^0 \partial_{v^k} = \widehat{\Omega}_{0k}^\infty+z^k \partial_t^\infty +\widehat{v}^k S-\widehat{v}^k z \nabla_z ,\qquad \partial_{z^k} \in \widehat{\mathbb{P}}^\infty_S,$$
it suffices to prove that $\widehat{Z}_\infty h(t_0,\cdot , \cdot)$ converges in $C^1_{\mathrm{loc}}$, as $t_0 \to +\infty$, for all $\widehat{Z}_\infty \in \widehat{\mathbb{P}}^\infty_S$. According to the commutation formula of Corollary \ref{CorCom} and Proposition \ref{Proestloworder}, there holds
$$ \sup_{\widehat{Z}_\infty \in \widehat{\mathbb{P}}^\infty_S} \big| \T_F^\infty \big( \widehat{Z}_\infty h \big) \big|(t,z,v)  \lesssim \bigg( \, \frac{\sqrt{\Lambda} }{  t^{\frac{4}{3}}}+\frac{\sqrt{\Lambda} \, \widehat{v}^{\underline{L}}(\XX_\C) }{  \langle t-|\XX_\C| \rangle^{\frac{4}{3}}} \bigg)\langle z \rangle \, \sup_{\widehat{Z}_\infty \in \widehat{\mathbb{P}}^\infty_S} \big|  \widehat{Z}_\infty h  \big|(t,z,v)  \lesssim  \frac{\sqrt{\Lambda} }{  t^{\frac{5}{4}}} \langle v \rangle \, \langle z \rangle^3 \sup_{\widehat{Z}_\infty \in \widehat{\mathbb{P}}^\infty_S} \big|  \widehat{Z}_\infty h  \big|(t,z,v) , $$
where, in the last step, we applied Lemma \ref{gainvadapted}. Fix $(z,v) \in \R^3_z \times \R^3_v$ and let $\Z_{s} := \Z(s,t_0,z,v)$, $\V_{s} :=\V(s,t_0,z,v)$ as well as 
$$ \Phi (s) := \, \sup_{\widehat{Z}_\infty \in \widehat{\mathbb{P}}^\infty_S} \big|  \widehat{Z}_\infty h  \big|(s,\Z_s,\V_s).$$
By Duhamel's principle, and since $\T^\infty_F(|g|)=\frac{g}{|g|} \T_F^\infty(g)$ in $W^{1,1}_{\mathrm{loc}}$, we have
$$ \forall t \geq T, \qquad  |\Phi(t_0) - \Phi(\tau)|  \lesssim  \int_{ s \in [(\tau,t_0)]}  \frac{ \Lambda }{  s^{\frac{5}{4}}} \langle \V_s \rangle \, \langle \Z_s \rangle^3 \Phi(s)  \dr s , \qquad [(\tau,t_0)]:=[\min(\tau,t_0),\max(\tau,t_0)].$$
By Lemma \ref{Lemcharacbase}, $\langle \V_s \rangle \, \, \langle \Z_s \rangle^3 \lesssim C_\Lambda \, \langle v \rangle \, \langle z \rangle^3$, for a certain constant $C_\Lambda >0$.  We then deduce from Gronwall's inequality that there exists $C^1_\Lambda, \, C^2_\Lambda >0$ such that
$$ \forall \, t \geq T, \qquad \sup_{ s \in [(\tau,t_0)]} \Phi (s) \leq \Phi(\tau) \exp \left(  C^1_\Lambda \langle v \rangle \, \langle z \rangle^3 \right) \lesssim \big( \langle v \rangle+ \langle z \rangle \big)  \exp \left(  C^1_\Lambda \langle v \rangle \, \langle z \rangle^3 \right) \lesssim \exp \left(  C^2_\Lambda \langle v \rangle \, \langle z \rangle^3 \right) ,$$
where we estimated $\Phi(\tau)$ by exploiting \eqref{eq:inidata}. Using Duhamel principle once again, we obtain that 
$$ \forall \, t_0 \geq T , \qquad \big| \widehat{Z}_\infty h(t_0,z,v)-\widehat{Z}_\infty h (\tau,z,v) \big| \lesssim \int_{ s \in [(\tau,t_0)]}  \frac{ \sqrt{\Lambda} }{  s^{\frac{5}{4}}} \, \langle \V_s \rangle \, \langle \Z_s \rangle^3 \Phi(s)  \dr s \lesssim \exp \left( \, \overline{D}_\Lambda \langle v \rangle \, \langle z \rangle^3 \right),$$
which concludes the proof.
\end{proof}

Although certain of the estimates that have just derived could be significantly improved, they are sufficient in order to prove regularity results for the limit of the characteristics.
\begin{Pro}\label{Procharacinv}
For all $t_0 \in [T,+\infty[$, $(z,v) \mapsto (\Z,\V)(\infty,t_0,z,v)$ is a $C^1$-diffeomorphism from $\R^3_z \times \R^3_v$ onto itself. Moreover, $(\Z,\V)(\infty,\cdot, \cdot , \cdot) \in C^1([T,+\infty [ \times \R^3_z \times \R^3_v)$.
\end{Pro}
\begin{proof}
Let, for $(t,t_0) \in [T,+\infty[ \times [T,+\infty]$, $J_{t,t_0}$ be the Jacobian determinant of the map $\Phi_{t,t_0}:(z,v) \mapsto (\Z,\V)(t,t_0,z,v)$. Then, Proposition \ref{Lemderivcharac} implies that $\Phi_{t,t_0} \in C^{1}(\R^3_z \times \R^3_v)$ and
\begin{equation}\label{Jacobianesti}
\forall \, (t,t_0,z,v) \in [T,+\infty[^2 \times \R^3_z \times \R^3_v, \qquad \qquad   |J_{t,t_0}|(z,v) \leq 6! \, e^{6D_\Lambda \, \langle v \rangle \, \langle z \rangle^3}.
\end{equation}
Since $\Phi_{t_0,t} \circ \Phi_{t,t_0}=\mathrm{id}$ for all $(t,t_0) \in [T,+\infty[^2$, we get from the previous estimate that, for all $t \geq T$,
$$|J_{t,t_0}|^{-1}(z,v) =|J_{t_0,t}| \big( \Z (t,t_0,z,v),\V(t,t_0,z,v) \big) \leq 6! \, \exp \left(6D_\Lambda \, \langle \V(t,t_0,z,v) \rangle \, \langle \Z(t,t_0,z,v) \rangle^3 \right) .$$
We then deduce from \eqref{Jacobianesti} and Lemma \ref{Lemcharacbase} that
$$ \exists \, \overline{D}_\Lambda >0, \; \forall \, t \geq T, \; \forall \, (z,v) \in \R^3_z \times \R^3_v, \qquad e^{-\overline{D}_\Lambda \, \langle v \rangle \, \langle z \rangle^3 }  \lesssim  |J_{t,t_0}|(z,v) \lesssim  e^{\overline{D}_\Lambda \, \langle v \rangle \, \langle z \rangle^3}.$$
Since these inequalities necessarily hold for $t_0=+\infty$, we obtain from the inverse function theorem that $\Phi_{t,\infty}:(z,v) \mapsto (\Z,\V)(t,\infty,z,v)$ is a local $C^1$-diffeomorphism. As we already know from Proposition \ref{Lemderivcharac} that it is invertible, $\Phi_{t,\infty}$ is in fact a global diffeomorphism. Finally, since $(t_0,z,v) \mapsto (\Z,\V)(t,t_0,z,v)$ is solution to the Vlasov equation $\T_F^\infty(h)=0$ for all $t \geq T$, the same holds for the limit $(\Z,\V)(\infty,\cdot,\cdot,\cdot)$ in the space of distributions. It implies that $\partial_t (\Z,\V)(\infty,\cdot,\cdot,\cdot)$ is continous, which concludes the proof.
\end{proof}

\subsubsection{Proof of Proposition \ref{ProCauchpb}}

Let $h_\infty : \R^3_z \times \R^3_v \to \R$ such that $\mathbb{E}_N^{N_z,N_v}[h_\infty]< +\infty$. Assume that $h \in C^1 \big( [T,+\infty[ \times \R^3_z \times \R^3_v \big)$ is a solution to the asymptotic Cauchy problem \eqref{Cauchy:asymp} and fix $(z,v) \in \R^3_z \times \R^3_v$. By Duhamel principle, we have, for all $(t, t_0) \in [T,+ \infty[^{\,2}$,
$$ h(t,z,v) \!=\! h\big(t_0 , \Z (t_0, t,z,v) , \V (t_0 ,t,z,v) \big)-h \big( t_0,\Z (\infty,t,z,v), \V (\infty ,t,z,v) \big)+h \big( t_0,\Z (\infty ,t,z,v), \V ( \infty, t,z,v) \big).$$
By assumption, the last term on the right hand side converges to $h_\infty \big(  \Z (\infty,t,z,v), \V (\infty,t,z,v) \big)$ as $t_0 \to + \infty$. By the mean value theorem, we have
\begin{align*}
& |h(t_0 , \Z (t_0,t,z,v) , \V (t_0,t,z,v) )-h(t_0,\Z (\infty ,t,z,v), \V (\infty,t,z,v) )|\\
 & \qquad \qquad \qquad \lesssim   \|\nabla_{z,v} h\|_{L^\infty_{t,z,v}}  \big(|\Z (t_0,t,z,v)-\Z (\infty,t,z,v)|+|\V (t_0,t,z,v)-\V (\infty,t,z,v)| \big)
\end{align*}
and, according to Lemma \ref{Lemcharacbase}, the right hand side goes to $0$ as $t_0 \to +\infty$. Hence, we have
\begin{equation}\label{def:soluvlasovfrominf}
 \forall \, (t,z,v) \in [T,+ \infty[ \times \R^3_z \times \R^3_v, \qquad h(t,z,v) := h_\infty \big( \Z(\infty,t,z,v), \V(\infty,t,z,v) \big),
 \end{equation}
so that uniqueness holds in the functional space $C^1 \big( [T,+\infty[ \times \R^3_z \times \R^3_v \big)$. Conversally, let $h$ be defined by \eqref{def:soluvlasovfrominf}. By Lemma \ref{Lemcharacbase} and Proposition \ref{Procharacinv}, it is a $C^1$ solution to the asymptotic Cauchy problem \eqref{Cauchy:asymp}. Let $\psi : \R^3_z \times \R^3_v \to \R_+$ be a mollifier and 
$$\psi_n(z,v):= n^6\psi(nz,nv), \qquad h_n := h \ast \psi_n , \qquad \qquad n \in \mathbb{N}^*,$$
so that $h_n$ is a $C^\infty$ solution to \eqref{Cauchy:asymp} with asymptotic data $h_\infty \ast \psi_n$. 
\begin{enumerate}
\item Assume first that $h_\infty$ is compactly supported. By Lemma \ref{Lemcharacbase}, $h_n(t,\cdot,\cdot)$ is then compactly supported as well for all $t \geq T$. We can then apply Proposition \ref{Proapriori} to $h_n$ and derive the estimates stated in Proposition \ref{ProCauchpb}, for the energy norms of $h$, by letting $n \to +\infty$.
\item Let $\chi \in C^\infty (\R,\R_+)$ be a cutoff function such that $\chi(s)=1$ for all $s \leq 1/4$ and $\chi (s)=0$ for all $s \geq 1/2$. Consider then the solution $h_n$ to the Vlasov equation arising from the asymptotic data $(z,v) \mapsto h_\infty (z,v) \chi (|z,v|n^{-1})$. We can then apply the previous step to $h_n$ and derive the expected estimates for $h$ by passing to the limit as $n \to + \infty$.
\end{enumerate}

\subsection{Construction of the solution and the $L^2_x$ estimates}

According to Proposition \ref{ProCauchpb}, we can define $g$ as the unique global solution to \eqref{Cauchy:asymp} with asymptotic data $f_\infty$. Furthermore, Proposition \ref{ProCauchpb} applied for the parameters $(N_z,N_v)=(10,19)$ implies that there exists $B>0$, depending only on $(N,D)$, such that
\begin{equation}\label{eq:gBound}
\mathbf{E}_{N-3}^{10,19}[g]+\mathbf{E}_N^{7,13}[g] \leq   \mathbb{E}_N^{10,19}[f_\infty] \exp \big(  \exp \big( B \sqrt{\Lambda} \, \big) \big) \leq \epsilon B_\Lambda ,
 \end{equation}
where $B_\Lambda := \exp \big(  \exp \big( B \sqrt{\Lambda} \, \big) \big)$. In particular, we get $g \in \mathbb{V}^{B_\Lambda,\epsilon}_N$.
\begin{Rq}\label{RqenergyforTh1}
Similarly, in view of the assumption \eqref{eq:assumpfinfty} verified by $f_\infty$, there exists $B' >0$ depending only on $N$ such that $$\mathbf{E}_N^{8,15}[g] \leq \mathbb{E}_N^{11,21}[f_\infty] \exp \big(  \exp \big( B' \sqrt{\Lambda} \, \big) \big) \lesssim \epsilon \exp \big(  \exp \big( B' \sqrt{\Lambda} \, \big) \big) .$$
This stronger bound is important in order to control sufficiently well the initial data in the perspective of constructing the scattering map in Theorem \ref{Theo3}. However, we do not need it in order to prove the existence of the solution $(f,F)$ to the Vlasov-Maxwell system stated by Theorem \ref{Theo1}.
\end{Rq} 

We conclude this section by deriving decay estimates for $\int_v \widehat{Z}^\beta f \dr v$, where $f(t,x,v)=g(t,x-t\widehat{v}-\C_{t,v},v)$. For this, we will apply Proposition \ref{Protechfordecay} to $\widehat{Z}^\beta f$. It motivates us to relate $\widehat{Z}^\beta f$ with $\widehat{Z}^\kappa_\infty g$ and to prove the convergence of the spatial averages of $\widehat{Z}^\kappa_\infty g$, which is the purpose of the next result.

\begin{Pro}\label{Proconvspat}
For any $|\beta| \leq N-1$, we have
$$ \forall \, t \geq T, \qquad \bigg\| |v^0|^8 \int_{\R^3_z} \widehat{Z}^\beta_\infty g(t,z,\cdot) - \widehat{Z}^\beta_\infty f_\infty (z,\cdot) \dr z \bigg\|_{L^2(\R^3_v)} \lesssim \sqrt{\epsilon \Lambda B_\Lambda} \, \frac{ \log^{3N}(t) }{t}.$$
\end{Pro}
\begin{proof}
We will use that for any sufficiently regular function $\psi :[T,+\infty[ \times \R^3_v$, we have, according to Minkowski integral inequality,
$$ \forall \, T \leq t_1 \leq t_2, \qquad  \big\| \psi (t_2,\cdot)-\psi (t_1 , \cdot) \big\|_{L^2(\R^3_v)} = \bigg\| \int_{t=t_1}^{t_2} \partial_t \psi (t,\cdot) \dr t \bigg\|_{L^2(\R^3_v)} \! \leq \int_{t=t_1}^{t_2} \big\| \partial_t \psi ( t , \cdot) \big\|_{L^2(\R^3_v)} \dr t   .$$
Applying first Corollary \ref{Corestipartialtg} and then the Cauchy-Schwarz inequality in $z$, we get, as $z \mapsto \langle z \rangle^{-2} \in L^2(\R^3_z)$,
\begin{align*}
\int_{\R^3_v} \langle v \rangle^{16} \bigg| \partial_t \int_{\R^3_z} \widehat{Z}^\beta_\infty h(t,z,v) \dr z \bigg|^2  \dr v  \lesssim &\sum_{|\kappa| \leq N}  \, \frac{\Lambda \log^{5N}(t)}{ t^4} \int_{\R^3_z}\int_{\R^3_v}  \langle v \rangle^{22}  \langle z \rangle^{8+2N-2\kappa_H} \, \big| \widehat{Z}_\infty^\kappa g \big|^2(t,z,v) \dr z \dr v \\
 & +  \log^{4N+4}(t) \, \mathbf{E}_N^{7,13}[g] \int_{\R^3_z}\int_{\R^3_v}  \langle t-|\XX_\C| \rangle^2  \left|\nabla_{t,x} \mathcal{L}_{Z^\kappa}F \right|^2(t,\XX_\C) \frac{\dr z \dr v}{  \langle z \rangle^4 \, \langle v \rangle^{8}}   .
\end{align*}
The $L^2$ bound \eqref{eq:gBound} for $g$ and Proposition \ref{Proesttoporder} imply the result.
\end{proof}
We now relate the derivatives of $f$ to the ones of $g$.
\begin{Lem}\label{Lemftog}
Let $|\beta| \leq N$. Then, we have the $L^2_{x,v}$ bound
$$ \forall \, t \geq T, \qquad \int_{\R^3_x} \int_{\R^3_v}  \langle x-t\widehat{v} \rangle^{14} \, \langle v \rangle^{26} \big| \widehat{Z}^\beta f \big|^2(t,x,v)  \dr x \dr v \lesssim \epsilon B_{ \Lambda}\log^{6N}(t)  .$$
For any $|\kappa| \leq N-2$, we have the $L^\infty_{x,v}$ bound
$$ \forall \, t \geq T, \qquad \sup_{(x,v) \in \R^3 \times \R^3} \langle x-t\widehat{v} \rangle^7 \, \langle v \rangle^{13} \big| \widehat{Z}^\beta f \big|(t,x,v)  \lesssim \sqrt{\epsilon B_{ \Lambda}}\log^{3N}(t)  .$$
If $|\beta| \leq N-1$, we have
$$ \forall \, t \geq T, \qquad \int_{\R^3_v} \langle v \rangle^{16}  \bigg|  \int_{\R^3_x} \widehat{Z}^\beta f (t,x,v) \dr x- \int_{\R^3_z} \widehat{Z}^\beta_\infty f_\infty (z,v) \dr z \bigg|^2 \dr v \lesssim \epsilon \Lambda B_\Lambda\frac{ \log^{6N}(t) }{t^2}.$$
\end{Lem}
\begin{proof}
By iterating Lemma \ref{Lemrelftog} and by exploiting the commutation relations \eqref{eq:commutePinf} between $\partial_t^\infty$ and $\partial_{z^i}$ with $Z_\infty \in \widehat{\mathbb{P}}_S^\infty$, we obtain the following result. For any $|\beta| \leq N$, we can write
\begin{equation*}
\big( \widehat{Z}^\beta f \big) (t,z+t\widehat{v}+\C_{t,v},v)- \widehat{Z}^\beta_\infty g(t,z,v)
\end{equation*}
as a combination, with coefficients which are polynomials in $\widehat{v}$, of terms of the form
$$  t^{-m}P_{p,q}(\C) \big( \partial_t^\infty \big)^n\partial_z^{\alpha_z} \widehat{Z}^\kappa_\infty g(t,z,v)    , \qquad n+|\alpha_z|+|\kappa| \leq |\beta|, \quad n+|\alpha_z| \geq 1, \quad 0 \leq  p \leq m+|\beta| \leq 2 |\beta|, \quad  q \leq |\beta|.$$
By \eqref{estiPolyC}, we have $|t^{-m}P_{p,q}(\C)| \lesssim t^{-m} \log^{2p}(t) \lesssim 1$ if $m \geq 1$ and $|t^{-m}P_{p,q}(\C)| \lesssim \log^{2N}(t)$ otherwise. Hence,
$$ \langle z+\C_{t,v} \rangle^7 \, \langle v \rangle^{13} \big| \big( \widehat{Z}^\beta f \big) (t,z+t\widehat{v}+\C_{t,v},v) \big| \lesssim \log^{2N+7}(t)  \sup_{|\xi| \leq |\beta|} \langle z \rangle^7 \, \langle v \rangle^{13} \big|\widehat{Z}^\xi_\infty g \big|(t,z,v).$$
The $L^2_{x,v}$ estimates then follows from the boundedness in $L^2_{z,v}$ of the derivatives of $g$ given by \eqref{eq:gBound}. The $L^\infty_{x,v}$ ones follow from the Sobolev embedding $H^2_{z,v} \hookrightarrow L^\infty_{z,v}$. For the spatial average, by integration by parts in $z$, we have
\begin{equation*}
\bigg|\int_{\R^3_x} \widehat{Z}^\beta f  (t,x,v) \dr x- \int_{\R^3_z} \widehat{Z}^\beta_\infty g(t,z,v) \dr z \bigg| \lesssim \sum_{1 \leq n \leq |\beta|} \sum_{|\kappa| \leq |\beta|- n } \log^{2N}(t) \bigg|\int_{\R^3_z} \big( \partial_t^\infty \big)^n \widehat{Z}^\kappa_\infty g(t,z,v) \dr z \bigg|.
\end{equation*}
Then, write $t\partial_t^\infty = S-z\cdot \nabla_z+\widehat{v}\cdot \nabla_z$ and perform again integration by parts, so that the result follows from Proposition \ref{Proconvspat} as well as \eqref{eq:gBound}.
\end{proof}

We are finally able to deduce the following $L^2_x$ estimates for the velocity averages of $f$.
\begin{Cor}\label{Corrho}
For any $|\beta| \leq N-1$ and $\mu \in \llbracket 0 , 3 \rrbracket$, we have 
$$
\forall \, t \geq T, \qquad \int_{\R^3_x} \langle t+|x| \rangle^{5} \bigg|\int_{\R^3_v}\widehat{v}_\mu \widehat{Z}^\beta f (t,x,v) \mathrm{d} v -   J^{\mathrm{asymp}}_{\mu}\big[ Z_\infty^\beta f_\infty \big](t,x) \bigg|^2  \dr x   \lesssim \epsilon \Lambda B_{ \Lambda}\log^{6N}(t) .
$$
\end{Cor}
\begin{Rq}
Recall that if $\beta_T \geq 1$, so that $\widehat{Z}^\beta$ contains at least one translation, then $J^{\mathrm{asymp}}[Z_\infty^\beta f_\infty]=0$.
\end{Rq}

In fact, this result is a direct consequence of the following one, applied for $\mathbf{q}(\omega,v)=\widehat{v}_\mu$, requiring more decay in $v$ but which will allow us to deal with the top order derivatives, where a parameter $\omega \in \mathbb{S}^2$ will naturally appear.
\begin{Cor}\label{Corrho2}
Let $\mathbf{q} : \mathbb{S}^2_\omega \times \R^3_v \to \R$ be a smooth function such that $|\mathbf{q}|(\cdot , v) + |\nabla_v \mathbf{q}|(\cdot , v)  \lesssim \langle v \rangle^3$ and $|\beta| \leq N-1$. We define
$$ \rho^\beta\big[ f,  f_\infty ; \mathbf{q} \big](t,x,\omega) :=  \int_{\R^3_v} \mathbf{q}(\omega,v) \widehat{Z}^{\beta} f(t,x,v)  \dr v-  \mathds{1}_{|x| < t} \,  \frac{1}{t^3} \mathbf{q} \bigg( \omega, \frac{\widecheck{ \; x \;}}{t} \bigg) \! \int_{\R^3_z}  \Big[ |v^0|^5 \widehat{Z}^\beta_\infty f_\infty \Big] \bigg( z, \frac{\widecheck{ \; x \;}}{t} \bigg) \dr z.$$
There holds, uniformly in $\omega \in \mathbb{S}^2_\omega$,
$$
\forall \, t \geq T , \qquad \int_{\R^3_x} \langle t+|x| \rangle^{5} \Big| \rho^\beta\big[ f,  f_\infty ; \mathbf{q} \big](t,x,\omega) \Big|^2 \dr x   \lesssim \epsilon \Lambda B_{ \Lambda}\log^{6N}(t).
$$
\end{Cor}
\begin{proof}
Apply Proposition \ref{Protechfordecay}, to the distribution function $\mathbf{q}(\omega ,v) \widehat{Z}^\beta f$ and for all $\omega \in \mathbb{S}^2_\omega$, together with Lemma \ref{Lemftog}. Then, remark that, by the change of variables $x=t\widehat{v}$ (recall Lemma \ref{cdv}) and Lemma \ref{Lemftog},
$$
\int_{\R^3_x} \langle t+|x| \rangle^{5} \bigg|  \frac{\mathds{1}_{|x|<t}}{t^3}\int_{\R^3_y} \Big[ |v^0|^5 \widehat{Z}^\beta f \Big] \bigg( t,y,\frac{\widecheck{\; x \;}}{t} \bigg)-\Big[ |v^0|^5 \widehat{Z}^\beta_\infty f_\infty \Big] \bigg( y,\frac{\widecheck{\; x \;}}{t} \bigg) \dr y \bigg|^2  \dr x \lesssim \epsilon \Lambda B_\Lambda \log^{6N}(t)
.$$
\end{proof}

In order to deal with the top order derivatives or with quadratic terms, we will make use of the next estimates.
\begin{Cor}\label{Corrho3}
For any $|\beta| \leq N$, we have
$$ \forall \, t \geq T, \qquad  \int_{\R^3_x} \langle t+|x| \rangle^3 \bigg|\int_{\R^3_v}\langle x-t\widehat{v} \rangle^3 \, \langle v \rangle^{9} \big| \widehat{Z}^\beta f (t,x,v)\big| \mathrm{d} v  \bigg|^2  \dr x   \lesssim \epsilon  B_{ \Lambda}\log^{6N}(t).$$
For any $|\kappa| \leq N-2$, there holds
$$ \forall \, (t,x) \in [T,+\infty[ \times \R^3, \qquad \int_{\R^3_v} \langle x- t \widehat{v} \rangle \, \langle v \rangle^5 \big| \widehat{Z}^\kappa f \big|(t,x,v) \dr v \lesssim  \sqrt{\epsilon  B_{ \Lambda}} \, \frac{|\max (t-|x|,1)|^{\frac{3}{2}}}{\langle t+|x| \rangle^{\frac{9}{2}}} \log^{6N}(t) .$$
\end{Cor}
\begin{proof}
For the $L^2_x$ estimate, apply Proposition \ref{Prosimpledecay} to $\langle x-t\widehat{v} \rangle^3 \, \langle v \rangle^{9} \, \widehat{Z}^\beta f (t,x,v)$ and use then Lemma \ref{Lemftog}. For the $L^\infty_x$ one, we apply Corollary \ref{Corpointestistrong} to $\langle x- t \widehat{v} \rangle \, \langle v \rangle^5\widehat{Z}^\kappa f$ and we use again Lemma \ref{Lemftog}.
\end{proof}

Finally, let us mention that we can derive a convergence estimate for $g$.
\begin{Pro}\label{Proconvgntoginfty}
For all $t \geq T$, we have
$$\| g(t,\cdot , \cdot)-f_\infty \|_{L^\infty(\R^3_z \times \R^3_v)} \lesssim \sqrt{\epsilon  \Lambda B_\Lambda} \, t^{-\frac{1}{2}} \log^3(t).$$
\end{Pro}
\begin{proof}
Following the proof of Proposition \ref{Proestloworder} and using Lemma \ref{gainvadapted}, one can derive
$$ |\partial_t g |(t,z,v) =|\partial_t g-\T^\infty_F(g)|(t,z,v) \lesssim \sqrt{\Lambda} t^{-\frac{3}{2}} \log^3(t) \sup_{|\beta|=1 } \, \langle z \rangle^{\frac{3}{2}} \, \langle v \rangle^2 \big| \widehat{Z}_\infty^\beta g\big|(t,z,v),$$
which, in view of the weighted $L^\infty_{z,v}$ bounds on the lower order derivatives of $g$, implies the result.
\end{proof}

\section{Construction of the electromagnetic field}\label{SecMax}

In this section, we study the asymptotic Cauchy problem
\begin{equation}\label{kevatalenn:tredebis}
\nabla^\mu G_{\mu \nu} = J(f)_\nu, \qquad \nabla^\mu {}^* \! G_{\mu \nu} =0, \qquad \qquad \lim_{r \to + \infty} r \underline{\alpha} (G)(r+u,r\omega)=\underline{\alpha}^{\infty} (u, \omega ),
\end{equation}
where $f$ is the distribution function constructed in Section \ref{SecVlasov}. 
\begin{Rq}
Given a source term $J$ and an $L^2_{u,\omega}$ radiation field $\underline{\alpha}^{\mathcal{I}^+}$, the asymptotic Cauchy problem for the Maxwell equations admits at most one solution in $L^{\infty}([T,+\infty[,L^2(\R^3_x))$ according to Theorem \ref{Thscat}.
\end{Rq}

Our goal is to prove the next result, where we recall from Definition \ref{DefsetMax} the set $\mathbb{M}^{D,\Lambda}_N$.
\begin{Pro}\label{ProforFnew}
There exists a unique classical solution $F^{\mathrm{new}}$ to \eqref{kevatalenn:tredebis} on $[T,+\infty[ \times \R^3$. Furthermore, if $\Lambda^2 B_\Lambda \epsilon \log^{-1}T$ is small enough, we have
$F^{\mathrm{new}} \in \mathbb{M}^{D,\Lambda}_N$.
\end{Pro}
\begin{Rq}\label{Rqconvalphantoalphainf}
We deduce from $F^{\mathrm{new}} \in \mathbb{M}^{D,\Lambda}_N$ as well as Propositions \ref{Thscatforsmooth} and \ref{Proradasymp} the convergence estimate
$$ \forall \, (t,r,\omega) \in [T,+\infty[ \times \R_+ \times \mathbb{S}^2, \qquad \qquad \big| \underline{\alpha}\big(F^{\mathrm{new}} \big)(t,r\omega)-\underline{\alpha}^\infty(t-r,\omega) \big| \lesssim \sqrt{\Lambda } \langle t+r \rangle^{-1} .$$
\end{Rq}
In fact, a part of the proof of Proposition \ref{ProforFnew} will also allow us to construct $F^1$, the first electromagnetic field of the sequence of approximate solutions. 
\begin{Pro}\label{ProforF1}
If $f=f_1$, there exists a unique classical solution $F^1$ to \eqref{kevatalenn:tredebis} on $[T,+\infty[ \times \R^3$. Moreover, if $\Lambda B_\Lambda \epsilon \log^{-1}T$ is small enough, we have
$\mathcal{E}^{K}_{N}\big[F^1-F^{\mathrm{asymp}}[f_\infty]\big](t) \leq D \Lambda$ for all $t \geq T$, so that $F^1 \in \mathbb{M}^{D,\Lambda}_N$.
\end{Pro}

 Our study of \eqref{kevatalenn:tredebis} is divided in three parts. 
\begin{enumerate}
\item We start by proving that there exists a unique classical solution $F_{ \mathrm{new}}$ to \eqref{kevatalenn:tredebis}, which further belongs to $W^{k,\infty} ([T,+\infty[,H^{N-k}(\R^3_x))$ for all $0 \leq k \leq N$. 
\item Then, we prove that the lower order derivatives verify $\mathcal{E}^{K}_{N-1}\big[F^{\mathrm{new}}-F^{\mathrm{asymp}}[f_\infty]\big] \leq \frac{D}{2} \Lambda$ on $[T,+\infty[$.
\item Finally, we deal with the top order derivatives by using the Glassey-Strauss decomposition of the electromagnetic fields, providing $F^{\mathrm{new}} \in \mathbb{M}^{D, \Lambda}_N$.
\end{enumerate}
\begin{Rq}
The first step could be avoided but it allows us to estimate all the derivatives of $F^{\mathrm{new}}$ separately. Otherwise, we would have to prove a higher order version of Proposition \ref{ProscatMax} taking moreover into account the difficulties related to the top order derivatives.
\end{Rq}

The properties of $f$ that we will exploit to deal with $1.$ are the estimates of Corollary \ref{Corrho3}, which turn out to be satisfied by $f_1$ as well (see Corollary \ref{Corlinbound}). For $2.$, we will use further the estimates, for the derivatives of $f$ up to order $N-1$, given by Corollary \ref{Corrho}. According to Corollary \ref{Corlinbound}, the derivatives of $f_1$, up to order $N$, verify the same estimates. Hence, in Sections \ref{Subsec1Max}--\ref{Subseclowerorder}, we also allow $f$ to denote $f_1$, in which case $N\!-\!1$ has to be formally replaced by $N$, so that the results proved there will imply Proposition \ref{ProforF1}.

\subsection{Well-posedness for the Maxwell equations with asymptotic data}\label{Subsec1Max}

\begin{Pro}
There exists a unique solution $F^{\mathrm{new}}$ to \eqref{kevatalenn:tredebis}, defined on $[T,+\infty[ \times \R^3$, which further verifies
$$ \forall \, |\gamma| \leq N, \qquad \qquad  \sup_{t \geq T}  \big\| \mathcal{L}_{Z^\gamma} F^{\mathrm{new}} (t, \cdot) \big\|^2_{L^2 (\R^3)} < + \infty .$$
\end{Pro}
\begin{proof}
Let $H$ be the unique soluton to
\begin{equation*}
\nabla^\mu H_{\mu \nu} = J(f)_\nu, \quad  \;\nabla^\mu {}^* \! H_{\mu \nu} =0, \qquad \quad \nabla_{\partial_{x^i}}H_{0 j} (T,\cdot)= \nabla_{\partial_{x^j}} H_{0 i} (T,\cdot), \quad \; \nabla_{\partial_{x^i}}{}^* \!H_{0 j}(T,\cdot)= \nabla_{\partial_{x^j}} {}^*\!H_{0 i} (T,\cdot) .
\end{equation*}
$1 \leq i , \, j \leq 3$. Note then that $\nabla^\mu \mathcal{L}_{Z^\gamma}(H)_{\mu \nu} = J(\widehat{Z}^\gamma f)_\nu$ by Proposition \ref{Com}. According to Proposition \ref{Proinidata0}, Lemma \ref{Lemftog} and Corollary \ref{Corrho}, we have 
$$ \forall \, |\gamma| \leq N, \qquad \big\| \mathcal{L}_{Z^\gamma} H (T, \cdot) \big\|_{L^2(\R^3)}+\sup_{t \geq T} \big\| \langle t-| \cdot| \rangle^{\frac{3}{2}} J \big( \widehat{Z}^\gamma f \big)(t,\cdot) \big\|_{L^2(\R^3)} <+ \infty.$$
By the general scattering result of Proposition \ref{blackboxscat}, applied to $\mathcal{L}_{Z^\gamma} H$ for any $|\gamma| \leq N-1$, $H$ admits a radiation field $\underline{\alpha}^{\mathcal{I}^+}$ along future null infinity, which satisfies
$$ \forall \, |\gamma| \leq N, \qquad \big\| \underline{\alpha}^{\mathcal{I}^+}_\gamma \big\|_{L^2(\R_u \times \mathbb{S}^2_\omega)} <+ \infty.$$
Consequently, in view of the assumptions on $\underline{\alpha}^{\infty}$, the radiation field $\underline{\alpha}^{\infty}_\gamma-\underline{\alpha}^{\mathcal{I}^+}_\gamma$ is bounded as well in $L^2_{u,\omega}$ for any $|\gamma| \leq N$. By Theorem \ref{Thscat}, there exists a unique $2$-form $H^{\mathrm{vac}}$ verifying
\begin{equation*}
\nabla^\mu H^{\mathrm{vac}}_{\mu \nu} = 0, \quad \nabla^\mu {}^* \! H_{\mu \nu}^{\mathrm{vac}} =0; \qquad \qquad \forall \, |\gamma| \leq N , \; \forall \, t \in \R_+, \quad  2\big\| \mathcal{L}_{Z^\gamma} H (t, \cdot) \big\|_{L^2(\R^3)}=\big\| \underline{\alpha}^{\infty}_\gamma-\underline{\alpha}^{\mathcal{I}^+}_\gamma \big\|_{L^2(\R_u \times \mathbb{S}^2_\omega)},
\end{equation*}
and $\underline{\alpha}^\infty \! - \underline{\alpha}^{\mathcal{I}^+}$ is the radiation field of $H^{\mathrm{vac}}$ along $\mathcal{I}^+$. Thus, the result holds and $F^{\mathrm{new}}:=H\!+H^{\mathrm{vac}}$.
\end{proof}

\subsection{Improved estimates for the lower order derivatives of the electromagnetic field}\label{Subseclowerorder} We refine now our estimates on $\mathcal{L}_{Z^\gamma}F^{\mathrm{new}}$, for $|\gamma| \leq N-1$. More precisely, we prove here the next result.

\begin{Pro}\label{Projustheretobeproved1}
If $D$ is chosen large enough and if $\Lambda B_\Lambda \epsilon \log^{-1} T$ is small enough, then 
$$ \forall \, t \geq T, \qquad \qquad \mathcal{E}^K_{N-1} \big[ F^{\mathrm{new}}-F^{\mathrm{asymp}}[f_\infty] \big](t) \leq \frac{D}{2} \Lambda .$$ 
\end{Pro}
Let $|\kappa| \leq N-1$. If $Z^\kappa$ is only composed by homogeneous vector fields $\Omega_{0k}$, $\Omega_{ij}$ and $S$, then, according to the commutation formula given by Proposition \ref{ProComMax}, we have
\begin{equation*}
\nabla^\mu \mathcal{L}_{Z^\kappa}(F^{\mathrm{new}}-F^{\mathrm{asymp}}[f_\infty])_{\mu \nu}= J(\widehat{Z}^\kappa f)_\nu-J^{\mathrm{asymp}}_\nu[\widehat{Z}^\kappa_\infty f_\infty], \qquad \qquad \nabla^\mu {}^* \! \mathcal{L}_{Z^\kappa}(F^{\mathrm{new}}-F^{\mathrm{asymp}}[f_\infty])_{\mu \nu} =0.
\end{equation*}
If $\kappa_T \geq 1$, $J(\widehat{Z}^\kappa f)$ and the source term of the commuted asymptotic Maxwell equations are strongly decaying so that we will be able to estimate both fields separately. Recall further that $J^{\mathrm{asymp}}_\nu[\widehat{Z}^\kappa_\infty f_\infty]=0$ in that case. This motivates the first step of the proof, which consists in proving that, for any $|\gamma| \leq N-1$, there exists an electromagnetic field $G^\gamma$ such that
\begin{equation}\label{eq:estiGgamma}
\hspace{-6mm} \nabla^\mu G_{\mu \nu}^\gamma= J(\widehat{Z}^\gamma f)_\nu-J^{\mathrm{asymp}}_\nu \big[\widehat{Z}^\gamma_\infty f_\infty \big], \qquad \nabla^\mu {}^* \! G_{\mu \nu}^\gamma =0; \qquad \qquad \sup_{t \geq T} \mathcal{E}^K [G^\gamma](t) \lesssim \epsilon T^{-1/2} \Lambda B_\Lambda 
\end{equation}
and its radiation field along future null infinity vanishes identically. For this, it suffices to check that we can apply Proposition \ref{ProscatMax} for $\underline{\alpha}^{\mathcal{I}^+}=0$, the source term $J:=J(\widehat{Z}^\gamma f)_\nu-J^{\mathrm{asymp}}_\nu[\widehat{Z}^\gamma_\infty f_\infty]$, $a=0$, $\delta =1/4$ and $B_{\mathrm{source}} \lesssim \epsilon T^{-1/2} \Lambda B_\Lambda$. In view of the strong $L^2_x$ decay estimates satisfied by $J$ given by Corollary \ref{Corrho} and Proposition \ref{CorassumpJforcondi}, we have indeed, for all $n \geq T$,
 $$ \forall \, t \geq T, \qquad \mathcal{E}^{K}\big[ S(J,n) \big] (t) \lesssim \langle t \rangle^{-\frac{1}{4}} \epsilon T^{-\frac{1}{2}} \Lambda B_\Lambda,$$
 so that the assumption \eqref{eq:condiSJn} is verified. Let us now prove the following result, which will also be useful for the treatment of the top order derivatives.
\begin{Pro}\label{ProGvacsolution}
There exists a solution $G^{\mathrm{vac}}$ to the vacuum Maxwell equations such that, for any $|\gamma| \leq N$,
$$  \forall \, t \in \R_+, \quad \qquad \mathcal{E}^K \big[ \mathcal{L}_{Z^\gamma} G^{\mathrm{vac}} \big](t) = \big\| \underline{\alpha}^{\infty}_\gamma- \underline{\alpha}^{\mathrm{asymp}}_\gamma [f_\infty] \big\|^2_{\mathcal{S}^K_{\mathcal{I}^+}} \lesssim \Lambda ,$$
and $\underline{\alpha}^\infty_\gamma - \underline{\alpha}^{\mathrm{asymp}}_\gamma[f_\infty]$ is the radiation field of $\mathcal{L}_{Z^\gamma} G^{\mathrm{vac}}$ along future null infinity $\mathcal{I}^+$.
\end{Pro}
\begin{proof}
Recall from Proposition \ref{Proradasymp} the properties satisfied by the radiation field $\underline{\alpha}^{\mathrm{asymp}}[f_\infty]$ of $F^{\mathrm{asymp}}[f_\infty]$, which imply, for all $|\gamma| \leq N$, 
$$\underline{\alpha}^\infty_\gamma -\underline{\alpha}^{\mathrm{asymp}}_\gamma[f_\infty] \in \mathcal{S}^K_{\mathcal{I}^+}, \qquad \qquad \qquad \big\| \underline{\alpha}^{\infty}_\gamma- \underline{\alpha}^{\mathrm{asymp}}_\gamma [f_\infty] \big\|^2_{\mathcal{S}^K_{\mathcal{I}^+}} \lesssim \Lambda + \epsilon \leq 2 \Lambda.$$ 
Consequently, $\underline{\alpha}^\infty -\underline{\alpha}^{\mathrm{asymp}}[f_\infty] \in \mathcal{S}^{\partial_{t},N}_{\mathcal{I}^+}$ and Theorem \ref{Thscat} implies that there exists a unique solution $G^{\mathrm{vac}}$ to the vacuum Maxwell equations such that
$$\forall \, |\gamma|=N, \qquad  \quad \| \mathcal{L}_{Z^\gamma} G^{\mathrm{vac}} (0,\cdot)\|_{\mathcal{S}^{\partial_t,N}_{\{t=0 \}}} = \big\| \underline{\alpha}^{\infty}_\gamma- \underline{\alpha}^{\mathrm{asymp}}_\gamma [f_\infty] \big\|_{\mathcal{S}^{\partial_t,N}_{\mathcal{I}^+}} $$
and $\underline{\alpha}^\infty - \underline{\alpha}^{\mathrm{asymp}}[f_\infty]$ is the radiation field of $G^{\mathrm{vac}}$ along $\mathcal{I}^+$. By Proposition \ref{blackboxscat}, $\underline{\alpha}^\infty_\gamma - \underline{\alpha}^{\mathrm{asymp}}_\gamma[f_\infty]$ is the radiation field of $\mathcal{L}_{Z^\gamma}G^{\mathrm{vac}}$ for any $|\gamma| \leq N$. The improved energy estimate given in the statement then follows from various applications of Theorem \ref{Thscat}, this time in the energy spaces $S^K_{\{t=0\}}$ and $S^{K}_{\mathcal{I}^+}$, as well as the conservation of $\mathcal{E}^K[\mathcal{L}_{Z^\gamma}G^{\mathrm{vac}}]$ over time.
\end{proof}

We are now able to control the weakly decaying derivatives. Fix then $|\gamma| \leq N-1$ such that $\gamma_T =0$, that is $Z^\gamma$ is only composed by homogeneous vector fields. By uniqueness considerations for the asymptotic Cauchy problem, we have $\mathcal{L}_{Z^\gamma} F^{\mathrm{new}}= \mathcal{L}_{Z^\gamma}G^{\mathrm{vac}}+G^\gamma+\mathcal{L}_{Z^\gamma} F^{\mathrm{asymp}}[f_\infty]$. Thus,
\begin{equation}\label{eq:forconcluProposi}
 \forall \, t \geq T, \qquad \qquad \mathcal{E}^K \big[ \mathcal{L}_{Z^\gamma}\big(F^{\mathrm{new}}-F^{\mathrm{asymp}}[f_\infty] \big) \big](t) \lesssim  \Lambda + \epsilon T^{-1/2} \Lambda B_\Lambda
 \end{equation}
according to Proposition \ref{ProGvacsolution} and \eqref{eq:estiGgamma}. Otherwise $\gamma_T \geq 1$ so $J^{\mathrm{asymp}}[\widehat{Z}_\infty^\gamma f_\infty]=0$ and the source terms of the Maxwell equations verified by $G^\gamma$ and $\mathcal{L}_{Z^\gamma} F^{\mathrm{new}}$ correspond. It leads us to consider a different solution to the vacuum Maxwell equations.
\begin{Pro}\label{ProMvacsolution}
There exists a unique solution $M^{\mathrm{vac}}$ to the vacuum Maxwell equations such that, 
\begin{itemize}
\item for any $|\gamma| \leq N$ and all $t \in \R_+$, we have $2\| \mathcal{L}_{Z^\gamma}M^{\mathrm{vac}}(t,\cdot) \|_{L^2_x}= \| \underline{\alpha}_\gamma^\infty\|_{L^2_{u,\omega}}$,
\item $\underline{\alpha}_\gamma^\infty$ is the radiation field of $\mathcal{L}_{Z^\gamma}M^{\mathrm{vac}}$ along future null infinity $\mathcal{I}^+$,
\item if $\gamma_T \geq 1$, then $\mathcal{E}^K \big[ \mathcal{L}_{Z^\gamma} M^{\mathrm{vac}} \big](t) = \big\| \underline{\alpha}^{\infty}_\gamma \big\|^2_{\mathcal{S}^K_{\mathcal{I}^+}} \lesssim \Lambda$ for all $t \in \R_+$.
\end{itemize}
\end{Pro}
\begin{proof}
For the first two points, we use the property of the operator $\mathscr{F}^+ \vert_{\mathcal{S}^{\partial_{t},N}_{\{t=0\}}}$ given by Theorem \ref{Thscat} and the assumption \eqref{eq:assumpunderalphainfty} on $\underline{\alpha}^\infty$. For the last part, fix $|\gamma| \leq N$ such that $\gamma_T \geq 1$ and recall Definition \ref{Defrecurrad}. Then, 

$$\int_{\R_u} \underline{\alpha}^\infty_\gamma (u,\cdot) \dr u =0,$$ so that, using again \eqref{eq:assumpunderalphainfty}, we have $\underline{\alpha}^{\infty}_\gamma \in \mathcal{S}_{\mathcal{I}^+}^{K}$. It then remains to apply once again Theorem \ref{Thscat} and to use that $\mathcal{E}^K[\cdot]$ is conserved for any sufficiently regular solution to the free Maxwell equations.
\end{proof}

We then deduce that $\mathcal{L}_{Z^\gamma} F^{\mathrm{new}}= \mathcal{L}_{Z^\gamma}M^{\mathrm{vac}}+G^\gamma$ as well as
\begin{equation}\label{kevatalenn:15}
 \forall \, t \geq T, \qquad \qquad \mathcal{E}^K \big[ \mathcal{L}_{Z^\gamma}F^{\mathrm{new}} \big](t) \lesssim  \Lambda + \epsilon T^{-1/2} \Lambda B_\Lambda.
 \end{equation}
Finally, by Corollary \ref{CorweightedcontrolFasymp}, we have $\mathcal{E}^K\big[\mathcal{L}_{Z^\gamma} F^{\mathrm{asymp}}[f_\infty] \big] \lesssim \epsilon \leq \Lambda $. Thus, \eqref{eq:forconcluProposi} is in fact verified for any $|\gamma| \leq N$ and Proposition \ref{Projustheretobeproved1} holds.

\subsection{The top order derivatives}\label{toporder}

We fix, for all this section, $|\gamma| \leq N-1$. In order to control sufficiently well $\nabla_{\partial_{x^k}} \mathcal{L}_{Z^\gamma}F^{\mathrm{new}}$, for $|\gamma|=N-1$, through Proposition \ref{ProscatMax}, we will be interested in the first order spatial derivatives of the following solution to the Maxwell equations. Let $\tau_0 \geq 4T$ and $G^{\tau_0}$ be uniquely defined as the solution to 
$$
 \nabla^\mu G_{\mu \nu}^{\tau_0} = \chi \big(t / \tau_0 \big)\Big[J\big( \widehat{Z}^\gamma f \big)_\nu-J^{\mathrm{asymp}}_\nu\big[ \widehat{Z}^\gamma_\infty f_\infty \big] \Big], \qquad \nabla^\mu {}^* \! G_{\mu \nu}^{\tau_0}=0, \qquad \qquad G^{\tau_0}(\tau_0,\cdot)=0, 
$$
where $\chi \in C^\infty ( \R,[0,1])$ verifies $\chi (s)=1$ for all $s \leq 1/4$ and $\chi(s)=0$ for all $s \geq 1/2$. Then, we will control the time derivative of $\mathcal{L}_{Z^\gamma}F^{\mathrm{new}}$ through the Maxwell equations (see for instance \eqref{VM2}--\eqref{VM3}). For the purpose of controlling the spatial derivatives of $G^{\tau_0}$, we first determine their Glassey-Strauss decomposition, allowing for a gain of regularity which can be compared to the elliptic regularity in the context of the Vlasov-Poisson system.
\subsubsection{Glassey-Strauss decomposition of the derivatives of the field}

We fix $\tau_0 \geq 4T$ and we simplify the notations by dropping the superscript $\tau_0$ of $G^{\tau_0}$. The decomposition of $\nabla_{\partial_{x^k}}G$ that we will use requires to introduce several integral kernels. We point out that, apart for one of them, the only information that we will need on these kernels is an upper bound, so that their precise expression are only useful for the statement of the decomposition of the derivatives of $G$.
\begin{Def}
For any $(i,j) \in \llbracket 1 , 3 \rrbracket^2$, we define the next smooth functions of the variables $(\omega,v) \in \mathbb{S}_\omega^2 \times \R^3_v$,
\begin{equation*}
\mathbf{w}_{0i}(\omega,v)=-\mathbf{w}_{i0}(\omega,v):= \omega_i+\widehat{v}_i, \qquad \mathbf{w}_{ij}(\omega,v)  := \omega_i\widehat{v}_j-\omega_j\widehat{v}_i.
\end{equation*}
Let further, for any $k \in \llbracket 1 , 3 \rrbracket$ and $0 \leq \mu , \, \nu \leq 3$, 
\begin{align*}
\mathbf{a}_{\mu\nu}^k(\omega,v)&:=-3\frac{\mathbf{w}_{\mu \nu}(\omega,v)\omega_k}{\langle v \rangle^4(1+\omega \cdot \widehat{v})^4}-3\frac{\mathbf{w}_{\mu \nu}(\omega,v)\widehat{v}_k}{\langle v \rangle^2(1+\omega \cdot \widehat{v})^3}+\frac{\delta_{k\mu}\widehat{v}_{\nu}-\delta_{k\nu}\widehat{v}_{\mu}}{\langle v \rangle^2(1+\omega \cdot \widehat{v})^2}, \\
\mathbf{b}^k_{\mu \nu} (\omega,v) &:= 3\frac{\mathbf{w}_{\mu \nu}(\omega,v)\omega_k}{\langle v \rangle^2(1+\omega \cdot \widehat{v})^3}-2\frac{\mathbf{w}_{\mu \nu}(\omega,v)\widehat{v}_k}{(1+\omega \cdot \widehat{v})^2}-\frac{\delta_{k\mu}\widehat{v}_{\nu}-\delta_{k\nu}\widehat{v}_{\mu}}{1+\omega \cdot \widehat{v}}, \\
\mathbf{c}^k_{\mu \nu}(\omega,v) &:= \frac{\omega_k \mathbf{w}_{\mu \nu}(\omega,v)}{(1+\omega \cdot \widehat{v})^2}, \qquad \qquad \qquad \mathbf{d}^k_{\mu \nu}(\omega,v)  := \frac{\omega_k \mathbf{w}_{\mu \nu}(\omega,v)}{\langle v \rangle^2(1+\omega \cdot \widehat{v})^3}.
\end{align*}
All these quantities are antisymmetric in $(\mu,\nu)$.
\end{Def}

Now, according to \cite[Corollary~$5.8$]{scat}, the following estimates hold.
\begin{Lem}\label{Lemboundkernel}
For any $ 1 \leq k \leq 3$, we have,
$$ \forall \, v \in \R^3_v, \qquad \Big( \big| \mathbf{a}^k\big| +\big|\nabla_v \mathbf{a}^k\big| +\big| \nabla_v \mathbf{b}^k\big|+\big| \mathbf{c}^k\big|+\big|\nabla_v \mathbf{c}^k\big|+\big| \nabla_v \nabla_v\mathbf{c}^k\big|+\big| \mathbf{d}^k\big|+\big|\nabla_v \mathbf{d}^k\big|\Big)(\cdot,v) \lesssim \langle v \rangle^3.$$
\end{Lem}

Recall the quantities $\rho^\beta\big[ f_n,  f_\infty ; \mathbf{q} \big](t,x,\omega)$ introduced in Corollary \ref{Corrho2}. We further lighten the notations by denoting the Lorentz force as
\begin{equation}\label{Lorentzf}
K^j(t,x,v):=\widehat{v}^{\mu} {F_\mu}^{j}(t,x), \qquad \qquad K_\xi^j(t,x,v) := \widehat{v}^{\mu} {\mathcal{L}_{Z^{\xi}}(F)_\mu}^{j}(t,x), \qquad 1 \leq j \leq 3, \quad   |\xi| \leq N. 
\end{equation}

\begin{Pro}\label{GSdecomoderiv}
Let $1 \leq k \leq 3$. Then, on $[T,\tau_0] \times \R^3_x$, $\nabla_{\partial_{x^k}}G$ can be written as
$$ 4\pi \nabla_{\partial_{x^k}} G= G^{\mathrm{ver}}_{k}+G^{\chi'}_{k}+G^{\chi''}_{k}+G^{TT}_{k}+G^{TS}_{k}+G^{SS}_{k}, $$
where the six $2$-forms are defined as follows. Fix indices $0 \leq \mu, \nu \leq 3$ and let us lighten the presentation by using, in the integrals below, the shorthands $$\omega :=\frac{y-x}{|y-x|}, \qquad \qquad t_y:=t+|y-x|.$$
 We have linear quantities,
\begin{itemize}
\item the vertex term
$$G^{\mathrm{ver}}_{k, \, \mu \nu}(t,x) :=  \int_{\sigma \in \mathbb{S}^2} \chi \left( \frac{t}{\tau_0} \right)  \rho^\gamma_k\big[ f,  f_\infty ; \mathbf{d}_{\mu \nu} \big](t,x,\sigma)  \dr \mu_{\mathbb{S}^2}.$$
\item The most singular term,
$$ G^{TT}_{k, \, \mu\nu}(t,x) :=  \int_{|y-x|\leq \tau_0-t }\chi \left( \frac{t+|y-x|}{\tau_0} \right) \rho^\gamma_k\big[ f,  f_\infty ; \mathbf{a}_{\mu \nu} \big](t+|y-x|,y,\omega) \frac{ \dr y}{|y-x|^3}$$
where the crucial identity $ \int_{|\sigma|=1} \mathbf{a}_{\mu \nu}^k(\sigma,\widehat{v})\dr \mu_{\mathbb{S}^2}=0$ holds for all $v \in \R^3_v$.
\item The terms arising from the cutoff function,
\begin{align*}
G^{\chi'}_{k, \, \mu\nu}(t,x) := \frac{1}{\tau_0} \int_{|y-x|\leq \tau_0-t }\chi' \left( \frac{t+|y-x|}{\tau_0} \right) \rho^\gamma_k\big[ f,  f_\infty ; \mathbf{b}_{\mu \nu} \big](t+|y-x|,y,\omega)   \frac{ \dr y}{|y-x|^2}, \\
G^{\chi''}_{k, \, \mu\nu}(t,x) := -\frac{1}{\tau_0^2} \int_{|y-x|\leq \tau_0-t }\chi'' \left( \frac{t+|y-x|}{\tau_0} \right) \rho^\gamma_k\big[ f,  f_\infty ; \mathbf{c}_{\mu \nu} \big](t+|y-x|,y,\omega)  \frac{ \dr y}{|y-x|}.
\end{align*}
\end{itemize}
Next, we have the nonlinear quantities $G^{TS}_{k}$ and $G^{SS}_{k}$. There exists integers $N^\gamma_{\xi , \beta}, \, N^{\gamma}_{\xi , \kappa , \beta} \in \mathbb{Z}$ such that
$$ G^{\, TS}_{k, \, \mu\nu}(t,x) := -\sum_{|\xi|+|\beta| \leq |\gamma|}N^{\gamma}_{\xi, \kappa} \int_{|y-x|\leq \tau_0-t } \chi \left( \frac{t_y}{\tau_0} \right)\int_{\R^3_v} \nabla_v \mathbf{b}_{\mu \nu}^k(\omega,v) \cdot \big(\widehat{Z}^{\beta}f \, K_{\xi} \big)(t_y,y,v) \frac{\mathrm{d}v \mathrm{d}y}{|y-x|^2}$$
and $G^{SS}_{k}$ is the sum of the four following terms, 
\begin{align*}
G^{SS,I}_{k, \, \mu \nu}&:=\sum_{|\xi|+|\kappa|+|\beta| \leq |\gamma|} N^{\gamma}_{\xi ,\kappa, \beta} \,\int_{|y-x|\leq \tau_0-t } \chi \left( \frac{t_y}{\tau_0} \right)\int_{\R^3_v}\left[\nabla_v \big( \nabla_v\mathbf{c}_{\mu\nu}^k(\omega,\cdot) \cdot K_{\xi} \big) \cdot K_{\kappa}\widehat{Z}^{\beta}f\right](t_y,y,v) \frac{\mathrm{d}v \mathrm{d}y}{|y-x|}, \\
G^{SS,II}_{k, \, \mu \nu}&:=\sum_{|\xi|+|\beta| \leq |\gamma|} N^{\gamma}_{\xi , \beta} \, \int_{|y-x|\leq \tau_0-t }\chi \left( \frac{t_y}{\tau_0} \right) \int_{\R^3_v} \nabla_v  \mathbf{c}_{\mu\nu}^k(\omega,v) \cdot \big( \T_0(K_\xi) \widehat{Z}^{\beta}f \big)(t_y,y,v) \frac{\mathrm{d}v \mathrm{d}y}{|y-x|}, \\
G^{SS,III}_{k, \, \mu \nu}&:=\sum_{|\xi|+|\beta| \leq |\gamma|} N^{\gamma}_{\xi , \beta} \, \int_{|y-x|\leq \tau_0-t } \chi \left( \frac{t_y}{\tau_0} \right) \int_{\R^3_v} \mathbf{c}_{\mu \nu}^{k} (\omega,v) \frac{\delta_j^i-\widehat{v}^i \widehat{v}_j}{v^0} \partial_{x^i}\big (K_{\xi}^j\widehat{Z}^{\beta}f \big)(t_y,y,v) \frac{\mathrm{d}v \mathrm{d}y}{|y-x|}, \\
G^{SS,IV}_{k, \, \mu \nu}&:=\sum_{|\xi|+|\beta| \leq |\gamma|} \frac{N^{\gamma}_{\xi , \beta} }{\tau_0} \int_{|y-x|\leq \tau_0-t }\chi' \left( \frac{t_y}{\tau_0} \right) \int_{\R^3_v} \nabla_v \mathbf{c}_{\mu \nu}^k(\omega,v) \cdot \big(\widehat{Z}^{\beta}f \, K_{\xi} \big)(t_y,y,v)  \frac{\mathrm{d}v \mathrm{d}y}{|y-x|}.
\end{align*}
\end{Pro}
\begin{proof}
Let $k \in \llbracket 1,3 \rrbracket$, $|\gamma|\leq N-1$ and consider a sufficiently regular function $h :[T,\tau_0] \times \R^3_x \times \R^3_v \to \R$ as well as $H$ satisfying
$$
 \nabla^\mu H_{\mu \nu} = \chi \big(t / \tau_0 \big)J( h )_\nu, \qquad \nabla^\mu {}^* \! H_{\mu \nu}=0, \qquad \qquad H(\tau_0,\cdot)=0. 
$$
In order to fit with the framework of Glassey-Strauss, recall further the electric part $E^i:= H_{0i}$ and the magnetic part $B^i = -\frac{1}{2}\varepsilon^{ijk} H_{jk}$ of $H$. We apply\footnote{See also the original version of the result, \cite[Theorem~$4$]{GlStrauss}.} \cite[Theorem~$5.4.1$]{Glassey} to $(\chi(t/\tau_0)h,E,B)$, where we denote abusively $(t,x,v) \mapsto \chi(t/\tau_0)$ by $\chi(t/\tau_0)$. This yields, as $H$ and its derivatives vanish at $t= \tau_0$,
$$ \nabla_{\partial_{x^k}} H_{\mu \nu} = H^{\mathrm{ver}}_{k, \, \mu \nu}+H^{TT}_{k, \, \mu \nu}+\bar{A}^{\, TS}_{k, \, \mu \nu}+\bar{A}^{\,SS}_{k, \, \mu \nu},$$
where
\begin{align*}
H^{\mathrm{ver}}_{k, \, \mu \nu}(t,x) &:=  \int_{\sigma \in \mathbb{S}^2} \chi \left( \frac{t}{\tau_0} \right) \int_{\R^3_v} \mathbf{d}^k_{\mu \nu} ( \sigma,  v ) h(t,x,v)  \dr v \dr \mu_{\mathbb{S}^2} , \\
 H^{TT}_{k, \, \mu\nu}(t,x) &:=  \int_{|y-x|\leq \tau_0-t }\chi \left( \frac{t+|y-x|}{\tau_0} \right) \int_{\R^3_v}\mathbf{a}^k_{\mu \nu} ( \omega, v) h(t+|y-x|,y,v)  \frac{ \dr v \dr y}{|y-x|^3} , \\
\bar{A}^{\, TS}_{k, \,\mu \nu}&:=\int_{|y-x| \leq \tau_0-t} \int_{\R^3_v} \mathbf{b}^k_{\mu \nu}(\omega,v) \T_0 \big( \chi (t/\tau_0 ) h \big)(t+|y-x|,y,v) \frac{\mathrm{d}v \mathrm{d}y}{|y-x|^2}, \\
\bar{A}^{\, SS}_{k, \, \mu \nu} &:=-\int_{|y-x| \leq \tau_0-t} \int_{\R^3_v} \mathbf{c}^k_{\mu \nu}(\omega,v)  \T_0 \T_0 \big( \chi (t/\tau_0 ) h \big)(t+|y-x|,y,v) \frac{\mathrm{d}v \mathrm{d}y}{|y-x|},
\end{align*}
as well as $ \int_{|\sigma|=1} \mathbf{a}^k_{\mu \nu}(\sigma,\widehat{v})\dr \mu_{\mathbb{S}^2}=0 $. We now apply this result for 
$$h:=\widehat{Z}^\gamma f(t,x,v)- h_n, \qquad \qquad h_n(t,x,v):= \int_{\R^3_z} \big[ |v^0|^5 \widehat{Z}_\infty^\gamma f_\infty \Big](z,v) \dr z \psi_n(x-t\widehat{v} ),$$
where $(\psi_n)_{n \geq 1}$ is a mollifier. Then, recall from Section \ref{Subsecfsing} that $(h_n)_{n \geq 1}$ converges to $f^{\mathrm{sing}} \big[ \widehat{Z}_\infty^\gamma f_\infty \big]$ as $n \to \infty$ and the corresponding convergence properties of its velocity averages (see in particular Lemma \ref{Lembasicpropsing}). As $\T_0(h_n)=0$, we obtain by letting $n \to +\infty$,
$$ H^{\mathrm{ver}}_{k} = G^{\mathrm{ver}}_{k}, \qquad H^{TT}_{k} = G^{TT}_k, \qquad  \bar{A}^{\, TS}_{k, \, \mu \nu} =G^{\chi'}_{k}+A^{\, TS}_{k, \, \mu \nu}, \qquad \bar{A}^{\, SS}_{k, \, \mu \nu}=G^{\chi''}_{k}+A^{\,SS}_{k, \, \mu \nu},$$
with, using the shorthand $t_y=t+|y-x|$,
\begin{align*}
A^{\, TS}_{k, \,\mu \nu}&:=\int_{|y-x| \leq \tau_0-t} \chi \left( \frac{t_y}{\tau_0} \right) \int_{\R^3_v} \mathbf{b}^k_{\mu \nu}(\omega,v) \T_0 \big(  \widehat{Z}^\gamma f \big)(t_y,y,v) \frac{\mathrm{d}v \mathrm{d}y}{|y-x|^2}, \\
A^{\, SS}_{k, \, \mu \nu} &:=-\int_{|y-x| \leq \tau_0-t} \int_{\R^3_v} \mathbf{c}^k_{\mu \nu}(\omega,v) \bigg[ \chi \left( \frac{t_y}{\tau_0} \right) \T_0 \T_0 \big(  \widehat{Z}^\gamma f \big)(t_y,y,v)+ \frac{1}{\tau_0}\chi' \! \left( \frac{t_y}{\tau_0} \right) \T_0 \big(  \widehat{Z}^\gamma f \big)(t_y,y,v) \bigg] \frac{\mathrm{d}v \mathrm{d}y}{|y-x|}.
\end{align*}
Recall now that the Lorentz force \eqref{Lorentzf} is divergence free, that is $\nabla_v \cdot K_\xi =0$ for any $|\xi| \leq N$. The commutation formula of Proposition \ref{Com} then yields
\begin{align}\label{eq:tokeepinmind}
\T_0 \big(\widehat{Z}^\gamma f\big)=  -\nabla_v \cdot \big( K \widehat{Z}^\gamma f \big)+ \sum_{|\beta| \leq |\gamma|-1} \sum_{|\xi|+|\beta| \leq |\gamma|} C^\gamma_{\xi,\beta} \, \nabla_v \cdot \big( K_\xi \, \widehat{Z}^\beta f \big).
\end{align}
We set $N^\gamma_{\xi,\beta}=-1$ if $|\xi|+|\beta|=0$ and $N^\gamma_{\xi,\beta}=C^\gamma_{\xi,\beta}$ otherwise. We get $A^{\, TS}_{k}=G^{\,TS}_k$ by performing integration by parts in $v$. Similarly, the part of $A^{\, SS}_{k}$ related to $\chi'$ gives rise to $G^{\, SS,IV}_k$. In order to deal with the last part of $A^{\, SS}_{k}$, we observe $[\T_0, \partial_{v^j}]=-\partial_{v^j}(\widehat{v}^i)\partial_{x^i}$, so that
\begin{align*}
&\T_0 \T_0 \big( \widehat{Z}^\gamma f \big)  =   \sum_{|\xi|+|\beta| \leq |\gamma|} N^\gamma_{\xi,\beta}\bigg[\nabla_v \cdot \Big( K_\xi \, \T_0(\widehat{Z}^\beta f) \Big)+ \nabla_v \cdot \Big(\T_0 \big( K_\xi \big) \, \widehat{Z}^\beta f \Big)-\frac{\delta_j^i-\widehat{v}^i \widehat{v}_j}{v^0} \partial_{x^i} \Big( K_\xi^j  \,\widehat{Z}^\beta f \Big) \bigg].
\end{align*} 
The last term generates $G^{SS,III}_k$ and the second one yields, by integration by parts in $v$, $G^{SS,II}_k$. Finally, we transform the first term using \eqref{eq:tokeepinmind} for $\gamma=\beta$. Performing twice integration by parts in $v$ provides the quantity $G^{SS,I}_k$, for certain integers $N^\gamma_{\xi,\kappa,\beta}$.
\end{proof}

\subsubsection{Estimates for the derivatives of $G$}
Our analysis of the terms of the Glassey-Strauss decomposition of $\nabla_{\partial_{x^k}}G$ will rely on the following result. Recall the shorthand $t_y:=t+|y-x|$.
\begin{Lem}\label{Lemcdvtech}
Let $B>0$ and $: \psi :[T,\tau_0] \times \R^3 \times \mathbb{S}^2 \to \R$ be a function such that
$$\forall \, (s,\omega) \in [T,\tau_0] \times  \mathbb{S}^2, \qquad \qquad \big\| \langle s + | \cdot | \rangle^{\frac{5}{2}} \, \psi (s,\cdot,\omega) \big\|_{L^2(\R^3_x)}^2 \leq B \, \langle s \rangle^{\frac{1}{8}}.$$
Let further $p \in \{1 , 2 , 3 \}$ as well as $b_p$ such that $b_1=b_2=0$ and $b_3=1$. Then, for all $t \in [T,\tau_0]$,
$$  \int_{\R^3_x} \langle t+|x| \rangle^4 \bigg| \int_{b_p \leq  |y-x| \leq \tau_0-t} \psi \bigg( t_y,y,\frac{y-x}{|y-x|} \bigg)  \frac{\mathrm{d}y}{\langle t_y \rangle^{3-p} \, |y-x|^p} \bigg|^2  \dr x  \lesssim \frac{B}{ \sqrt{t}} .$$
\end{Lem}
\begin{proof}
Let us denote $\frac{y-x}{|y-x|}$ by $\overline{\omega}$. The Cauchy-Schwarz inequality in $y$ gives
$$  \bigg| \int_{ b_p \leq |y-x| \leq \tau_0-t} \psi(t_y,y,\overline{\omega})   \frac{\mathrm{d}y}{\langle t_y \rangle^{3-p} \, |y-x|^p} \bigg|^2 \leq \int_{ |y-x| \leq \tau_0-t} \langle t_y+|y| \rangle^5\big| \psi(t_y,y,\overline{\omega}) \big|^2  \frac{\mathrm{d}y}{\langle t_y \rangle^{3-p} \, |y-x|^{p+\frac{1}{4}}} \, \mathcal{Q}^p_{t,x},$$
where
$$ \mathcal{Q}^p_{t,x} := \int_{b_p \leq |y-x| \leq \tau_0-t}  \frac{\mathrm{d}y}{\langle t+|y-x|+|y| \rangle^{5} \, \langle t+|y-x| \rangle^{3-p} \, |y-x|^{p-\frac{1}{4}}}.$$
Now, we use $|y|/2 \geq |x|/2-|y-x|/2$ and we parameterize the domain of integration, which is a cone, by $(s,\omega)\mapsto (t+s,x+s\omega)$. As $s^2 \langle t+s\rangle^{p-3}s^{-p+1/4} \lesssim \langle s \rangle^{-3/4}$ for $s \geq b_p$, one obtains
$$
\mathcal{Q}^p_{t,x}   \leq 4 \pi  \int_{s=0}^{\tau_0-t} \frac{ \dr s}{\langle t+|x|+s \rangle^5 \langle s \rangle^{\frac{3}{4}}} \lesssim \frac{1}{\langle t+|x| \rangle^{4+\frac{3}{4}}} \lesssim \frac{1}{\langle t+|x| \rangle^{4} \, t^{\frac{3}{4}}} .
$$
Next, we parameterize again the cone by $(s,\omega)\mapsto (t+s,x+s\omega)$. Hence, applying Fubini's theorem and performing then the change of variables $X=x+s\omega$, we get
\begin{align*}
&\int_{\R^3_x} \int_{b_p \leq |y-x| \leq \tau_0-t} \langle t_y+|y| \rangle^5\big| \psi(t_y,y,\overline{\omega}) \big|^2  \frac{\mathrm{d}y}{\langle t_y \rangle^{3-p} \, |y-x|^{p+\frac{1}{4}}}  \dr x \\
& \qquad \qquad =  \int_{s=b_p}^{\tau_0-t} \int_{\mathbb{S}^2_\omega} \int_{\R^3_X} \langle t+s+|X| \rangle^5 \big|\psi(t+s,X,\omega) \big|^2 \dr X \dr \mu_{\mathbb{S}^2_\omega} \frac{ \dr s}{\langle s \rangle \, |s|^{\frac{1}{4}} } \leq 4\pi B \int_{s=b_p}^{\tau_0-t} \frac{\langle t+s \rangle^{\frac{1}{8}} \dr s}{\langle s \rangle \, |s|^{\frac{1}{4}} } \lesssim B \, t^{\frac{1}{8}}.
\end{align*}
\end{proof}
We start by dealing with the linear terms in the decomposition of $\nabla_{\partial_{x^k}} G$.
\begin{Pro}\label{Proestivertex}
Let $ 1 \leq k \leq 3$. Then,
$$ \forall \, t \in [T, \tau_0], \qquad \int_{\R^3_x} \langle t+|x| \rangle^4 \Big( \big|G^{\mathrm{ver}}_k \big|^2+\big|G^{\chi'}_k \big|^2+\big|G^{\chi''}_k \big|^2 +\big|G^{TT}_k \big|^2\Big)(t,x) \dr x \lesssim \frac{\epsilon \Lambda B_\Lambda}{\sqrt{t}}.$$
\end{Pro}
\begin{proof}
The estimate for $G^{\mathrm{ver}}_k$ simply follows from the Cauchy-Schwarz inequality in the variables $ \sigma \in \mathbb{S}^2$ and the $L^2_x$ decay estimate of Corollary \ref{Corrho2}. For the other terms, in order to apply the previous Lemma \ref{Lemcdvtech}, we will use that
$$\forall \, (s,\omega) \in [T,\tau_0]\times \mathbb{S}^2_\omega, \qquad \big\| \langle s + | \cdot | \rangle^{\frac{5}{2}} \, \rho^\gamma_k[f,f_\infty;\mathbf{q}_{\mu \nu}] (s,\cdot,\omega) \big\|_{L^2_x}^2 \lesssim \epsilon  \Lambda B_\Lambda \,\langle s \rangle^{\frac{1}{8}},$$ 
for any integral kernel $\mathbf{q} \in \{ \mathbf{a}, \mathbf{b}, \mathbf{c} \}$. This ensues from Corollary \ref{Corrho2}.  We treat $G^{\chi'}_k$ as well as $G^{\chi''}_k$ by noticing that $\tau_0^{-1} \lesssim \langle t_y \rangle^{-1}$ on the domain of integration and by applying Lemma \ref{Lemcdvtech}, for $p=1$ and $p=2$. In order to deal with the singular factor $|y-x|^{-3}$ of the integrand in the term $G^{TT}_k$, we write
\begin{align*}
G^{TT}_k(t,x)&=G^{TT,[0,1]}_k(t,x)+G^{TT,[1,\tau_0-t]}_k(t,x), \\
 G^{TT,[a,b]}_k(t,x)&:= \int_{a \leq |y-x|\leq b }\chi \left( \frac{t_y}{\tau_0} \right) \rho^\gamma_k\big[ f,  f_\infty ; \mathbf{a}_{\mu \nu} \big](t_y,y,\omega) \frac{ \dr y}{|y-x|^3}.
 \end{align*}
Note that we can assume $\tau_0-t \geq 1$ since the integrand vanishes if $\tau_0-t \leq 1$ in view of the support of $\chi$ and $\tau_0 \geq 4$. Then, $G^{TT,[1,\tau_0-t]}_k(t,x)$ can be estimated by applying again Lemma \ref{Lemcdvtech}, for $p=3$. Next, since $\int_{\mathbb{S}^2_\sigma} \mathbf{a}^k_{\mu \nu}(\sigma , \cdot )\dr \mu_{\mathbb{S}^2_\sigma}=0$, there holds
$$ \big|G^{TT,[0,1]}_k\big|(t,x) \leq \int_{s=0}^1 \bigg| \int_{\mathbb{S}^2_\omega} \rho^\gamma_k\big[ f,  f_\infty ; \mathbf{a}_{\mu \nu} \big](t+s,x+s\omega,\omega)-\rho^\gamma_k\big[ f_n,  f_\infty ; \mathbf{a}_{\mu \nu} \big](t+s,x,\omega) \mu_{\mathbb{S}^2_\omega} \bigg| \frac{ \dr s}{s}.$$
We apply the fundamental theorem of calculus $\tau \mapsto \rho^\gamma_k\big[ f,  f_\infty ; \mathbf{a}_{\mu \nu} \big](t+s,x+\tau s \omega ,\omega)$ in order to compensate the singular factor $s^{-1}$. We get
\begin{align*}
 \big|G^{TT,[0,1]}_k \big|^2(t,x) &  \leq \bigg| \int_{s=0}^1 \int_{\mathbb{S}^2_\omega}\int_{\tau=0}^1 \Big| \nabla_x \rho^\gamma_k\big[ f,  f_\infty ; \mathbf{a}_{\mu \nu} \big](t+s,x+\tau \omega ,\omega) \Big| \dr \tau \frac{s|\omega| \dr \mu_{\mathbb{S}^2_\omega}\dr s}{s} \bigg|^2  \\
 & \leq 4 \pi \int_{s=0}^1 \int_{\mathbb{S}^2_\omega}\int_{\tau=0}^1 \Big| \nabla_x \rho^\gamma_k\big[ f,  f_\infty ; \mathbf{a}_{\mu \nu} \big](t+s,x+\tau \omega ,\omega) \Big|^2 \dr \tau  \dr \mu_{\mathbb{S}^2_\omega}\dr s.
 \end{align*}
As $\langle t+|x| \rangle \lesssim \langle  t+s+|x+\tau \omega|\rangle$ for all $(s,\tau) \in [0,1]^2$, we obtain from Fubini's theorem and the change of variables $X=x+\tau  \omega$, 
 $$ \int_{\R^3_x} \langle t+|x| \rangle^4 \big|G^{TT,[0,1]}_k \big|^2(t,x) \dr x \lesssim  \sup_{(s, \omega ) \in [0,1] \times \mathbb{S}^2} \int_{\R^3_X} \langle t+s+|X| \rangle^4 \Big| \nabla_x \rho^\gamma_k\big[ f,  f_\infty ; \mathbf{a}_{\mu \nu} \big](t+s,X ,\omega) \Big|^2 \dr X .$$
It turns out that each of the two terms defining $\nabla_x \rho^\gamma_k \big[ f,  f_\infty ; \mathbf{a}_{\mu \nu} \big]$ are decaying sufficiently fast in order to bound them separatly. Consider $T \leq \tau \leq \tau_0$, $X \in \R^3$, $\omega \in \mathbb{S}^2$ as well as $1 \leq i \leq 3$ and remark that
$$\Big| \partial_{x^i} \rho^\gamma_k \big[ f,  f_\infty ; \mathbf{a}_{\mu \nu} \big] \Big|(\tau,X,\omega) \leq \bigg| \int_{\R^3_v} \mathbf{a}_{\mu \nu}^k (\omega,v) \partial_{x^i} \widehat{Z}^\gamma f (\tau,X,v) \dr v \bigg|+ \bigg|   \partial_{x^i} J^{\mathrm{asymp}}_0 \big[ \mathbf{a}_{\mu \nu}^k(\omega , \cdot) \widehat{Z}^\gamma_\infty f_\infty \big] (\tau,X) \bigg|  .$$
Since $|\mathbf{a}^k_{\mu \nu}(\cdot,v)|+|v^0 \nabla_v \mathbf{a}^k_{\mu \nu}(\cdot,v)| \lesssim \langle v \rangle^4$, we get by applying Corollary \ref{CorgainderivVla},
$$ \bigg| \int_{\R^3_v} \mathbf{a}_{\mu \nu}^k (\omega,v) \partial_{x^i} \widehat{Z}^\gamma f (\tau,X,v) \dr v \bigg| \lesssim \sum_{|\beta| \leq |\gamma|+1} \frac{1}{\langle \tau +|X| \rangle} \int_{\R^3_v} \langle v \rangle^6 \, \langle X-\tau \widehat{v} \rangle\, \big| \widehat{Z}^\beta f \big| (\tau,X,v) \dr v .$$
Similarly, by Lemma \ref{Lemfirstasymp0} and the assumptions on $f_\infty$, we have
 $$\bigg|   \partial_{x^i} J^{\mathrm{asymp}}_0 \big[ \mathbf{a}_{\mu \nu}^k(\omega , \cdot) \widehat{Z}^\gamma_\infty f_\infty \big] (\tau,X) \bigg| \lesssim \frac{1}{\tau^4} \mathds{1}_{|X|<\tau}  \sqrt{\epsilon}  \lesssim \frac{\sqrt{\epsilon} }{\langle \tau+|X| \rangle^4}.$$
 Consequently,
  $$ \int_{\R^3_x} \langle t+|x| \rangle^4 \big|G^{TT,[0,1]}_k \big|^2(t,x) \dr x \lesssim  \sup_{s \in [0,1] } \sup_{|\beta| \leq N}  \int_{\R^3_X} \langle t+s+|X| \rangle^2 \bigg| \int_{\R^3_v} \langle v \rangle^6 \, \langle X-\tau \widehat{v} \rangle\, \big| \widehat{Z}^\beta f \big| (\tau,X,v) \dr v \bigg|^2 \dr X+\frac{\epsilon}{ \langle t \rangle} $$
  and it remains to use the $L^2_x$ estimate of Corollary \ref{Corrho3}.
\end{proof}

We now focus on the nonlinear terms.
\begin{Pro}\label{ProGTS}
For any $1 \leq k \leq 3$, we have
$$ \forall \, t \in [T, \tau_0], \qquad  \int_{\R^3_x} \langle t+|x| \rangle^{4} \big| G^{\, TS}_k (t,x) \big|^2 \dr x \lesssim \frac{\epsilon \Lambda B_\Lambda }{\sqrt{t}} .$$
\end{Pro}
\begin{proof}
Recall from Lemma \ref{Lemboundkernel} that $|\nabla_v \mathbf{b}^k|(\cdot,v) \lesssim \langle v \rangle^3$. As $|\chi| \leq 1$, $\big| G^{\, TS}_k (t,x) \big|$ is bounded by a linear combination of the following terms,
$$ \mathfrak{B}^{\xi,\beta}_{t,x} :=   \int_{|y-x|\leq \tau_0-t }  \left| \mathcal{L}_{Z^\xi}F\right|(t_y,y)\int_{\R^3_v} \langle v \rangle^3 \big| \widehat{Z}^\beta f \big|(t_y,y,v) \mathrm{d}v \frac{ \mathrm{d}y}{|y-x|^2}, $$
where $|\xi|+|\beta| \leq |\gamma|$. Assume first that $|\xi| \leq N-3$, so that, by \eqref{eq:Mpoint} and Lemma \ref{gainv}, 
$$ \left| \mathcal{L}_{Z^\xi}F\right|(t_y,y) \lesssim \sqrt{\Lambda} \, \langle t_y+|y| \rangle^{-1} \, \langle t_y-|y| \rangle^{-1} \lesssim \sqrt{\Lambda} \, \langle t_y+|y| \rangle^{-2} \, \langle y- t_y\widehat{v} \rangle \, \langle v \rangle^2.$$
We then deduce, by combining the previous Lemma \ref{Lemcdvtech} with the $L^2_x$ estimate of Corollary \ref{Corrho3} that
\begin{align*}
 \int_{\R^3_x} \langle t+|x|& \rangle^4  \big| \mathfrak{B}^{\xi,\beta}_{t,x}\big|^2 \dr x \\
  & \lesssim  \int_{\R^3_x}  \langle t+|x| \rangle^4 \bigg| \int_{ |y-x| \leq \tau_0-t} \frac{\sqrt{\Lambda}}{\langle t_y +|y| \rangle} \int_{\R^3_v} \langle v \rangle^5 \langle y-t_y \widehat{v} \rangle \big| \widehat{Z}^\beta f \big|(t_y,y,v) \mathrm{d} v   \frac{\mathrm{d}y }{\langle t_y \rangle \, |y-x|^2} \bigg|^2 \dr x  \lesssim  \frac{\epsilon \Lambda B_\Lambda}{  \langle t \rangle^{\frac{1}{2}}}.
 \end{align*}
Assume now that $|\xi| \geq N-2$, so that $|\beta| \leq 1$ and, according to Corollary \ref{Corrho3},
\begin{equation*}
 \int_{\R^3_v} \langle v \rangle^3 \big| \widehat{Z}^\beta f \big|(t_y,y,v) \mathrm{d}v \lesssim \sqrt{\epsilon B_\Lambda} \, \frac{|\max( t_y-|y|,1) |^{\frac{3}{2}} }{ \langle t_y+|y| \rangle^{\frac{9}{2}}} \log^{6N}(t_y)\lesssim \sqrt{\epsilon B_\Lambda} \, \frac{|\max( t_y-|y|,1) |^{\frac{3}{2}} }{ \langle t_y+|y| \rangle^{\frac{9}{2}}} \langle t_y \rangle^{\frac{1}{10}}. 
\end{equation*}
Then, with $\psi(s,z):=\langle s \rangle^{\frac{1}{10}}|\max( s-|z|,1) |^{\frac{3}{2}} \, \langle s+|z| \rangle^{ -\frac{7}{2}}  \left| \mathcal{L}_{Z^\xi}F\right|(s,z)$, we have
$$ \int_{\R^3_x} \langle t+|x| \rangle^4 \big| \mathfrak{B}^{\xi,\beta}_{t,x}\big|^2 \dr x \lesssim \epsilon B_\Lambda   \int_{\R^3_x} \langle t+|x| \rangle^4 \bigg|  \int_{ |y-x| \leq \tau_0-t} \psi(t_y,y) \frac{\mathrm{d}y }{\langle t_y \rangle \, |y-x|^2} \bigg|^2 \dr x .$$
It then remains to apply Lemma \ref{Lemcdvtech} and to prove
 $$ \sup_{T \leq s \leq \tau_0} \big\| \langle s + | \cdot | \rangle^{\frac{5}{2}} \, \psi (s,\cdot) \big\|_{L^2(\R^3_x)}^2 \lesssim \Lambda \langle s \rangle^{\frac{1}{8}}.$$
For this, note that $\langle s+|z| \rangle^{ -1}  | \mathcal{L}_{Z^\xi}F|(s,z) \lesssim \sup_{Z \in \mathbb{K}}  |\nabla_{t,x} \mathcal{L}_{Z^\kappa}F|(s,z)$, where $Z^\xi=ZZ^\kappa$, and use \eqref{eq:ML2}.
\end{proof}

Finally, we deal with $G^{SS}_k$.
\begin{Pro}\label{ProGSS}
For any $1 \leq k \leq 3$, we have
$$ \forall \, t \in [T, \tau_0], \qquad  \int_{\R^3_x} \langle t+|x| \rangle^{4} \big| G^{\, SS}_k (t,x) \big|^2 \dr x \lesssim \frac{\epsilon \Lambda^2 B_\Lambda }{\sqrt{t}} .$$
\end{Pro}
\begin{proof}
Let us first prove that $ G^{\, SS}_k (t,x) $ can be bounded by a linear combination of terms of the form
$$\sqrt{\Lambda} \int_{|y-x|\leq \tau_0-t } \Big( \langle t_y \rangle \big|\nabla_{t,x} \mathcal{L}_{Z^\xi}F (t_y,y)\big|+ \big| \mathcal{L}_{Z^\xi}F (t_y,y)\big| \Big) \! \int_{\R^3_v}\langle v \rangle^5 \, \langle y-t_y \widehat{v} \rangle \, \big|\widehat{Z}^{\beta}f \big|(t_y,y,v) \mathrm{d}v\frac{ \mathrm{d}y}{ \langle t_y \rangle \, |y-x|}, $$
where $|\xi|+|\beta| \leq N$ and $|\xi| \leq N-1$. Indeed, all the integral kernels and their derivatives are bounded by $\langle v \rangle^3$ according to Lemma \ref{Lemboundkernel}. Moreover,
\begin{itemize}
\item we control $ G^{\, SS,IV}_k (t,x)$ by using $ \tau_0^{-1} \leq t_y^{-1}$.
\item For $ G^{\, SS,I}_k (t,x)$, we remark that $|\xi| +|\kappa| \leq N-1$ implies that either $K_\xi(t_y,y,v)$ or $K_\kappa (t_y,y,v)$ can be estimated pointwise by $\sqrt{\Lambda} \, \langle t_y \rangle^{-1}$.
\item We control $ G^{\, SS,II}_k (t,x)$ as well as the terms obtained when $\partial_{x^i}$ is applied to the Lorentz force in $ G^{\, SS,III}_k (t,x)$ by simply using the triangle inequality for integrals.
\item Finally, when $\partial_{x^i}$ is applied to the distribution function in $G^{\, SS,III}_k (t,x)$, we use Corollary \ref{CorgainderivVla}. 
\end{itemize}
Fix now $|\xi|+|\beta| \leq N$ and assume first that $|\xi| \leq N-4$. Then, using first the pointwise decay estimates \eqref{eq:Mpoint} together with Lemma \ref{improderiv} and then Lemma \ref{gainv},
$$ \langle t_y \rangle \, \left|\nabla_{t,x} \mathcal{L}_{Z^\xi}F \right|(t_y,y)+ \left| \mathcal{L}_{Z^\xi}F \right|(t_y,y) \lesssim \sqrt{\Lambda} \, \langle t_y-|y| \rangle^{-2} \lesssim \sqrt{\Lambda} \, \langle  t_y +|y| \rangle^{-2} \, \langle y-t_y \widehat{v} \rangle^2 \, \langle v \rangle^4.$$
Otherwise we write $ |\mathcal{L}_{Z^\xi}F|(t_y,y)\lesssim \langle t_y+|y| \rangle|\nabla_{t,x}\mathcal{L}_{Z^\kappa}F|(t_y,y)$ for $Z^\xi=ZZ^\kappa$. Moreover, we have $|\beta| \leq 3$ so Corollary \ref{Corrho3} provides
$$ \int_{\R^3_v}\langle v \rangle^5 \, \langle y-t_y \widehat{v} \rangle \, \big|\widehat{Z}^{\beta}f \big|(t_y,y,v) \dr v \lesssim \sqrt{\epsilon  B_\Lambda} \, \frac{|\max (t_y -|y| ,1)|^{\frac{3}{2}} }{\langle t_y+|y| \rangle^{\frac{9}{2}}} \langle t_y \rangle^{\frac{1}{10}}.$$ 
Consequently, we have
\begin{align*}
 \big| G^{\, SS}_k (t,x) \big|  \lesssim & \sum_{|\beta| \leq N} \int_{|y-x|\leq \tau_0-t } \frac{\Lambda}{\langle t_y+|y| \rangle} \int_{\R^3_v}\langle v \rangle^9 \, \langle y-t_y \widehat{v} \rangle^3 \, \big|\widehat{Z}^{\beta}f \big|(t_y,y,v) \dr v \frac{ \mathrm{d}y \dr x}{\langle t_y \rangle^2 \, |y-x|} \\
 & +\sum_{|\xi| \leq N-1}\sqrt{\epsilon \Lambda B_\Lambda }  \int_{|y-x|\leq \tau_0-t } \langle t_y \rangle^{\frac{1}{10}} \frac{|\max (t_y -|y| ,1)|^{\frac{3}{2}} }{\langle t_y+|y| \rangle^{\frac{5}{2}}}\big| \nabla_{t,x} \mathcal{L}_{Z^\xi}F \big|(t_y,y) \frac{ \mathrm{d}y \dr x}{ \langle t_y \rangle^2|y-x|}.
\end{align*}
It remains to use Lemma \ref{Lemcdvtech} as well as the $L^2_x$ estimates given by Corollary \ref{Corrho3} and \eqref{eq:ML2}.
\end{proof}

\subsubsection{Improved estimates for the top order derivatives}\label{subsubsecmachin}

In order to conclude the proof of Proposition \ref{ProforFnew}, it remains to establish that, if $D$ is chosen large enough and if $\Lambda^2 B_\Lambda \epsilon \log^{-1} T$ is small enough, then
\begin{equation}\label{eq:toprovefortheendofsec}
\hspace{-9mm}  \forall \, |\gamma|=N-1, \qquad \qquad  \sup_{t \geq T} \, \mathcal{E}^{K,1} \big[ \nabla_{\partial_{x^\lambda}} \mathcal{L}_{Z^\gamma} \big( F^{\mathrm{new}}-F^{\mathrm{asymp}}[f_\infty] \big) \big](t) \leq \frac{D \Lambda}{8},
 \end{equation}
for any $0 \leq \lambda \leq 3$. We first deal with the case $\lambda \neq 0$ and we then fix $|\gamma|=N-1$ as well as $1 \leq k \leq 3$. Hence, Combining the Glassey-Strauss decomposition of $G^n$ given by Proposition \ref{GSdecomoderiv} with Propositions \ref{Proestivertex}--\ref{ProGSS}, we obtain
$$ \forall \, n \geq 4T, \; \forall \, t \geq T, \qquad   \mathcal{E}^{K,1} \big[ \nabla_{\partial_{x^k}} G^n \big](t) \lesssim \epsilon T^{-\frac{1}{4}} \Lambda^2 B_\Lambda \langle t \rangle^{-\frac{1}{4}}.$$
We claim there exists $G^{\gamma,k}$ such that
\begin{equation}\label{eq:estiGgammabis}
\hspace{-6mm} \nabla^\mu G_{\mu \nu}^{\gamma,k}= \partial_{x^k}J(\widehat{Z}^\gamma f)_\nu-\partial_{x^k}J^{\mathrm{asymp}}_\nu \big[\widehat{Z}^\gamma_\infty f_\infty \big], \qquad \nabla^\mu {}^* \! G_{\mu \nu}^{\gamma,k} =0; \qquad \quad \sup_{t \geq T} \mathcal{E}^{K,1} [G^{\gamma,k}](t) \lesssim \frac{\epsilon  \Lambda^2 B_\Lambda}{T^{\frac{1}{4}}} 
\end{equation}
and its radiation field along $\mathcal{I}^+$ is identically zero. This is obtained by applying Proposition \ref{ProscatMax} for $\underline{\alpha}^{\mathcal{I}^+}=0$, the source term $J:=\partial_{x^k}J(\widehat{Z}^\gamma f)_\nu-\partial_{x^k}J^{\mathrm{asymp}}_\nu[\widehat{Z}^\gamma_\infty f_\infty]$, $a=1$, $\delta =1/4$ and $B_{\mathrm{source}} \lesssim \epsilon T^{-1/4} \Lambda^2 B_\Lambda$. For this, note that $S(J,n)=\nabla_{\partial_{x^k}}G^n$, so that \eqref{eq:condiSJn} is verified.

Recall from Propositions \ref{ProGvacsolution}--\ref{ProMvacsolution} the solutions to the vacuum Maxwell equations $G^{\mathrm{vac}}$ and $M^{\mathrm{vac}}$. Consider first the case $\gamma_T=0$, so that
$$\nabla_{\partial_{x^k}}\mathcal{L}_{Z^\gamma} F^{\mathrm{new}}=G^{\gamma,k}+ \nabla_{\partial_{x^k}}\mathcal{L}_{Z^\gamma}G^{\mathrm{vac}}+\nabla_{\partial_{x^k}}\mathcal{L}_{Z^\gamma} F^{\mathrm{asymp}}[f_\infty]$$
and \eqref{eq:toprovefortheendofsec} holds for this multi-index $\gamma$ and $\lambda =k$ in view of Proposition \ref{ProGvacsolution}, Lemma \ref{improderiv} as well as \eqref{eq:estiGgammabis}. Otherwise $\gamma_T \geq 1$, so $J^{\mathrm{asymp}}[\widehat{Z}_\infty^\gamma f_\infty]=0$ and $\nabla_{\partial_{x^k}}\mathcal{L}_{Z^\gamma} F^{\mathrm{new}}=G^{\gamma,k}+ \nabla_{\partial_{x^k}}\mathcal{L}_{Z^\gamma}M^{\mathrm{vac}}$. Proposition \ref{ProMvacsolution}, combined with Lemma \ref{improderiv}, and \eqref{eq:estiGgammabis} then imply
\begin{equation}\label{kevatalenn:16}
 \forall \, t \geq T, \qquad \qquad  \mathcal{E}^{K,1} \big[ \nabla_{\partial_{x^k}} \mathcal{L}_{Z^\gamma}  F^{\mathrm{new}} \big](t) \lesssim \Lambda +\epsilon T^{-\frac{1}{4}} \Lambda^2 B_\Lambda. 
 \end{equation}
Finally, Lemma \ref{improderiv} and Corollary \ref{CorweightedcontrolFasymp} imply that $ \mathcal{E}^{K,1} \big[ \nabla_{\partial_{x^k}} \mathcal{L}_{Z^\gamma}  F^{\mathrm{asymp}}[f_\infty] \big](t) \lesssim \epsilon$ if $\gamma_T \geq 1$, so that \eqref{eq:toprovefortheendofsec} holds for $\lambda =k$ in this case as well.

It remains to treat the case $\lambda =0$. For this, note that if $H$ is solution to the Maxwell equations with source term $J$, that is $\nabla^\mu H_{\mu \nu}=J_\nu$ and $\nabla^\mu {}^* \! H_{\mu \nu}=0$, then
\begin{equation}\label{eq:timederivequation}
 \hspace{-6mm} \partial_t H_{0i}=\partial_{x^1}H_{1i}+\partial_{x^2}H_{2i}+\partial_{x^3}H_{3i}-J_i, \qquad \partial_t H_{ij}=\partial_{x^i} H_{0j}+\partial_{x^j} H_{i0}, \qquad \qquad 1 \leq i, \, j \leq 3.
 \end{equation}
We then estimate the time derivative by using the commutation formula of Proposition \ref{ProComMax} and by controlling the source term through Corollary \ref{Corrho} if $\gamma_T=0$ or Corollaries \ref{CorgainderivVla} and \ref{Corrho3} if $\gamma_T \geq 1$. 

\section{Proof of Theorem \ref{Theo1}}\label{SecCauchy}
The main part of this section will be devoted to the proof of the following result. For simplicity, given $C \in \R$, we will denote $\exp \big( \exp \big(C\sqrt{\Lambda} \,\big) \big)$ by $C_\Lambda$. Recall the constants $ \overline{\varepsilon}$ and $B_\Lambda$ given by Proposition \ref{Picard}.
\begin{Lem}\label{LemCauchyseq}
There exists $0<\varepsilon_0 \leq \overline{\varepsilon}$ and $C \geq B$, which depends only on $N$, such that the next statement holds. If $\epsilon \log^{-1}(T) C_\Lambda \leq \varepsilon_0$, then the sequence $(f_n,F^n)_{n \geq 1}$ is Cauchy in $L^\infty \big( [T,+ \infty[ , L^2(\R^3_x \times \R^3_v) \times L^2(\R^3_x) \big)$.
\end{Lem}

Before proving it, let us show that, together with the results of the previous sections, it implies Theorem \ref{Theo1} except the uniqueness statement. Indeed, we obtain from Lemma \ref{LemCauchyseq} that there exists a distribution function $f :[T,+ \infty[ \times \R^3_x \times \R^3_v \to \R$ and a $2$-form $F$, defined on $[T,+\infty[$, such that
$$ f_n \xrightarrow[n \to + \infty]{}  f \qquad \text{in $L^\infty \big( [T,+ \infty[ , L^2(\R^3_x \times \R^3_v)  \big)$}, \qquad \qquad F^n \xrightarrow[n \to + \infty]{} F \qquad  \text{in $ L^\infty \big( [T,+ \infty[ ,  L^2(\R^3_x) \big)$.}$$
We then deduce, in view of the definition \eqref{Cauchy1}--\eqref{Cauchy2} of $(f_n,F^n)_{n \geq 1}$,
$$ \T_F (f)=0 \qquad \text{in $\mathcal{D} \big(]T,+\infty[ \times \R^3_x \times \R^3_v \big) $}, \qquad \qquad \nabla^\mu F_{\mu \nu}= J(f)_\nu, \quad \nabla^\mu {}^* \! F_{\mu \nu}=0 \qquad \text{in $\mathcal{D} \big(]T,+\infty[ \times \R^3_x \big) $},$$
so that $(f,F)$ is a weak solution to the Vlasov-Maxwell system \eqref{VM1}--\eqref{VM3}. Since $(F^n)_{n \geq 1}$ is a bounded sequence in $ \mathbb{M}^{D,\Lambda}_N$, we have for any $|\gamma|=N-1$ and all $t \geq T$, 
$$  \mathcal{E}_{N-1}^K \big[F-F^{\mathrm{asymp}}[f_\infty] \big](t) + \mathcal{E}^{K,1} \big[ \nabla_{t,x} \mathcal{L}_{Z^\gamma} \big(F-F^{\mathrm{asymp}}[f_\infty] \big) \big](t)  \leq D \Lambda   .$$
Let $g(t,z,v):=f(t,z+t\widehat{v}+\C_{t,v},v)$. By Remark \ref{RqenergyforTh1}, if we choose $C \geq B'$, there holds, for all $t \geq T$,
\begin{align*}
  \sup_{|\beta| \leq N} \int_{\R^3_z} \int_{\R^3_v} \langle z \rangle^{16+2N-2\beta_H} \langle v \rangle^{30} \big| \widehat{Z}^\beta_\infty g (t,z,v) \big|^2 \dr v \dr z  &\leq  \epsilon \exp \big(  \exp \big( B' \sqrt{\Lambda} \, \big) \big)\leq C_\Lambda \epsilon .
\end{align*}
Consequently, $(f,F)$ is a classical solution to the Vlasov-Maxwell system. By interpolating we also have convergence in stronger topologies so that, by Proposition \ref{Proconvgntoginfty} and Remark \ref{Rqconvalphantoalphainf},
\begin{equation}\label{eq:rateofconvtoscat}
 \lim_{t \to +\infty} g(t,z,v) = f_\infty(z,v), \qquad \qquad \qquad  \lim_{r \to + \infty} r \underline{\alpha} (F)(t+u,r \omega) = \underline{\alpha}^\infty (u,\omega).
 \end{equation}
Hence, we have modified scattering for $f$, along the correction of the linear trajectories $t \mapsto (x+t\widehat{v}+\C_{t,v},v)$, to $f_\infty$. Moreover, $\underline{\alpha}^\infty$ is the radiation field along future null infinity $\mathcal{I}^+$ of the electromagnetic field $F$.
\begin{Rq}
Since $F \in \mathbb{M}^{D,\Lambda}_N$ and $(f,F)$ is a solution to the Vlasov-Maxwell system with asymptotic data $(f_\infty,\underline{\alpha}^\infty)$, all the results of Sections \ref{SecVlasov}--\ref{SecMax} apply to $(f,F)$. In particular, Proposition \ref{Proconvgntoginfty} and Remark \ref{Rqconvalphantoalphainf} provide a uniform rate of convergence for \eqref{eq:rateofconvtoscat}.
\end{Rq}
\subsection{The approximate solutions form a Cauchy sequence} Since the arguments are very similar to part of those used during the last two Sections \ref{SecVlasov}--\ref{SecMax}, we will be more sketchy.  Consider, for $n \geq 1$,
$$ \delta_{\mathrm{Vlasov}}^n :=  \mathbf{E}^{7,13}_1[g_{n+1}-g_n ], \qquad \qquad  \delta_{\mathrm{Maxwell}}^n := \sup_{t \geq T} \mathcal{E}^K \big[F^{n+1}-F^{n} \big](t) +\sup_{t \geq T} \mathcal{E}^{K,1}\big[\nabla_{t,x} \big( F^{n+1}-F^{n} \big) \big](t)   $$
and let us control first $g_{n+1}-g_n$.
\begin{Pro}\label{ProdiffVla}
For all $n \geq 2$, we have $\delta_{\mathrm{Vlasov}}^n \lesssim \epsilon  T^{-\frac{1}{4}}C_\Lambda   \delta_{\mathrm{Maxwell}}^{n-1}$, where $C >0$ depends only on $N$, and
$$
 \forall \, t \geq T, \qquad \bigg\| |v^0|^8 \int_{\R^3_z} \big[ g_{n+1}-g_n \big](t,z,\cdot) \dr z \bigg\|_{L^2(\R^3_v)}  \lesssim  \frac{\sqrt{ \epsilon \Lambda C_\Lambda} \log^4(t)}{T^{\frac{1}{8}} t} \sqrt{\delta_{\mathrm{Maxwell}}^{n-1}} .
$$
\end{Pro}
\begin{proof}
Let $n \geq 2$. Following the proof of the estimate \eqref{eq:ML2conv}, one can obtain
\begin{equation}\label{eq:ML2convCauchy}
\hspace{-8mm} \forall \, t \geq T, \qquad \quad  \sup_{|\gamma| \leq 1} \int_{\R^3_x} \int_{\R^3_v}  t^4 \left|  \mathcal{L}_{Z^\gamma} (F^n-F^{n-1})  \right|^2(t,x) \frac{\dr v \mathrm{d} x}{\langle x-t\widehat{v} -\C_{t,v} \rangle^{7} \langle v \rangle^{13}}  \lesssim   \frac{\log^9(t)}{t}  \delta_{\mathrm{Maxwell}}^{n-1}.
\end{equation}
Let $|\beta| \leq 1$. Then, since $\T_{F^{n}}^\infty(g_{n+1})=\T^\infty_{F^{n-1}}(g_n)=0$, we have
\begin{align*}
 \T_{F^{n}}^\infty \Big[ \langle v \rangle^{13} \, \langle z \rangle^{8-\beta_H} \,\widehat{Z}^\beta_\infty &\big(g_{n+1}-g_n \big) \Big] = \T_{F^{n}}^\infty \big( \langle v \rangle^{13} \, \langle z \rangle^{8-\beta_H} \big) \widehat{Z}^\beta_\infty\big( g_{n+1}-g_n \big) \\ 
 &+\langle v \rangle^{13} \, \langle z \rangle^{8-\beta_H} \,\big[\T^\infty_{F^{n}},\widehat{Z}^\beta_\infty \big]\big(g_{n+1}-g_n \big)   +\langle v \rangle^{13} \, \langle z \rangle^{8-\beta_H} \widehat{Z}^\beta \big( \T_{F^{n}}^\infty (g_n)-\T^\infty_{F^{n-1}}(g_n) \big),
 \end{align*}
 where $\big[\T^\infty_{F^{n}},\widehat{Z}^\beta_\infty \big]=0$ if $|\beta|=0$. According to Corollary \ref{CorCom} and Proposition \ref{Proestloworder},
 \begin{align*}
 & \big|\T_{F^{n}} \big( \langle v \rangle^{13} \, \langle z \rangle^{8-\beta_H} \big) \widehat{Z}^\beta_\infty\big( g_{n+1}-g_n \big)+\langle v \rangle^{13} \, \langle z \rangle^{8-\beta_H} \,\big[\T^\infty_{F^{n}},\widehat{Z}^\beta_\infty \big]\big(g_{n+1}-g_n \big)\big|(t,z,v) \\
 & \qquad \qquad \qquad \qquad \qquad \lesssim \! \Big( \sqrt{ \Lambda}   t^{-\frac{4}{3}}+\sqrt{\Lambda} \,  \widehat{v}^{\underline{L}}(\XX_\C)  \, \langle t-|\XX_\C| \rangle^{-\frac{4}{3}} \Big) \sup_{|\kappa| \leq 1} \langle v \rangle^{13} \, \langle z \rangle^{8-\kappa_H} \big|  \widehat{Z}^\kappa_\infty \big(g_{n+1}-g_n \big) \big|(t,z,v) .
 \end{align*} 
 Moreover, we have
 $$\T_{F^{n}}^\infty -\T^\infty_{F^{n-1}}  =  \frac{\delta_j^i-\widehat{v}_j \widehat{v}^i}{tv^0} \widehat{v}^\mu t^2  {\big[F^{n-1}-F^{n}\big]_{\mu}}^{j}(t,\XX_\C)\partial_{z^i}+\frac{\widehat{v}^\mu}{v^0} {\big[F^{n}-F^{n-1}\big]_{\mu}}^{j}(t,\XX_\C)\left( v^0\partial_{v^j}-v^0 \partial_{v^j} \C^i_{t,v} \partial_{z^i} \right) .$$
Now, applying Lemmata \ref{LemCom1}--\ref{LemCom2} to $F=F^{n}-F^{n-1}$ and $f_\infty=0$, one can derive\footnote{Alternatively, one can apply Lemma \ref{LemCom2} to $F^n$ as well as $F^{n-1}$ and consider the difference of the resulting quantities.}
\begin{align*}
\big| \langle v \rangle^{13} \, \langle z \rangle^{8-\beta_H} \widehat{Z}^\beta \big( \T_{F^{n}}^\infty (g_n)-&\T^\infty_{F^{n-1}}(g_n) \big) \big|\lesssim \frac{1}{v^0  t} \sup_{|\xi|,  |\kappa| \leq 1} t^2\big|  \mathcal{L}_{Z^\xi}  \big( F^{n}-F^{n-1} \big)\big|(t,\XX_\C)  \, \big| \, \langle v \rangle^{13} \langle z \rangle^{9-\kappa_H}  \widehat{Z}_\infty^\kappa g_n \big| ,
\end{align*}
where we used $|v^0 \partial_{v^j} \C_{t,v}^i| \lesssim \log^2(t) \lesssim t$. Next, as $H^2_{z,v} \hookrightarrow L^\infty_{z,v}$, we have $\big| \, \langle v \rangle^{19} \langle z \rangle^{13-\kappa_H}  \widehat{Z}_\infty^\kappa g_n \big| \lesssim \sqrt{\epsilon B_\Lambda}$ in view of the bound \eqref{eq:gBound}. Consequently, using \eqref{eq:ML2convCauchy} and performing $L^2_{z,v}$ estimates as in Proposition \ref{Proapriori}, we get, since $\widehat{Z}^\beta_\infty (g_{n+1}-g_n)(t,\cdot , \cdot) \to 0$ as $t \to +\infty$ by construction,
$$\delta^n_{\mathrm{Vlasov}} \leq  \sup_{t \geq T} \, \sup_{u \in \R} \mathbb{E}^{7,13}_1[g_{n+1}-g_n ](t,u) \lesssim \frac{\epsilon B_\Lambda}{T^{\frac{1}{4}}} C'_\Lambda \delta_{\mathrm{Maxwell}}^{n-1} =\frac{\epsilon C_\Lambda}{T^{\frac{1}{4}}} \delta_{\mathrm{Maxwell}}^{n-1}, \qquad C',\, C >0.$$
For the $L^2_v$ decay estimate of the spatial average of $g_{n+1}-g_n$, note first that it goes to $0$ as $t \to +\infty$. Then, we use the relation $\T_{F^n}(g_{n+1}-g_n)=(\T_{F^{n-1}}-\T^\infty_{F^n})(g_n)$ and integration by parts in $z$, as the ones performed in Corollary \ref{Corestipartialtgg}. Using also $| \, \langle v \rangle^{16} \langle z \rangle^7  \widehat{Z}_\infty^\kappa g_n | \lesssim  \sqrt{\epsilon B_\Lambda}$ given by \eqref{eq:gBound}, \eqref{eq:Mpoint} as well as Lemmata \ref{improderiv} and \ref{gainvadapted}, we get 
\begin{align*}
 \bigg|\partial_t \int_{\R^3_z} \big( g_{n+1}-g_n \big)(t,z,v) \dr z \bigg| &\lesssim \frac{\sqrt{\Lambda} \log^4(t)}{t^2} \sup_{|\kappa| \leq 1}\int_{\R^3_z}  \langle z \rangle^2 \, \langle v \rangle^3 \, \big|\widehat{Z}^\kappa_\infty \big(g_{n+1}-g_n \big)\big|(t,z,v) \dr z \\
 & \quad +\frac{\sqrt{\epsilon B_\Lambda} \log^2(t)}{ t^2} \sup_{|\xi| \leq 1} \int_{\R^3_z}   t^2 \big|\mathcal{L}_{Z^\xi} \big(F^{n}-F^{n-1} \big)\big|(t,\XX_\C) \frac{\dr z}{\langle z \rangle^6\langle v \rangle^{15}}.
 \end{align*}
We then obtain the convergence estimate by following the proof of Proposition \ref{Proconvspat} and by using \eqref{eq:ML2convCauchy} as well as the estimate obtained for $\delta_{\mathrm{Vlasov}}^{n}$.
\end{proof}

We control now $F^{n+1}-F^n$.
\begin{Pro}\label{Prodiffelec}
For all $n \geq 2$ and $t \geq T$, we have
\begin{align*}
 \delta_{\mathrm{Maxwell}}^{n} \lesssim  \epsilon T^{-\frac{1}{4}}\Lambda^2C_\Lambda \delta_{\mathrm{Maxwell}}^{n-1}.
 \end{align*}
\end{Pro}
\begin{proof}
Let $\mathbf{q}:\mathbb{S}^2_\omega \times \R^3_v \to \R$ such that $|\mathbf{q}|(\cdot,v)+|\nabla_v \mathbf{q}|(\cdot,v) \lesssim |v^0|^3$. We have from Proposition \ref{Protechfordecay}, applied to $ \mathbf{q}(\omega ,v) (f_{n+1}-f_n)$, and the previous Proposition \ref{ProdiffVla},
$$
\forall \, t \geq T , \qquad \; \sup_{\omega \in \mathbb{S}^2_\omega} \, \int_{\R^3_x} \langle t+|x| \rangle^{5} \bigg| \int_{\R^3_v} \mathbf{q}(\omega ,v ) \big( f_{n+1}-f_n \big)(t,x,v) \dr v\bigg|^2 \dr x   \lesssim \epsilon T^{-\frac{1}{4}} \Lambda C_\Lambda  \log^{8}(t) \delta_{\mathrm{Maxwell}}^{n-1}.
$$
Similarly, according to Proposition \ref{Prosimpledecay}, we have for any $|\beta| \leq 1$,
$$ \forall \, t \geq T, \qquad  \int_{\R^3_x} \langle t+|x| \rangle^3 \bigg|\int_{\R^3_v}\langle x-t\widehat{v} \rangle^3 \, \langle v \rangle^{9} \big| \widehat{Z}^\beta (f_{n+1}-f_n) \big|(t,x,v) \mathrm{d} v  \bigg|^2  \dr x   \lesssim \epsilon T^{-\frac{1}{4}} C_\Lambda  \log^{12}(t) \delta_{\mathrm{Maxwell}}^{n-1}.$$
We refer to the proof of Corollary \ref{Corrho2} for more details. Consider now, for $\tau_0 \geq 4T$,
\begin{alignat*}{2}
 \nabla^\mu G^{\tau_0}_{\mu \nu} &= \chi \big(t / \tau_0 \big)J\big(  f_{n+1}-f_n \big)_\nu, \qquad &&\nabla^\mu {}^* \! G_{\mu \nu}^{\tau_0}=0, \qquad G(\tau_0,\cdot)={}^* \! G (\tau_0,\cdot)=0. 
 \end{alignat*}
Bounding $\log^{8}(t)$ by $t^{1/4}$, we obtain from Proposition \ref{CorassumpJforcondi} that
 $$ \forall \, t \geq T, \qquad \mathcal{E}^K\big[ G^{\tau_0} \big](t) \lesssim  \epsilon T^{-\frac{1}{4}} \Lambda C_\Lambda  \delta^{n-1}_{\mathrm{Maxwell}} \, \langle t \rangle^{-\frac{3}{4}}.$$
Hence, since the radiation field of $F^{n+1}-F^n$ vanishes, we have $\mathcal{E}^K [ F^{n+1}-F^n ] \lesssim \epsilon T^{-1/4} \Lambda C_\Lambda \delta^{n-1}_{\mathrm{Maxwell}}$ on $[T,+\infty[$ according to Proposition \ref{ProscatMax}. Following the analysis carried out in Section \ref{toporder} and using the pointwise decay estimates given by \eqref{eq:Mpoint}, for $F^{n+1}$ and $F^n$, as well as Corollary \ref{Corrho3}, for $f_{n+1}$ and $f_n$, we get
\begin{align*}
& \sup_{t \geq T} \, \langle t \rangle^{\frac{1}{4}} \, \mathcal{E}^{K,1} \big[ \nabla_{\partial_{x^k}}G^{\tau_0} \big](t) \lesssim \epsilon T^{-\frac{1}{4}} \Lambda^2 C_\Lambda \delta_{\mathrm{Maxwell}}^{n-1},
\end{align*}
for any $1 \leq k \leq 3$. We then obtain from Proposition \ref{ProscatMax} the stated estimates for the spatial derivatives of $F^{n+1}-F^n$. The time derivative can be handled by exploiting \eqref{eq:timederivequation}.
\end{proof}
Hence, if $\epsilon T^{-1/4}\Lambda^2 C_\Lambda$ is small enough, there exists $0 \leq \eta <1$ such that $\delta^n_{\mathrm{Maxwell}}+\delta^n_{\mathrm{Vlasov}} \leq \eta^n \delta^1_{\mathrm{Maxwell}}$ for all $n \geq 2$. By considering a slightly larger constant $C$, this implies Lemma \ref{LemCauchyseq}. 

\subsection{Uniqueness of the solution} To conclude the proof of Theorem \ref{Theo1}, it remains us to prove that $(f,F)$ is the only solution to the Vlasov-Maxwell system with scattering data $(f_\infty,\underline{\alpha}^\infty)$ such that
\begin{equation}\label{Un:uniq}
 \mathbf{E}_8^{7,13}[g]+ \sup_{t \geq T} \mathcal{E}^K_7 \big[F-F^{\mathrm{asymp}}[f_\infty] \big](t) +\sup_{|\gamma|=7} \, \sup_{t \geq T} \mathcal{E}^{K,1}\big[\nabla_{t,x} \mathcal{L}_{Z^\gamma}\big(F-F^{\mathrm{asymp}}[f_\infty] \big) \big](t)  <+\infty.  \tag{Uniq}
 \end{equation}
Let $(h,H)$ be a solution to the same asymptotic Cauchy problem and verifying the bound \eqref{Un:uniq}. Since local well-posedness hold, $(h,H)=(f,F)$ is a consequence of the next result.

\begin{Lem}
There exists $T' \geq T$ such that $(h,H)=(f,F)$ on $[T',+\infty[$.
\end{Lem}
\begin{proof}
Let $\mathbf{h}(t,z,v):=h(t,z+t\widehat{v}+\C_{t,v},v)$. By assumption, there exists $\Lambda' \geq 1$ such that
\begin{align*}
\mathbf{E}_8^{7,13}[\mathbf{h}]+ \sup_{t > T } \mathcal{E}^K \big[ H-F^{\mathrm{asymp}}[f_\infty] \big](t)+\sup_{|\gamma|=7} \, \sup_{t \geq T} \mathcal{E}^{K,1}\big[\nabla_{t,x} \mathcal{L}_{Z^\gamma}\big(H-F^{\mathrm{asymp}}[f_\infty] \big) \big](t) &\leq D \Lambda',
 \end{align*}
that is $F \in \mathbb{M}^{D,\Lambda'}_8$. We can assume, without loss of generality, that $\Lambda \leq \Lambda'$. Consider now $T' \geq T$ such that $\epsilon \log^{-1}(T')C_{\Lambda'} \leq \varepsilon_0$, where the constants $C_{\Lambda'}$ and $\varepsilon_0$ are given by Lemma \ref{LemCauchyseq}. Applying Proposition \ref{ProdiffVla} to, with a slight abuse of notations,
$$ (f_{n+1},F^{n})=(h,H), \qquad (f_n,F^{n-1})=(f,F), \qquad T=T', \qquad \Lambda=\Lambda',$$
and Proposition \ref{Prodiffelec} to $(f_{n+1},F^{n+1})=(h,H)$ and $(f_n,F^{n})=(f,F)$, we get $(h,H)=(f,F)$ on $[T',+\infty[$. 
\end{proof}

\section{Improved velocity decay assumptions on $f_\infty$ when the regularity order $N$ is large}\label{Secweakerassump}

We explain here how to get rid of the extra hypothesis \eqref{eq:extracondi}, which is relevant as soon as $N \geq 9$ but only significant for $N \gg 1$. This assumption \eqref{eq:extracondi} ensures that all the derivatives of $F^{\mathrm{asymp}}[f_\infty]$ are well-defined up to order $N$. 
\subsection{Commutation properties of $\mathbb{K}$ with $F^{\mathrm{asymp}}[\cdot]$}
The idea, if $\mathcal{L}_{Z^\gamma}F^{\mathrm{asymp}}[f_\infty]$ cannot be controlled sufficiently well, consists in considering $F^{\mathrm{asymp}}[\widehat{Z}^\gamma_\infty f_\infty]$, which requires less decay in $v$ on $f_\infty$ in order to be well-defined. To check it, compare $\overline{\mathbb{E}}_{|\gamma|+1}[f_\infty]$ with $\overline{\mathbb{E}}_1[\widehat{Z}^\gamma_\infty f_\infty]$. We then estimate the difference of this two fields.
\begin{Pro}\label{ProcommuFasympZ}
Let $1 \leq |\gamma| \leq N$ such that $\gamma_T=0$. If $\overline{\mathbb{E}}_{N+1}[f_\infty] < +\infty$, we have
$$ \forall \, (t,x) \in [t_0,+\infty[ \times \R^3, \qquad \big|\mathcal{L}_{Z^\gamma}\big( F^{\mathrm{asymp}} \big[  f_\infty \big] \big)- F^{\mathrm{asymp}} \big[ \widehat{Z}^\gamma_\infty f_\infty \big]-\mathcal{L}_{Z^\gamma} ( \, \overline{F} \,) \big|(t,x) \lesssim \frac{\overline{E}_{N+1}[f_\infty]}{\langle t+|x| \rangle \, \langle t-|x| \rangle^2}.$$
\end{Pro}
\begin{proof}
This estimate is related to the asymptotic expansion of their initial data. More precisely, we have
$$ \forall \, 1 \leq |\gamma| \leq N+1, \qquad \qquad  Q^\gamma:=\int_{\R^3_z} \int_{\R^3_v} \widehat{Z}^\gamma_\infty f_\infty (z,v) \dr z \dr v =0.$$
These identities are obtained by integration by parts, for $|\gamma|=1$, and a direct induction for the general case. Thus, the stated estimate ensues from the following observations.
\begin{itemize}
\item The inhomogeneous part, according to Definition \ref{Defdecompasymp}, of $F^{\mathrm{asymp}}[\widehat{Z}^\gamma_\infty f_\infty]$ and $\mathcal{L}_{Z^\gamma}F^{\mathrm{asymp}} [  f_\infty ]$ are equal.
\item Up to terms decaying as $\overline{E}_{N_0+1}[f_\infty] \, \langle t+r \rangle^{-1} \, \langle t-r \rangle^{-2}$, the difference of their homogeneous part is equal to $\mathcal{L}_{Z^\gamma}\overline{F}$. This is a consequence of the analysis carried out in Section \ref{subsubsechompar} (see in particular \eqref{eq:estihomFtilde}).
\end{itemize}
\end{proof}

Thus, although $ F^{\mathrm{asymp}} \big[ \widehat{Z}^\gamma_\infty f_\infty \big]$ captures the asymptotic behavior of $\mathcal{L}_{Z^\gamma} F^{\mathrm{asymp}} [  f_\infty ]$ in the interior of the light cone $|x|<t$, and then along timelike trajectories, it fails to do so in the exterior region. Fortunately, this failure can be quantified by an electromagnetic field that we know well, $\mathcal{L}_{Z^\gamma} \overline{F}$, which is a derivative of the pure charge part of $F^{\mathrm{asymp}}[f_\infty]$.

\subsection{Alternative functional space for the electromagnetic fields}

Motivated by the previous discussion, we introduce the following sets of $2$-forms.
\begin{Def}
Let $D \geq 1$ and $\widetilde{\mathbb{M}}_N^{D,\Lambda}$ be the set of the $2$-forms $F$ defined on $[T,+\infty[ \times \R^3$ and verifying
\begin{alignat}{2}
 \sup_{t \geq T} \mathcal{E}^K_2 \big[ F-F^{\mathrm{asymp}}[f_\infty] \big](t) & \leq D \Lambda ,  \label{kevat:condi1} \\
\sup_{t \geq T} \mathcal{E}^K_2 \big[ \mathcal{L}_{Z^\gamma} \big( F-\overline{F} \big)-F^{\mathrm{asymp}}[ Z_\infty^\gamma f_\infty ]  \big](t) & \leq D \Lambda , \qquad \qquad 1 \leq |\gamma| \leq N-3 \label{kevat:condi2} \\
\hspace{-16mm} \sup_{t \geq T} \mathcal{E}^{K,1} \big[ \nabla_{t,x} \mathcal{L}_{Z^\gamma} \big( F-\overline{F} \big)-\nabla_{t,x}F^{\mathrm{asymp}}[ Z_\infty^\gamma f_\infty ] \big](t) & \leq D \Lambda , \qquad \qquad |\gamma|=N-1. \label{kevat:condi3}
 \end{alignat}
\end{Def}

Then the following remarks are importants.
\begin{itemize}
\item Controlling two derivatives of $ \mathcal{L}_{Z^\gamma} ( F-\overline{F} \,)-F^{\mathrm{asymp}}[ Z_\infty^\gamma f_\infty ] $ is important in order to derive pointwise decay estimates.
\item The assumptions of Theorem \ref{Theo1} satisfied by $f_\infty$ imply that $\overline{\mathbb{E}}_3[Z_\infty^\gamma f_\infty] \lesssim \epsilon$ for all $|\gamma| \leq N-3$, so that the results of Section \ref{SecMaxasymp} can be applied to the distribution function $Z_\infty^\gamma f_\infty$ and for $N_0=2$. 
\item Similarly, $\overline{\mathbb{E}}_2[Z_\infty^\xi f_\infty] \leq \epsilon$ for all $|\xi| \leq N-1$. In particular, we can apply Proposition \ref{Proasymp} as well as the first part of Proposition \ref{Proradasymp} to $Z_\infty^\xi f_\infty$ and for $N_0=1$.
\item For all $F \in \widetilde{\mathbb{M}}_N^{D,\Lambda}$, the five estimates \eqref{eq:Mpoint}--\eqref{eq:ML2conv} hold. This can be obtained by using the estimates satisfied by $\overline{F}$ and its derivatives, which are given by Proposition \ref{Propurecharge}, and by following the proof of Proposition \ref{PropropMaxfield}. 
\end{itemize}

\subsection{Modifications in the proof of Theorem \ref{Theo1}}

The extra assumption \eqref{eq:extracondi} was not directly exploited in Section \ref{SecVlasov}, where the analysis of the Vlasov equation were carried out. Similarly, since we proved that $(f_n,F^n)_{n \geq 1}$ is a Cauchy sequence in a space of low regularity, we did not use \eqref{eq:extracondi} neither\footnote{More precisely, we only used it through the estimates \eqref{eq:Mpoint}--\eqref{eq:ML2conv}, which hold as well in this setting.} in Section \ref{SecCauchy}. In addition to Proposition \ref{PropropMaxfield}, already mentioned, the analysis of Section \ref{SecMax} has to be modified as well. More precisely, we need to prove that $(F^n)_{n \geq 1}$ is a bounded sequence in $\widetilde{\mathbb{M}}_N^{D,\Lambda}$. For this, we will use the next result which ensues from Proposition \ref{Prorad}.
\begin{Lem}\label{LemradFbar}
For any multi-index $\gamma$, the radiation field of $\mathcal{L}_{Z^\gamma} \overline{F}$ along future null infinity vanishes identically. Moreover, for any $0 \leq \mu, \, \nu \leq 3$, $r^2 \mathcal{L}_{Z^\gamma} ( \,\overline{F} \,)_{\mu \nu}(r+u,r\omega)$ converges as $r \to +\infty$ in $L^\infty_{\mathrm{loc}}(\R_u \times \mathbb{S}^2_\omega)$.
\end{Lem}

By keeping the notations of Section \ref{SecMax}, we have to prove $F^{\mathrm{new}} \in \widetilde{\mathbb{M}}_N^{D,\Lambda}$ for $D$ chosen large enough and if $\epsilon \log^{-1}T$ small enough, which is a stronger statement than $F^{\mathrm{new}} \in \mathbb{M}_N^{D,\Lambda}$. Note now that the condition \eqref{kevat:condi1} holds by the results of Section \ref{Subseclowerorder}. Note that Proposition \ref{Propurecharge} and Lemma \ref{improderiv} imply, for any multi-index $\kappa$ such that $\kappa_T \geq 1$,
$$ \forall \, t \in \R_+ , \qquad \qquad \mathcal{E}^{K,1} \big[ \mathcal{L}_{Z^\kappa} \overline{F}  \big]  (t) \lesssim \epsilon.$$
Since $F^{\mathrm{asymp}}\big[ \widehat{Z}^\kappa_\infty f_\infty \big]=0$ for such multi-indices $\kappa$, we obtain from \eqref{kevatalenn:15} and \eqref{kevatalenn:16} that \eqref{kevat:condi2}--\eqref{kevat:condi3} hold for any $\gamma$ verifying $\gamma_T \geq 1$. Fix now $1 \leq |\gamma| \leq N-3$ as well as $|\xi|=N-1$ such that $\gamma_T=\xi_T=0$ and let us prove that \eqref{kevat:condi2} as well as \eqref{kevat:condi3} hold for them. Recall the fields $G^{\gamma}$ and $G^{\xi,k}$, which were introduced in \eqref{eq:estiGgamma} and \eqref{eq:estiGgammabis}. Compared with the analysis performed in Section \ref{SecMax}, we will need to consider slightly different decompositions of the derivatives of $F^{\mathrm{new}}$. More precisely, let us prove that
\begin{enumerate}
\item for any $|\kappa| \leq 2$, there exists an electromagnetic fields $\overline{G}^{\kappa,\gamma}$ such that
$$ \mathcal{L}_{Z^\kappa Z^\gamma}F^{\mathrm{new}}\! = \mathcal{L}_{Z^\kappa}\big( F^{\mathrm{asym}} \big[ \widehat{Z}^\gamma_\infty f_\infty]+\mathcal{L}_{Z^\gamma}  \overline{F} \, \big)+\mathcal{L}_{Z^\kappa } G^{\gamma,\mathrm{vac}}+G^{\kappa\gamma} +\overline{G}^{\kappa,\gamma}, \qquad  \sup_{t \geq T} \mathcal{E}^K \big[\, \overline{G}^{\kappa,\gamma} \big](t) \lesssim  \epsilon T^{-\frac{1}{2}} ,$$
where, for any $|\beta| \leq N-1$ and $1 \leq k \leq 3$, $G^{\beta,\mathrm{vac}}$ is the unique solution to the vacuum Maxwell equations with radiation field $\underline{\alpha}^{\infty}_\beta-\underline{\alpha}^{\mathrm{asymp}}_\beta \big[ \widehat{Z}^\beta_\infty f_\infty]$. Moreover, $G^{\kappa \gamma}=G^{\zeta}$, where $Z^\kappa Z^\gamma = Z^\zeta$.
\item For any $|\xi|=N-1$, there exists $\overline{G}^{\xi,k}$ such that
$$\nabla_{\partial_{x^k}}\mathcal{L}_{Z^\xi} F^{\mathrm{new}} \! =\nabla_{\partial_{x^k}}  \big(F^{\mathrm{asymp}}\big[ \widehat{Z}^\xi_\infty f_\infty \big]+\mathcal{L}_{Z^\xi}  \overline{F} \, \big)+\nabla_{\partial_{x^k}}G^{\xi,\mathrm{vac}}+G^{\xi,k}+\overline{G}^{\xi,k}\!, \quad \; \sup_{t \geq T} \mathcal{E}^K \big[ \, \overline{G}^{\xi,k} \big](t) \lesssim \epsilon T^{-\frac{1}{2}} .$$
\end{enumerate}
Note then that according to Theorem \ref{Thscat} and Lemma \ref{improderiv},
$$ \forall \, t \geq T, \qquad \mathcal{E}^K \big[\mathcal{L}_{Z^\kappa } G^{\gamma,\mathrm{vac}} \big](t) +\mathcal{E}^{K,1} \big[\nabla_{\partial_{x^k}} G^{\xi,\mathrm{vac}} \big](t) \lesssim \Lambda, \qquad \qquad |\kappa| \leq 2.$$
If $1.$ and $2.$ hold true, \eqref{kevat:condi2}--\eqref{kevat:condi3} would be verified in view of the estimates derived for $G^{\gamma}$ and $G^{\xi,k}$ in Section \ref{SecMax}. For simplicity, we denote by $\overline{J}^\beta$ the source term of Maxwell equations verified by $\mathcal{L}_{Z^\beta} \overline{F}$. Let us justify that, for any multi-index $\beta$, we can define $\overline{G}^{\beta}$ as the unique finite energy electromagnetic field, with identically zero radiation fields, such that
\begin{equation*}
 \nabla^\mu \overline{G}_{\mu \nu}^{\beta}= -\overline{J}^{\, \beta}_\nu, \qquad \qquad \nabla^\mu {}^* \! \overline{G}_{\mu \nu}^{\beta} =0 .
\end{equation*}
Moreover, it turns out that $\mathcal{E}^K \big[\overline{G}^\beta \big] \lesssim \epsilon T^{-1/2}$ on $\R_+$. For this, we use Proposition \ref{ProscatMax}, applied for $\underline{\alpha}^{\mathcal{I}^+}=0$, $\delta=1/2$ and $B_{\mathrm{source}}=\epsilon T^{-1/2}$, together with the next result.
\begin{Lem}
For any $|\kappa| \leq N$, $a \in \R_+$ and all $n \in \mathbb{N}^*$, we have, uniformly in $a$,
$$ \forall \, t \in [0,n], \qquad \mathcal{E}^{K,a} \big[ S( \, \overline{J}^\kappa\!,n) \big](t) \lesssim \epsilon \, \langle t \rangle^{-1},$$
where $S( \, \overline{J}^\kappa\!,n)$ is introduced in Definition \ref{DefSJn}.
\end{Lem}
\begin{proof}
Let $n \geq 1$, $|\kappa| \leq N$, $H:= S(\overline{J}^\kappa,n)$ and remark that we only need to consider the case $a=0$. Indeed, $H$ is supported in the exterior of the light cone by finite speed of propagation. Note then, according to Proposition \ref{ProenergyforscatMax} and since $H(n,\cdot)=0$,
$$\forall \, t \geq 0, \qquad \qquad \mathcal{E}^K [H](t)+\sup_{u \in \R} \mathcal{F}^K [H]_t^n(u) \leq 8 \int_{\tau=t}^n \int_{\R^3_x}  \big| K^{\mu} H_{\mu \nu} \,  \overline{J}^{\kappa,\nu}\big| (\tau,x) \dr x \dr \tau . $$
By expanding the integrand according to the null frame $(\underline{L},L,e_\theta,e_\varphi)$, we obtain, exactly as in \eqref{eq:expanullframeeee},
\begin{align*}
2\big| K^{\mu} H_{\mu \nu} \, \overline{J}^{\kappa,\nu}\big| & \leq \langle t-r \rangle^{2} |\rho (H)| \big|\overline{J}^{\kappa}_{\underline{L}} \big|+\langle t+r \rangle^{2} |\rho (H)| \big|\overline{J}^{\kappa}_{L} \big|+\langle t-r \rangle^{2} |\underline{\alpha} (H)| \big| \slashed{\overline{J}}^{\kappa} \big|+\langle t+r \rangle^{2} |\alpha (H)| \big| \slashed{\overline{J}}^{\kappa} \big|,
\end{align*}
where $ \slashed{\overline{J}}^{\kappa} $ is the angular part of $\overline{J}^{\kappa}$. According to Lemma \ref{LemComMaxwell}, $\overline{J}^{\kappa}$ is a linear combination of $\mathcal{L}_{Z^\beta}\overline{J}$. Thus, by Proposition \ref{Propurecharge},
\begin{align*}
\big| K^{\mu} H_{\mu \nu} \, \overline{J}^{\kappa,\nu}\big| & \lesssim  \sqrt{\epsilon}\, \langle t+r \rangle^{-2} |\rho (H)| \mathds{1}_{1 \leq r-t \leq 2} +\sqrt{\epsilon} \, \langle t+r \rangle^{-3} |\underline{\alpha} (H)|   \mathds{1}_{1 \leq r-t \leq 2} +\sqrt{\epsilon} \, \langle t+r \rangle^{-1} |\alpha (H)|   \mathds{1}_{1 \leq r-t \leq 2}.
\end{align*}
Now, by the Cauchy-Schwarz inequality in $x$, we have
\begin{align*}
 \int_{\tau=t}^n \int_{1 \leq |x| -t \leq 2}  \frac{|\rho (H)|(\tau,x)}{\langle \tau + |x| \rangle^2}   +\frac{|\underline{\alpha} (H)|(\tau,x)}{\langle \tau + |x| \rangle^3} \dr x \dr \tau & \lesssim  \sup_{[t,n]} \big| \mathcal{E}^K [H] \big|^{\frac{1}{2}} \int_{\tau=t}^n \bigg| \int_{| r-t| \leq 2} \frac{r^2 \,\dr r}{\langle t+r \rangle^6} \bigg|^{\frac{1}{2}} \dr \tau \\
 & \lesssim  \langle t \rangle^{-\frac{1}{2}} \sup_{[t,n]} \big| \mathcal{E}^K [H] \big|^{\frac{1}{2}} .
 \end{align*}
For the other term, we perform the change of variables $(u,\underline{u},\omega)=(\tau-|x|,\tau+|x|,x/|x|)$ and we apply the Cauchy-Schwarz inequality in $(\underline{u},\omega)$. We obtain
\begin{align*}
\int_{\tau=t}^n \int_{1 \leq |x| -t \leq 2} \frac{|\alpha (H)|(\tau,x)}{\langle \tau + |x| \rangle} \dr x \dr \tau & \leq  \int_{u=-1}^{-2}  \int_{\underline{u}=2t-u}^{2n-u} \int_{\mathbb{S}^2_\omega}  \frac{|\alpha (H)|(u,\underline{u},\omega)}{\langle\underline{u} \rangle}  \, r^2 \dr \mu_{\mathbb{S}^2_\omega} \dr \underline{u} \dr u \\
& \lesssim \sup_{ \R} \big| \mathcal{F}^K [H]_t^n \big|^{\frac{1}{2}} \int_{u=-1}^{-2} \bigg| \int_{\underline{u}=2t-u}^{2n-u} \frac{r^2  \dr \underline{u} }{\langle\underline{u} \rangle^4}  \bigg|^{\frac{1}{2}} \dr u \lesssim \langle t \rangle^{-\frac{1}{2}} \sup_{ \R} \big| \mathcal{F}^K [H]_t^n \big|^{\frac{1}{2}}.
\end{align*}
We finally get the result by combining the previous inequalities.
\end{proof}

Then, we set $\overline{G }^{\kappa,\gamma}\!:= \overline{G}^{\kappa \gamma}$ and, for $\beta$ such that $\partial_{x^k} Z^\xi =Z^\beta$, $\overline{G}^{\xi,k}\!:= \overline{G}^\beta$. Since uniqueness holds for the asymptotic Cauchy problem for the Maxwell equations, the properties $1.$ and $2.$ hold, which concludes the proof of Theorem \ref{Theo1}.

\appendix

\section{Elliptic estimates}\label{secApp}

We prove here classical estimates for solutions to the Poisson equation, which are in particular relevant in order to control initial electromagnetic fields. 

\begin{Lem}\label{Leminidata}
Let $\psi : \R^3 \to \R$ be a sufficiently regular function and $\phi$ be the unique function in $H^1(\R^3)$ such that $\Delta \, \phi = \psi$. Then, 
$$ \int_{\R^3_x}  \left| \nabla \phi  (x) \right|^2 \dr x  \lesssim \int_{\R^3_x} \langle x \rangle^{\frac{5}{2}} |\psi (x)|^2 \dr x . $$
Moreover, if $ \int_{\R^3} \psi =0$, we have
$$ \int_{\R^3_x} \langle x \rangle^{\frac{5}{2}} \left| \nabla \phi  (x) \right|^2 \dr x  \lesssim \int_{\R^3_x} \langle x \rangle^{6} |\psi (x)|^2 \dr x, \qquad   \int_{\R^3_x} \langle x \rangle^{\frac{9}{2}} \left| \nabla^2 \phi  (x) \right|^2 \dr x  \lesssim \int_{\R^3_x} \langle x \rangle^6 |\psi (x)|^2 \dr x + \int_{\R^3_x} \langle x \rangle^{\frac{1}{2}} |\phi (x)|^2 \dr x .$$
as well as 
$$\sup_{x \in \R^3} \langle x \rangle^3  \left| \nabla \phi (x) \right| \lesssim \sup_{y \in \R^3} \, \langle y \rangle^{5} |\psi (y) |.$$
\end{Lem}
\begin{Rq}
We refer to the proof of Proposition \ref{Proinidata} in order to deal with the case $ \int_{\R^3} \psi =Q \neq 0$.
\end{Rq}
\begin{proof}
Using the Green function of the Laplacian, we have, for almost all $x \in \R^3$,
\begin{equation}\label{Green}
 4\pi \nabla \phi (x)  = \int_{\R^3_y} \frac{x-y}{|x-y|^3} \psi (y) \dr y = \int_{|x-y| \leq \frac{|x|}{2}}\frac{x-y}{|x-y|^3} \psi (y) \dr y+ \int_{|x-y| \geq \frac{|x|}{2}}\frac{x-y}{|x-y|^3} \psi (y) \dr y  .
 \end{equation}
 We start by the unweighted $L^2$ estimate, for which we use Young's convolution inequality. For this, we need to be careful since the Kernel does not belong to any $L^p$ space. We proceed as follows,
\begin{align*}
\int_{\R^3_x} \bigg| \int_{\R^3_y} \frac{x-y}{|x-y|^3} \psi (y) \dr y \bigg|^2 \dr x & \leq \int_{\R^3_x} \bigg| \int_{\R^3_y} \frac{\mathds{1}_{|x-y| \leq 1}}{|x-y|^2} |\psi (y)| \dr y \bigg|^2 \dr x +\int_{|x| \leq 1} \bigg| \int_{\R^3_y} \frac{\mathds{1}_{|x-y| \geq 1}}{|x-y|^2} |\psi (y)| \dr y \bigg|^2 \dr x  \\
 & \leq \bigg| \int_{\R^3_z} \! \mathds{1}_{|z| \leq 1}\frac{\dr z}{|z|^2} \| \psi \|_{L^2(\R^3)} \bigg|^2+ \bigg|\int_{\R^3_z} \! \mathds{1}_{|z| \geq 1}\frac{\dr z}{|z|^{\frac{10}{3}}} \bigg|^{\frac{6}{5}}\|\psi\|^{2}_{L^{10/9}(\R^3_x)} \! \lesssim  \| \psi \|_{L^2}^2+\|\psi\|^{2}_{L^{10/9}}.
 \end{align*}
Then, we apply Hölder inequality in order to get $\|\psi\|_{L^{10/9}} \lesssim \| \langle \cdot \rangle^{\frac{5}{4}} \psi \|_{L^2}$. We now focus on the weighted $L^\infty$ estimate and we treat first the region $\{ |x| \leq 1 \}$. We have
$$ \forall \, |x| \leq 1, \qquad  |\nabla \phi|(x) \leq  \int_{\R^3_y} \frac{\dr y}{\langle y \rangle^2 |x-y|^2} \dr y \, \sup_{z \in \R^3} \, \langle z \rangle^2 |\psi (z)| \lesssim \sup_{z \in \R^3} \, \langle z \rangle^2 |\psi (z)|.$$
Fix now $|x| \geq 1$ and use \eqref{Green} in order to decompose $\nabla \phi(x)$ into three parts,
$$
 \nabla \phi (x)  = \int_{|x-y| \leq \frac{|x|}{2}}\frac{x-y}{|x-y|^3} \psi (y) \dr y+ \frac{x}{|x|^3}\int_{|x-y| \geq \frac{| x |}{2}} \psi (y) \dr y+\int_{|x-y| \geq \frac{| x |}{2}}\frac{(x-y)|x|^3-x|x-y|^3}{|x|^3|x-y|^3} \psi (y) \dr y.
$$
Then, as $|x-y| \leq |x|/2$ implies $\langle y \rangle \geq |x|/2$, we have
$$ \bigg|\int_{|x-y| \leq \frac{|x|}{2}}\frac{x-y}{|x-y|^3} \psi (y) \dr y \bigg| \lesssim \frac{1}{|x|^4} \int_{|x-y| \leq \frac{|x|}{2}}\frac{1}{|x-y|^2} \langle y \rangle^{4}\left|   \psi (y) \right|  \dr y \lesssim \frac{1}{|x|^3} \sup_{y \in \R^3} \, \langle y \rangle^{4} |\psi (y) |  .$$
Next, since the average of $\psi$ vanishes by assumption, we have 
$$  \bigg| \frac{x}{|x|^3}\int_{|x-y| \geq \frac{|x|}{2}} \psi (y) \dr y \bigg| = \bigg| \frac{x}{|x|^3}\int_{|x-y| \leq \frac{|x|}{2}} \psi (y) \dr y \bigg| \lesssim \frac{1}{| x |^6}  \int_{|x-y| \leq \frac{|x|}{2}} \langle y \rangle^{4} |\psi (y) | \dr y \lesssim \frac{1}{|x|^3} \sup_{z \in \R^3} \, \langle z \rangle^{4} |\psi (z) | .$$
For the last integral, note first that for $a =x-y$,
\begin{equation}\label{eq:techLapla}
|a|x|^3-x|a|^3|=|(a-x)|x|^3+x(|x|^3-|a|^3)| \lesssim |a-x||x|(|a|^2+|x|^2)=|y||x|(|x-y|^2+|x|^2).
\end{equation} 
As $|y-x| \geq |x|/2$ on the domain of integration, we get
$$
\int_{|x-y| \geq \frac{| x |}{2}}\frac{|(x-y)|x|^3-x|x-y|^3|}{|x|^3|x-y|^3} \big|\psi (y) \big| \dr y  \lesssim \int_{|x-y| \geq \frac{|x|}{2}}  \frac{|y||\psi (y)|}{|x|^2|x-y|}   \dr y  \lesssim   \frac{1}{|x|^3}\int_{|x-y| \geq \frac{|x|}{2}}  \frac{\dr y}{\langle y \rangle^4 }   \sup_{z \in \R^3} \, \langle z \rangle^{5} |\psi (z) |.$$
Since the integral on the right hand side has a finite value, this concludes the proof of the $L^\infty$ estimate. For the weighted $L^2$ estimate, we consider again these three terms. It suffices to deal with the region $\{|x| \geq 1\}$ since we already treated the complement one. By decomposing the Kernel as $1/|z|^2 \mathds{1}_{|z| \leq 1}+1/|z|^2 \mathds{1}_{|z| > 1}$ and by applying Young's convolution inequality
\begin{align*}
&\bigg| \int_{|x| \geq 1} \langle x \rangle^{\frac{5}{2}} \bigg| \int_{|x-y| \leq \frac{\langle x \rangle}{2}}\frac{x-y}{|x-y|^3} \psi (y) \dr y \bigg|^2 \dr x \bigg|^{\frac{1}{2}}  \lesssim \bigg| \int_{\R^3_x} \bigg| \int_{|x-y| \leq \frac{|x|}{2}}\frac{1}{|x-y|^2} \langle y \rangle^{\frac{5}{4}} \left|   \psi (y) \right|  \dr y \bigg|^2 \dr x \bigg|^{\frac{1}{2}} \\
& \qquad \qquad  \leq \int_{\R^3_z} \mathds{1}_{|z| \leq 1} \frac{\dr z}{|z|^2}  \big\| \langle \cdot \rangle^{\frac{5}{4}} \, \psi \big\|_{L^2(\R^3)}+\bigg| \int_{\R^3_z} \mathds{1}_{|z| \geq 1} \frac{\dr z}{|z|^4} \bigg|^{\frac{1}{2}}  \big\| \langle \cdot \rangle^{\frac{5}{4}} \, \psi \big\|_{L^1(\R^3)} \lesssim \big\| \langle \cdot \rangle^{\frac{5}{4}} \, \psi \big\|_{L^2(\R^3)}+\big\| \langle \cdot \rangle^{\frac{5}{4}} \, \psi \big\|_{L^1(\R^3)}.
\end{align*}
For the second term, we use again that the average of $\psi$ vanishes and then that $|y| \geq |x| / 2$ on the domain of integration in order to get
$$ \int_{|x| \geq 1} \langle x \rangle^{\frac{5}{2}} \bigg| \frac{x}{|x|^3}\int_{|x-y| \geq \frac{| x |}{2}} \psi (y) \dr y \bigg|^2 \dr x \lesssim \int_{|x| \geq 1} \frac{1}{|x|^{\frac{7}{2}}} \bigg| \int_{|x-y| \leq \frac{|x|}{2}} \langle y \rangle \left|   \psi (y) \right|  \dr y \bigg|^2 \dr x \lesssim \big\| \langle \cdot \rangle \, \psi \big\|^2_{L^1(\R^3)}  .$$
Finally, for the last term, we use \eqref{eq:techLapla} and then that $|x-y| \gtrsim |x|$ on the domain of integration. It yields
$$
 \int_{|x| \geq 1} \langle x \rangle^{\frac{5}{2}} \bigg|\int_{|x-y| \geq \frac{| x |}{2}}\frac{(x-y)|x|^3-x|x-y|^3}{|x|^3|x-y|^3} \psi (y) \dr y \bigg|^2 \dr x   \lesssim  \int_{|x| \geq 1} \frac{1}{|x|^{\frac{7}{2}}} \bigg| \int_{|x-y| \geq \frac{|x|}{2}}\langle y \rangle  |\psi (y)|   \dr y \bigg|^2 \dr x  \lesssim \big\| \langle \cdot \rangle \, \psi \big\|^2_{L^1} $$
and the result ensues from $ \big\| \langle \cdot \rangle^{5/4} \, \psi \big\|_{L^1(\R^3)}\lesssim  \big\| \langle \cdot \rangle^3 \, \psi \big\|^2_{L^2(\R^3)}$. For the second order derivatives, we remark that
$$ \Delta \big( \langle x \rangle^{\frac{9}{4}} \phi \big)= \langle x \rangle^{\frac{9}{4}} \psi+\frac{9}{2} \langle x \rangle^{\frac{1}{4}} x^i\partial_{x^i} \phi + \frac{27}{4}  \langle x \rangle^{\frac{1}{4}} \phi +\frac{9}{16}|x|^2 \langle x \rangle^{-\frac{7}{4}} \phi.$$
By elliptic regularity, we then get
$$ \big\| \langle \cdot \rangle^{\frac{9}{4}} \nabla^2 \phi \big\|_{L^2_x} \lesssim \big\| \nabla^2 \big(\langle \cdot \rangle^{\frac{9}{4}} \phi \big) \big\|_{L^2_x}+ \big\| \langle \cdot \rangle^{\frac{5}{4}} \nabla \phi \big\|_{L^2_x}+ \big\| \langle \cdot \rangle^{\frac{1}{4}} \phi \big\|_{L^2_x} \lesssim \big\| \langle \cdot \rangle^{\frac{9}{4}} \psi \big\|_{L^2_x}+\big\| \langle \cdot \rangle^{\frac{5}{4}} \nabla \phi \big\|_{L^2_x}+ \big\| \langle \cdot \rangle^{\frac{1}{4}} \phi \big\|_{L^2_x}.$$
\end{proof} 

We are now able to state higher order statements. For this, it will be convenient to exploit the following ordered set of vector fields,
$$
\mathbb{Q} := \{ \partial_{x^k}, \, \Omega_{ij}, \, S_x \, | \, 1 \leq k \leq 3, \, 1 \leq i, \, j \leq 3 \}=\{ \Gamma^\ell \, | \, 1 \leq \ell \leq 7\}, \qquad S_x:= x^\ell \partial_{x^\ell}.
$$
which verify $\mathbb{Q} \subset \mathbb{K}$ on $\{t=0 \}$ as well as the following commutation relations,
\begin{equation}\label{comudelta}
 \forall \, \Gamma \in \mathbb{Q} \setminus {S_x}, \quad [\Delta, \Gamma]=0,  \qquad \qquad [\Delta,S_x]=2\Delta.
 \end{equation}
If $\gamma \in \llbracket 1 , 7 \rrbracket^p$, we define the operator $\Gamma^\gamma \in \mathbb{Q}^{|\gamma|} $ as $\Gamma^{\gamma_1} \dots \Gamma^{\gamma_p}$. Note that these weighted derivatives capture decay in $x$. Indeed, in view of the relations $|x|^2\partial_{x^j}=x^jS_x+x^i\Omega_{ij}$, for any sufficiently regular function $\Phi$,
\begin{equation}\label{eq:equinormDelta}
 \sup_{|\gamma| \leq N} |\nabla \Gamma^\gamma \Phi|(x) \lesssim \sup_{|\kappa| \leq N} \langle x \rangle^{|\kappa|}| \nabla \partial_x^\kappa \Phi|(x) \lesssim \sup_{|\gamma| \leq N} |\nabla \Gamma^\gamma \Phi|(x).
 \end{equation}
 \begin{Lem}\label{Lemdatahigh0}
Let $N \geq 0$, $t_0 \geq 0$, $f :  [t_0,+ \infty[ \times \R^3_x \times \R^3_v \to \R$ be a sufficiently regular function and $\phi$ be the unique $H^1(\R^3)$ solution to $\Delta \, \phi = \int_v f(t_0,\cdot,v)\dr v$. Then, 
$$ \forall \, |\kappa| \leq N, \qquad  \int_{\R^3_x} \langle x \rangle^{2|\kappa|} \left|\nabla \partial_x^\kappa \phi(x) \right|^2   \dr x \lesssim_{t_0} \sup_{|\beta| \leq N} \int_{\R^3_x} \int_{\R^3_v} \langle x \rangle^{\frac{5}{2}} \langle v \rangle^4 \big| \widehat{Z}^\beta f (t_0,x,v) \big|^2 \dr v \dr x. $$
\end{Lem}
\begin{proof}
Apply first Lemma \ref{Leminidata} to $\Gamma^\gamma \phi$, for any $|\gamma| \leq N$, together with the commutation relations \eqref{comudelta}. Then, we use \eqref{eq:equinormDelta} for $p=1$ and we remark that, by integration by parts in $v$ and since $S_x=S-t_0 \partial_t$,
$$ \int_{\R^3_x} \int_{\R^3_v} \langle x \rangle^{\frac{5}{2}} \bigg| \int_{\R^3_v} \Gamma^\xi f(t_0,x,v) \dr v \bigg|^2 \dr x \lesssim \langle t_0 \rangle^N \sup_{|\beta| \leq N} \int_{\R^3_x} \int_{\R^3_v} \langle x \rangle^{\frac{5}{2}} \bigg| \int_{\R^3_v} \widehat{Z}^\beta f(t_0,x,v) \dr v \bigg|^2 \dr x.$$
It remains to use the Cauchy-Schwarz inequality in $v$.
\end{proof}
\begin{Lem}\label{Lemdatahigh}
Let $N \geq 0$, $\psi : \R^3 \to \R$ be a sufficiently regular function supported in $\{x \in \R^3 \, | \, |x| \leq 3\}$ and such that $\int_{\R^3} \psi=0$. Let further $\phi$ be the unique $H^1(\R^3)$ solution to $\Delta \, \phi = \psi$. Then, for any $|\kappa| \leq N$,
$$ \forall \, x \in \R^3, \qquad  \langle x \rangle^{3+|\kappa|} \left|\nabla \partial_x^\kappa \phi(x) \right| \lesssim \big\| \psi \big\|_{W^{N,\infty}(\R^3)}.  $$
If $N \geq 1$, we have, for any $|\kappa| \leq N+1$,
$$  \int_{\R^3_x} \langle x \rangle^{\frac{5}{2}+2|\kappa|} \left|\nabla \partial_x^\kappa \phi(x) \right|^2   \dr x \lesssim \big\| \psi \big\|_{H^{N}(\R^3)}. $$
\end{Lem}
\begin{proof}
By integration by parts and an induction, we have
$$ \forall \, |\gamma| \leq N, \qquad  \int_{\R^3} \Gamma^\gamma \psi =0.$$
Apply then Lemma \ref{Leminidata} to $\Gamma^\gamma \phi$, for any $|\gamma| \leq N$, together with the commutation relations \eqref{comudelta}. Since $\psi$ is compactly supported, one obtains both the $L^\infty_x$ and the $L^2_x$ estimates, up to $|\kappa| \leq N$, by using \eqref{eq:equinormDelta}. For the record, we in particular have
$$ \forall \, |\beta| \leq N, \qquad \big\| \langle \cdot \rangle^{\frac{5}{4}} \Gamma^\beta \phi \big\|_{L^2(\R^3_x)} \lesssim \big\| \psi \big\|_{H^{N}(\R^3)}. $$
 Assume now that $N \geq 1$ and remark that
$$ \forall \, |\kappa|=N+1, \qquad  \langle x \rangle^{|\kappa|} \left|\nabla \partial_x^\kappa \phi(x) \right|\lesssim \sup_{|\gamma| \leq N+1}  \left| \nabla \Gamma^\gamma \phi(x) \right| \lesssim \sup_{|\gamma| = N} \langle x \rangle \left| \nabla^2 \Gamma^\gamma \phi(x) \right| + \sup_{|\beta| \leq N} \left|\nabla \Gamma^\beta \phi(x) \right| .$$ 
We have already proved the required weighted $L^2$ estimate for the second term on the right hand side. Consider then $|\gamma| =N$, $|\xi|=N-1$ and $\Gamma \in \mathbb{Q}$ such that $\Gamma^\gamma=\Gamma \Gamma^\xi$. As $\Gamma$ is either a translation $\partial_{x^i}$ or a homogeneous vector field, we have
$$ \langle x \rangle^{\frac{1}{4}} \left| \Gamma^\gamma \phi(x) \right| \leq \langle x \rangle^{\frac{5}{4}} \left|\nabla \Gamma^\xi \phi(x) \right|,$$
so that we can control $\langle \cdot \rangle^{\frac{9}{4}} \left| \nabla^2 \Gamma^\gamma \phi \right|$ in $L^2$ by applying once again from Lemma \ref{Leminidata} and \eqref{comudelta}.
\end{proof}

\renewcommand{\refname}{References}
\bibliographystyle{abbrv}
\bibliography{biblio}

\end{document}